\documentclass[12pt]{article}
\usepackage{amsmath,amsthm,amssymb,bm,mathrsfs}
\usepackage{graphicx}
\usepackage[dvipsnames]{xcolor}

\usepackage{algorithmic} 
\usepackage{algorithm}
\usepackage{amsmath, amssymb} 
\usepackage{bm}    
\usepackage{caption,enumerate,listings}
\usepackage{setspace}
\usepackage[OT1]{fontenc}
\usepackage[colorlinks=true,
            linkcolor=blue,
            citecolor=blue,
            urlcolor=blue]{hyperref}
\usepackage{pifont}
\usepackage[margin=1in]{geometry}
\usepackage[protrusion=false,expansion=true]{microtype}
\usepackage[title]{appendix}
\usepackage{bbm}
\usepackage[authoryear,round]{natbib}
\usepackage{amsmath}
\usepackage{bibunits}
\usepackage{booktabs}
\makeatletter
\renewcommand\section{\@startsection{section}{1}{\z@}%
  {-1.5ex \@plus -0.5ex \@minus -.2ex}
  {0.3ex \@plus 0.2ex \@minus .2ex}
  {\normalfont\Large\bfseries}}
\makeatother
\usepackage{extpfeil}
\usepackage{subcaption}

\usepackage{comment}
\usepackage{enumitem}
\usepackage{multirow}
\def\spacingset#1{\renewcommand{\baselinestretch}%
{#1}\small\normalsize} \spacingset{1}
\newcommand{\argmin}{\mathop{\mathrm{argmin}}}

\newtheorem{lemma}{{\bf Lemma}}
\newtheorem{corollary}{{\bf Corollary}}
\newtheorem{theorem}{{\bf Theorem}}
\newtheorem{assumption}{{\bf Assumption}}
\newtheorem{definition}{{\bf Definition}}
\newtheorem{remark}{{\bf Remark}}
\newtheorem{example}{{\bf Example}}
\usepackage{setspace}
\usepackage{booktabs}
\theoremstyle{plain}
\usepackage{titlesec}
\makeatletter
\g@addto@macro\normalsize{%
  \setlength{\abovedisplayskip}{7pt}%
  \setlength{\belowdisplayskip}{7pt}%
  \setlength{\abovedisplayshortskip}{7pt}%
  \setlength{\belowdisplayshortskip}{7pt}%
}
\makeatother

\titleformat{\section}
  {\Large\bfseries}  
  {\thesection}
  {0.7em}
  {}

\titleformat{\subsection}
  {\large\bfseries}
  {\thesubsection}
  {0.5em}
  {}
\begin{document}

\title{\Large \bf Do More Predictions Improve Statistical Inference?
Filtered Prediction-Powered Inference}

\author{
Shirong Xu\thanks{Department of Statistics and Data Science, Xiamen University, China}
 \and 
Will Wei Sun\thanks{Mitchell E. Daniels, Jr. School of Business, Purdue University}
}
\date{ }
\maketitle

\begin{abstract}
Recent advances in artificial intelligence have enabled the generation of large-scale, low-cost predictions with increasingly high fidelity. As a result, the primary challenge in statistical inference has shifted from data scarcity to data reliability. Prediction-powered inference methods seek to exploit such predictions to improve efficiency when labeled data are limited. However, existing approaches implicitly adopt a use-all philosophy, under which incorporating more predictions is presumed to improve inference. When prediction quality is heterogeneous, this assumption can fail, and indiscriminate use of unlabeled data may dilute informative signals and degrade inferential accuracy. In this paper, we propose Filtered Prediction-Powered Inference (FPPI), a framework that selectively incorporates predictions by identifying a data-adaptive filtered region in which predictions are informative for inference. We show that this region can be consistently estimated under a margin condition, achieving fast rates of convergence. By restricting the prediction-powered correction to the estimated filtered region, FPPI adaptively mitigates the impact of biased or noisy predictions. We establish that FPPI attains strictly improved asymptotic efficiency compared with existing prediction-powered inference methods. Numerical studies and a real-data application to large language model evaluation demonstrate that FPPI substantially reduces reliance on expensive labels by selectively leveraging reliable predictions, yielding accurate inference even in the presence of heterogeneous prediction quality.

\end{abstract}


\begin{bibunit}[apalike]
\newpage
\spacingset{1.7}
\section{Introduction}

Modern artificial intelligence (AI), particularly large language models (LLMs) and foundation models, can generate vast amounts of high-quality data or predictions across diverse domains at minimal cost \citep{ji2025predictions,zhou2026detecting}. This capability creates a new opportunity for statistical inference: when labeled data are scarce, pretrained AI models can produce cheap outcome predictions on abundant unlabeled data, which may be used to augment limited real-world observations and improve estimation efficiency. However, this opportunity comes with a fundamental challenge. Model-generated predictions are not ground-truth observations, and treating them as such can induce bias and invalidate inference. Conversely, restricting analysis to labeled data alone preserves validity but often yields estimators with high variance and low statistical power.

Prediction-powered inference (PPI; \citealt{angelopoulos2023prediction}) was proposed to reconcile this tension. The central idea of PPI is to treat model predictions as auxiliary information: labeled data are used to estimate residuals and correct potential biases. By combining high-quality predictions with real labeled data, PPI can enhance estimator efficiency, reduce variance, and provide more accurate uncertainty quantification, even when labeled data are scarce. However, a critical limitation of PPI is its heavy reliance on the predictive model’s quality \citep{angelopoulos2023PPI++,fisch2024stratified}. When predictions are inaccurate, PPI may fail to deliver improvements and, in some cases, can even degrade inference performance. To address this limitation, PPI\texttt{++} \citep{angelopoulos2023PPI++} was developed as an enhanced framework that incorporates a tuning parameter to regulate the influence of predictions, thereby guaranteeing asymptotic improvement regardless of the predictive model’s accuracy. Despite this theoretical guarantee, empirical improvements from PPI\texttt{++} are often small or negligible in practice \citep{mani2025no}.


A central reason for this limitation is that existing PPI and PPI\texttt{++} frameworks implicitly treat the predictive model as a homogeneous source of information. In modern applications, however, prediction quality is rarely uniform across the covariate space. A model may perform exceptionally well on a large subset of observations, yet fail systematically on others due to distribution shift or model misspecification \citep{tian2023transfer,amortila2024mitigating,Chen05012026}. When predictions of heterogeneous quality are aggregated into a single global correction term, informative signals and misleading noise can offset one another. As a result, incorporating more predictions may paradoxically worsen finite-sample inference, even when high-quality predictions are present. This observation motivates a fundamental question:
\begin{center}
\textit{How to selectively leverage predictions of heterogeneous quality in order to improve the efficiency of prediction-powered inference?}
\end{center}
To answer this question, we first present a simple motivating example that reveals a structural limitation of the PPI\texttt{++} framework and illustrates the gains from our filtering procedure.

\subsection{Motivation Example}
\label{SubSec:Motivation}
We begin with a motivating example (Example \ref{Exam:PPIF_new}) illustrating a scenario in which PPI\texttt{++} fails to improve statistical inference compared with the classical estimator (i.e., the case of $\lambda=0$). The example corresponds to the well-studied mean estimation problem under the PPI\texttt{++} framework \citep{angelopoulos2023prediction,angelopoulos2023PPI++,mani2025no}. 


\begin{example}[\textbf{PPI\texttt{++} Fails}]
\label{Exam:PPIF_new}
Suppose the labeled dataset $\mathcal{D}_L=\{(x_i,y_i)\}_{i=1}^n$ follows the model $Y = X + \varepsilon$, where $X\sim \mathcal N(\theta^\star,1)$ and $\varepsilon\sim \mathcal N(0,\sigma^2)$ are independent, and $\theta^\star =1$ is the parameter of interest. We further assume that the prediction model is given by $f(x)=(x-\theta^\star)^2-1$, and that the unlabeled data $\{\widetilde x_j\}_{j=1}^N$ follow the same marginal distribution as $X$, i.e., $\widetilde x_j\sim \mathcal N(\theta^\star,1)$.  
Under this setting, the PPI\texttt{++} estimator for $\theta^\star$ takes the form
\begin{align*}
\widehat{\theta}_{+}(\lambda)
= \frac{1}{n}\sum_{i=1}^n y_i
+ \lambda\left(
\frac{1}{N}\sum_{j=1}^N f(\widetilde x_j)
- \frac{1}{n}\sum_{i=1}^n f(x_i)
\right).
\end{align*}
The variance of $\widehat{\theta}_{+}(\lambda)$ is $\mathrm{Var}(\widehat{\theta}_{+}(\lambda))
= \frac{1+\sigma^2}{n} + 2\lambda^2 \left(\frac{1}{n} + \frac{1}{N}\right)$, which is minimized at $\lambda = 0$. Hence, PPI\texttt{++} offers no variance improvement over the labeled-only estimator $\frac{1}{n}\sum_{i=1}^n y_i$.
\end{example}

The proof of Example \ref{Exam:PPIF_new} is provided in the Section A.2 of the supplementary file. In Example \ref{Exam:PPIF_new}, it is evident that the PPI\texttt{++} estimator attains its minimal variance at $\lambda = 0$, at which point it reduces to the classical estimator $\frac{1}{n}\sum_{i=1}^n y_i$, relying solely on the labeled dataset $\mathcal{D}_L$. This behavior arises because the covariance between $f(X)$ and $Y$ (or $X$) is globally zero, which is crucial for establishing the improvement of the PPI\texttt{++} estimator over the classical approach \citep{angelopoulos2023PPI++,miao2025assumption}. However, note that when $Y = X + \varepsilon$ and $f(X) = (X - \theta^\star)^2 - 1$, the prediction model $f(X)$ indeed carries information about $X$, which can be fully leveraged to improve estimation efficiency. In other words, while some of the data provided by $f(X)$ is informative, other parts can be misleading, and the useful and misleading information can effectively cancel each other out.

We consider the following filtered-type estimator under the same setting of Example \ref{Exam:PPIF_new}:
\begin{align}
\label{FPPI}
\widehat{\theta}(\lambda,t)
= \frac{1}{n}\sum_{i=1}^n y_i
+ \lambda \Bigg(
\frac{1}{N}\sum_{j=1}^N f(\widetilde{x}_j)\cdot \mathbf{1}(\widetilde{x}_j > t)
- \frac{1}{n}\sum_{i=1}^n f(x_i) \cdot \mathbf{1}(x_i > t)
\Bigg),
\end{align}
where $\mathbf{1}(x_i > t)$ is the indicator function, equal to 1 if $x_i > t$ and 0 otherwise. Here, the form of indicator function is included here primarily to illustrate the basic idea and can take other forms for filtering. It is worth noting that, since the filtering procedure is applied to both the labeled and unlabeled covariates, the estimator $\widehat{\theta}(\lambda,t)$ remains unbiased for $\theta^\star$ for any $t \in \mathbb{R}$, that is $\mathbb{E}(\widehat{\theta}(\lambda,t))=\theta^\star$ for any $t \in \mathbb{R}$. 

To demonstrate the effectiveness of $\widehat{\theta}(\lambda,t)$, we focus on the specific filtering threshold $t=1$ and derive the variance of $\widehat{\theta}(\lambda,1)$ in Example \ref{Exam:FPPI}. Here, we emphasize that the threshold $t=1$ is chosen merely for illustrative purposes; in practice, selecting an alternative threshold could result in substantially larger improvement.

\begin{example}
 \label{Exam:FPPI}   
Under the same setting in Example \ref{Exam:PPIF_new}, the variance of the proposed filtered-type estimator in \eqref{FPPI} with $t=1$ is given by $\textnormal{Var}(\widehat{\theta}(\lambda,1))
= \frac{1+\sigma^2}{n}
+ \lambda^2 \left(\frac{1}{n}+\frac{1}{N}\right)
- \frac{\lambda}{n}\sqrt{\frac{2}{\pi}}$. Clearly, $\textnormal{Var}(\widehat{\theta}(\lambda,1))$ is minimized at $\lambda^\star= \frac{1}{\sqrt{2\pi}\left(1+\frac{n}{N}\right)}$, which yields the minimal variance
\begin{align*}
\textnormal{Var} \left(\widehat{\theta}(\lambda^\star,1)\right)
= \frac{1+\sigma^2}{n}
- \underbrace{\frac{1}{2\pi n\left(1+\frac{n}{N}\right)}}_{\textnormal{variance reduction}}
< \frac{1+\sigma^2}{n}
= \textnormal{Var} \left(\frac{1}{n}\sum_{i=1}^n y_i\right).
\end{align*}
\end{example}

As shown in Example \ref{Exam:FPPI}, for the proposed filtered-type estimator, filtering out the observations $\{x_i\}_{i=1}^n$ and $\{\widetilde{x}_j\}_{j=1}^N$ that are smaller than $t=1$ leads to a substantial variance reduction relative to the PPI\texttt{++} estimator. This result indicates that, when a prediction model is used to assign labels, even if it is only globally weakly correlated with the response $Y$, selectively utilizing an informative subset of the data can still substantially improve estimation and inference. Figure~\ref{fig:Example1} illustrates the comparison between the two methods.


\begin{figure}[h]
    \centering
                \begin{subfigure}[b]{0.325\textwidth}
        \centering
        \includegraphics[width=\textwidth]{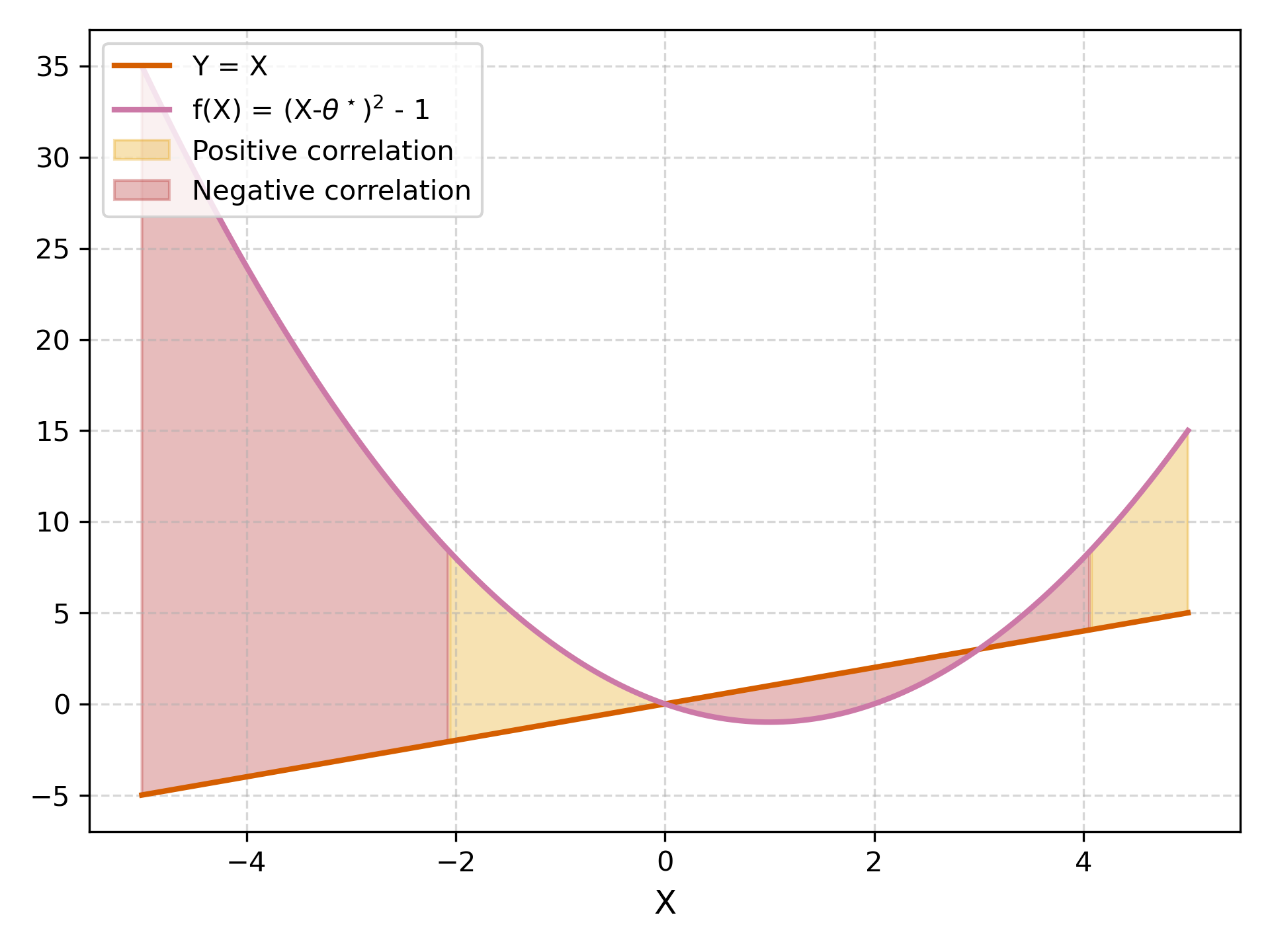}
    \end{subfigure}
            \begin{subfigure}[b]{0.325\textwidth}
        \centering
        \includegraphics[width=\textwidth]{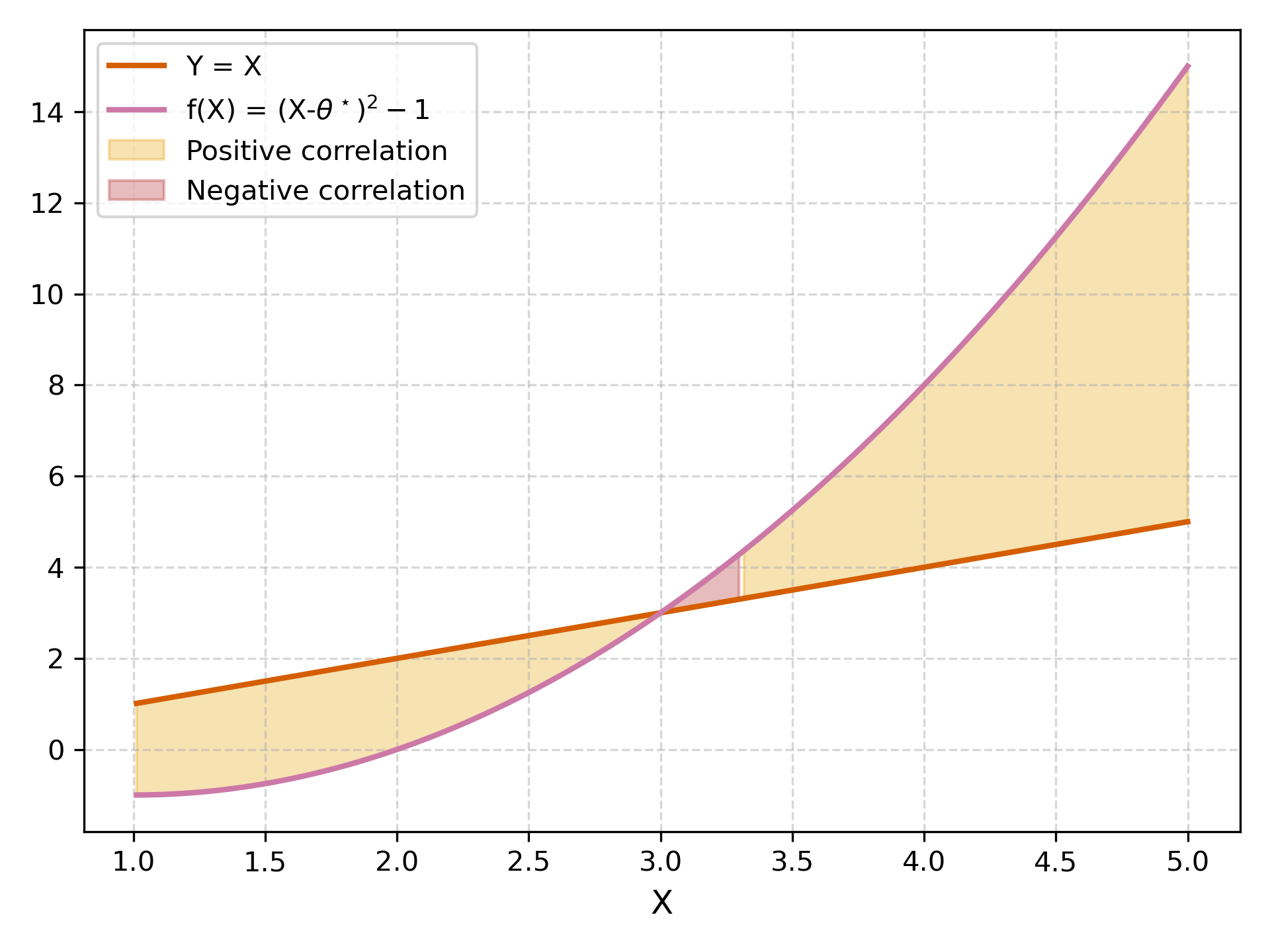}
    \end{subfigure}
        \begin{subfigure}[b]{0.325\textwidth}
        \centering
        \includegraphics[width=\textwidth]{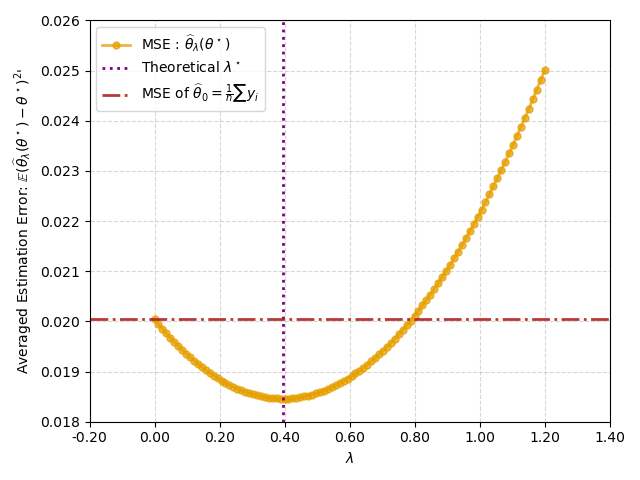}
    \end{subfigure}
    \caption{The left panel highlights regions in Example \ref{Exam:PPIF_new} where $Y$ and $f(X)$ are positively or negatively correlated. The middle panel depicts the filtered regions where $Y$ and $f(X)$ exhibit highly positive correlation. The right panel shows the estimation error of $\widehat{\theta}(\lambda,1)$ in Example \ref{Exam:FPPI} across varying values of $\lambda$.}
    \label{fig:Example1}
\end{figure}

To conclude, Examples \ref{Exam:PPIF_new} and \ref{Exam:FPPI} illustrate both the limitations of the original PPI\texttt{++} framework and the potential advantages of a filtered-type estimator. Specifically, in mean estimation, when the global covariance between the prediction model $f(X)$ and the response $Y$ is zero, the PPI\texttt{++} estimator fails to improve upon the classical estimator. In contrast, the results for the filtered-type estimator (Figure \ref{fig:Example1}) show that by selectively utilizing only the informative subset of the data, via a simple filtering procedure, one can achieve a substantial reduction in variance while maintaining unbiasedness. These findings underscore the importance of incorporating local information in semi-supervised estimation. 

\subsection{Our Contributions and Related Work}
Based on the motivation in Section \ref{SubSec:Motivation}, we propose a novel Filtered Prediction-Powered Inference (FPPI) framework to improve the efficiency of statistical inference as illustrated in Figure \ref{fig:Framework}. The key idea is that, for any given predictive model, not all predictions contribute equally to parameter estimation, and incorporating all predictions indiscriminately may lead to only modest improvements due to conflicting or noisy effects. To address this, our method selectively filters out predictions that are likely to introduce substantial bias, thereby focusing estimation on the most informative subset of the data and improving overall model fitting performance.
\begin{figure}[h]
    \centering
    \includegraphics[scale=0.16]{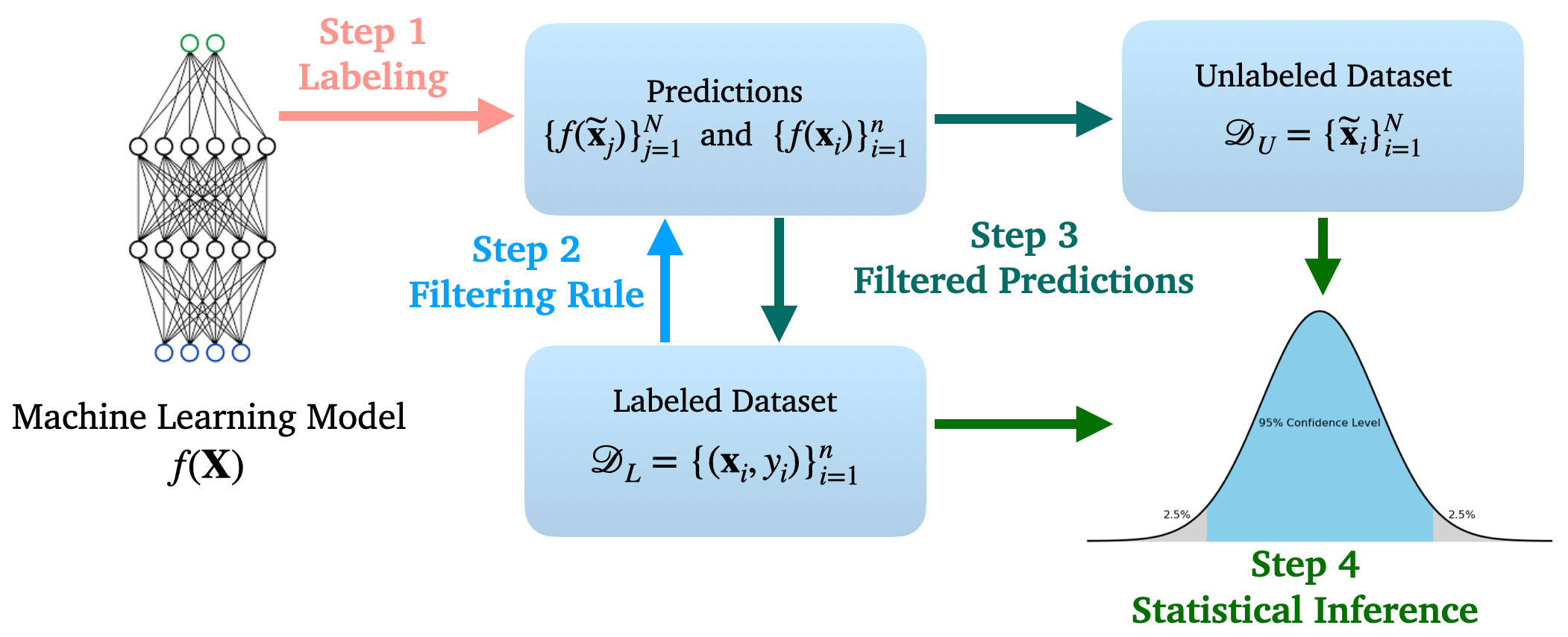}
    \caption{The general procedure of the proposed FPPI framework.}
    \label{fig:Framework}
\end{figure}
Our contributions are summarized as follows:
\begin{itemize}
    \item \textbf{Selective filtering:} The proposed FPPI identifies and retains only the predictions that are informative for estimation and inference. This approach not only improves the accuracy of the estimator but also mitigates the negative impact of heterogeneous prediction quality.
    
    \item \textbf{Margin condition for distribution characterization:} We introduce a margin condition to rigorously characterize the distribution of covariates relative to the boundary depending on the true function and the fitted model. This condition provides a precise measure of how accurately the filtered region can be recovered, allowing for potentially fast convergence rates.
    
    \item \textbf{Fast and theoretically guaranteed region estimation:} Leveraging the margin condition, we show that the informative region can be estimated efficiently. In particular, when the covariate data distribution is sufficiently separated from the boundary, the convergence rate of the estimated region can be extremely fast, and under certain conditions, it can even achieve an exponential rate. This ensures that the filtering procedure is both statistically reliable and practically feasible.
    
    \item \textbf{Enhanced statistical efficiency:} By concentrating estimation on the selected region, our method substantially improves the asymptotic efficiency of parameter estimation compared with standard PPI\texttt{++} approaches. Through the use of a leave-one-out technique, the proposed procedure leverages the labeled dataset for both region estimation and parameter estimation simultaneously, thereby avoiding the need for data splitting. In addition, our method delivers notable gains in finite-sample performance especially when predictions are heterogeneous and noisy, as demonstrated by extensive simulations and a real-data application of the LLM evaluation.
\end{itemize}

\textbf{Related work.} Our work builds on the framework of prediction-powered inference \citep{angelopoulos2023prediction, angelopoulos2023PPI++, zrnic2024cross, gronsbell2024another} and is conceptually related to active statistical inference \citep{zrnic2024active, li2025robust}. However, existing active inference methods primarily focus on selective sampling of labeled data, rather than addressing heterogeneous prediction quality from a predictive model. In particular, \citet{zrnic2024active} propose a sampling rule that leverages a prediction model to assess label quality and improve inference under a fixed sampling budget, with a more robust extension developed in \citet{li2025robust}. In addition, \citet{miao2025assumption} generalizes PPI\texttt{++} by introducing a multivariate weighting scheme that ensures element-wise variance reduction in multi-dimensional estimation problems, while a more general matrix-form weighting scheme is developed in \citet{gan2024prediction}. These approaches are unified in \citet{ji2025predictions}, which generalizes PPI by formulating it in terms of a recalibrated imputed loss. A key distinction of our work from the existing literature is that our proposed method handles heterogeneous prediction quality through sample-level filtering rather than globally reweighting the estimation objective. From a methodological perspective, our proposed method is complementary to existing prediction-powered inference frameworks.


\textbf{Notations.} Throughout this paper, boldface letters denote multivariate quantities, while their unbolded counterparts denote univariate quantities. For example, $\bm{x} \in \mathbb{R}^p$ denotes a $p$-dimensional vector, whereas $x$ denotes a scalar. The $\ell_2$-norm of a vector $\bm{x} \in \mathbb{R}^p$ is defined as $\Vert\bm{x}\Vert_2 = \left( \sum_{i=1}^p x_i^2 \right)^{1/2}$. For a multivariate continuous random variable $\bm{X}$, $P_{\bm{X}}(\bm{x})$ denotes its probability density function evaluated at $\bm{x}$, and $\mathbb{P}_{\bm{X}}$ denotes the associated probability measure. Expectations with respect to the randomness of $\bm{X}$ are denoted by $\mathbb{E}_{\bm{X}}$. For $\mathcal{X} \subset \mathbb{R}^p$, we let $\mathcal{B}(\mathcal{X})$ denote the Borel $\sigma$-algebra on $\mathcal{X}$. For two sequences $\{a_n\}_{n\geq 1}$ and $\{b_n\}_{n\geq 1}$, we write $a_n \lesssim b_n$ if there exists a universal constant $C > 0$ such that $a_n \leq C b_n$ for $n \geq 1$.

\section{Prelinimaries}
\label{Sec:Pre}

Prediction-powered inference (PPI; \citealp{angelopoulos2023prediction}) is a semi-supervised framework that leverages predictive models to enhance valid statistical inference while quantifying predictive uncertainty. Specifically, PPI assumes access to a small labeled dataset $\mathcal{D}_L = \{(\bm x_i, y_i)\}_{i=1}^n$, where both covariates $\bm x_i \in \mathcal{X} \subset \mathbb{R}^p$ and responses $y_i \in \mathbb{R}$ are observed, together with a larger unlabeled dataset $\mathcal{D}_U = \{\widetilde{\bm x}_j\}_{j=1}^N$ that contains only covariates. In practice, it is typically assumed that $\mathcal{D}_L$ and $\mathcal{D}_U$ are independent and that $n \ll N$, since collecting features is more cost-effective than collecting labels. In addition, suppose there is a prediction model $f:\mathcal{X} \rightarrow \mathbb{R}$ that generates pseudo-labels for the unlabeled dataset $\mathcal{D}_U$. That is, the set $\{f(\widetilde{\bm x}_j)\}_{j=1}^N$ constitutes the pseudo-labels associated with $\mathcal{D}_U$.

In general, the labeled data are used to estimate residuals and quantify predictive uncertainty, while the unlabeled data are leveraged to improve the efficiency of statistical inference for the resulting estimators. A convenient way to describe this framework is via the general PPI\texttt{++} formulation \citep{angelopoulos2023PPI++}, written in an M-estimation form as
\begin{align}
\label{PPI++}
\widehat{\bm{\theta}}_{+}(\lambda)
= \argmin_{\bm{\theta}\in \mathbb{R}^p}
\Bigg\{
\frac{1}{n}\sum_{i=1}^n L_{\bm{\theta}}(\bm{x}_i,y_i)
+\lambda \Bigg(
\frac{1}{N}\sum_{j=1}^N L_{\bm{\theta}}(\widetilde{\bm{x}}_j,f(\widetilde{\bm x}_j))
- \frac{1}{n}\sum_{i=1}^n L_{\bm{\theta}}(\bm{x}_i,f(\bm x_i))
\Bigg)
\Bigg\},
\end{align}
where $L_{\bm{\theta}}(\bm{x}_i,y_i)$ denotes a general loss function. Typical examples include the squared loss for linear regression
$L_{\bm{\theta}}(\bm{x}_i,y_i)=(\bm{x}_i^\top \bm{\theta}-y_i)^2$
and the negative log-likelihood
$L_{\bm{\theta}}(\bm{x}_i,y_i)=-\log P(y_i\mid \bm x_i,\bm{\theta})$ as special cases.

The second term in (\ref{PPI++}) acts as a prediction-based correction: it uses the large unlabeled sample to approximate the population risk under model predictions, and subtracts the corresponding empirical quantity on the labeled sample to correct for prediction bias. The tuning parameter $\lambda$ controls how strongly this correction is incorporated. When $\lambda = 1$, the estimator reduces to the original PPI procedure \citep{angelopoulos2023prediction}. When $\lambda = 0$, the correction term vanishes and we recover the classical labeled-data-only estimator. Thus, PPI\texttt{++} provides a continuum between label-only inference and fully prediction-powered inference, allowing the method to adapt to the quality of the prediction model.


\section{Filtered Prediction-Powered Inference}
\label{Sec:FPPI}
In this section, we first formalize the concept of the filtered prediction-powered estimator, building on the empirical insights from Section \ref{SubSec:Motivation}. We then examine its theoretical properties and potential advantages across a broad class of estimation problems.

Following the PPI\texttt{++} framework, we assume that $\{(\bm{x}_i, y_i)\}_{i=1}^n$ are i.i.d. samples from the joint distribution of $(\bm X, Y)$, with $\bm{X} \in \mathcal{X} \subset \mathbb{R}^p$ and $Y \in \mathbb{R}$. Throughout the paper, we assume that $\mathcal{X}$ is a compact set. We further assume that $\{\widetilde{\bm x}_j\}_{j=1}^N$ are i.i.d. unlabeled covariates drawn from the same marginal distribution as $\bm X$, independent of the labeled dataset. As in PPI\texttt{++}, the prediction model $f$ is considered fixed. The objective function for obtaining the proposed filtered-type estimator is given by
\begin{align}
\label{FPPI:Framework}
\mathcal{L}_{\textnormal{FPPI}}(\bm{\theta}|\lambda,\mathcal{S}) \triangleq 
    \frac{1}{n}\sum_{i=1}^n L_{\bm{\theta}}(\bm{x}_i,y_i)+
    \lambda \cdot \left(
    \frac{1}{N}\sum_{j=1}^N Q_{\bm{\theta}}(\widetilde{\bm{x}}_j|\lambda,\mathcal{S})-
    \frac{1}{n}\sum_{i=1}^n Q_{\bm{\theta}}(\bm{x}_i|\lambda,\mathcal{S})
    \right),
\end{align}
where $Q_{\bm{\theta}}(\bm{x}_i|f,\mathcal{S}) = L_{\bm{\theta}}(\bm{x}_i,f(\bm x_i))
    \cdot \bm{1}_{\mathcal{S}}(\bm{x}_i)$, $\mathcal{S} \in \mathcal{B}(\mathcal{X})$ is a subset of $\mathcal{X}$ with positive measure, and $\bm{1}_{\mathcal{S}}(\bm{x}_i)$ denotes the indicator function that equals $1$ when $\bm{x}_i \in \mathcal{S}$. For any given subset $\mathcal{S}$ and tuning parameter $\lambda$, we define the filtered prediction-powered estimator as 
$$
\widehat{\bm{\theta}}_{\text{FPPI}}(\lambda,\mathcal{S})
= \argmin_{\bm{\theta} \in \mathbb{R}^p}
\mathcal{L}_{\textnormal{FPPI}}(\bm{\theta}\mid \lambda,\mathcal{S}).
$$ 
We note that $\mathcal{L}_{\textnormal{FPPI}}(\bm{\theta}\mid \lambda,\mathcal{S})$ is an unbiased estimator of $\mathbb{E}[L_{\bm{\theta}}(\bm{X},Y)]$, as in the PPI\texttt{++} framework. As a result, evaluating the performance of the proposed estimator relative to PPI\texttt{++} primarily amounts to comparing their variances. This naturally leads to two fundamental questions regarding the proposed estimator $\widehat{\bm{\theta}}_{\text{FPPI}}(\lambda,\mathcal{S})$:
\begin{itemize}
    \item[\textbf{Q1}:] How can one systematically construct or identify a \emph{good} filtering set $\mathcal{S}_0$ based on the available data in practice?
    \item[\textbf{Q2}:] Once such a good set $\mathcal{S}_0$ is obtained, can we theoretically guarantee that the resulting estimator $\widehat{\bm{\theta}}_{\text{FPPI}}(\lambda,\mathcal{S})$ achieves strictly better statistical efficiency (e.g., smaller variance or mean squared error) than the original PPI\texttt{++} estimator?
\end{itemize}

To address the two questions above, we analyze the proposed estimator under two settings: mean estimation and generalized linear models (GLMs). Our analysis shows that there always exists an informative region in which incorporating auxiliary predictions reduces the estimation error and thus improves statistical efficiency. Moreover, we demonstrate that estimating such an informative region can be highly effective in certain scenarios. Throughout the following analysis, we impose the following mild assumption for the subsequent studies.
\begin{assumption}
\label{Ass:SubGaussian}
Assume that the conditional distribution of $Y$ given $\bm{X}$ is sub-Gaussian with conditional mean $m(\bm{X})$. Specifically, there exists a constant $\sigma^2>0$, referred to as a uniform proxy variance, such that for all $\lambda \in \mathbb{R}$, $\mathbb{E} \left[\exp \left\{\lambda\big(Y - m(\bm{X})\big)\right\}\,\middle|\, \bm{X}\right]
    \le \exp \left(\frac{\lambda^2 \sigma^2}{2}\right)$.
\end{assumption}

Assumption~\ref{Ass:SubGaussian} requires that, conditional on any given value of $\bm{X}$, the random variable $Y$ is sub-Gaussian. A simple special case is the additive noise model $Y = m(\bm{X}) + \varepsilon$ for some regression function $m(\cdot)$, in which case Assumption~\ref{Ass:SubGaussian} reduces to assuming that the noise term $\varepsilon$ is sub-Gaussian. More generally, the quantity $Y - m(\bm{X})$ may represent $\bm{X}$-dependent noise, allowing the conditional distribution of the noise to vary with $\bm{X}$. This formulation naturally accommodates models with heteroskedastic or non-additive noise. For example, in logistic regression, $m(\bm{X}) = \mathbb{E}(Y \mid \bm{X})$ represents the success probability, and $Y - m(\bm{X})$ corresponds to a binary, $\bm{X}$-dependent error term.

\subsection{Mean Estimation}
In this section, we study the mean estimation problem as a warm-up to gain theoretical insight into the proposed method. Specifically, the parameter of interest is $\theta^\star = \mathbb{E}(Y)$, and the framework in (\ref{FPPI:Framework}) can be reformulated as
\begin{align*}
    \mathcal{L}_{\textnormal{FPPI}}(\theta|\lambda,\mathcal{S}) 
    =  
    \frac{1}{n}\sum_{i=1}^n (y_i-\theta)^2+
    \lambda \cdot \left(
    \frac{1}{N}\sum_{j=1}^N (f(\widetilde{\bm{x}}_j)-\theta)^2 \cdot \bm{1}_{\mathcal{S}}(\widetilde{\bm{x}}_j)-
    \frac{1}{n}\sum_{i=1}^n (f(\bm{x}_i)-\theta)^2 \cdot \bm{1}_{\mathcal{S}}(\bm{x}_i)\right),
\end{align*}
In this setting, the proposed filtered prediction-powered estimator is given by
\begin{align}
\label{Eqn:Mean}
    \widehat{\theta}_{\mathrm{FPPI}}(\lambda,\mathcal{S}) = 
    \frac{1}{n}\sum_{i=1}^n y_i +
    \lambda \cdot\left(
\frac{1}{N}\sum_{j=1}^N f(\widetilde{\bm{x}}_j) \cdot \bm{1}_{\mathcal{S}}(\widetilde{\bm{x}}_j) -
\frac{1}{n}\sum_{i=1}^n f(\bm{x}_i) \cdot \bm{1}_{\mathcal{S}}(\bm{x}_i)
    \right).
\end{align}

Here, we note that $\widehat{\theta}_{\mathrm{FPPI}}(\lambda,\mathcal{S})$ generalizes the PPI\texttt{++} estimator, which is recovered as a special case when $\mathcal{S}=\mathcal{X}$. To investigate the properties of $\widehat{\theta}_{\mathrm{FPPI}}(\lambda,\mathcal{S})$, we establish Theorem \ref{thm:fppi_mean_var}, which characterizes the variance of the proposed estimator and identifies the condition for $\lambda$ under which the variance is minimized.

\begin{theorem}
\label{thm:fppi_mean_var}
For any measurable set $\mathcal{S}\subset \mathbb{R}^p$ and any fixed prediction model $f(\cdot)$, the filtered prediction-powered estimator defined in \eqref{Eqn:Mean} is unbiased, i.e.,
$\mathbb{E}(\widehat{\theta}_{\mathrm{FPPI}}(\lambda,\mathcal{S}))=\theta^\star$.
Moreover, its variance is given by
\begin{align*}
\mathrm{Var}(\widehat{\theta}_{\mathrm{FPPI}}(\lambda,\mathcal{S}))
= \frac{1}{n}\mathrm{Var}(Y)
+ \lambda^2\frac{N+n}{Nn}
\mathrm{Var} \left(f(\bm X)\mathbf{1}_{\mathcal S}(\bm X)\right)
- \frac{2\lambda}{n}
\mathrm{Cov} \left(Y,f(\bm X)\mathbf{1}_{\mathcal S}(\bm X)\right).
\end{align*}
For any subset $\mathcal{S} \subset \mathcal{X}$ such that $\mathrm{Var}\big(f(\bm X)\mathbf 1_{\mathcal S}(\bm X)\big)>0$, setting the optimal weight 
\begin{align}
    \label{Optimal_Weight}
    \lambda^\star_{\mathcal{S}}
= \frac{\mathrm{Cov}\big(Y,  f(\bm X)\mathbf 1_{\mathcal S}(\bm X)\big)}
       {\mathrm{Var}\big(f(\bm X)\mathbf 1_{\mathcal S}(\bm X)\big)}
\cdot \frac{1}{1+n/N}
\end{align}
 leads to the minimized variance
\begin{align*}
\mathrm{Var}\big(\widehat{\theta}_{\mathrm{FPPI}}(\lambda^\star_{\mathcal{S}},\mathcal S)\big)
= \frac{1}{n}\mathrm{Var}(Y)
- \frac{N}{n(N+n)}  
\frac{\mathrm{Cov}^2\big(Y,f(\bm X)\mathbf 1_{\mathcal S}(\bm X)\big)}
     {\mathrm{Var}\big(f(\bm X)\mathbf 1_{\mathcal S}(\bm X)\big)}.
\end{align*}
\end{theorem}

The result in Theorem \ref{thm:fppi_mean_var} generalizes Example 6.1 of \citet{angelopoulos2023PPI++} by introducing a filtered region. A key implication is that, if there exists a subset $\mathcal{S}_0$ such that
$$
\frac{\mathrm{Cov}^2\big(Y, f(\bm X)\mathbf 1_{\mathcal{S}_0}(\bm X)\big)}
{\mathrm{Var}\big(f(\bm X)\mathbf 1_{\mathcal{S}_0}(\bm X)\big)}
>
\frac{\mathrm{Cov}^2\big(Y, f(\bm X)\big)}
{\mathrm{Var}\big(f(\bm X)\big)},
$$
then the proposed estimator $\widehat{\theta}_{\mathrm{FPPI}}(\lambda^\star_{\mathcal{S}_0},\mathcal{S}_0)$ achieves a smaller variance than the standard PPI\texttt{++} estimator. This raises two natural questions: (1) Does such a subset $\mathcal{S}_0$ exist? (2) If it exists, how can we construct the region $\mathcal{S}_0$ that maximizes the variance reduction?

To address the above two questions, we establish Theorem~\ref{thm:filtered_optimality}, which identifies a specific region in which the proposed estimator achieves strictly higher estimation efficiency.

\begin{theorem}
\label{thm:filtered_optimality}
Suppose $\mathbb{E}[f(\bm X)] = 0$, $\mathrm{Var}(f(\bm X)) > 0$, and $\textnormal{Cov}(Y,f(\bm{X})) \geq 0$. Define the filtered set
\begin{equation}
    \label{RegionS}
    \mathcal{S}_0 = \Big\{
\bm{x} \in \mathcal{X} : \big(m(\bm{x}) - \mathbb{E}(Y)\big)\cdot f(\bm{x}) > 0
\Big\},
\end{equation}
and assume $\mathbb{P}_{\bm X}(\mathcal{S}_0) \in (0,1)$. Then, the filtered prediction-powered estimator satisfies
$$
\mathrm{Var}\Big(\widehat{\theta}_{\mathrm{FPPI}}(\lambda^\star_{\mathcal S_0}, \mathcal S_0)\Big) 
< 
\underbrace{\mathrm{Var}\Big(\widehat{\theta}_{\mathrm{FPPI}}(\lambda^\star_{\mathcal X}, \mathcal X)\Big)}_{\textnormal{PPI\texttt{++} Estimator}},
$$
where $\lambda^\star_{\mathcal S_0}$ is defined in (\ref{Optimal_Weight}) with $\mathcal{S}=\mathcal{S}_0$.
\end{theorem}

Theorem \ref{thm:filtered_optimality} guarantees the existence of a nontrivial filtered region $\mathcal S_0$ under mild conditions, showing that utilizing data in this region leads to a strict variance improvement over the standard PPI\texttt{++} estimator, thereby enhancing inferential efficiency. Moreover, the theorem has several important implications. 
\begin{itemize}
    \item First, the assumption $\mathbb{E}[f(\bm X)] = 0$ is essentially without loss of generality. In practice, the prediction function can always be centered, which does not affect prediction-powered estimators due to the built-in debiasing step. Moreover, the assumption $\operatorname{Cov}(Y, f(\bm X)) \ge 0$ is mild. If $\operatorname{Cov}(Y, f(\bm X)) < 0$, one can simply flip the sign of the prediction function, replacing $f$ with $-f$, to ensure that $\operatorname{Cov}(Y, -f(\bm X)) \ge 0$.

    \item Second, the construction of $\mathcal S_0$ is driven by the sign consistency between the centered conditional mean $m(\bm{x})-\mathbb E[Y]$ and the prediction $f(\bm x)$. When this product is nonnegative at a point $\bm x$, the prediction provides locally informative and directionally correct guidance for $Y$, and including such points for learning is beneficial. In contrast, if the sign were perfectly aligned everywhere, then $f$ would already be uniformly informative across the entire domain and no filtering would be necessary, so the existence of a proper subset $\mathcal S_0$ reflects the heterogeneity of predictive quality across the feature space.
    \item Third, the condition $\mathbb P_{\bm X}(\mathcal S_0)\in (0,1)$ ensures that the informative region is neither degenerate nor trivial: it guarantees that $f(\bm X)$ predicts the direction of the conditional mean correctly on a non-negligible subset of $\mathcal{X}$, while still allowing for systematic disagreement elsewhere. This also implicitly excludes the degenerate case in which $Y$ and $f(\bm X)$ are perfectly correlated, that is $f(\bm{X})$ and $m(\bm{x})-\mathbb E[Y]$ are perfectly positively linearly correlated. Under these conditions, Theorem \ref{thm:filtered_optimality} rigorously shows that filtering according to $\mathcal S_0$ leads to a strictly smaller variance than the standard PPI\texttt{++} estimator, thereby formalizing the intuition that selectively leveraging predictions from informative regions can yield genuine efficiency gains.
\end{itemize}

Theorem~\ref{thm:filtered_optimality} is further corroborated by the simulation results reported in Scenario~I of Section~\ref{Sec:Exp}. In particular, the region $\mathcal{S}_0$ defined in \eqref{RegionS} depends on $\mathbb{E}(Y | \bm X=\bm x)$, whose estimation can be difficult, especially in high-dimensional setting \citep{yang2015minimax}. Nevertheless, note that specifying $\mathcal{S}_0$ only requires correctly determining the sign of $\big(m(\bm{x}) - \mathbb{E}(Y)\big) f(\bm{x})$, which is generally easier than estimating the full conditional expectation $m(\bm{x})$. To illustrate this point, we examine the estimation of $\mathcal{S}_0$ in the following sections for discrete covariates (Section \ref{subsub:Dis}) and continuous covariates (Section \ref{Subsec:Con}).

\subsubsection{Case 1: Discrete Covariates}
\label{subsub:Dis}
Assume that $\bm{X}$ takes $T$ distinct values, i.e., $\bm{X} \in \mathcal{K} = \{\bm{k}_1, \ldots, \bm{k}_T\}$ with $\mathbb{P}(\bm{X}=\bm{k}_t) > 0$ for all $t \in [T]$. For example, if $\bm{X} \in \{0,1,2\} \times \{0,1\}$, then $T = 6$, corresponding to all possible combinations of these discrete values. For each $\bm{k}_t \in \mathcal{K}$, let $n_{\bm{k}_t}$ denote the number of labeled samples in $\mathcal{D}_L = \{(\bm{x}_i, y_i)\}_{i=1}^n$ corresponding to $\bm{k}_t$, so that $\sum_{\bm{k}_t \in \mathcal{K}} n_{\bm{k}_t} = n$. For each $\bm{k} \in \mathcal{K}$, let $m(\bm{k}_t) = \mathbb{E}(Y \mid \bm{X} = \bm{k}_t)$ denote the conditional mean, and define $\theta^\star = \mathbb{E}(Y) = \mathbb{E}[m(\bm{X})]$, which is the quantity of interest. Given the labeled dataset $\mathcal{D}_L$, we estimate $m(\bm{k}_t)$ and $\mathbb{E}(Y)$ by $\widehat{m}(\bm{k}_t) = \frac{1}{n_{\bm{k}}} \sum_{i=1}^n y_i \cdot \mathbf{1}(\bm{x}_i = \bm{k}_t)$ and $\widehat{\theta} = \frac{1}{n} \sum_{i=1}^n y_i$, respectively. Using these estimates, the filtered region $\mathcal{S}_0$ can be estimated as
\begin{align}
    \label{Esti:Region}
    \widehat{\mathcal{S}}_0 = \{\bm{k}_t \in \mathcal{K} : (\widehat{m}(\bm{k}_t) - \widehat{\theta}) f(\bm{k}_t) > 0 \}.
\end{align}

\begin{theorem}[Exact Recovery of $\mathcal{S}_0$] 
\label{Thm:Disc} Define $\Delta_t = m(\bm{k}_t)-\mathbb{E}(Y) \neq 0$ and $\sigma_Y^2 = \sigma^2+\Delta_{\textnormal{diff}}^2/4$ with $\Delta_{\textnormal{diff}}=\max_{t \in [T]} m(\bm{k}_t) - \min _{t \in [T]} m(\bm{k}_t)$. Under Assumption \ref{Ass:SubGaussian}, it then holds that
    \begin{align*}
        \mathbb{P}(\widehat{\mathcal{S}}_0 \neq \mathcal{S}_0) \leq 2\sum_{t=1}^T
        \exp\left(-\frac{n_{\bm{k}_t} \Delta_t^2}{8\sigma^2}\right)
        + 2\sum_{t=1}^T
        \exp\left(-\frac{n \Delta_t^2}{8\sigma_Y^2}\right).
    \end{align*}
\end{theorem}

Theorem \ref{Thm:Disc} shows that the mis-recovery probability decays \emph{exponentially fast} in both the per-category sample sizes $n_{\bm{k}_t}$ and the total sample size $n$. Consequently, the signs of $\widehat{m}(\bm{k}_t)-\widehat{\theta}$ stabilize rapidly. This property is particularly advantageous in practical settings, where the labeled dataset $\mathcal{D}_L$ is typically small. In contrast to parameter estimation, recovering the filtered region is substantially easier, and Theorem~\ref{Thm:Disc} quantifies this statistical simplicity precisely. 

\subsubsection{Case 2: Continuous Covariates}
\label{Subsec:Con}
In this section, we extend the analysis of estimating $\mathcal{S}_0$ to the setting with continuous covariates. To support this extension, we introduce a margin condition that characterizes the distribution of $(\bm{X},Y)$ near the boundary and its impact on the recoverability of $\mathcal{S}_0$.

\begin{definition}[Margin Condition]
\label{ass:margin}
We say that the the distribution of $(\bm{X},Y)$ satisfies the margin condition with parameters $(C,\tau)$ if there exist constants $C>0$ and $\tau>0$ such that
$$
\mathbb{P}_{\bm{X}} \left( \left| m(\bm{X}) - \mathbb{E}(Y) \right| \le t \right) \le C t^{\tau}, 
\quad \text{for all } t>0.
$$
\end{definition}
Definition~\ref{ass:margin} describes how the probability mass of $\bm{X}$ concentrates near the boundary 
$\{\bm{X} : m(\bm{X}) = \mathbb{E}(Y)\}$. This characterization is crucial for analyzing the error in estimating $\mathcal{S}_0$, as it determines how many covariate values lie close to the decision boundary. In particular, the parameter $\tau$ captures the sharpness of the margin: a larger $\tau$ implies that the distribution of $\bm{X}$ places less mass near the boundary, resulting in a clearer separation and, consequently, an easier set estimation problem. This definition is similar to the low-noise assumption in the classification literature \citep{liu2006multicategory,tsybakov2007fast,tian2025neyman}. To illustrate the achievability of large values of $\tau$, we present the following univariate example (Example~\ref{ex:polynomial-density}).
\begin{example}
\label{ex:polynomial-density}
Let $X$ be a univariate random variable supported on $\mathcal{X}=[-1,1]$ with probability density function $P_X(x)=\frac{\ell+1}{2}|x|^{\ell}$ with $\ell>0$ and $m(x)=\mathbb{E}(Y|X=x)=x$. Then, the distribution of $(X,Y)$ satisfies Assumption \ref{ass:margin} with $C=1$ and $\tau = \ell+1$.
\end{example}

Note that estimating $\mathcal{S}_0$ essentially reduces to estimating the conditional mean $\mathbb{E}(Y \mid \bm{X} = \bm{x})$, which is a standard regression problem in supervised learning. Let $\widehat{m}(\bm{x})$ be an estimate of $m(\bm{x}) = \mathbb{E}(Y \mid \bm{X} = \bm{x})$ obtained from $\mathcal{D}_L$. To start with, we make the following general assumption on $\widehat{m}(\bm{x})$.
\begin{assumption}[Consistency of $\widehat{m}$]
\label{Ass:TailConvergence}
Assume that there exist some constants $c_1, c_2,t_1 > 0$ such that, for all $0<t<t_1 $ and for almost every $\bm{x} \in \mathcal{X}$,
\begin{equation}
    \mathbb{P}_{\mathcal{D}_L}\left( \big| \widehat{m}(\bm{x}) - m(\bm{x}) \big| \ge t \right)
    \leq c_1 \exp\left( - c_2 \alpha_n t^2 \right),
\end{equation}
where $\mathbb{P}_{\mathcal{D}_L}$ denotes the probability with respect to the randomness in the labeled dataset $\mathcal{D}_L$, and $\{\alpha_n\}_{n \geq 1}$ is a sequence tending to infinity.
\end{assumption}

Assumption \ref{Ass:TailConvergence} requires that the estimator $\widehat m(\bm{x})$ is consistent. This is a mild condition, as it imposes no specific convergence rate. Extensive literature has established achievable values for $\alpha_n$ under various settings. Specifically, classical results in nonparametric regression \citep{tsybakov2007fast} yield $\alpha_n = n^{\frac{2\beta}{2\beta+p}}$, where $\beta$ denotes a smoothness parameter.

For the continuous covariate case, $\mathcal{S}_0$ can be estimated as
\begin{align}
\label{Contin_S_est}
    \widehat{\mathcal{S}}_0 = \left\{\bm{x}\in \mathcal{X}:\left(\widehat{m}(\bm{x})-\frac{1}{n}\sum_{i=1}^n y_i\right)f(\bm{x}) >0\right\}.
\end{align}
Unlike the discrete covariate setting, where $\mathcal{S}_0$ consists of finite elements, in the continuous case $\mathcal{S}_0$ is a measurable subset of $\mathcal{X}$. To quantify the discrepancy between $\widehat{\mathcal{S}}_0$ and $\mathcal{S}_0$, we evaluate their symmetric difference under the distribution of $\bm{X}$, which is given by 
$$
\textbf{Mis-Recovery Probability: }
\mathbb{P}_{\bm{X}}\left(\widehat{\mathcal{S}}_0 \Delta \mathcal{S}_0
\right) =\mathbb{P}_{\bm{X}}\left((\widehat{\mathcal S}_0 \cup  \mathcal{S}_0) \setminus (\widehat{\mathcal S}_0 \cap  \mathcal{S}_0)
\right),
$$
where $\mathbb{P}_{\bm X}(\cdot)$ denotes probability measure with respect to the marginal distribution of $\bm X$.

\begin{theorem}
    \label{thm:continuous}
Suppose that $(\bm{X}, Y)$ satisfies the margin condition with parameters $(C,\tau)$ as defined in Definition~\ref{ass:margin}. Under Assumptions~\ref{Ass:SubGaussian} and \ref{Ass:TailConvergence}, we obtain
\begin{align}
\label{Error_Bound}
\mathbb{E}_{\mathcal{D}_L}\Big( \mathbb{P}_{\mathbf{X}}(\widehat{\mathcal{S}}_0 \Delta \mathcal{S}_0) \Big)
\lesssim \left(
\frac{\log (\alpha_n \wedge n)}{\alpha_n  \wedge n}
\right)^{\frac{\tau}{2}},
\end{align}
where $a \wedge b = \min\{a, b\}$ and $\alpha_n$ is as defined Assumption~\ref{Ass:TailConvergence}.
\end{theorem}

Theorem~\ref{thm:continuous} shows that, up to logarithmic factors, the expected error probability decays at the rate $(\alpha_n \wedge n)^{-\frac{\tau}{2}}$, where the exponent $\tau$ is the margin parameter defined in Definition~\ref{ass:margin}. A larger value of $\tau$ corresponds to a sharper margin, meaning that covariates are less likely to lie near the boundary $\{\bm{X}:m(\bm{X}) = \mathbb{E}(Y)\}$. Consequently, the estimator can more easily distinguish the two regions, leading to faster convergence of $\widehat{\mathcal{S}}_0$ to $\mathcal{S}_0$.

Although the rate in~(\ref{Error_Bound}) can be fast when $\tau$ is large, a natural question is whether an exponential decay rate, similar to that in the discrete-covariate setting (Theorem~\ref{Thm:Disc}), can be achieved. Such a rate would indicate that recovering the region $\mathcal{S}_0$ is substantially easier. To this end, we consider such scenario in which the covariates are bounded away from the boundary $\{\bm{X}:m(\bm{X})=\mathbb{E}(Y)\}$.
\begin{theorem}
   \label{Prop:Exp}
    Under Assumptions \ref{Ass:SubGaussian} and \ref{Ass:TailConvergence}, suppose there exists a constant $c_3>0$ such that  $|m(\bm{X})-\mathbb{E}(Y)|\geq c_3$ almost surely, it then holds that
    \begin{align*}
       \mathbb{E}_{\mathcal{D}_L}\Big( \mathbb{P}_{\mathbf{X}}(\widehat{\mathcal{S}}_0 \Delta \mathcal{S}_0) \Big) 
\leq c_1 \exp\left(-\frac{c_2 \alpha_n c_3^2}{4}\right)+
2  \exp\left(-\frac{n c_3^2}{8\sigma_Y^2}\right),
    \end{align*}
    where $c_1$, $c_2$, and $\alpha_n$ are as defined in Assumption \ref{Ass:TailConvergence} and $\sigma_Y^2$ is the proxy variance of $Y$.
\end{theorem}

Theorem \ref{Prop:Exp} demonstrates that the expected mis-recovery probability of $\widehat{\mathcal{S}}_0$ decays exponentially if $\alpha_n \asymp n^{b}$ for some $b>0$. This observation is in line with Theorem \ref{Thm:Disc}, indicating that recovering the efficiency-improving region $\mathcal{S}_0$ is inherently easier than model estimation. Similarly, to illustrate the validity of the additional assumption in Theorem~\ref{Prop:Exp}, we present the following univariate example (Example~\ref{ex:fast-density}).

\begin{example}
\label{ex:fast-density}
Let $X$ be a univariate random variable supported on $\mathcal{X}=[-1,-c] \cup [c,1]$ with $c \in (0,1)$ and probability density function $P_X(x)=\frac{\ell+1}{2 ( 1 - c^{\ell+1} )}|x|^{\ell}$ with $\ell>0$ and $m(X)=X$. Then, $|m(X)-\mathbb{E}(Y)| \geq c$ for any $X$.
\end{example}

\subsection{Asymptotic Normality}
In this section, we establish the asymptotic normality of the proposed mean estimator in the continuous covariate setting. We begin by introducing the following assumption on the leave-one-out stability of $\widehat{m}(\cdot)$. This assumption requires that removing a single observation results in only a small uniform change in the estimator \citep{bousquet2002stability}.

\begin{assumption}[Uniform Leave-one-out Stability]
\label{ass:loo_stability}
For a labeled dataset $\mathcal{D}_L = \{(\bm{x}_i,y_i)\}_{i=1}^n$, there exists a deterministic sequence $\{\psi_n\}_{n \geq 1}$ tending to infinity and some constants $c_3,c_4,t_2$ such that for any $0<t<t_2$,
$$
\mathbb{P}_{\mathcal{D}_L}\left(\sup_{\bm{x}\in\mathcal X}
\big|\widehat m(\bm{x})-\widehat m^{(-i)}(\bm{x})\big| \geq t\right)
\le c_3 \exp\left(-c_4 \psi_n t^2\right) ,
$$
where $\widehat m^{(-i)}$ denotes the estimator constructed from the dataset $\mathcal{D}_L$ without the $i$-th observation.
\end{assumption}

To justify Assumption~\ref{ass:loo_stability}, we provide two illustrative examples, $k$-nearest neighbor regression and simple linear regression, in Section A.1.2 of the supplementary file, where we establish that $\psi_n \gtrsim \alpha_n$ in both settings. A key advantage of Assumption \ref{ass:loo_stability} is that it allows the labeled dataset to be used to estimate $\mathcal{S}_0$ and $\lambda^\star_{\mathcal{S}_0}$, as well as parameter estimation without requiring sample splitting. For notational convenience, define $f_{\mathcal{S}}(\bm{x}) = f(\bm{x})\cdot \bm{1}_{\mathcal{S}}(\bm{x})$. Given the estimator $\widehat{\mathcal{S}}_0$ defined in \eqref{Contin_S_est}, we construct the plug-in FPPI mean estimator as 
\begin{align*}
   \widehat{\theta}_{\mathrm{FPPI}}(\widehat \lambda_{\mathcal{S}_0},\widehat{\mathcal{S}}_0)  = 
    \frac{1}{n}\sum_{i=1}^n y_i +
  \widehat \lambda_{\mathcal{S}_0}\cdot\left(
\frac{1}{N}\sum_{j=1}^N f_{\widehat{\mathcal{S}}_0}(\widetilde{\bm{x}}_j) -
\frac{1}{n}\sum_{i=1}^n f_{\widehat{\mathcal{S}}_0}(\bm{x}_i)
    \right),
\end{align*}
where $
    \widehat \lambda_{\mathcal{S}_0} = \frac{
\frac{1}{n}\sum_{i=1}^n
\left(y_i-\bar y\right)f_{\widehat{\mathcal{S}}_0}(\bm{x}_i)
}{
\frac{1}{N}
\sum_{j=1}^N
\big(f_{\widehat{\mathcal{S}}_0}(\widetilde{\bm{x}}_j)-\bar{f}_{\widehat{\mathcal{S}}_0}\big)^2
}
\cdot \frac{1}{1+n/N}$ and $\bar{f}_{\widehat{\mathcal{S}}_0} = \frac{1}{n+N}\left[\sum_{j=1}^N
f_{\widehat{\mathcal{S}}_0}(\widetilde{\bm{x}}_j)+\sum_{i=1}^n f_{\widehat{\mathcal{S}}_0}(\bm{x}_i)\right]$. The general algorithm is summarized in Algorithm \ref{alg:filtered-fppi-continuous}.

\begin{algorithm}[h]
\caption{FPPI Estimator via Region Estimation for Continuous Covariates}
\label{alg:filtered-fppi-continuous}
\begin{algorithmic}[1]
\STATE \textbf{Input:} Labeled dataset $\mathcal{D}_L = \{(\bm{x}_i, y_i)\}_{i=1}^n$, 
       Unlabeled/synthetic dataset $\widetilde{\mathcal{D}} = \{\widetilde{\bm{x}}_j\}_{j=1}^N$, 
       Regression estimator $\widehat m(\cdot)$, and feature function $f(\cdot)$.

\STATE Fit the regression estimator $\widehat m(\bm{x})$ using the labeled data $\mathcal{D}_L$.

\STATE Compute the overall mean of labeled outcomes:
$$
\bar y = \frac{1}{n} \sum_{i=1}^n y_i
$$

\STATE Estimate the efficiency-improving region:
$$
\widehat{\mathcal{S}}_0 = \left\{\bm{x} \in \mathcal{X}: (\widehat m(\bm{x}) - \bar y) f(\bm{x}) > 0 \right\}
$$

\STATE Compute the plug-in FPPI coefficient:
$$
\widehat \lambda_{\widehat{\mathcal{S}}_0} =
\frac{
\frac{1}{n} \sum_{i=1}^n (y_i - \bar y) f_{\widehat{\mathcal{S}}_0}(\bm{x}_i)
}{
\frac{1}{N} \sum_{j=1}^N \big(f_{\widehat{\mathcal{S}}_0}(\widetilde{\bm{x}}_j) - \bar f_{\widehat{\mathcal{S}}_0}\big)^2
} \cdot \frac{1}{1+n/N}, \quad
\bar f_{\widehat{\mathcal{S}}_0} = \frac{1}{n+N}\left(\sum_{i=1}^n f_{\widehat{\mathcal{S}}_0}(\bm{x}_i) + \sum_{j=1}^N f_{\widehat{\mathcal{S}}_0}(\widetilde{\bm{x}}_j)\right)
$$

\STATE Construct the FPPI mean estimator:
$$
\widehat{\theta}_{\mathrm{FPPI}}(\widehat \lambda_{\mathcal{S}_0},\widehat{\mathcal{S}}_0) =
\bar y + \widehat \lambda_{\mathcal{S}_0} \left(
\frac{1}{N} \sum_{j=1}^N f_{\widehat{\mathcal{S}}_0}(\widetilde{\bm{x}}_j) -
\frac{1}{n} \sum_{i=1}^n f_{\widehat{\mathcal{S}}_0}(\bm{x}_i)
\right)
$$

\STATE \textbf{Output:} $\widehat{\theta}_{\mathrm{FPPI}}$
\end{algorithmic}
\end{algorithm}

Here, we emphasize that $\{\widetilde{\bm{x}}_j\}_{j=1}^N$ is independent of $\widehat{\mathcal{S}}_0$, which depends only on the labeled dataset. Corollary \ref{Coro:Mean_Inference} establishes the asymptotic distribution of the proposed estimator with plug-in estimates.

\begin{corollary}
\label{Coro:Mean_Inference}
Suppose that $\psi_n \gtrsim (n \wedge \alpha_n)$. Under Assumptions \ref{Ass:SubGaussian}–\ref{ass:loo_stability} with
$\mathbb{P}_{\mathbf{X}}\left(\widehat{\mathcal{S}}_0 \Delta \mathcal{S}_0\right) = o_p\left(n^{-1/2}\right)$, we have, as $\frac{n}{N} \to r$,
    \begin{align*}
        \sqrt{n}\left(\widehat{\theta}_{\mathrm{FPPI}}(\widehat \lambda_{\mathcal{S}_0},\widehat{\mathcal{S}}_0) - \theta^\star
        \right) \xrightarrow{d} \mathcal{N}\left(
0,\mathrm{Var}(Y)
- \frac{1}{(1+r)}  
\frac{\mathrm{Cov}^2\big(Y,f(\bm X)\mathbf 1_{\mathcal S_0}(\bm X)\big)}
     {\mathrm{Var}\big(f(\bm X)\mathbf 1_{\mathcal S_0}(\bm X)\big)}
        \right).
    \end{align*}
\end{corollary}

Corollary \ref{Coro:Mean_Inference} shows that the plug-in estimator attains a smaller asymptotic variance than the classical PPI\texttt{++} estimator, as characterized in Theorem \ref{thm:filtered_optimality}. This advantage, however, holds under the condition $\mathbb{P}_{\mathbf{X}} \left(\widehat{\mathcal{S}}_0 \,\Delta\, \mathcal{S}_0\right)=o_p \left(n^{-1/2}\right)$, which requires the region recovery error to vanish faster than the $n^{-1/2}$ rate. This condition is satisfied when $(\alpha_n \wedge n)^{-\tau/2}=o(n^{-1/2})$ up to a logarithmic factor. Moreover, under the strong separation condition in Theorem \ref{Prop:Exp}, the condition holds automatically. In addition, Corollary \ref{Coro:Mean_Inference} requires $\psi_n \gtrsim (n \wedge \alpha_n)$, which is a mild condition and is satisfied by standard nonparametric estimation procedures (see the discussion in Section A.1.2 of the supplementary file).

\begin{remark}
Note that assumptions $\psi_n \gtrsim (n \wedge \alpha_n)$ and $\mathbb{P}_{\mathbf{X}} \left(\widehat{\mathcal{S}}_0 \,\Delta\, \mathcal{S}_0\right)=o_p(n^{-1/2})$ are needed only when full labeled data are used for the plug-in estimator. If a data-splitting scheme is employed instead, both requirements can be removed. See Section A.1.3 in the supplement.   
\end{remark}

\subsection{Generalized Linear Models}
\label{Subsec:GLM}
In this section, we specialize the framework in \eqref{FPPI:Framework} to the generalized linear model (GLM) setting. For estimation and inference, we impose a GLM as a \emph{working model}, without assuming that it coincides with the true conditional distribution of $Y$ given $\bm X$. Accordingly, our inferential target is the pseudo-true parameter $\bm\theta^\star
= \argmin_{\bm\theta \in \mathbb{R}^p}
\mathbb E \big[ L_{\bm{\theta}}(\bm{X},Y) \big]$, where $L_{\bm{\theta}}(\bm{X},Y)$ denotes the negative log-likelihood induced by the canonical exponential family model with density
\begin{align}
\label{GLM:formula}
P(Y \mid \bm{X}, \bm{\theta})
= h(Y)\exp \left\{ (\bm X^\top \bm\theta)Y
- A(\bm X^\top \bm\theta) \right\},
\end{align}
with $A(\cdot)$ denoting the log-partition function. In this paper, we assume that the true parameter $\bm{\theta}^\star$ is unique and lies in a convex, compact parameter space $\bm{\Theta}$, for example $\bm{\Theta} = [-M, M]^p$ for some sufficiently large constant $M$. In addition, the specification in \eqref{GLM:formula} encompasses linear, logistic, and Poisson regression as special cases, with detailed derivations provided in Section A.1.1 of the supplementary file.

Given the labeled dataset $\mathcal{D}_L$ and the unlabeled dataset $\mathcal{D}_U$, the FPPI objective function under the canonical GLM is defined as
\begin{align}
\label{FPPI:GLM}
\mathcal L_{\text{FPPI}}(\bm\theta \,|\,\lambda,\mathcal{S})
=&\;
\frac{1}{n}\sum_{i=1}^n
\Big[
A(\bm x_i^\top \bm\theta)
- (\bm x_i^\top \bm\theta) y_i
\Big]
+ \frac{\lambda}{N}\sum_{j=1}^N
\Big[
A(\widetilde{\bm x}_j^\top \bm\theta)
- (\widetilde{\bm x}_j^\top \bm\theta)
f(\widetilde{\bm x}_j)
\Big]\bm 1_{\mathcal S}(\widetilde{\bm x}_j)  \notag\\
&\;
- \frac{\lambda}{n}\sum_{i=1}^n
\Big[
A(\bm x_i^\top \bm\theta)
- (\bm x_i^\top \bm\theta)
f(\bm x_i)
\Big]\bm 1_{\mathcal S}(\bm x_i).
\end{align}
For any $\lambda$ and $\mathcal{S}$, the corresponding FPPI estimator is defined as
\begin{align}
    \label{FPPI:GLM_Esti}
    \widehat{\bm\theta}_{\text{FPPI}}(\lambda,\mathcal{S})
= \arg\min_{\bm\theta\in \bm{\Theta}}
\mathcal L_{\text{FPPI}}(\bm\theta \,|\,\lambda,\mathcal{S}).
\end{align}
In what follows, we aim to understand how the choice of $\lambda$ and the filtered region $\mathcal{S}$ influence the asymptotic distribution of $\widehat{\bm\theta}_{\mathrm{FPPI}}(\lambda,\mathcal{S})$.

\begin{assumption}
\label{ass:PD}
The population Hessian matrix $\bm \Sigma = \mathbb E \left[ A''(\bm X^\top \bm\theta^\star)\,\bm X \bm X^\top \right]$ is positive definite. Equivalently, there exists a constant $\kappa>0$ such that $\Lambda_{\min}(\bm \Sigma) \ge \kappa$. In addition, $A(\cdot)$ is twice continuously differentiable and satisfies $A''(\eta) > 0$ for all $\eta \in \mathbb{R}$.
\end{assumption}

Assumption~\ref{ass:PD} imposes a non-degeneracy condition on the population Hessian. It guarantees that the population quasi-risk is strictly convex and, consequently, that the pseudo-true parameter $\bm\theta^\star$ is uniquely identifiable. This assumption is mild and standard in generalized linear models. In particular, the condition $A''(\eta)>0$ is automatically satisfied by canonical exponential-family working models, such as linear, logistic, and Poisson regression. For linear regression with a Gaussian working likelihood, Assumption~\ref{ass:PD} reduces to the classical requirement that the covariate covariance matrix be positive definite \citep{white1982maximum}.

\begin{theorem}
\label{thm:FPPI_GLM_final}
Denote $\mu(\bm X) = A'(\bm X^\top \bm\theta^\star)$ and $R(\bm X,Y) = Y - \mu(\bm X)$. Under Assumptions~\ref{Ass:SubGaussian} and \ref{ass:PD}, if $n, N \to \infty$ with $n/N \to r$, the FPPI estimator
$\widehat{\bm\theta}_{\mathrm{FPPI}}(\lambda,\mathcal S)$ defined in
(\ref{FPPI:GLM_Esti}) satisfies
\begin{align}
\label{GLM:AsymDist}
\sqrt n
\big(
\widehat{\bm\theta}_{\mathrm{FPPI}}(\lambda,\mathcal S)
-\bm\theta^\star
\big)
\xrightarrow{d}
\mathcal N\Big(
\bm 0,\;
\bm \Sigma^{-1}
\big(
\bm\Omega
+\lambda^2(1+r)\bm M_{\mathcal S}
-2\lambda \bm\Gamma_{\mathcal S}
\big)
\bm \Sigma^{-1}
\Big),
\end{align}
where $\bm{\Omega}$, $\bm{M}_{\mathcal{S}}$, and $\bm{\Gamma}_{\mathcal{S}}$ are defined as $\bm\Omega
=
\mathbb E \left[
R^2(\bm X,Y)\bm X\bm X^\top
\right]$, $\bm M_{\mathcal S}
=
\textnormal{Var}\Big(
\bm X
\big(f(\bm X)-\mu(\bm X))
\mathbf 1_{\mathcal S}(\bm X)
\Big)$, and $\bm \Gamma_{\mathcal S}=\mathbb E \left[
R(\bm X,Y)
\big(f(\bm X)-\mu(\bm X)\big)
\bm X\bm X^\top
\mathbf 1_{\mathcal S}(\bm X)
\right]$, respectively. The asymptotic mean squared error (AMSE) is given by
$$
\mathrm{AMSE}(\lambda,\mathcal S)
\triangleq n\cdot\mathbb{E}[\Vert
\widehat{\bm\theta}_{\mathrm{FPPI}}(\lambda,\mathcal S)
-\bm\theta^\star\Vert_2^2]
=
\mathrm{tr} \left(
\bm \Sigma^{-1}
\big(
\bm\Omega
+\lambda^2(1+r)\bm M_{\mathcal S}
-2\lambda \bm\Gamma_{\mathcal S}
\big)
\bm \Sigma^{-1}
\right).
$$
If $\mathrm{tr}(\bm\Sigma^{-1}\bm M_{\mathcal S}\bm\Sigma^{-1})\neq 0$, then $\mathrm{AMSE}(\lambda,\mathcal S)$ is minimized at $\lambda^\star_{\mathcal S}=\frac{\mathrm{tr}(\bm\Sigma^{-1}\bm\Gamma_{\mathcal S}\bm\Sigma^{-1})}{
(1+r)\mathrm{tr}(\bm\Sigma^{-1}\bm M_{\mathcal S}\bm\Sigma^{-1})}$, yielding the minimal AMSE
$$
\mathrm{AMSE}(\lambda^\star_{\mathcal S},\mathcal S)
=
\mathrm{tr}(\bm\Sigma^{-1}\bm\Omega\bm\Sigma^{-1})
-
\frac{
\big[\mathrm{tr}(\bm\Sigma^{-1}\bm\Gamma_{\mathcal S}\bm\Sigma^{-1})\big]^2
}{
(1+r)\,
\mathrm{tr}(\bm\Sigma^{-1}\bm M_{\mathcal S}\bm\Sigma^{-1})
}.
$$
\end{theorem}

Theorem \ref{thm:FPPI_GLM_final} establishes the asymptotic normality of $\widehat{\bm\theta}_{\mathrm{FPPI}}(\lambda,\mathcal{S})$ defined in~\eqref{FPPI:GLM_Esti} and provides an explicit expression for its asymptotic covariance matrix. This characterization reveals how the tuning parameter $\lambda$ and the choice of the filtered region $\mathcal{S}$ jointly influence the efficiency of statistical inference. In particular, the structure of the asymptotic variance shows that, for any given region $\mathcal{S}$, there exists an optimal tuning parameter $\lambda_{\mathcal S}^\star$ that minimizes the asymptotic mean squared error, thereby yielding the most efficient FPPI estimator attainable within that region.

A natural question that arises is how to identify a region $\mathcal{S}$ for which $\textnormal{AMSE}(\lambda^\star_{\mathcal{S}},\mathcal{S})<\textnormal{AMSE}(\lambda^\star_{\mathcal{X}},\mathcal{X})$. To address this question, we further develop Corollary~\ref{Coro:GLM}, which specifies a region $\mathcal{S}_0$ enables the FPPI estimator to achieve a strictly smaller asymptotic variance than that obtained by using the entire sample space.

\begin{corollary}
    \label{Coro:GLM}
Define the filtered region 
$$\mathcal{S}_0 = \big\{ \bm{x} \in \mathcal{X} : \big(f(\bm{x}) - 
    \mu(\bm{x})\big) \cdot \big(m(\bm{x}) - \mu(\bm{x})\big) > 0 \big\},$$
    where $\mu(\cdot)$ is as defined in Theorem \ref{thm:FPPI_GLM_final}. If $\mathbb{P}_{\bm X}(\mathcal S_0)\in(0,1)$, $\mathrm{tr} \left(\bm\Sigma^{-1}\bm\Gamma_{\mathcal X}\bm\Sigma^{-1}\right)>0$,
and $\mathbb E \left[\bm X\{Y-f(\bm X)\}\right]=\bm 0$, then, under the assumptions of Theorem~\ref{thm:FPPI_GLM_final}, it holds that
\begin{align*}
    \textnormal{AMSE}(\lambda^\star_{\mathcal{S}_0}, \mathcal{S}_0)
    < \underbrace{\textnormal{AMSE}(\lambda^\star_{\mathcal{X}}, \mathcal{X})}_{\textnormal{PPI\texttt{++}}} \leq \textnormal{tr}(\bm{\Sigma}^{-1}\bm{\Omega}\bm{\Sigma}^{-1}).
\end{align*}
\end{corollary}

Corollary~\ref{Coro:GLM} indicates that $\mathcal{S}_0$ achieving a smaller asymptotic error should primarily consist of data points for which the prediction model $f(\bm X)$ and the true regression function $m(\bm X)$ exhibit aligned tendencies relative to the benchmark generalized linear approximation $\mu(\bm{X})=A'(\bm X^\top \bm\theta^\star)$. In other words, $\mathcal{S}_0$ should include those data satisfying 
$$
\textbf{Sign Consistency: }
\operatorname{sign} \big(f(\bm X)-\mu(\bm{X})\big)
=
\operatorname{sign} \big(m(\bm X)-\mu(\bm{X})\big).
$$
This result is intuitive: when fitting a generalized linear model to $(\bm X, Y)$, the misspecification induces a non-constant bias $m(\bm X)- \mu(\bm X)$. The prediction model becomes informative for variance reduction precisely on those data points where its induced bias $f(\bm X)-\mu(\bm X)$ aligns in direction with the true bias. In such regions, the pseudo-labels behave similarly to labeled observations and thus contribute positively to the estimation efficiency.

A natural question that arises in the GLM setting is whether the region $\mathcal{S}_0$ can be recovered accurately. In practice, $\mathcal{S}_0$ can be estimated by first estimating $m(\bm X)$ and the parameter $\bm\theta^\star$ from the labeled dataset $\mathcal{D}_L$. Similar to the phenomena observed in Theorems~\ref{thm:continuous} and~\ref{Prop:Exp}, the difficulty of recovering $\mathcal{S}_0$ through the estimation of $m(\bm X)$ and $\bm\theta^\star$ is primarily governed by the behavior of the covariate distribution near the boundaries $\{\bm{X}:m(\bm X)=\mu(\bm{X})\}$ and $\{\bm{X}:f(\bm X)=\mu(\bm{X})\}$. To formalize this intuition, we establish Theorem~\ref{Thm:GLM_Estimate_S0}.

\begin{theorem}
\label{Thm:GLM_Estimate_S0}
Define $\widehat{\mathcal{S}}_0 
= \big\{ \bm{x} \in \mathcal{X} : 
\big(f(\bm{x}) -  \widehat{\mu}(\bm{x})\big)
\big(\widehat{m}(\bm{x}) - \widehat{\mu}(\bm{x})\big) > 0 
\big\}$ with $ \widehat{\mu}(\bm{x})=A'(\bm{x}^\top \widehat{\bm{\theta}}_{\textnormal{MLE}})$ and $\widehat{\bm{\theta}}_{\text{MLE}} = \argmin_{\bm{\theta} \in \mathbb{R}^p} \sum_{i=1}^n \big[A(\bm{x}_i^\top \bm{\theta})- y_i \bm{x}_i^\top \bm{\theta}  \big]$. Under Assumptions \ref{Ass:SubGaussian} and \ref{Ass:TailConvergence}, $\mathbb{E}_{\mathcal{D}_L}\Big(
\mathbb{P}_{\mathbf{X}}(\widehat{\mathcal{S}}_0 \Delta \mathcal{S}_0)
\Big)$ admits the following bounds.
\begin{itemize}
\item[(1)] \textbf{\textnormal{(General Margin Condition)}}.
If there exist constants $C>0$ and $\tau>0$ such that $
\mathbb{P}_{\mathbf{X}} \left(
\min \left\{
|m(\mathbf{X}) - \mu(\bm{X})|,
|f(\bm{X}) - \mu(\bm{X})|
\right\}
\le t
\right)
\le C t^{\tau}$, for all $t>0$, then
$$
\mathbb{E}_{\mathcal{D}_L}\Big(
\mathbb{P}_{\mathbf{X}}(\widehat{\mathcal{S}}_0 \Delta \mathcal{S}_0)
\Big)
\lesssim
\left(
\frac{p + \log (n \wedge \alpha_n)}
     {n \wedge \alpha_n}
\right)^{\tau/2}.
$$
\item[(2)] \textbf{\textnormal{(Strong Separation)}}.
If $\min \left\{
|f(\bm{X}) - \mu(\bm{X})|,
|m(\bm{X}) - \mu(\bm{X})|
\right\}
\ge c_3$ for some positive constant $c_3$, then for any $\bm{X} \in \mathcal{X}$, then for some universal constant $C>0$,
$$
\mathbb{E}_{\mathcal{D}_L}\Big(
\mathbb{P}_{\mathbf{X}}(\widehat{\mathcal{S}}_0 \Delta \mathcal{S}_0)
\Big)
\;\lesssim\;
\exp \big(- C (n \wedge \alpha_n) c_3^2 \big).
$$
\end{itemize}
\end{theorem}

Theorem~\ref{Thm:GLM_Estimate_S0} characterizes the asymptotic behavior of the plug-in estimator $\widehat{\mathcal{S}}_0$ for the filtered region $\mathcal{S}_0$ in the generalized linear model setting. It shows that $\widehat{\mathcal{S}}_0$ is a consistent estimator of $\mathcal{S}_0$ as the effective sample size $n \wedge \alpha_n \to \infty$, provided that both $m(\bm X)$ and $\bm\theta^\star$ can be estimated consistently. Under a general margin condition, where the probability mass of $\bm X$ near the decision boundaries $\{\bm{X}:m(\bm X)=\mu(\bm X)\}$ and $\{\bm{X}:f(\bm X)=\mu(\bm X)\}$ vanishes at rate $t^{\tau}$, the expected symmetric difference error converges to zero at rate $O\left((n \wedge \alpha_n)^{-\tau/2}\right)$, up to logarithmic factors. In contrast, under a strong separation condition, where the regression function and the prediction function are uniformly away from the orcale linear function, the estimation error decays exponentially fast in $n \wedge \alpha_n$.

\begin{figure}[ht]
    \centering
            \begin{subfigure}[b]{0.485\textwidth}
        \centering
        \includegraphics[scale=0.354]{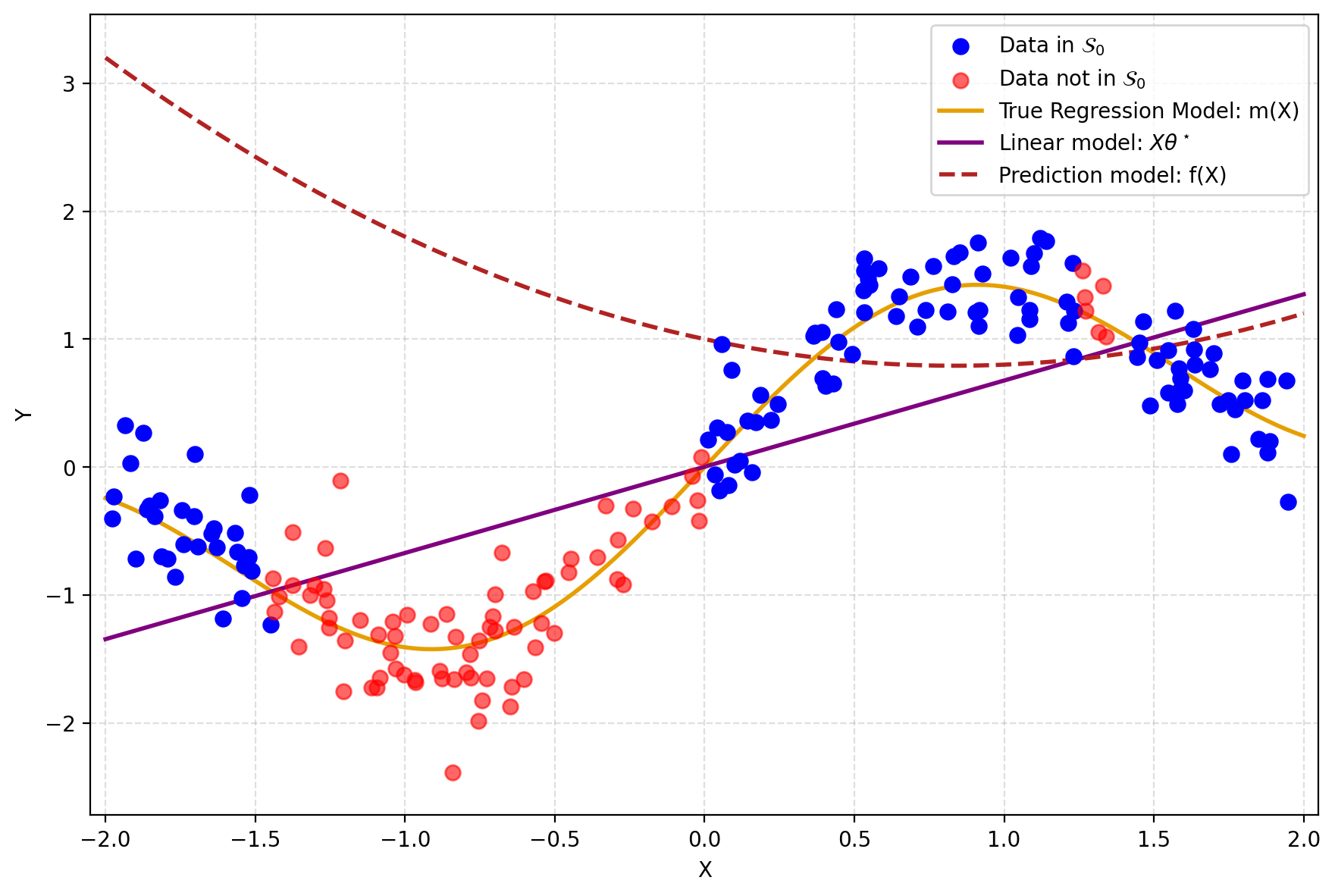}
    \end{subfigure}
        \begin{subfigure}[b]{0.485\textwidth}
        \centering
        \includegraphics[scale=0.354]{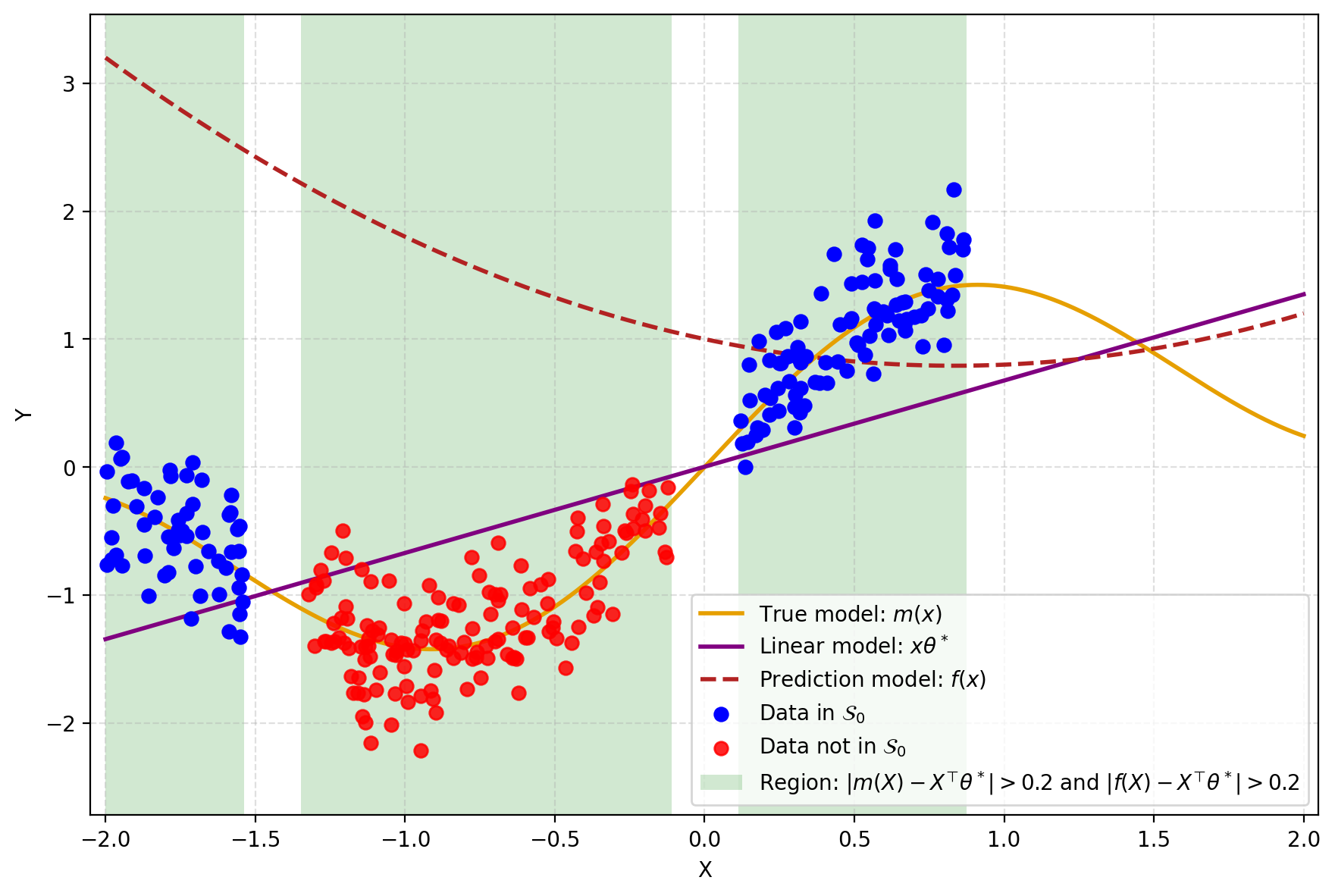}
    \end{subfigure}
    \caption{Two cases in Theorem~\ref{Thm:GLM_Estimate_S0} (Linear Regression Example). 
(\textbf{Left:}) The data lie close to the boundaries $\{\bm{X}:m(\bm{X}) = \bm{X}^\top \bm{\theta}^\star\}$ and $\{\bm{X}:f(\bm{X}) = \bm{X}^\top \bm{\theta}^\star\}$. (\textbf{Right:}) The data are well separated from both boundaries, making the recovery of $\mathcal{S}_0$ easier.
}
    \label{fig:Demonstration}
\end{figure}

To illustrate the distinction between the two regimes in Theorem~\ref{Thm:GLM_Estimate_S0}, we present an example in Figure~\ref{fig:Demonstration}. The left panel depicts the general margin condition, under which the covariate distribution places nonnegligible probability mass near the decision boundaries. In this setting, blue points represent observations within the region $\mathcal{S}_0$ that should be retained for training, whereas red points correspond to observations that should be excluded. The right panel illustrates the strong separation regime, where the support of the data (shaded region, corresponding to $\mathcal{X}$) is bounded away from the boundaries $\{\bm{X}:m(\bm X)=\mu(\bm{X})\}$ and $\left\{\bm{X}:f(\bm X)=\mu(\bm{X})\right\}$. In this case, the recovery of $\mathcal{S}_0$ that includes the blue points, occurs at an exponentially fast rate.

\begin{algorithm}
\caption{Filtered Prediction-Powered Estimator for the Generalized Linear Models}
\label{alg:filtered-fppi-init}
\begin{algorithmic}[1]
\STATE \textbf{Input:} Labeled data $\mathcal{D}_L = \{(\bm{x}_i, y_i)\}_{i=1}^n$, unlabeled data $\mathcal{D}_U = \{\widetilde{\bm{x}}_j\}_{j=1}^N$, prediction function $f(\cdot)$, and $\widehat{m}(\cdot)$  \\
\STATE \textbf{Output:} The FPPI estimator

\STATE Compute the MLE using the labeled data $\mathcal{D}_L$:
$\widehat{\bm{\theta}}_{\text{MLE}} = \argmin_{\bm{\theta} \in \mathbb{R}^p} \sum_{i=1}^n \big[A(\bm{x}_i^\top \bm{\theta})- y_i \bm{x}_i^\top \bm{\theta}  \big]$.
\STATE Estimate the filtered region $\mathcal{S}_0$ by
\begin{align*}
    \widehat{\mathcal{S}}_0 = \left\{
    \bm{x} \in \mathcal{X}:
    \big(\widehat{m}(\bm{x})-A'(\bm{x}^\top \widehat{\bm{\theta}}_{\text{MLE}})\big) \big(f(\bm{x})-A'(\bm{x}^\top \widehat{\bm{\theta}}_{\text{MLE}}\big)>0
    \right\}.
\end{align*}

\STATE Estimate $\bm{\Sigma}=\mathbb{E}\left[A''(\bm X^\top\bm\theta^\star)\bm X\bm X^\top\right]$, $\bm{\Gamma}_{\mathcal{S}_0}$, and $\bm{M}_{\mathcal{S}_0}$ by
\begin{align*}
   \widehat{\bm\Sigma} &= \frac{1}{N}\sum_{j=1}^N A''(\widetilde{\bm x}_j^\top  \widehat{\bm{\theta}}_{\text{MLE}}) \widetilde{\bm x}_j \widetilde{\bm x}_j^\top, \quad\widehat{\bm M}_{\widehat{\mathcal{S}}_0}= \frac{1}{N} \sum_{j=1}^N \mathbf 1_{\widehat{\mathcal{S}}_0}(\widetilde{\bm{x}}_j) \big(f(\widetilde{\bm{x}}_j) - A'(\widetilde{\bm{x}}_j^\top  \widehat{\bm{\theta}}_{\text{MLE}})\big)^2 \widetilde{\bm{x}}_j \widetilde{\bm{x}}_j^\top \\
\widehat{\bm \Gamma}_{\widehat{\mathcal{S}}_0} &= \frac{1}{n}\sum_{i=1}^n \mathbf 1_{\widehat{\mathcal{S}}_0}(\bm{x}_i) \big(y_i - A'(\bm x_i^\top  \widehat{\bm{\theta}}_{\text{MLE}})\big)\big(f(\bm x_i) - A'(\bm x_i^\top  \widehat{\bm{\theta}}_{\text{MLE}})\big) \bm x_i \bm x_i^\top
\end{align*}
\STATE Estimate the optimal $\lambda_{\mathcal{S}_0}$ by $ \widehat{\lambda}_{\mathcal{S}_0}=
    \frac{\mathrm{tr}(\widehat{\bm\Sigma}^{-1}\widehat{\bm\Gamma}_{\widehat{\mathcal{S}}_0}\widehat{\bm\Sigma}^{-1})}{
(1+r)\mathrm{tr}(\widehat{\bm\Sigma}^{-1}\widehat{\bm M}_{\widehat{\mathcal{S}}_0}\widehat{\bm\Sigma}^{-1})}$.
\STATE Employ the gradient descent to minimize $\mathcal L_{\text{FPPI}}(\bm\theta \,|\,\widehat{\lambda}_{\widehat{\mathcal{S}}_0},\widehat{\mathcal{S}}_0)$.
\STATE \textbf{Return} The minimizer of $\mathcal L_{\text{FPPI}}(\bm\theta \,|\,\widehat{\lambda}_{\mathcal{S}_0},\widehat{\mathcal{S}}_0)$ as defined in (\ref{FPPI:GLM_Esti}).
\end{algorithmic}
\end{algorithm}

\begin{corollary}
\label{Coro:GLM_Inference}
Let $\widehat{\bm \theta}_{\mathrm{FPPI}}(\widehat \lambda_{\mathcal{S}_0},\widehat{\mathcal{S}}_0)$ be the estimator from Algorithm \ref{alg:filtered-fppi-init}. Suppose that $\psi_n \gtrsim (n \wedge \alpha_n)$. Under Assumptions \ref{Ass:SubGaussian}-\ref{ass:PD} with
$\mathbb{P}_{\mathbf{X}}\left(\widehat{\mathcal{S}}_0 \Delta \mathcal{S}_0\right) = o_p\left(n^{-1/2}\right)$, we have, as $\frac{n}{N} \to r \in (0,\infty)$,
    \begin{align*}
    \sqrt{n}\left(\widehat{\bm \theta}_{\mathrm{FPPI}}(\widehat \lambda_{\mathcal{S}_0},\widehat{\mathcal{S}}_0) - \bm \theta^\star
        \right) \xrightarrow{d} \mathcal N\Big(
\bm 0,\;
\bm \Sigma^{-1}
\big(
\bm\Omega
+(\lambda_{\mathcal{S}_0}^\star)^2(1+r)\bm M_{\mathcal S_0}
-2\lambda_{\mathcal{S}_0}^\star \bm\Gamma_{\mathcal S_0}
\big)
\bm \Sigma^{-1}
\Big),
    \end{align*}
    where $\bm{\Sigma}$, $\bm{\Omega}$, $\bm{M}_{\mathcal{S}_0}$, and $\bm{\Gamma}_{\mathcal{S}_0}$ are as defined in Theorem \ref{thm:FPPI_GLM_final}.
\end{corollary}

Corollary \ref{Coro:GLM_Inference} establishes the asymptotic distribution of the proposed FPPI estimator for GLMs and shows that it achieves higher efficiency than the PPI\texttt{++} estimator. Consistent with Corollary \ref{Coro:Mean_Inference}, the conditions $\psi_n \gtrsim (n \wedge \alpha_n)$ and $\mathbb{P}_{\mathbf{X}} \left(\widehat{\mathcal{S}}_0 \,\Delta\, \mathcal{S}_0\right)=o_p(n^{-1/2})$ are primarily required because the entire labeled dataset is used to construct the plug-in components. If a data-splitting strategy is adopted instead, both requirements can be eliminated.

\section{Simulation}
\label{Sec:Exp}

\textbf{Scenario I: Mean Estimation.} In this experiment, we investigate the performance of the mean estimator that only uses $\mathcal{D}_L$, PPI\texttt{++}, and FPPI in a discrete covariate setting. Specifically, the covariate $X$ takes four possible values $\{0,1,2,3\}$, each with probability $0.25$. The response is generated according to $Y = m(X) + \varepsilon$, where $(m(0),\ldots,m(3)) = (-2,-1,5,10)$ and $\varepsilon \sim \mathcal{N}(0,1)$. This setup yields the target parameter $\theta^\star = \mathbb{E}(Y) = 3$. To examine how different auxiliary models influence the filtered estimator, we consider three prediction functions: (1) $(f_1(0),\ldots,f_1(3)) = (1,1,-5,3)$, (2) $(f_2(0),\ldots,f_2(3)) = (-1,1,-1,1)$, and (3) $(f_3(0),\ldots,f_3(3)) = (-5,1,1,3)$. Their covariances with $Y$ are $0.5$, $1.5$, and $11$, respectively, indicating increasing predictive accuracy. Following the filtered estimation framework, the corresponding filtered regions are $\{3\}$, $\{0,3\}$, and $\{0,2,3\}$, respectively, reflecting the varying quality of the prediction functions.

For each $f_j$, we consider labeled sample sizes $n \in \{50 \times i: i \in [10]\}$ and unlabeled sample sizes $N \in \{10^4, 2 \times 10^4\}$. Each $(n, N)$ configuration is replicated $10^3$ times to estimate the empirical mean squared error. As a benchmark, we also include the classical sample mean estimator $\widehat{\theta}_0 = n^{-1}\sum_{i=1}^n y_i$. The proposed FPPI estimator is constructed using the estimated region $\widehat{\mathcal{S}}_0$ defined in (\ref{Esti:Region}) and the optimal $\lambda^\star_{\mathcal{S}_0}$ is estimated based on available data. The experimental results are summarized in Figure~\ref{fig:Scenario1_Discrete}.

\begin{figure}[ht!]
    \centering
    \begin{subfigure}[b]{0.328\textwidth}
        \includegraphics[width=\textwidth]{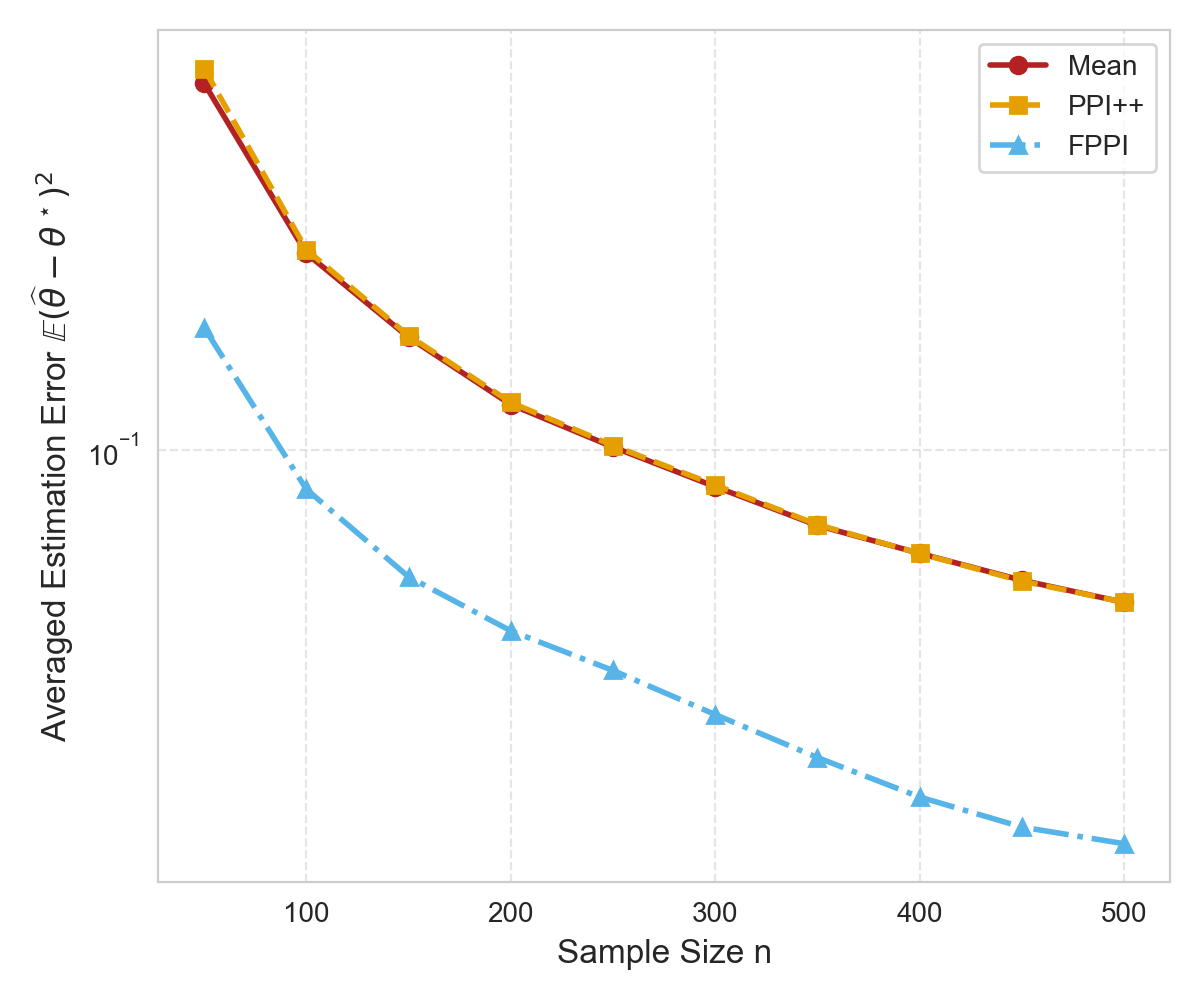}
        \caption{$f_1(x)$ and $N=10^4$}
    \label{F1_f1}
    \end{subfigure}
    \hfill
    \begin{subfigure}[b]{0.328\textwidth}
        \includegraphics[width=\textwidth]{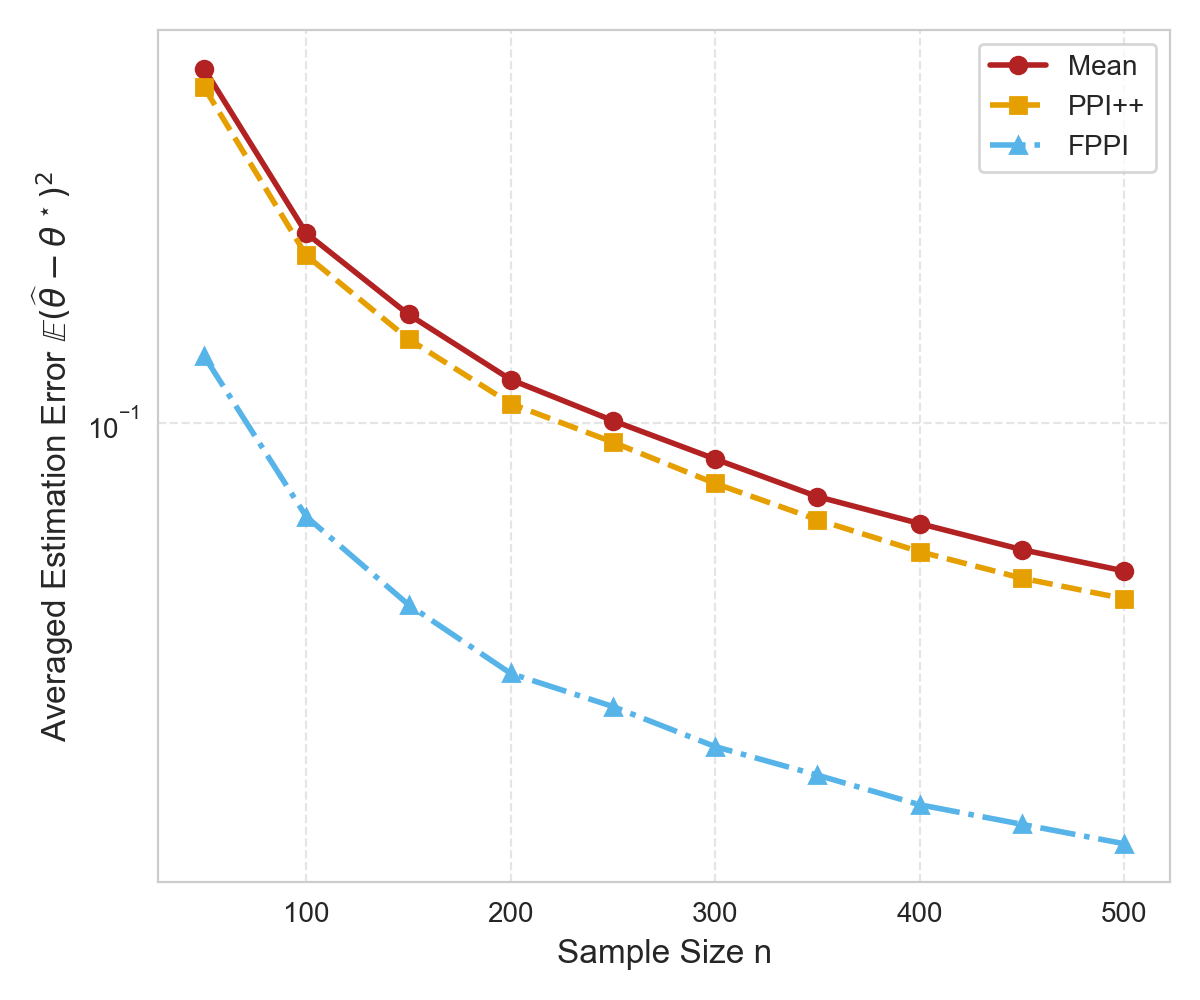}
        \caption{$f_2(x)$ and $N=10^4$}
    \end{subfigure}
    \hfill
    \begin{subfigure}[b]{0.328\textwidth}
        \includegraphics[width=\textwidth]{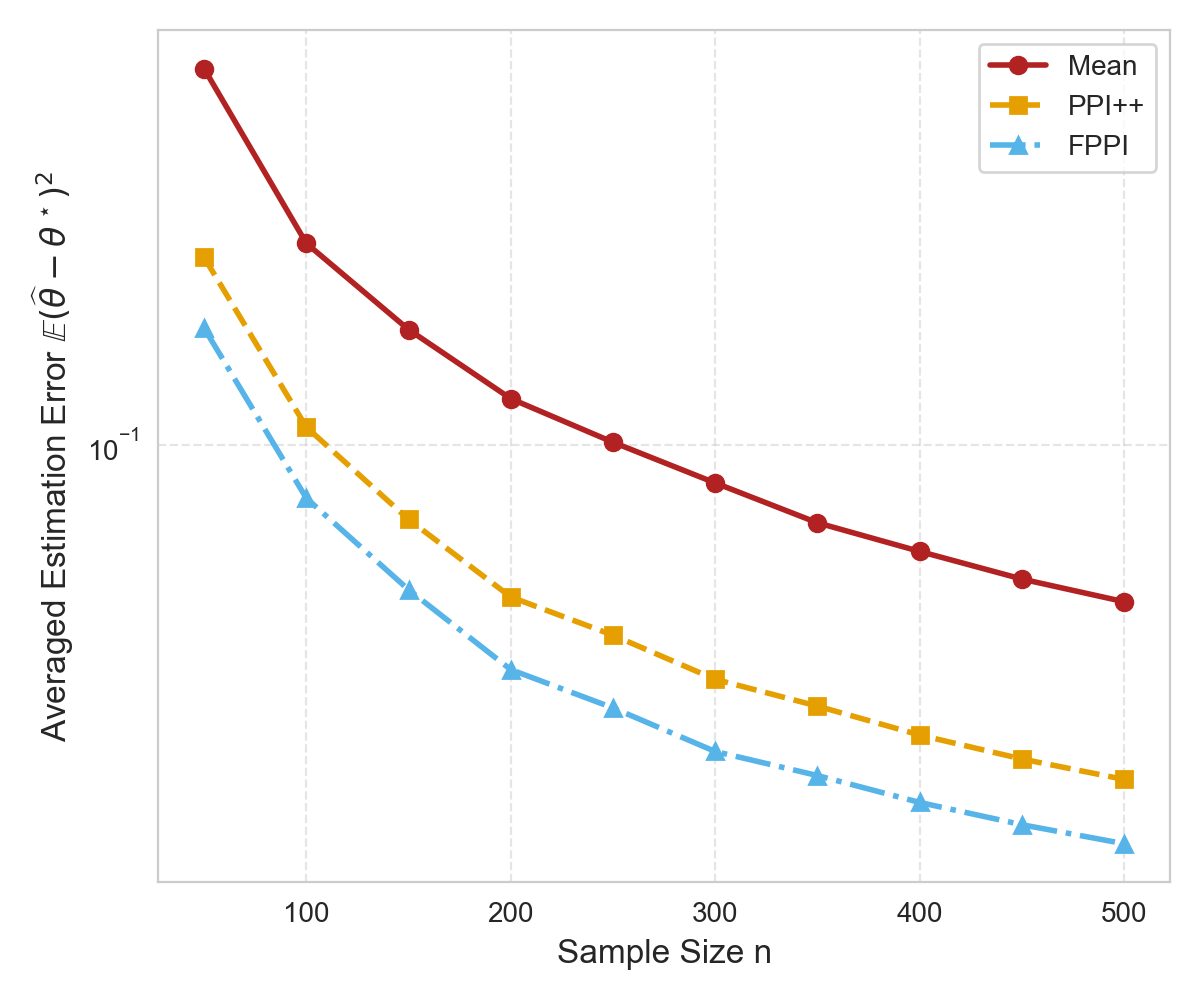}
        \caption{$f_3(x)$ and $N=10^4$}
    \end{subfigure}
    \begin{subfigure}[b]{0.328\textwidth}
        \includegraphics[width=\textwidth]{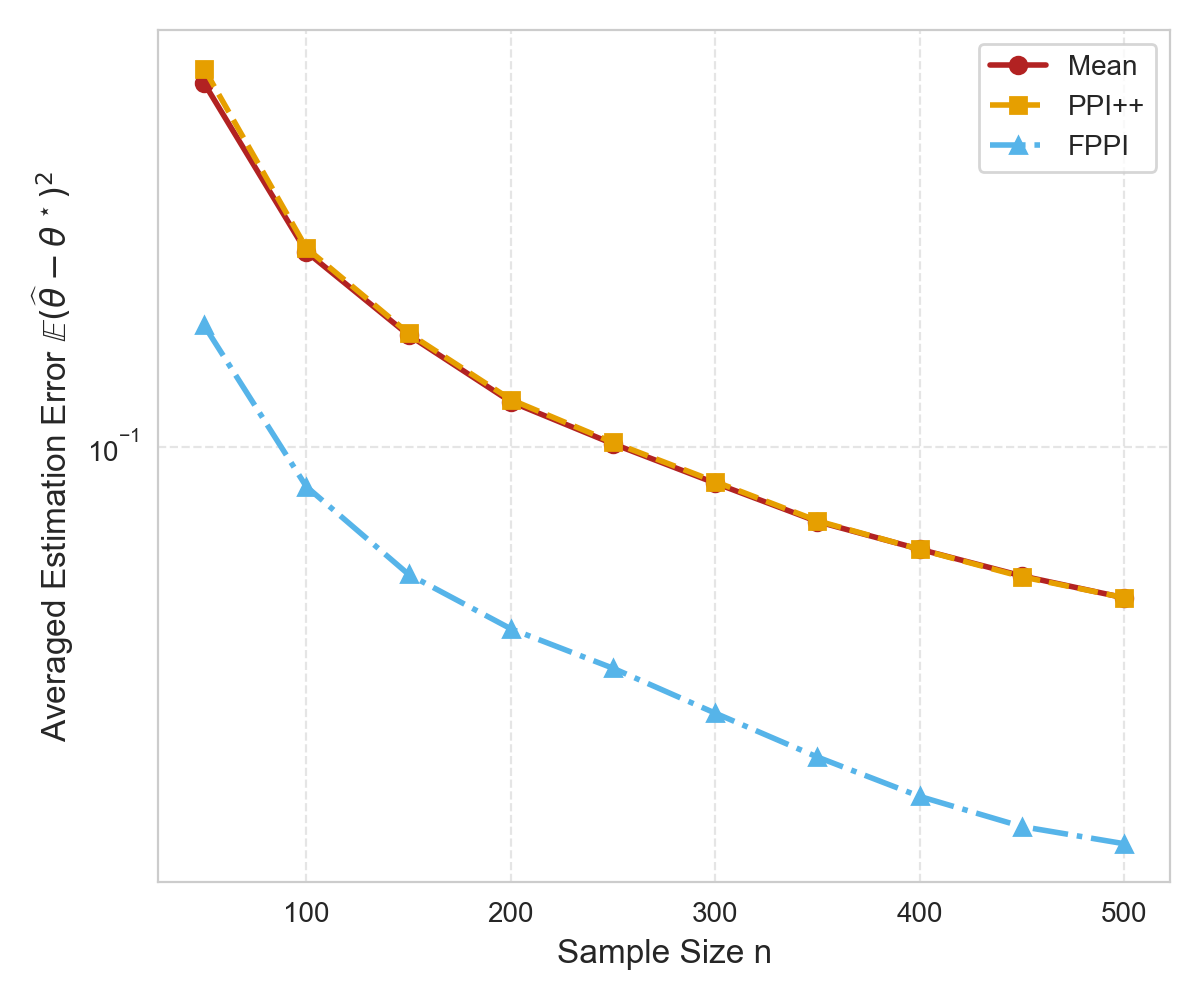}
        \caption{$f_1(x)$ and $N=2 \times 10^4$}
                \label{F1_f2}
    \end{subfigure}
    \hfill
    \begin{subfigure}[b]{0.328\textwidth}
        \includegraphics[width=\textwidth]{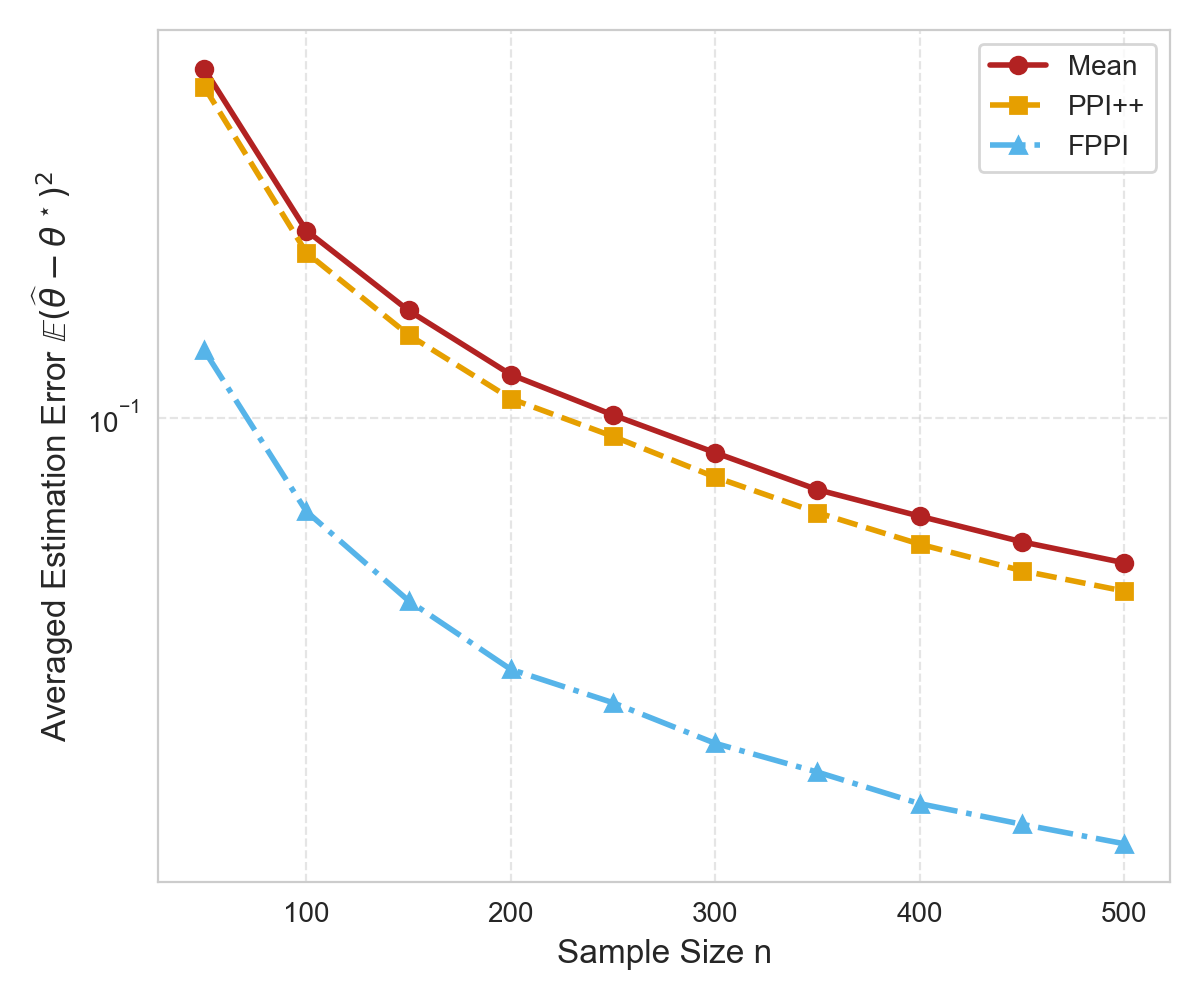}
        \caption{$f_2(x)$ and $N=2 \times 10^4$}
    \end{subfigure}
    \hfill
    \begin{subfigure}[b]{0.328\textwidth}
        \includegraphics[width=\textwidth]{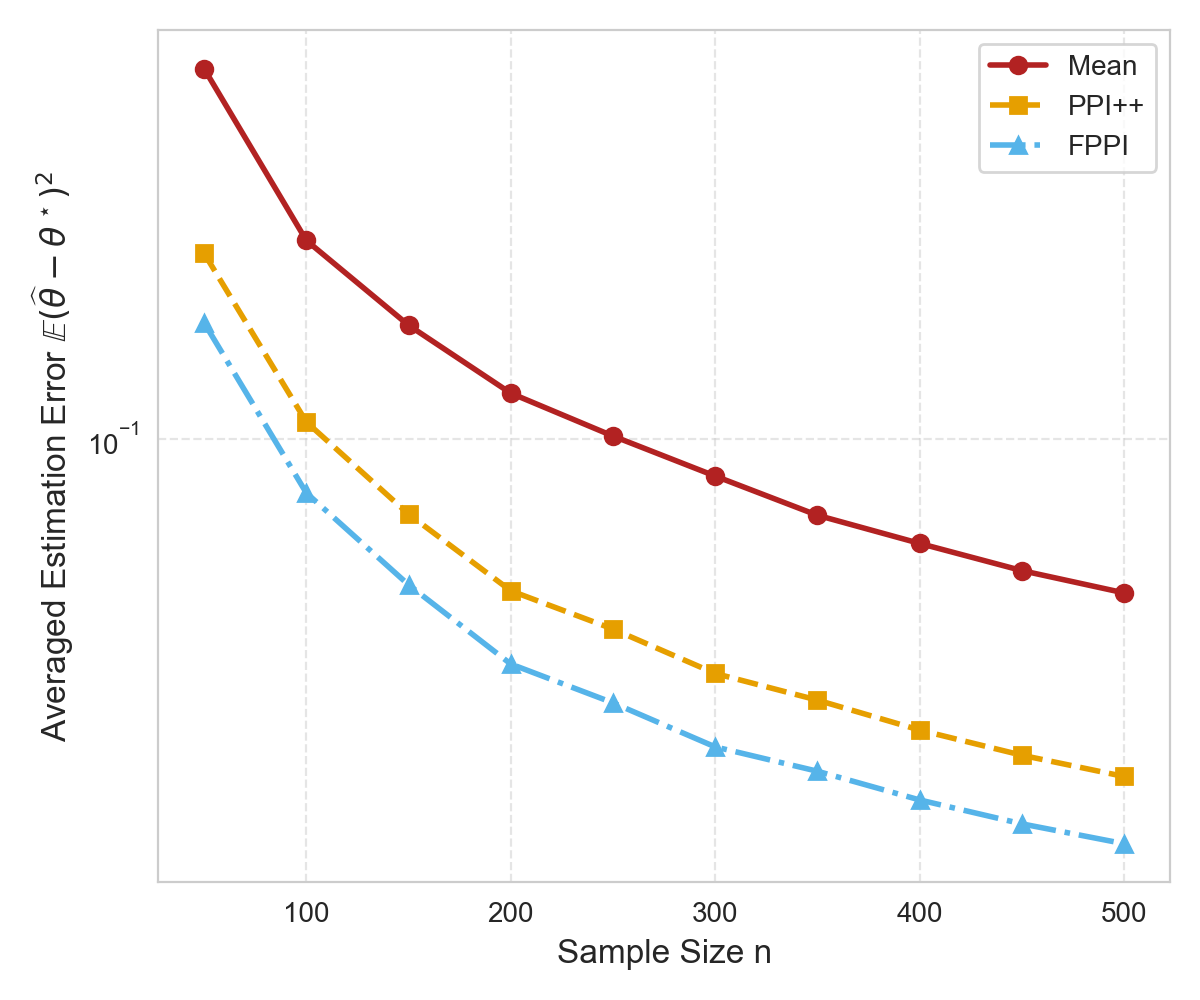}
        \caption{$f_3(x)$ and $N=2 \times 10^4$}
    \end{subfigure}
\caption{\textbf{Scenario I}: Monte Carlo simulation results showing the estimation errors (or variances) of Classic mean, PPI\texttt{++}, and FPPI estimators for different prediction functions and unlabeled sample sizes $N$.}
    \label{fig:Scenario1_Discrete}
\end{figure}

The results in Figure~\ref{fig:Scenario1_Discrete} show that, for mean estimation, FPPI effectively leverages auxiliary predictions that are positively aligned with the target signal by incorporating them in a selective manner. In contrast, PPI\texttt{++} aggregates auxiliary information globally and is therefore more susceptible to variance inflation when the prediction function is partially misspecified. This mechanism explains why FPPI consistently outperforms PPI\texttt{++} across all panels. Although PPI\texttt{++} is theoretically guaranteed to improve upon the classical estimator when the optimal tuning parameter is known, such improvement cannot be empirically guaranteed in practice, since the tuning parameter~$\lambda$ must be estimated from data (see Figures~\ref{F1_f1} and~\ref{F1_f2}).

\noindent\textbf{Scenario II: Linear Regression.} 
In this experiment, we aim to verify the theoretical results presented in Corollary \ref{Coro:GLM_Inference}. 
Specifically, we consider the model $Y = m(\bm{x}) + \varepsilon$ with $m(\bm{x}) = \sum_{i=1}^4 \sin(x_i) + \sum_{i=1}^4 \cos(x_i)$, $\bm X \sim \mathcal{N}(\bm 0, I_4)$ and $\varepsilon \sim \mathcal{N}(0, 2^2)$. We examine several prediction functions: 
(1) $f_1(\bm x) = \sum_{i=1}^4 [2.5 \sin(x_i)- 0.5 \cos(x_i)]$, 
(2) $f_2(\bm x) = \sum_{i=1}^4 [2.25 \sin(x_i) - 0.25 \cos(x_i)]$, and (3) $f_3(\bm x) = 2\sum_{i=1}^4  \sin(x_i)$. According to Corollary \ref{Coro:GLM_Inference}, the estimated filtered regions are 
$$
\mathcal{S}_{0}^{(j)} = \left\{ \bm x \in \mathbb{R}^4 : \big( \widehat{m}(\bm{x}) - \bm{x}^\top \widehat{\bm \theta} \big) \cdot \big( f_j(\bm x) - \bm{x}^\top \widehat{\bm \theta} \big) > 0 \right\}, \quad j \in [3].
$$
Here, $\widehat{m}(\cdot)$ is obtained via a $K$-nearest neighbor regressor with $K=15$ trained on the labeled dataset $\mathcal{D}_L$, and $\widehat{\bm \theta}$ is the OLS estimator based on $\mathcal{D}_L$. We consider labeled sample sizes $n \in \{100 \times i: i \in [10]\}$ and unlabeled sample size $N \in \{ 10^4,2\times 10^4\}$. Each case is replicated $10^3$ times, and the averaged estimation errors are reported in Figure~\ref{fig:Scenario_II_LR}.

\begin{figure}[h]
    \centering
    \begin{subfigure}[b]{0.328\textwidth}
        \includegraphics[width=\textwidth]{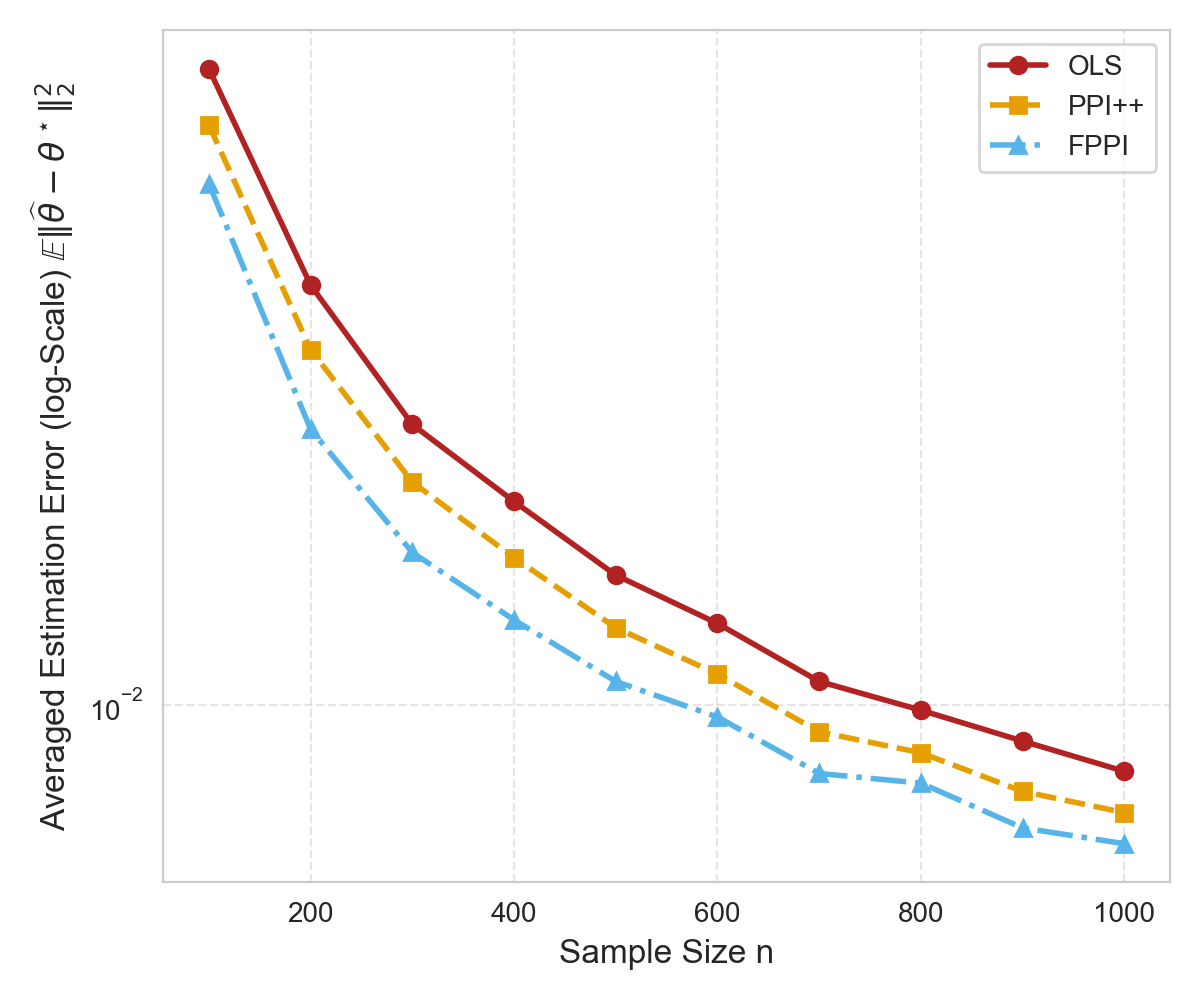}
        \caption{$f_1(x)$ and $N=10^4$}
    \end{subfigure}
    \hfill
    \begin{subfigure}[b]{0.328\textwidth}
        \includegraphics[width=\textwidth]{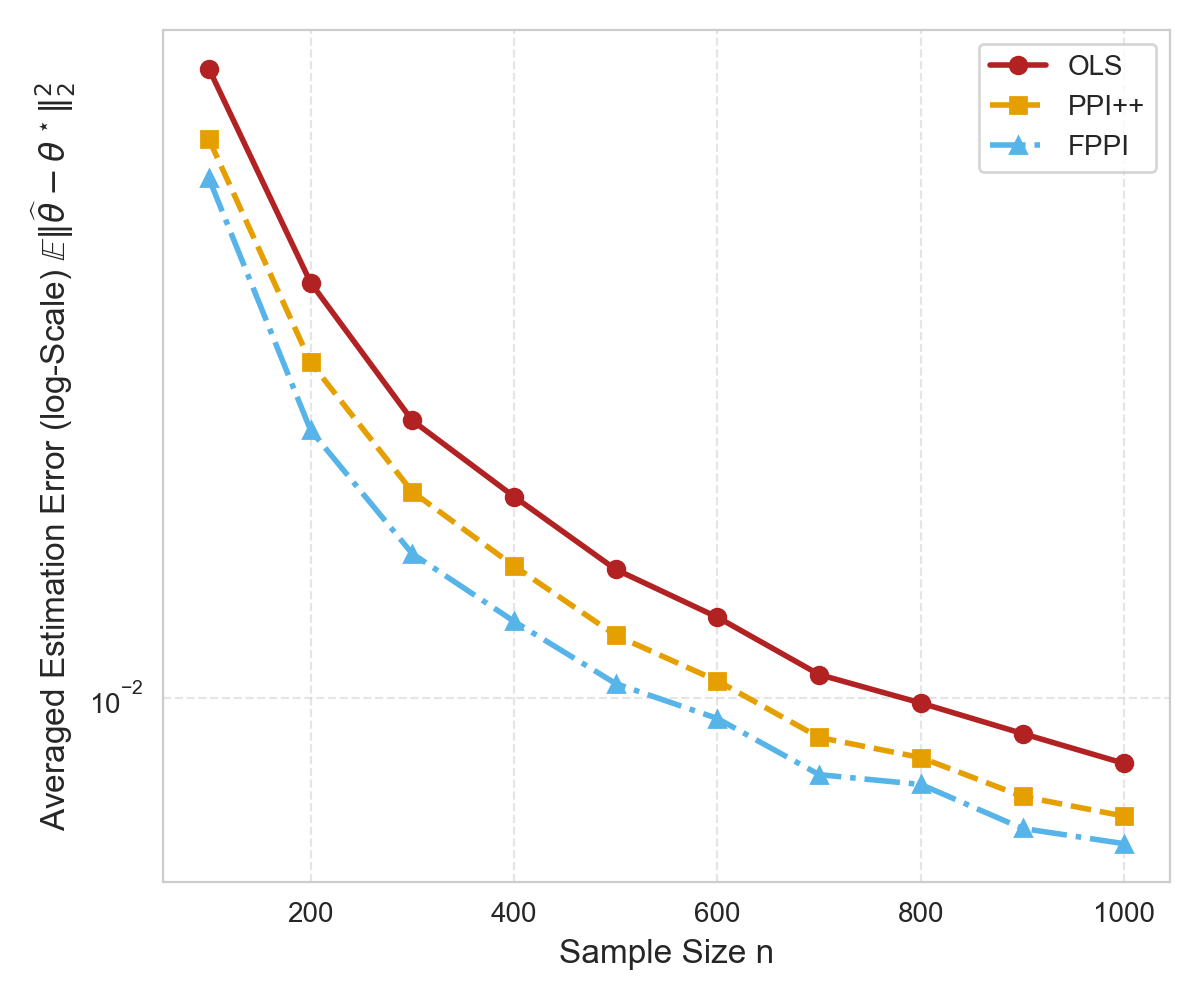}
        \caption{$f_2(x)$ and $N=10^4$}
    \end{subfigure}
    \hfill
    \begin{subfigure}[b]{0.328\textwidth}
        \includegraphics[width=\textwidth]{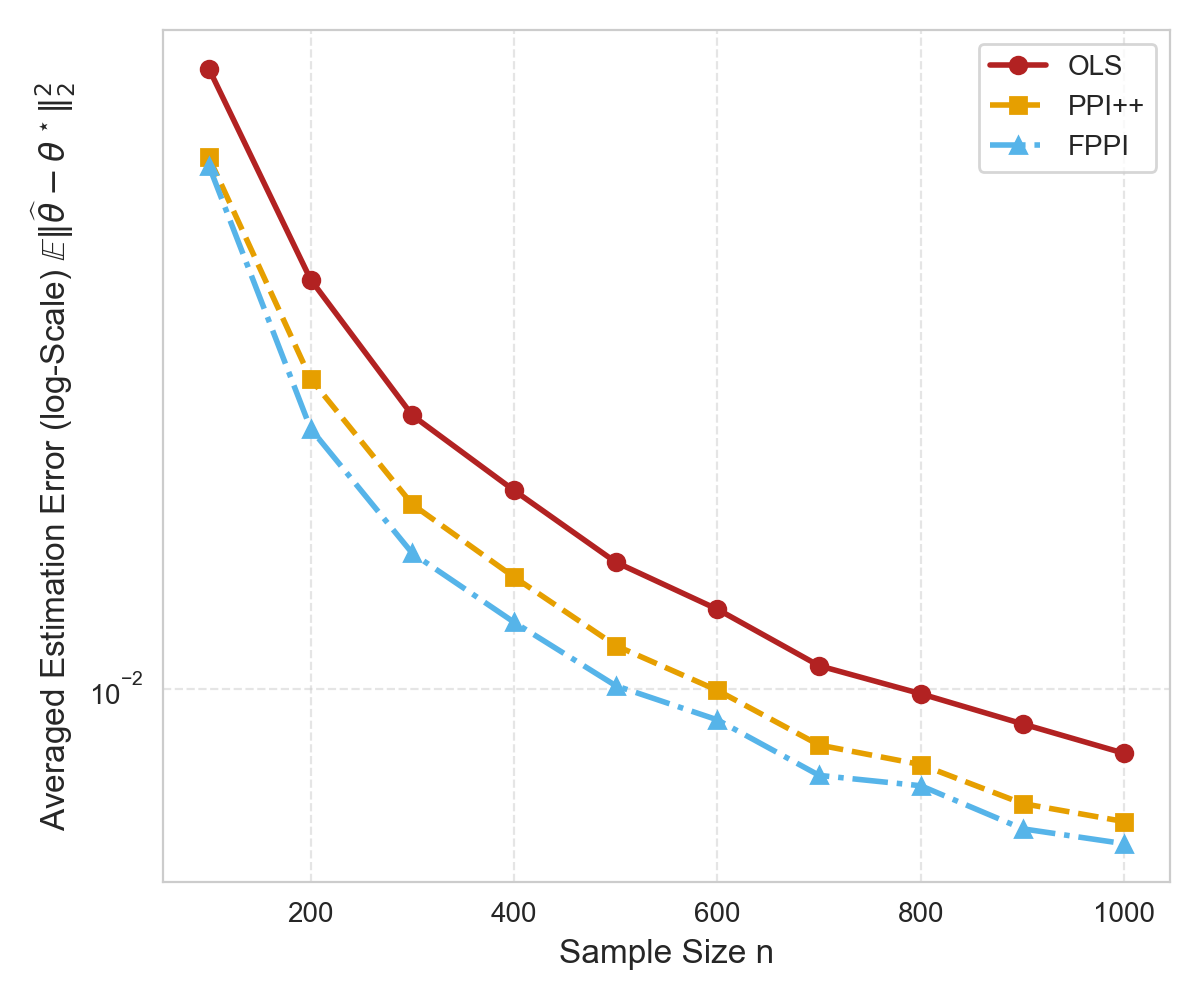}
        \caption{$f_3(x)$ and $N=10^4$}
    \end{subfigure}
    \begin{subfigure}[b]{0.328\textwidth}
        \includegraphics[width=\textwidth]{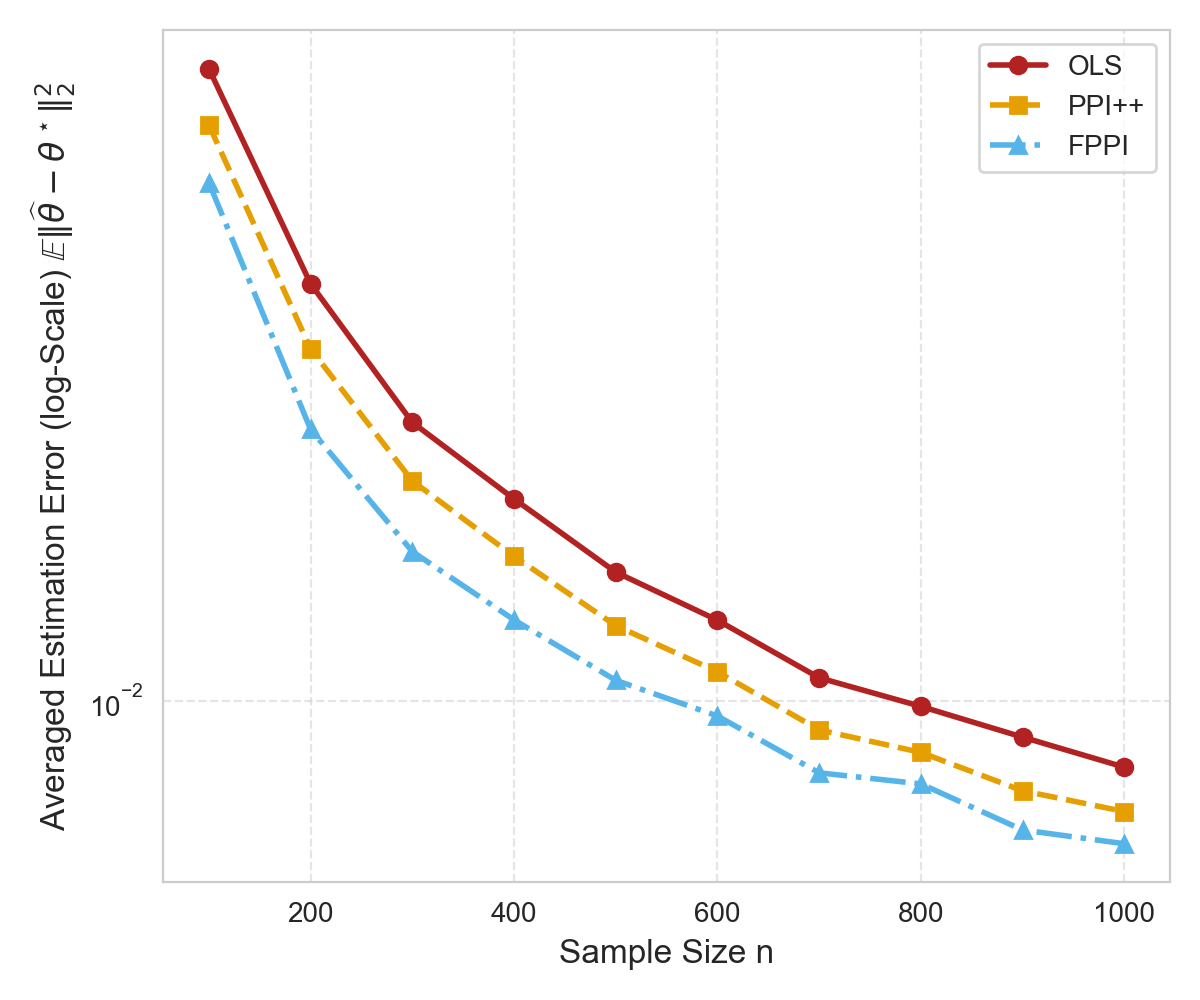}
        \caption{$f_1(x)$ and $N=2 \times 10^4$}
    \end{subfigure}
    \hfill
    \begin{subfigure}[b]{0.328\textwidth}
        \includegraphics[width=\textwidth]{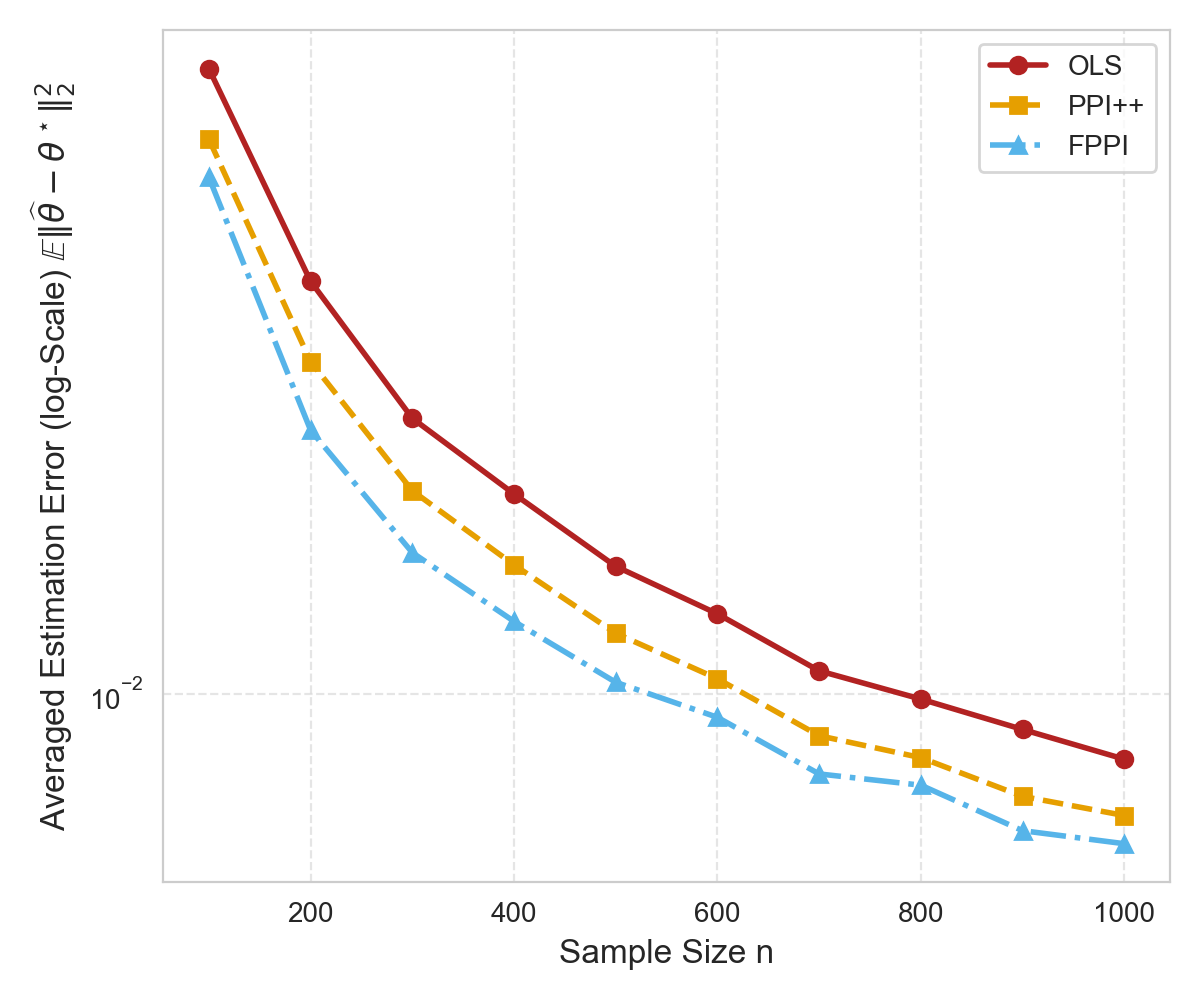}
        \caption{$f_2(x)$ and $N=2 \times 10^4$}
    \end{subfigure}
    \hfill
    \begin{subfigure}[b]{0.328\textwidth}
        \includegraphics[width=\textwidth]{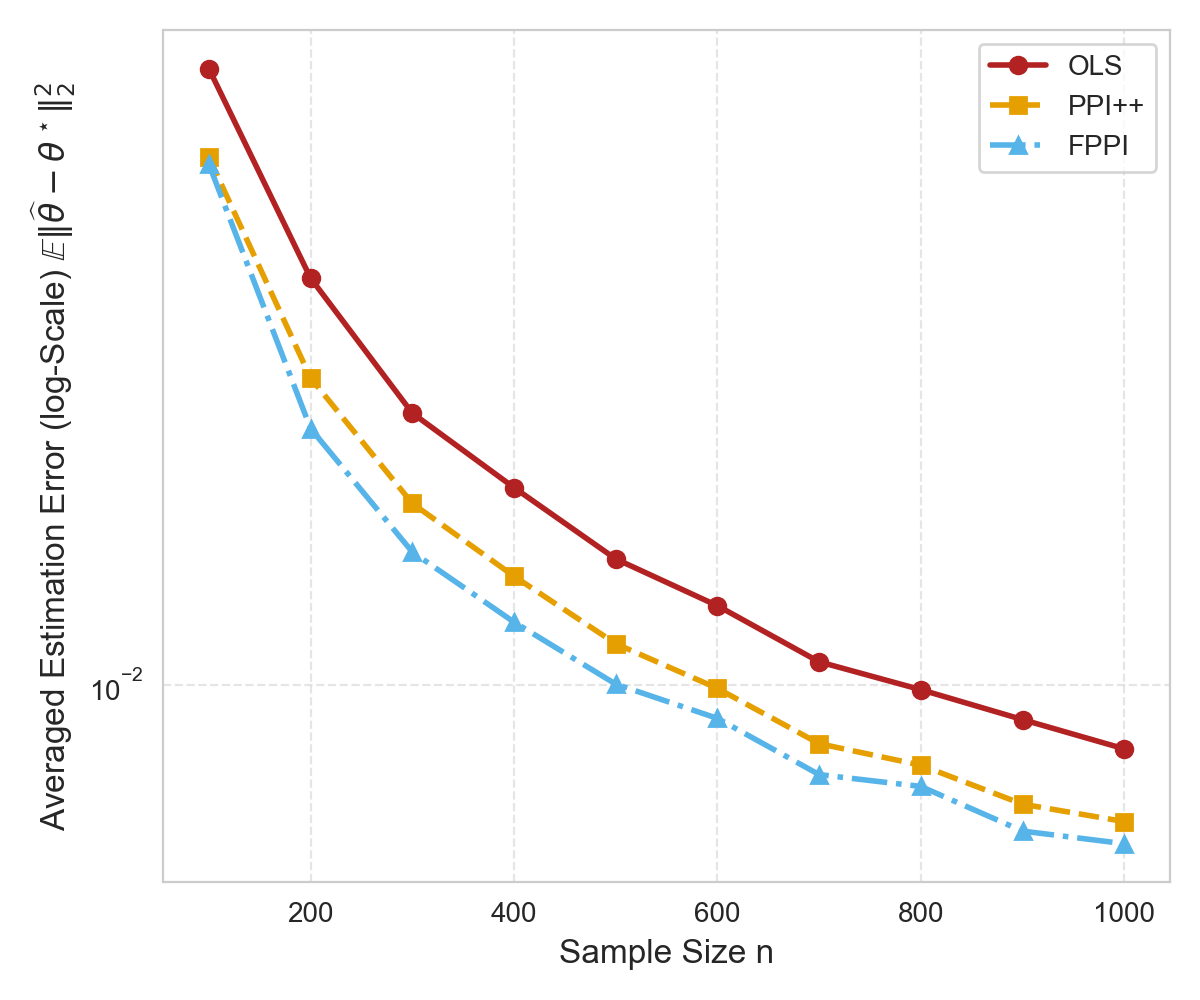}
        \caption{$f_3(x)$ and $N=2 \times 10^4$}
    \end{subfigure}
\caption{\textbf{Scenario II:} Monte Carlo simulation results illustrating the estimation errors of the classical OLS, PPI\texttt{++}, and FPPI estimators on a logarithmic scale, under different prediction functions and unlabeled sample sizes $N$.}
    \label{fig:Scenario_II_LR}
\end{figure}

As can be seen in Figure~\ref{fig:Scenario_II_LR}, both PPI\texttt{++} and FPPI uniformly outperform the classical OLS estimator across all sample sizes $n$, confirming that incorporating auxiliary predictions can substantially reduce estimation error. More importantly, FPPI consistently achieves smaller estimation errors than PPI\texttt{++} in all configurations. This gap is particularly pronounced for smaller $n$ and for more challenging signal functions, indicating that selectively incorporating well-aligned auxiliary information is more effective than global aggregation.

\noindent\textbf{Scenario III: Logistic Regression.} In this experiment, We consider a binary classification task under a logistic regression model with $p=4$, where the covariates are generated from a multivariate normal distribution. The data are generated according to the true regression function $\sigma\big(m(\bm{x}) - 2\big)$. We evaluate three candidate prediction functions $\sigma\big(f_j(\bm{x})\big)$ for $j \in [3]$, as specified in Scenario~II. To approximate the unknown regression function $\mu^\star(\bm X)$, we employ a $k$-nearest neighbors classifier with $k=15$. The resulting estimator is used to construct the filtered region $\widehat{\mathcal S}_0$. As a baseline, we include the standard logistic regression maximum likelihood estimator (MLE) based on $\mathcal{D}_L$ only, implemented using \texttt{scikit-learn}~\citep{pedregosa2011scikit}. Both FPPI and PPI\texttt{++} are optimized via gradient descent. We adopt the same sample size configurations $(n, N)$ as in Scenario~II. For each $(n, N)$ pair, all methods are evaluated over $10^3$ independent Monte Carlo replications, and performance metrics are averaged across replications and reported in Figure \ref{fig:Scenario_III_LC}.

\begin{figure}[h]
    \centering
    \begin{subfigure}[b]{0.328\textwidth}
        \includegraphics[width=\textwidth]{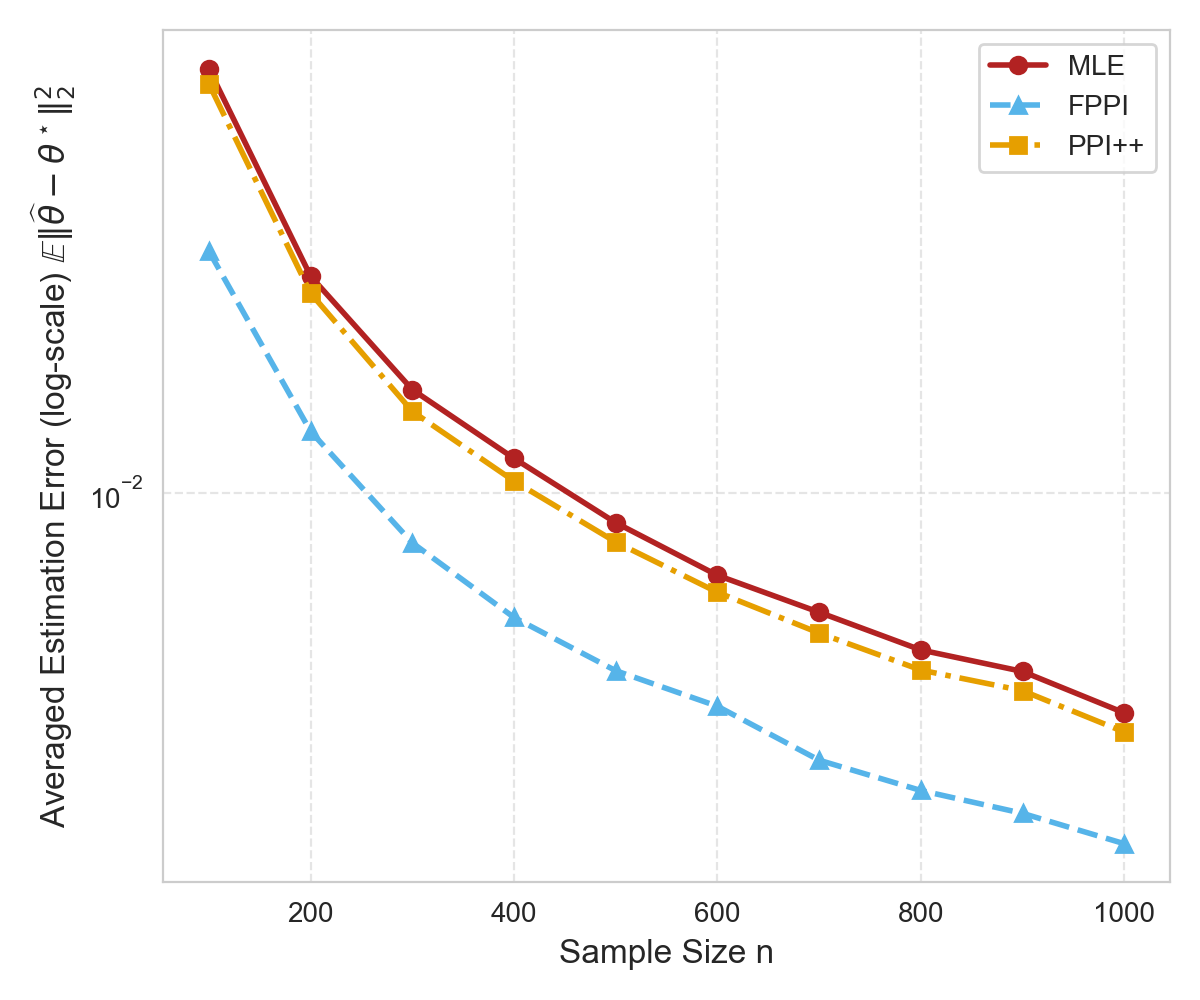}
        \caption{$f_1(x)$ and $N=10^4$}
    \end{subfigure}
    \hfill
    \begin{subfigure}[b]{0.328\textwidth}
        \includegraphics[width=\textwidth]{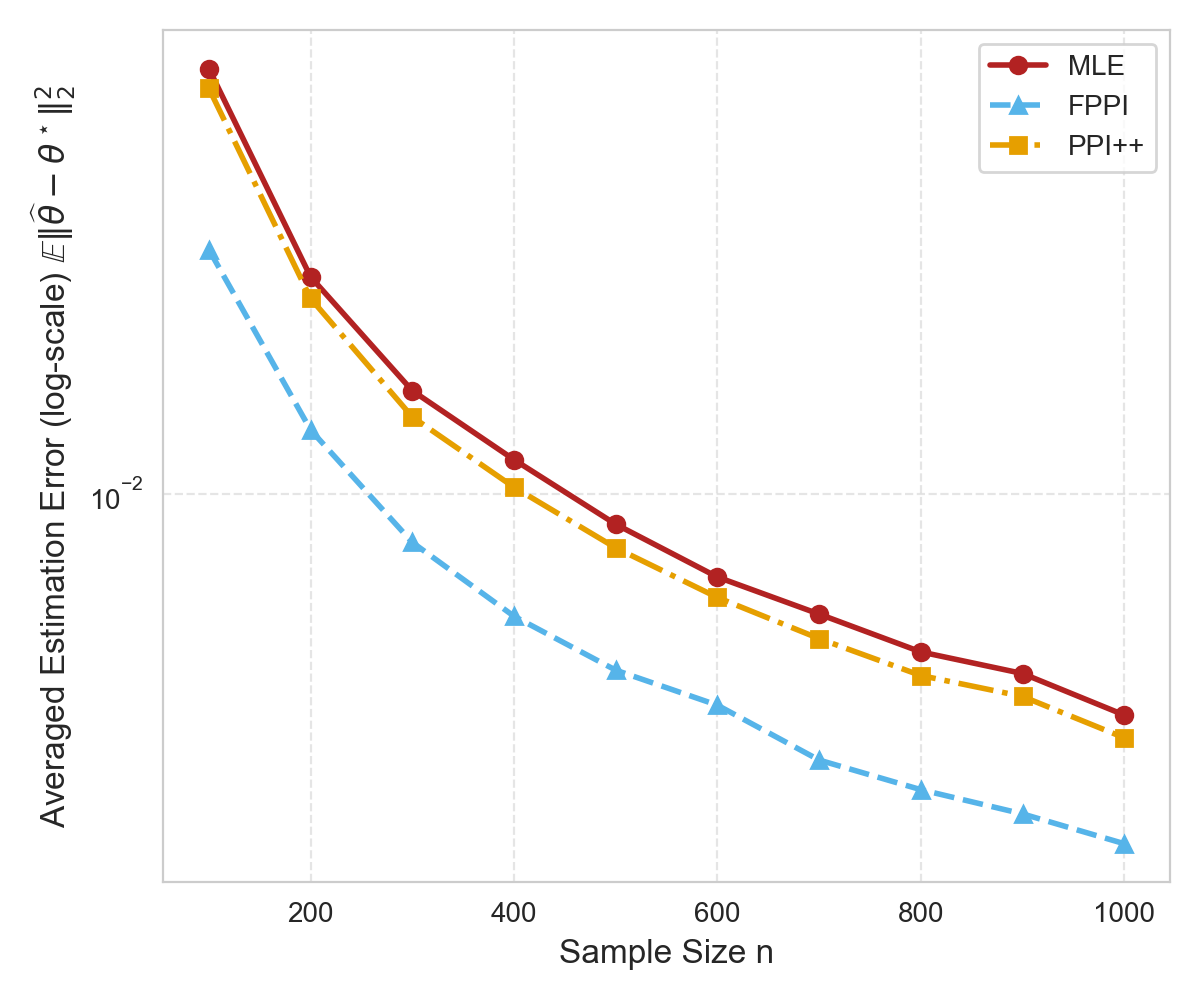}
        \caption{$f_2(x)$ and $N=10^4$}
    \end{subfigure}
    \hfill
    \begin{subfigure}[b]{0.328\textwidth}
        \includegraphics[width=\textwidth]{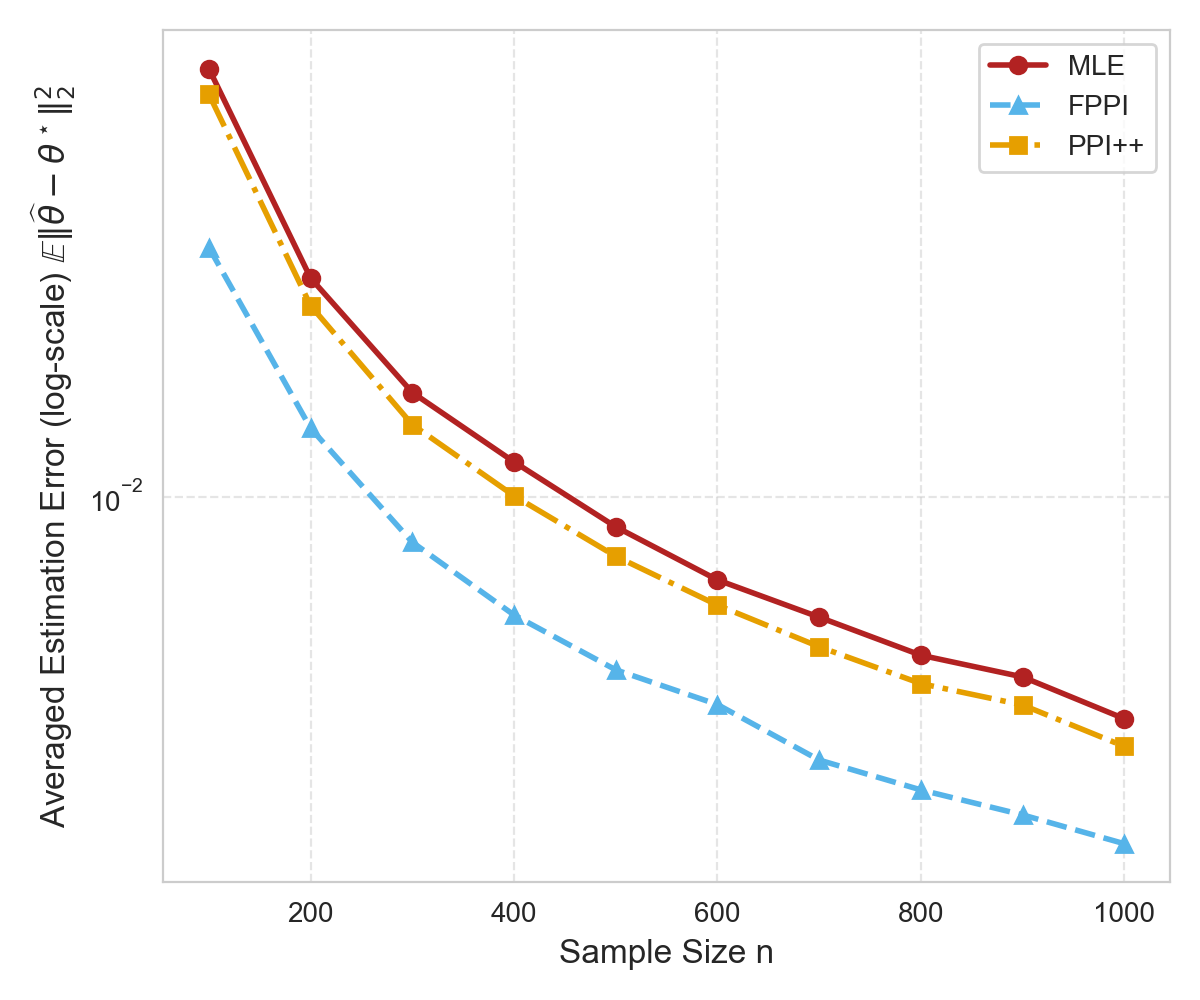}
        \caption{$f_3(x)$ and $N=10^4$}
    \end{subfigure}
    \begin{subfigure}[b]{0.328\textwidth}
        \includegraphics[width=\textwidth]{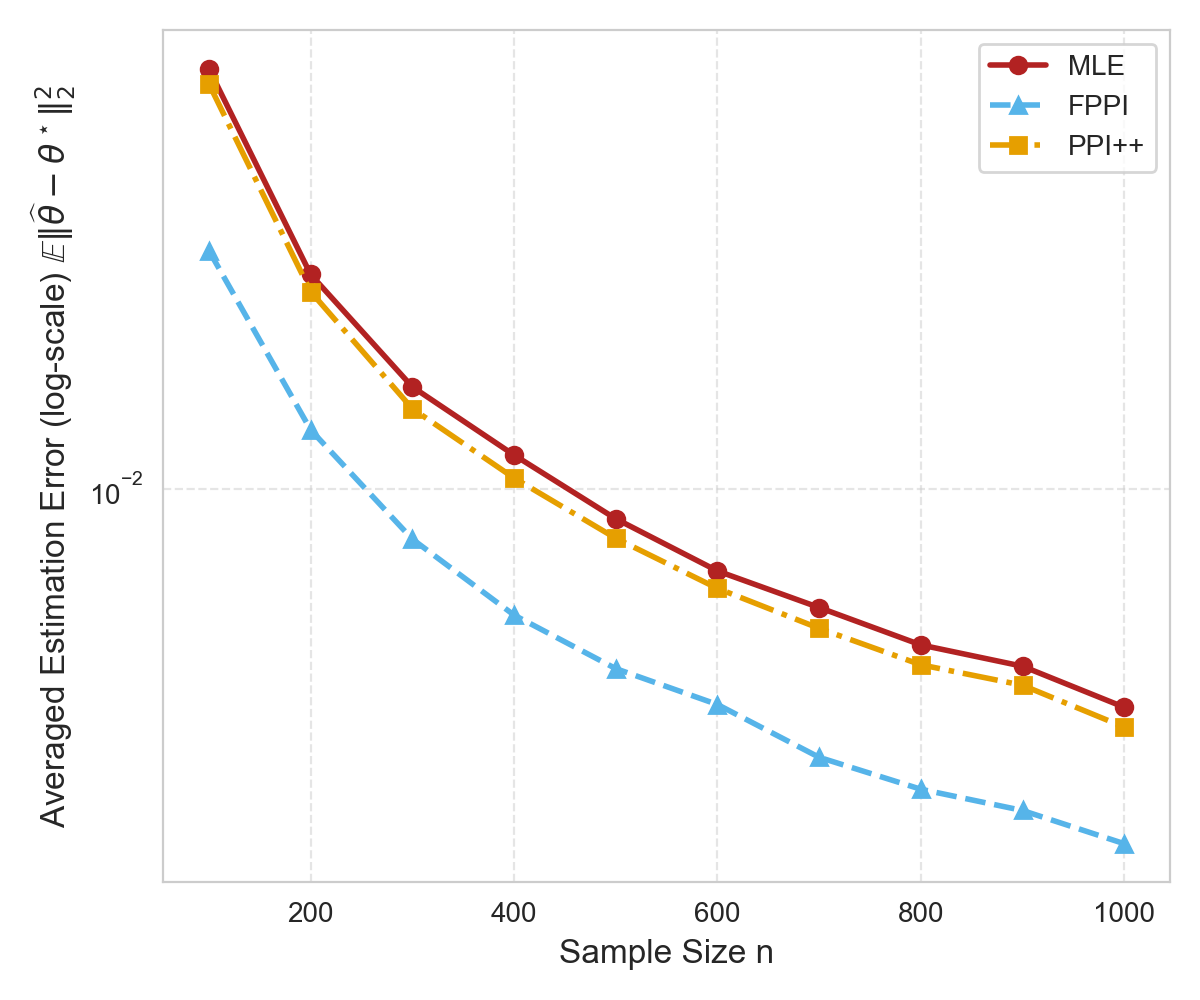}
        \caption{$f_1(x)$ and $N=2 \times 10^4$}
    \end{subfigure}
    \hfill
    \begin{subfigure}[b]{0.328\textwidth}
        \includegraphics[width=\textwidth]{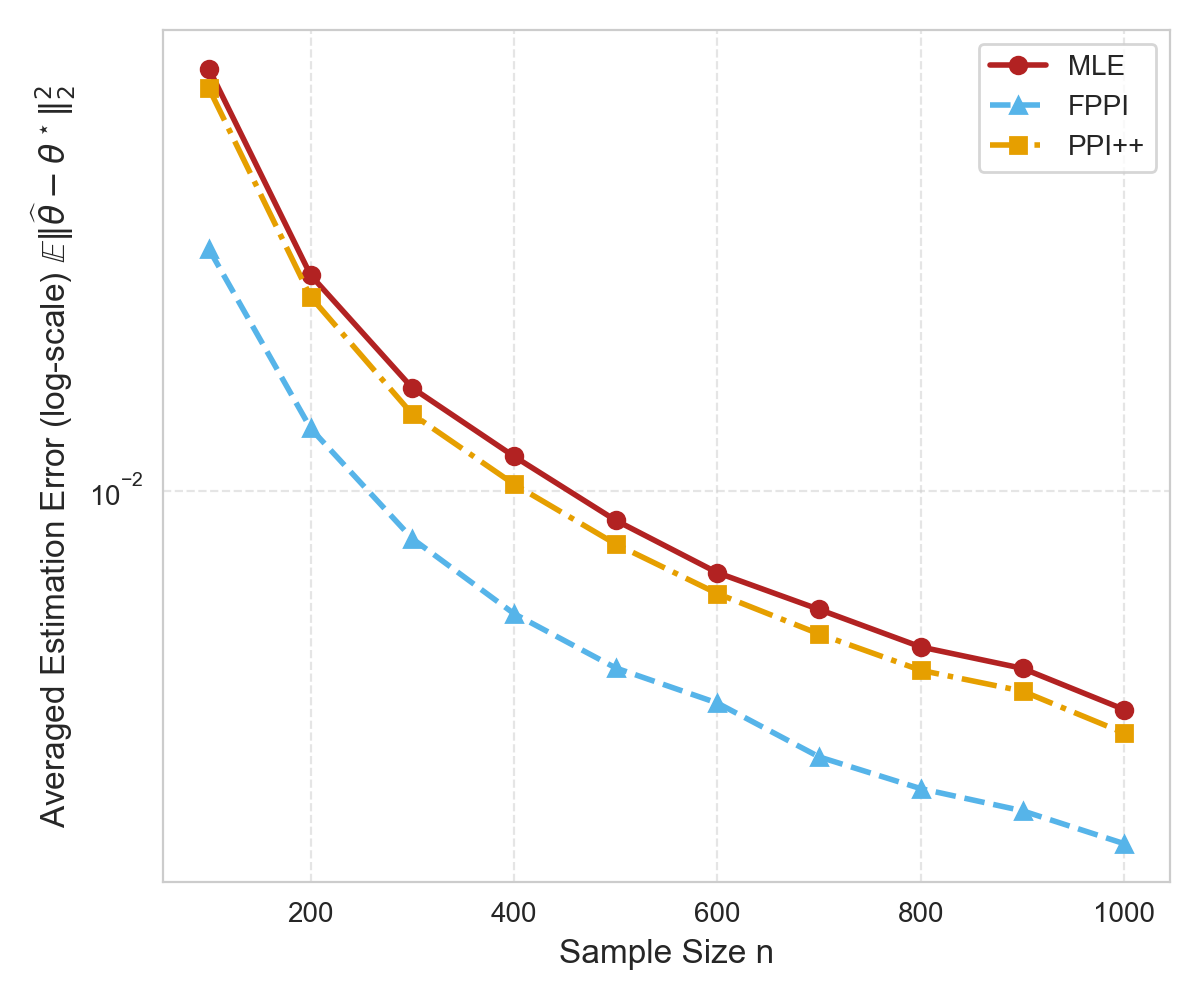}
        \caption{$f_2(x)$ and $N=2 \times 10^4$}
    \end{subfigure}
    \hfill
    \begin{subfigure}[b]{0.328\textwidth}
        \includegraphics[width=\textwidth]{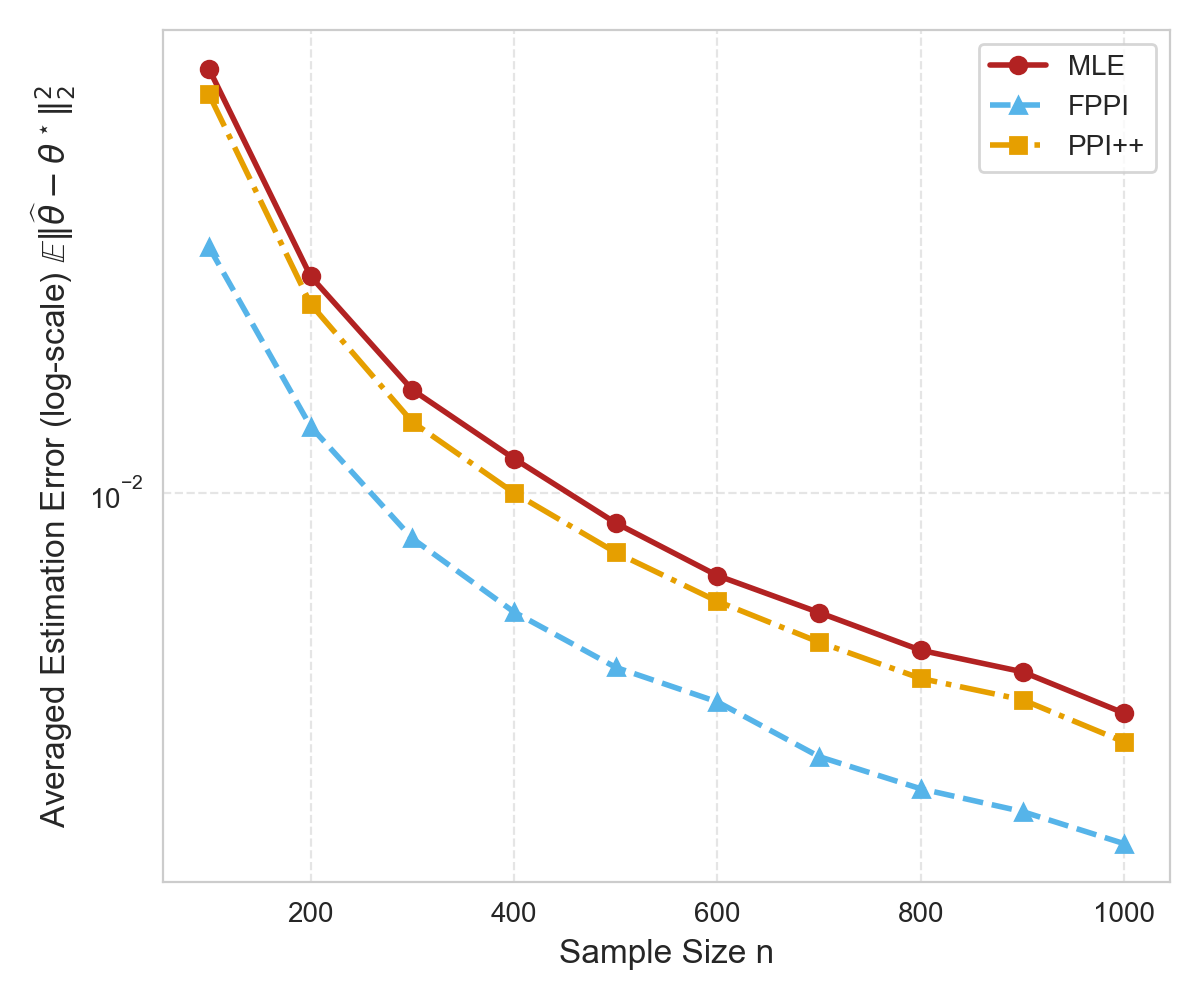}
        \caption{$f_3(x)$ and $N=2 \times 10^4$}
    \end{subfigure}
\caption{\textbf{Scenario III:} Monte Carlo simulation results illustrating the estimation errors of the classical MLE, PPI\texttt{++}, and FPPI estimators for logistic regresion on a logarithmic scale, under different prediction functions and unlabeled sample sizes $N$.}
    \label{fig:Scenario_III_LC}
\end{figure}

The results of Scenario III are fully consistent with those of Scenarios I and II, demonstrating that the proposed FPPI estimator attains smaller estimation errors than both the baseline MLE and the PPI\texttt{++} estimator. These findings further substantiate the underlying principle that not all data are equally informative: selectively incorporating observations from regions where the prediction function aligns well with the true regression function can lead to substantial performance improvements.

\section{An Application to LLM Evaluation}

Current chatbot evaluation platforms, such as LMArena, rely almost exclusively on human judgments to compare large language models (LLMs). In this section, we investigate whether a LLM can serve as an effective auxiliary evaluator for estimating model win rates in human preference comparisons. Specifically, we aim to determine whether LLM-generated preference labels—applied to only a small subset of human-labeled data—can accurately recover model rankings, thereby reducing reliance on costly and time-consuming human evaluations. The motivation for this approach is that human preferences are both noisy and expensive to collect, whereas LLM judges are inexpensive and weakly correlated. This setting aligns precisely with the regime in which the proposed method is most effective, as LLM-generated text can be misleading, as documented in prior work \citep{zhou2026detecting}.

We illustrate this idea using the LMArena Human Preference 55K dataset \citep{chiang2024chatbot}, which is publicly available at Huggingface\footnote{\url{https://huggingface.co/datasets/lmarena-ai/arena-human-preference-55k}}. The dataset contains approximately 55{,}000 pairwise human comparisons between LLMs. Each record includes a prompt (\texttt{prompt}), two models being compared (\texttt{model\_a}, \texttt{model\_b}), their respective responses (\texttt{response\_a}, \texttt{response\_b}), and a human-labeled outcome indicating whether model~A wins, model~B wins, or the comparison is a tie. In our analysis, we exclude tied outcomes, yielding a dataset of 39{,}716 pairwise comparisons with binary labels.

\begin{figure}[ht]
    \centering
    \includegraphics[scale=0.4]{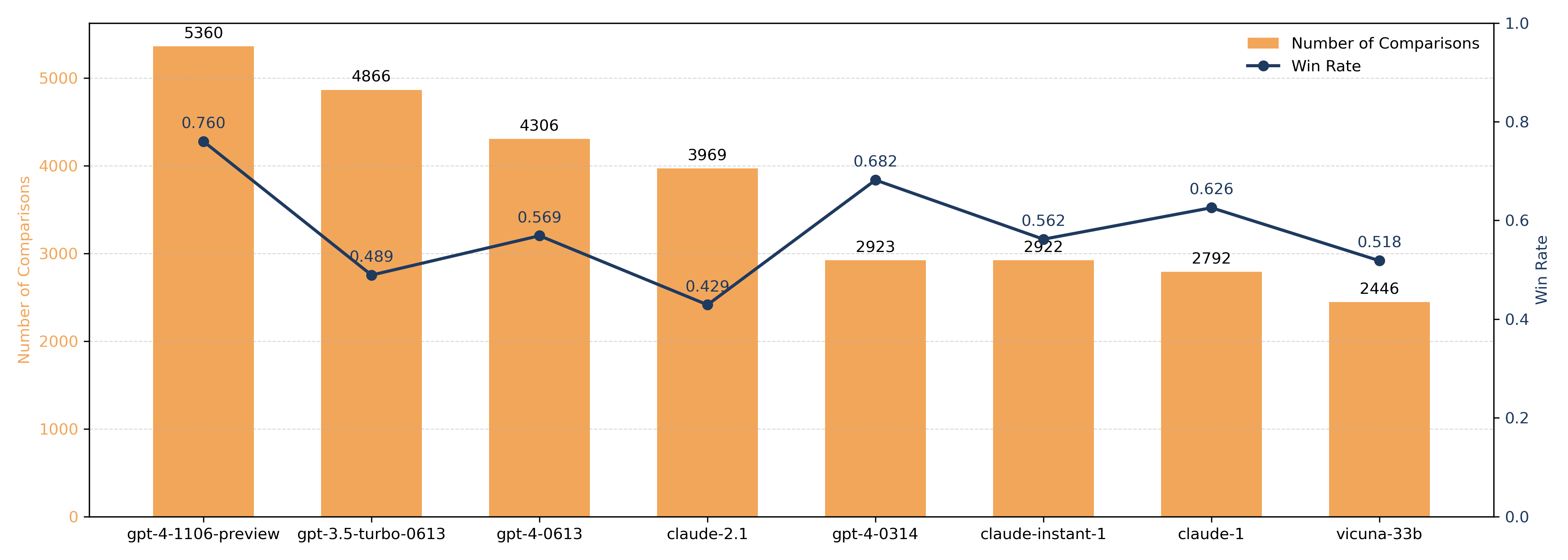}
\caption{Comparison statistics for the eight most frequently evaluated models in the Arena Human Preference dataset, excluding tied outcomes. The win rate is defined as the proportion of wins out of the total number of comparisons, and the number of comparisons indicates how often each model is evaluated. All reported values are rounded to three decimal places.}

    \label{fig:Real1}
\end{figure}

The dataset contains comparisons for 70 language models. For each model $i$, we compute its win rate as
\begin{align*}
    \text{Win-Rate of Model } i = 
    \frac{\text{Number of Wins for Model } i}
         {\text{Total Number of Comparisons Involving Model } i}.
\end{align*}
Figure \ref{fig:Real1} presents the win rates of the eight most frequently compared models. Notably, \texttt{GPT-4-1106-Preview} is both the most frequently compared model and the one achieving the highest win rate. In this real-world application, we focus on this model (\texttt{GPT-4-1106-Preview}) and aim to infer its win rate against other language models. Formally, we are interested in the population win rate defined as 
\begin{align*}
\theta^\star =
\mathbb{P}\Big(
\texttt{GPT-4-1106-Preview} \succ \texttt{Other LMs}
\Big) ,
\end{align*}
where $\succ$ denotes a win. As illustrate in Figure \ref{fig:Real1}, this dataset contains 5360 comparisons involving the model \texttt{GPT-4-1106-Preview} and the win-rate is approximately 0.760. This win rate is of interest to the LLM community because, as more advanced models are released, people are always curious about how much improvement has been achieved compared to previous models. Under this setting, the covariate $\bm{x}$ and label $y$ correspond to $(\texttt{response\_a}, \texttt{response\_b})$ and the indicator of whether \texttt{GPT-4-1106-Preview} wins the comparison, where $a$ or $b$ denotes the response generated by \texttt{GPT-4-1106-Preview}.

\textbf{LLM Labeling.} To obtain LLM-based preference labels, we employ \texttt{GPT-5.2} via the OpenAI API as the evaluation model. For every comparison, we construct a structured evaluation prompt consisting of: (1) \texttt{prompt}, (2) \texttt{response\_a}, and (3) \texttt{response\_b}. The judge is instructed to output exactly one token \texttt{A} or \texttt{B}, indicating which response is better. We evaluate the consistency between human and LLM preferences and obtain an accuracy of 68.34\% on the full dataset, indicating that 68.34\% of LLM-generated labels agree with ground-truth human preferences. For comparisons involving \texttt{GPT-4-1106-Preview}, the agreement rate increases to 72.05\%. The corresponding correlation coefficient is 0.1822, suggesting that, for the target task, the prediction model provides a weak but positive auxiliary signal.

\textbf{Data Filtering Procedure.} The main idea of data filtering is to identify comparisons for which LLM judgments are well aligned with human preferences. To this end, we train binary classifiers to predict whether the LLM-selected winner agrees with the human-selected winner, using only textual information from the prompt and the two model responses. We consider two complementary feature extraction strategies and train a separate $l_2$-regularized logistic regression classifier for each. (1) The first uses TF--IDF representations, where prompts and responses are vectorized separately with English stop-word removal, a minimum document frequency of 5, and vocabulary caps of 500 features for prompts and 2{,}000 features for responses. These features capture surface-level lexical and stylistic patterns that may correlate with agreement. (2) The second strategy uses dense semantic embeddings derived from a pre-trained sentence transformer model (all-mpnet-base-v2\footnote{https://huggingface.co/sentence-transformers/all-mpnet-base-v2}). We encode the prompt and the LLM responses independently and use their concatenated embeddings as input features. These representations capture higher-level semantic relationships beyond token counts. The resulting classifier predicts whether human and LLM preferences agree for a given comparison. We use these predictions as filtering rules. The overall framework of the data filtering process is illustrated in Figure~\ref{fig:realF}.

\begin{figure}[h]
    \centering
    \includegraphics[scale=0.188]{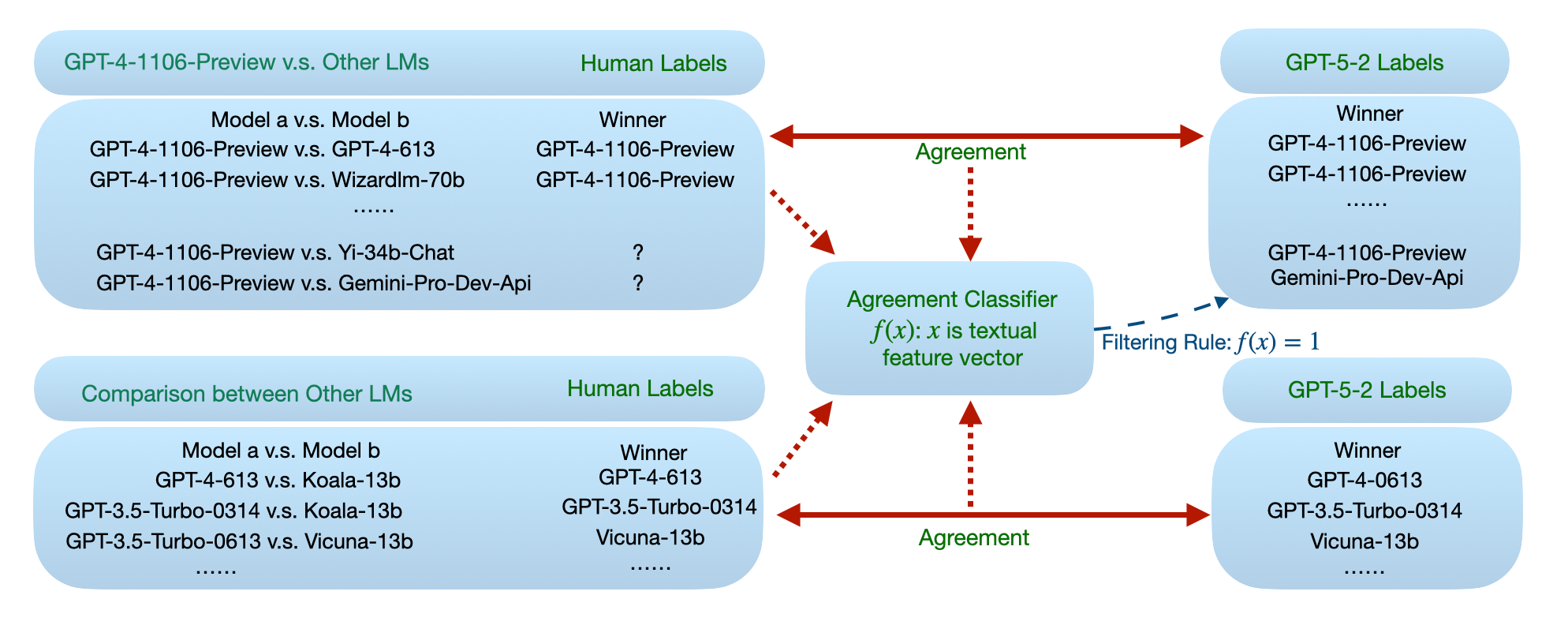}
    \caption{The overall workflow of the data filtering procedure for labels generated by \texttt{GPT-5.2}.}
    \label{fig:realF}
\end{figure}

\textbf{Experimental Setup.} We compare the proposed estimator with the classical sample mean estimator, as well as the PPI and PPI\texttt{++} estimators. We consider settings $(n, N) \in \{100\times i : 1 \le i \le 5\} \times \{10^3,2\times 10^3\}$, where $n$ denotes the number of samples with human labels and $N$ the number of samples without human labels. For each pair $(n, N)$, we repeat the experiment 1{,}000 times using independent random sampling. In each repetition, we randomly select $n$ comparisons from the 5{,}360 total comparisons as the labeled dataset, and another $N$ comparisons as the unlabeled dataset. For each estimator, we compute the squared error $(\widehat{\theta} - \theta^\star)^2$. Performance is evaluated using the average mean squared error (MSE) across repetitions and the experimental results are reported in Figure~\ref{fig:result_real}.

\begin{figure}[h]
    \centering
    \begin{subfigure}[b]{0.444\textwidth}
        \includegraphics[scale=0.385]{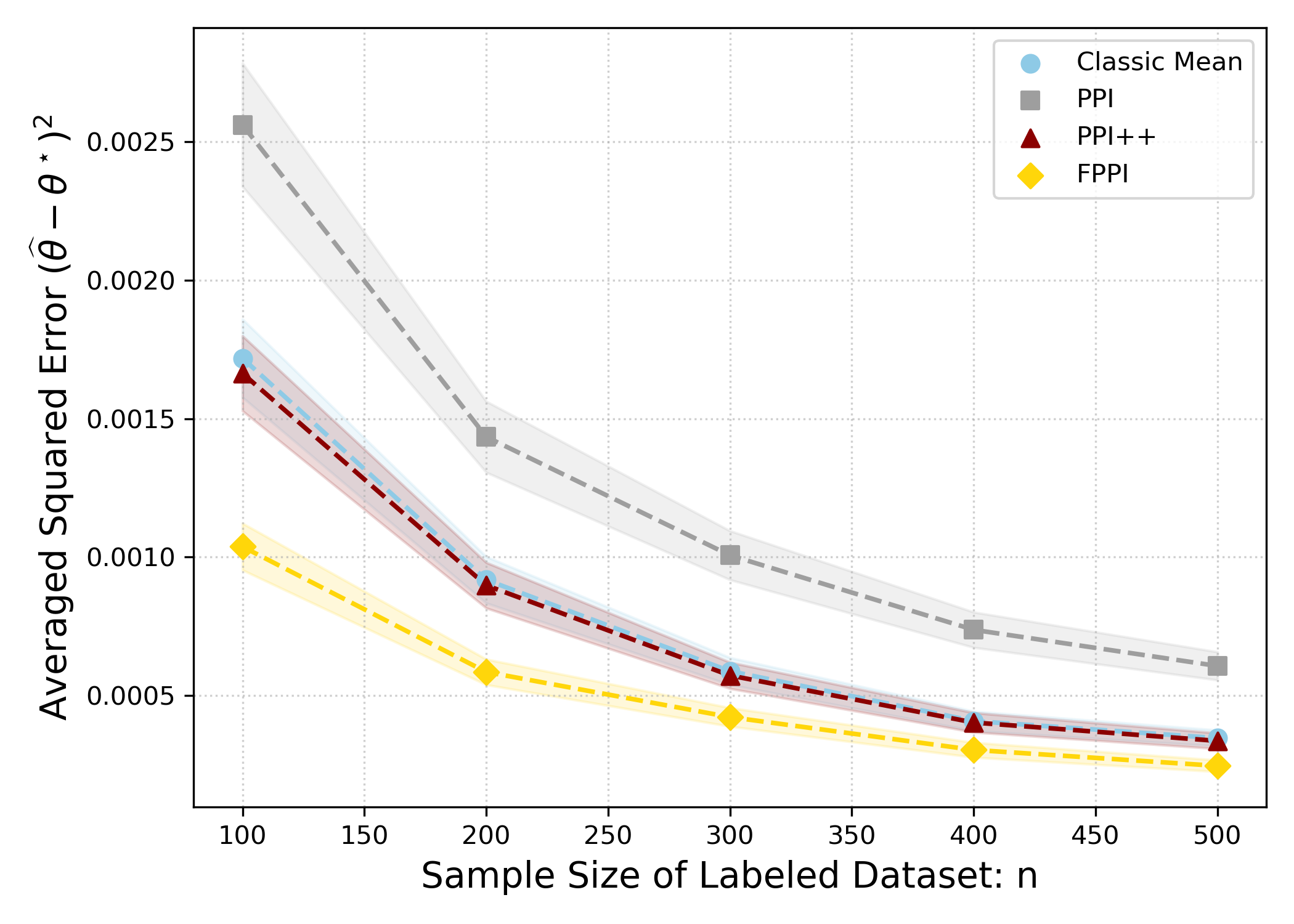}
        \caption{$N=2{,}000$: TF-IDF}
    \end{subfigure}
    \begin{subfigure}[b]{0.444\textwidth}
        \includegraphics[scale=0.385]{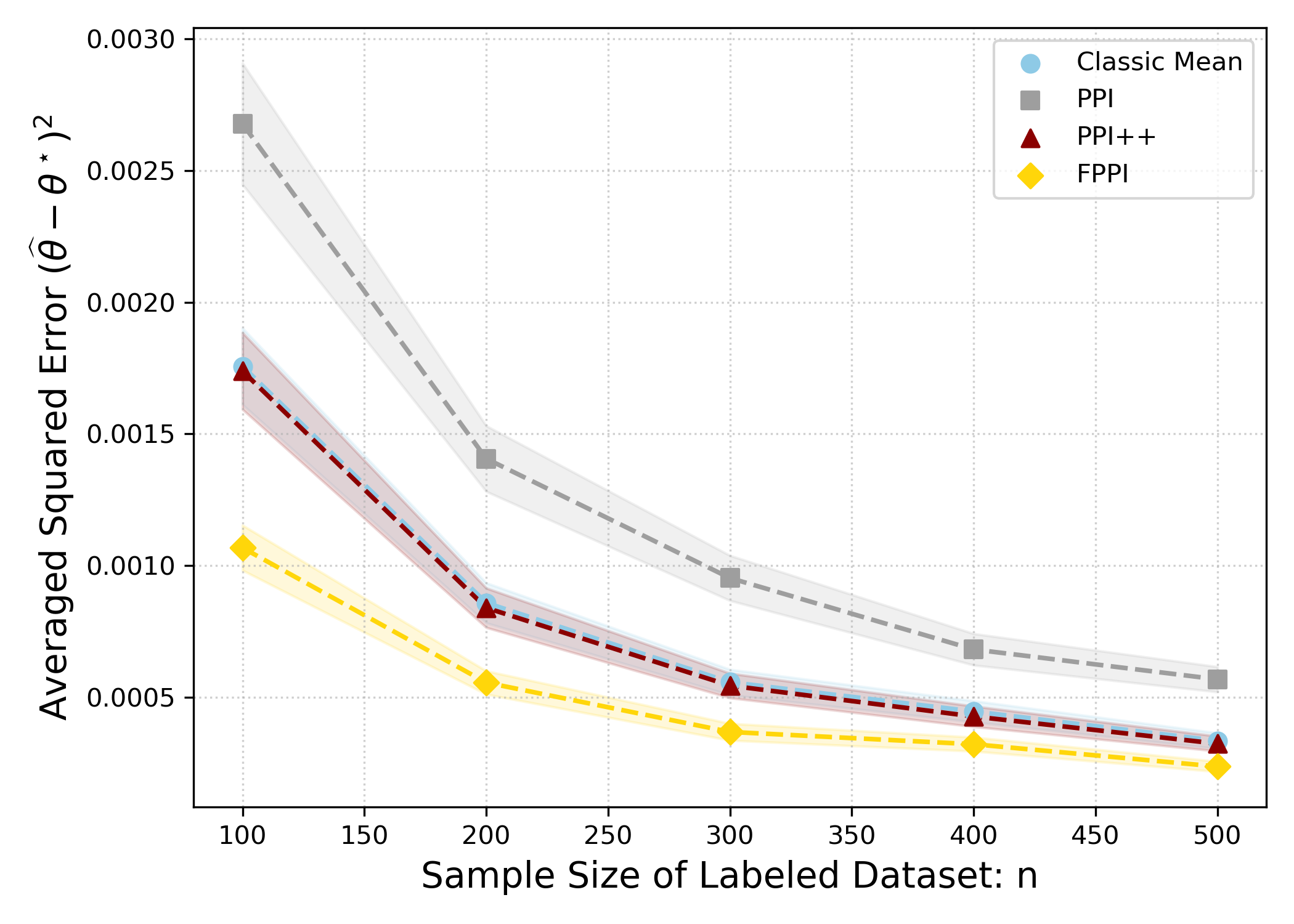}
        \caption{$N=4{,}000$: TF-IDF}
    \end{subfigure}
        \begin{subfigure}[b]{0.444\textwidth}
        \includegraphics[scale=0.385]{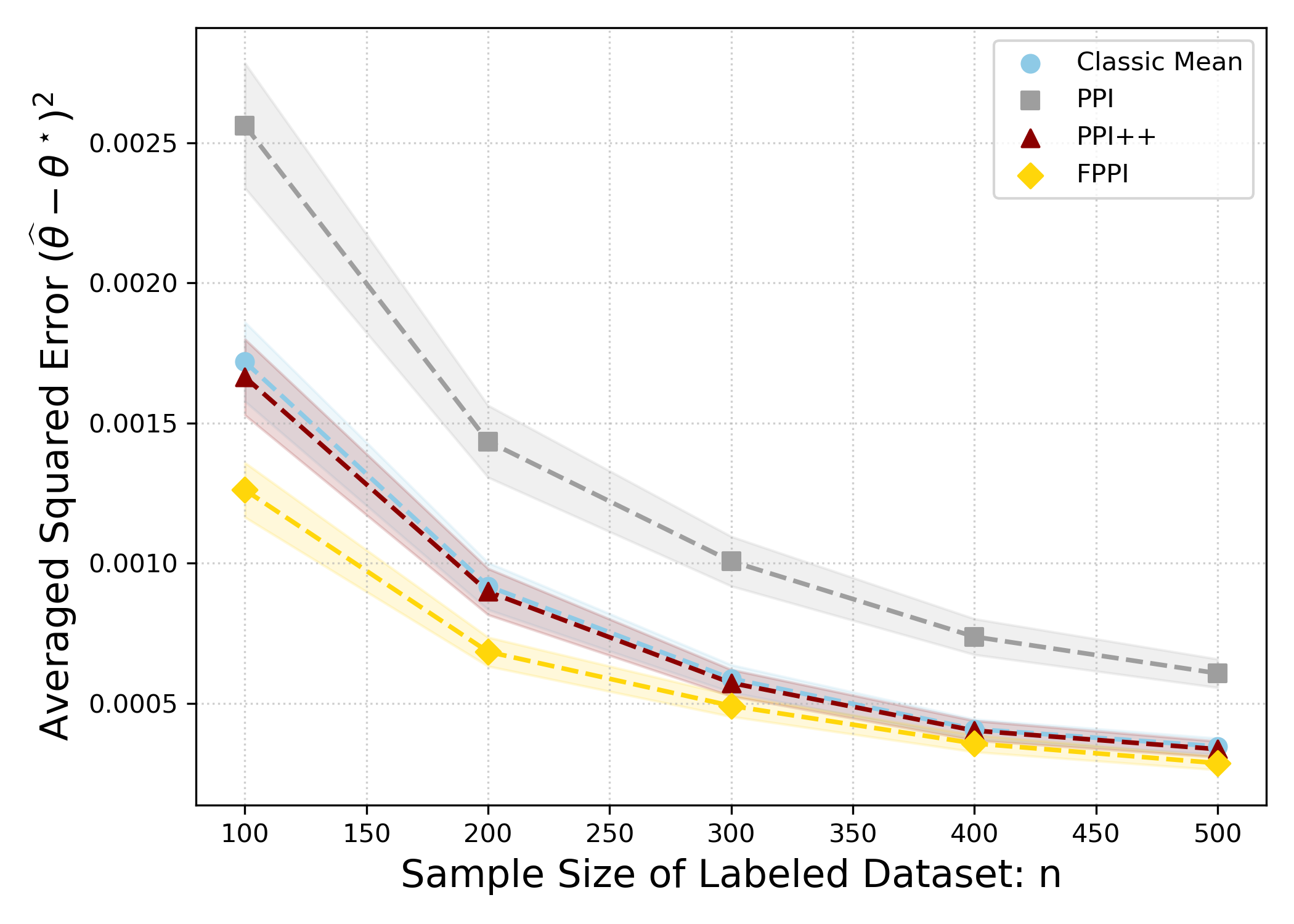}
        \caption{$N=2{,}000$: Embeddings}
    \end{subfigure}
    \begin{subfigure}[b]{0.444\textwidth}
        \includegraphics[scale=0.385]{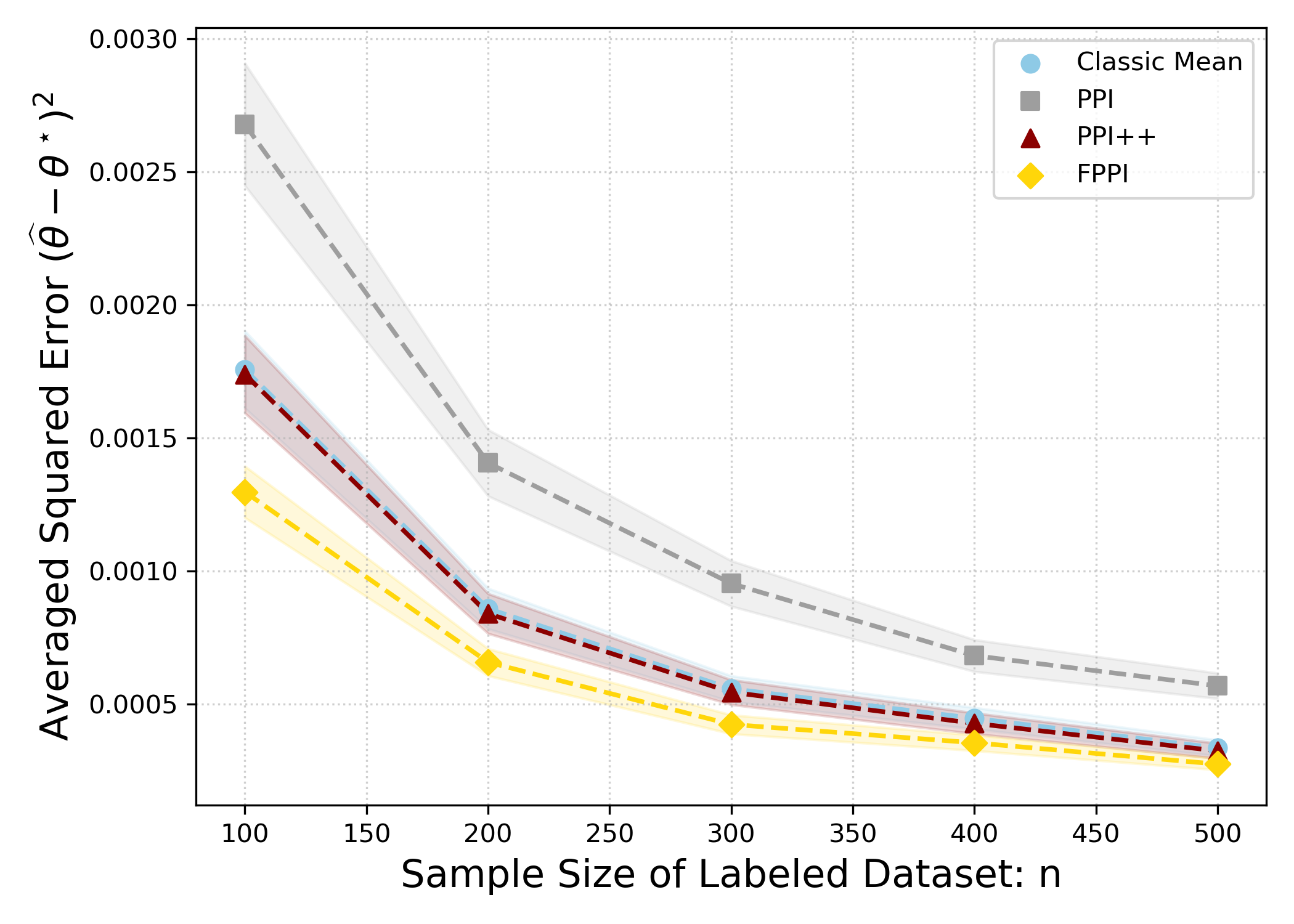}
        \caption{$N=4{,}000$: Embeddings}
    \end{subfigure}
\caption{The average mean squared errors of the four estimators across different $(n, N)$ settings and feature extraction strategies, computed over 1{,}000 replications.}
    \label{fig:result_real}
\end{figure}

Figure~\ref{fig:result_real} shows that the proposed FPPI estimator significantly outperforms its competitors across all scenarios. This improvement arises because the logistic regression step helps identify comparisons where LLM preferences align with human preferences. It is also worth noting that the PPI\texttt{++} estimator shows only marginal improvement (although still statistically significant). As discussed earlier, LLM preferences are only weakly correlated with human preferences; therefore, performance gains are not guaranteed in finite samples for PPI\texttt{++}, a phenomenon also noted in the existing literature \citep{mani2025no}.

\spacingset{1.4}

\renewcommand\refname{References}
\bibliographystyle{authoryear}
\putbib[Ref]
\end{bibunit}
\clearpage

\appendix
\setcounter{page}{1}
\setcounter{equation}{0}
\setcounter{section}{0}
\renewcommand{\thesection}{A.\arabic{section}}
\renewcommand{\thelemma}{A\arabic{lemma}}

\begin{center}
{\Large\bf Supplementary Materials} \\
\medskip
{\Large\bf ``Do More Predictions Improve Statistical Inference?
Filtered Prediction-Powered Inference"}  \\
\bigskip
\vspace{0.2in} 
\end{center}

\begin{bibunit}[apalike]
In this supplementary file, we present some additional discussions and the proofs of all examples, lemmas, and theorems from the main text, along with several additional auxiliary lemmas. To begin with, we recall the following definitions for theoretical purposes.
\begin{itemize}
    \item \textbf{$\varepsilon$-covering number:} Let $(\mathcal{T}, d)$ be a metric space and let $\varepsilon>0$.
A set $\mathcal{N}\subseteq \mathcal{T}$ is called an $\varepsilon$-net of
$\mathcal{T}$ if for every $t\in\mathcal{T}$ there exists
$s\in\mathcal{N}$ such that $d(t,s)\le\varepsilon$. The $\varepsilon$-covering number of $\mathcal{T}$ with respect to $d$ is
denoted by $N(\varepsilon,\mathcal{T},d)$ and is defined as
$$
N(\varepsilon,\mathcal{T},d)
=
\min\big\{|\mathcal{N}|:\mathcal{N}\subseteq\mathcal{T}
\text{ is an }\varepsilon\text{-net of }\mathcal{T}\big\}.
$$
\item For a square matrix $\bm{A}$, we denote its smallest eigenvalue by $\Lambda_{\min}(\bm{A})$ and its largest eigenvalue by $\Lambda_{\max}(\bm{A})$.
\item \textbf{Sub-Gaussian Random Variable: } A mean-zero random variable $X$ is a subgaussian random variable with a variance proxy $\sigma^2$ if 
$$
\mathbb{E}(\exp(\lambda X)) \leq e^{\frac{\lambda^2 \sigma^2}{2}} \mbox{ for all } \lambda \in \mathbb{R}.
$$
Equivalently, $(\mathbb{E}(X^p))^{\frac{1}{p}} \leq K\sqrt{p} $ for all $p \geq 1$ and some constant $K>0$ \citep{wainwright2019high}.
\item \textbf{Sub-Exponential Random Variable:} 
A mean-zero random variable $X$ is called sub-exponential if it satisfies
$$
\mathbb{E} \left[\exp(\lambda X)\right]
\le \exp \left(\frac{\sigma^{2}\lambda^{2}}{2}\right) \text{ for all } |\lambda| < b^{-1},
$$
for some constants $\sigma>0$ and $b>0$. Equivalently, $X$ has uniformly bounded moments:
$$
\left(\mathbb{E}|X|^{p}\right)^{1/p} \le K   p,
\qquad \text{for all } p \ge 1,
$$
for some constant $K>0$ \citep{wainwright2019high}.
\item \textbf{$\psi_2$-norm: } For a random variable $Z$, its $\psi_2$-norm is defined by
$$
\|Z\|_{\psi_2}
=
\inf\Bigl\{
\lambda>0:\;
\mathbb E[\exp(Z^2/\lambda^2)]\le 2
\Bigr\}.
$$

\end{itemize}

\clearpage
\section{Additional Discussions}
\label{SecAppendix:AD}

\subsection{Special Cases of Generalized Linear Models}
It is worth noting that the generalized linear models considered in Section~\ref{Subsec:GLM} encompass several widely used regression models as special cases, including linear regression, logistic regression, and Poisson regression. To illustrate this unifying formulation,
Table~\ref{tab:GLM_special_cases} summarizes the explicit forms of the corresponding functions $h(y)$, $S(y)$, and $A(\bm X^\top \bm\theta)$ for each model. Below, we provide detailed derivations for these three special cases.

\begin{table}[h]
\centering
\caption{Special cases of the GLM.}
\label{tab:GLM_special_cases}
\begin{tabular}{lcc}
\toprule
Model & $h(y)$ & $A(\bm{X}^\top \bm{\theta})$ \\
\midrule
Linear (Gaussian)
& $\displaystyle \frac{1}{\sqrt{2\pi\sigma^2}}
\exp \left(-\frac{y^2}{2\sigma^2}\right)$
& $\displaystyle \frac{(\bm{X}^\top \bm{\theta})^2}{2\sigma^2}$ \\

Logistic (Bernoulli)
& $1$
& $\displaystyle \log(1+e^{\bm{X}^\top \bm{\theta}})$ \\

Poisson
& $1/y!$
& $\displaystyle e^{\bm{X}^\top \bm{\theta}}$ \\
\bottomrule
\end{tabular}
\end{table}

\noindent
\textbf{Linear regression (Gaussian).} Assume that $Y \mid \bm X=\bm x \sim \mathcal N(\bm x^\top \bm\theta,\sigma^2)$. The conditional density of $Y$ given $\bm X=\bm x$ is
$$
P_Y(y\mid \bm{X}=\bm x)
=\frac{1}{\sqrt{2\pi\sigma^2}}
\exp \left\{
-\frac{1}{2\sigma^2}(y-\bm x^\top\bm\theta)^2
\right\}.
$$
Therefore, the density can be written as $P(y\mid \bm x)
= h(y)\exp \left\{
(\bm x^\top\bm\theta) S(y)
- A(\bm x^\top\bm\theta)
\right\}$, where
$$
h(y)=\frac{1}{\sqrt{2\pi\sigma^2}}
\exp \left(-\frac{y^2}{2\sigma^2}\right),\quad
S(y)=\frac{y}{\sigma^2},\quad
A(\bm x^\top\bm\theta)=\frac{(\bm x^\top\bm\theta)^2}{2\sigma^2}.
$$

\noindent\textbf{Logistic regression (Bernoulli).} Assume that $Y \mid \bm X=\bm x \sim \mathrm{Bernoulli}(\pi)$, where
$\pi = \exp(\bm x^\top \bm\theta)\big/\big(1+\exp(\bm x^\top \bm\theta)\big)$. The conditional probability function of $Y$ given $\bm X=\bm x$ is
$$
P_Y(y\mid \bm X=\bm x)
= \pi^y (1-\pi)^{1-y}.
$$
Therefore, the probability function can be written as
$P_Y(y\mid \bm x)
= h(y)\exp \left\{
(\bm x^\top\bm\theta) S(y)
- A(\bm x^\top\bm\theta)
\right\}$, where
$$
h(y)=1,\quad
S(y)=y,\quad
A(\eta)=\log \big(1+e^{\bm x^\top\bm\theta}\big).
$$

\noindent
\textbf{Poisson regression.} Assume that $Y \mid \bm X=\bm x \sim \mathrm{Poisson} \left(e^{\bm x^\top\bm\theta}\right)$.
The conditional probability mass function of $Y$ given $\bm X=\bm x$ is
$$
P_Y(y\mid \bm X=\bm x)
=\frac{\exp \left(-e^{\bm x^\top\bm\theta}\right)
\left(e^{\bm x^\top\bm\theta}\right)^y}{y!}.
$$
Therefore, the density can be written as
$P(y\mid \bm x)
= h(y)\exp \left\{
(\bm x^\top\bm\theta) S(y)
- A(\bm x^\top\bm\theta)
\right\}$, where
$$
h(y)=\frac{1}{y!},\quad
S(y)=y,\quad
A(\bm x^\top\bm\theta)=e^{\bm x^\top\bm\theta}.
$$

\subsection{Conditions for Asymptotic Normality}
Corollaries \ref{Coro:Mean_Inference} and \ref{Coro:GLM_Inference} are established under the condition that $\psi_n \gtrsim (\alpha_n \wedge n)$. We next demonstrate that this requirement is mild and is satisfied by a broad class of estimation procedures. For illustration, we use $k$-nearest neighbor regression (Example \ref{ex:knn_stability_tail}) \citep{doring2018rate} and simple linear regression (Example \ref{ex:linear_stability}) to show that the condition typically holds in practice.

\begin{example}[$k$-Nearest Neighbors]
\label{ex:knn_stability_tail}
Assume the regression model $Y = m(\bm X) + \varepsilon$, where $\varepsilon$ is independent mean-zero sub-Gaussian random variable with variance proxy $\sigma^2$, and the regression function $m(\cdot)$ is $L$-Lipschitz, that is,
$$
|m(\bm x) - m(\bm x')| \le L \|\bm x - \bm x'\|_2, \quad \forall \bm x,\bm x' \in \mathcal X,
$$
for some constant $L > 0$. Let $\widehat m(\bm x) = \frac{1}{k}\sum_{j\in \mathcal N_k(\bm x)} y_j$ be the $k$-nearest neighbor estimator. Under the conditions of Lemma \ref{lemma:knn_radius}, then Assumptions~\ref{Ass:TailConvergence} and~\ref{ass:loo_stability} both hold with $\alpha_n \asymp n^{\frac{2}{2+d}}$ and $ \psi_n \asymp \frac{n^{\frac{4}{2+d}}}{\log n} $ under the optimal choice $k \asymp n^{\frac{2}{2+d}}$. This implies that $\psi_n \gtrsim \alpha_n$, and hence $\psi_n \gtrsim (\alpha_n \wedge n) $.
\end{example}

Example \ref{ex:knn_stability_tail} shows that the condition $\psi_n \gtrsim (\alpha_n \wedge n)$ is satisfied by $k$-nearest neighbor regression, illustrating that this requirement is mild and readily met in practice. Moreover, the result in Example \ref{ex:knn_stability_tail} can be extended to a broader class of regression functions $m(\cdot)$ under more general smoothness conditions \citep{samworth2012optimal}.

\begin{example}[Simple linear regression]
\label{ex:linear_stability}
Consider the simple linear regression model $y_i = x_i \beta^\star + \varepsilon_i$ for $i=1,\dots,n$, where the noises $\varepsilon_i$ are independent mean-zero sub-Gaussian random variables with variance proxy $\sigma^2$. The design points $\{x_i\}_{i=1}^n$ are treated as fixed (non-random) and satisfy $|x_i| \le M_{\mathcal X}$ and $\frac{1}{n}\sum_{i=1}^n x_i^2 \ge c_0 > 0$. Let $\widehat\beta=\left(\sum_{i=1}^n x_i^2\right)^{-1} \sum_{i=1}^n x_i y_i$ and $\widehat m(x)=x\widehat\beta$ be the ordinary least squares estimator and predictor. Then Assumption~\ref{ass:loo_stability} holds with $\psi_n \asymp n^2$. Consequently, $\psi_n \gtrsim (\alpha_n \wedge n)$ for any $\alpha_n = o(n^2)$, while $\alpha_n \asymp n$ for linear regression.
\end{example}

Example~\ref{ex:linear_stability} shows that the condition $\psi_n \gtrsim (\alpha_n \wedge n)$ is also satisfied by ordinary least squares in simple linear regression. In fact, linear regression enjoys even stronger leave-one-out stability, with $\psi_n \asymp n^2$, reflecting the global averaging nature of the estimator. This demonstrates that Assumption~\ref{ass:loo_stability} naturally holds for a broad class of classical parametric estimators.

\subsection{Data Splitting Scheme}

As discussed in the main text, the proposed method enables the labeled dataset to be used simultaneously for region estimation and parameter estimation. This approach requires that $\mathbb{P}_{\bm{X}}(\widehat{\mathcal{S}}\Delta\mathcal{S}_0)=o_p(n^{-1/2})$. In this section, we show that this requirement can be removed under a data-splitting scheme. For demonstration, we only consider the mean estimation problem and the extension to the GLMs can be similarly derived.

\begin{algorithm}[t]
\caption{Sample-Splitting FPPI Estimator for Continuous Covariates}
\label{alg:filtered-fppi-split}
\begin{algorithmic}[1]
\STATE \textbf{Input:} Labeled dataset $\mathcal{D}_L = \{(\bm{x}_i, y_i)\}_{i=1}^n$, 
       Unlabeled dataset $\widetilde{\mathcal{D}}_U = \{\widetilde{\bm{x}}_j\}_{j=1}^N$, 
       Regression estimator $\widehat m(\cdot)$, and feature function $f(\cdot)$.

\STATE Randomly split the labeled dataset $\mathcal{D}_L$ into:
$$
\mathcal{D}_1 = \{(\bm{x}_i, y_i)\}_{i=1}^{n_1}, \quad
\mathcal{D}_2 = \{(\bm{x}_i, y_i)\}_{i=n_1+1}^{n}, \quad n_1+n_2=n
$$

\STATE Fit the regression estimator $\widehat m(\bm{x})$ using the first part $\mathcal{D}_1$.

\STATE Compute the mean outcome in $\mathcal{D}_1$:
$$
\bar y_1 = \frac{1}{n_1} \sum_{i=1}^{n_1} y_i
$$

\STATE Estimate the efficiency-improving region:
$$
\widehat{\mathcal{S}}_0 = \left\{\bm{x} \in \mathcal{X}: (\widehat m(\bm{x}) - \bar y_1) f(\bm{x}) > 0 \right\}
$$

\STATE Compute the plug-in FPPI coefficient using the second part $\mathcal{D}_2$:
\begin{align*}
    \widehat \lambda_{\mathcal{S}_0}& =
\frac{
\frac{1}{n_2} \sum_{i=n_1+1}^{n} (y_i - \bar y_2) f_{\widehat{\mathcal{S}}_0}(\bm{x}_i)
}{
\frac{1}{N} \sum_{j=1}^N \big(f_{\widehat{\mathcal{S}}_0}(\widetilde{\bm{x}}_j) - \bar f_{\widehat{\mathcal{S}}_0}\big)^2
} \cdot \frac{1}{1+n_2/N}, 
\quad
\bar y_2 = \frac{1}{n_2} \sum_{i=n_1+1}^{n} y_i, \\
\bar f_{\widehat{\mathcal{S}}_0} &= \frac{1}{n_2+N} \left(\sum_{i=n_1+1}^{n} f_{\widehat{\mathcal{S}}_0}(\bm{x}_i) + \sum_{j=1}^N f_{\widehat{\mathcal{S}}_0}(\widetilde{\bm{x}}_j)\right).
\end{align*}

\STATE Construct the Sample-Splitting FPPI mean estimator:
$$
\widehat{\theta}_{\mathrm{FPPI}}(\widehat \lambda_{\mathcal{S}_0},\widehat{\mathcal{S}}_0) =
\bar y_2 + \widehat \lambda_{\mathcal{S}_0} \left(
\frac{1}{N} \sum_{j=1}^N f_{\widehat{\mathcal{S}}_0}(\widetilde{\bm{x}}_j) -
\frac{1}{n_2} \sum_{i=n_1+1}^{n} f_{\widehat{\mathcal{S}}_0}(\bm{x}_i)
\right)
$$

\STATE \textbf{Return:} $\widehat{\theta}_{\mathrm{FPPI}}(\widehat \lambda_{\mathcal{S}_0},\widehat{\mathcal{S}}_0)$
\end{algorithmic}
\end{algorithm}

\begin{corollary}
\label{Coro:Mean_Inference_Split}
Suppose $n_2/n \xrightarrow{}d \in (0,1)$. Under Assumptions \ref{Ass:SubGaussian}–\ref{Ass:TailConvergence}, the mean estimator produced by Algorithm \ref{alg:filtered-fppi-split} satisfies, in the regime where $\frac{n_2}{N} \to r$, the following result:
    \begin{align*}
        \sqrt{n_2}\left(\widehat{\theta}_{\mathrm{FPPI}}(\widehat \lambda_{\mathcal{S}_0},\widehat{\mathcal{S}}_0) - \theta^\star
        \right) \xrightarrow{d} \mathcal{N}\left(
0,\mathrm{Var}(Y)
- \frac{1}{(1+r)}  
\frac{\mathrm{Cov}^2\big(Y,f(\bm X)\mathbf 1_{\mathcal S_0}(\bm X)\big)}
     {\mathrm{Var}\big(f(\bm X)\mathbf 1_{\mathcal S_0}(\bm X)\big)}
        \right).
    \end{align*}
\end{corollary}

Corollary \ref{Coro:Mean_Inference_Split} shows that sample splitting allows us to substantially simplify the theoretical requirements of FPPI. In particular, by separating region estimation and parameter estimation, we no longer need leave-one-out stability arguments nor explicit assumptions on the convergence rate of the region recovery step. This decoupling makes the analysis considerably cleaner and more transparent. The trade-off, however, is a loss of sample efficiency: only the subsample of size $n_2$ contributes to the final estimator, leading to an asymptotic scaling of $\sqrt{n_2}$ rather than $\sqrt{n}$. Thus, while sample splitting weakens the technical assumptions, it comes at the cost of a reduced effective sample size for estimation.

\section{Proof of Examples}
\vspace{6mm}

\noindent\textbf{Proof of Example \ref{Exam:PPIF_new}.} Recall that $Y = X + \varepsilon$, where $X \sim \mathcal{N}(\theta^\star,1)$, 
$\varepsilon \sim \mathcal{N}(0,\sigma^2)$ are independent, and $\theta^\star=1$.
Let
$$
\overline y = \frac{1}{n}\sum_{i=1}^n y_i, \quad
\overline f_L = \frac{1}{n}\sum_{i=1}^n f(x_i), \quad
\overline f_U = \frac{1}{N}\sum_{j=1}^N f(\widetilde x_j),
$$
where $f(x)=(x-\theta^\star)^2-1$.

\paragraph{Variance of the PPI\texttt{++} estimator.}

The PPI\texttt{++} estimator can be written as
$$
\widehat{\theta}_{+}(\lambda)
= \overline y + \lambda (\overline f_U - \overline f_L).
$$
Its variance satisfies
$$
\mathrm{Var}(\widehat{\theta}_{+}(\lambda))
= \mathrm{Var}(\overline y)
+ \lambda^2 \mathrm{Var}(\overline f_U - \overline f_L)
+ 2\lambda\, \mathrm{Cov}(\overline y, \overline f_U - \overline f_L).
$$
Since $y_i = x_i + \varepsilon_i$ with independence,
$$
\mathrm{Var}(\overline y)
= \frac{1}{n}\mathrm{Var}(Y)
= \frac{1}{n}\big(\mathrm{Var}(X)+\mathrm{Var}(\varepsilon)\big)
= \frac{1+\sigma^2}{n}.
$$
Because labeled and unlabeled samples are independent,
$$
\mathrm{Var}(\overline f_U - \overline f_L)
= \mathrm{Var}(\overline f_U) + \mathrm{Var}(\overline f_L).
$$
Let $Z = X-\theta^\star \sim \mathcal{N}(0,1)$. Then
$f(X)=Z^2-1$, so $\mathbb{E}[f(X)]=0$ and
$$
\mathrm{Var}(f(X)) = \mathrm{Var}(Z^2-1)=2.
$$
Therefore,
$$
\mathrm{Var}(\overline f_L)=\frac{2}{n}, \quad
\mathrm{Var}(\overline f_U)=\frac{2}{N},
$$
and hence
$$
\mathrm{Var}(\overline f_U - \overline f_L)
= 2\Big(\frac{1}{n}+\frac{1}{N}\Big).
$$
Next,
$$
\mathrm{Cov}(y_i,f(x_i))
= \mathrm{Cov}(x_i,f(x_i)) + \mathrm{Cov}(\varepsilon_i,f(x_i))
= \mathrm{Cov}(x_i,f(x_i)),
$$
since $\varepsilon_i$ is independent of $x_i$.
Moreover,
$$
\mathrm{Cov}(x_i,f(x_i))
= \mathbb{E}[(\theta^\star+Z)(Z^2-1)]
= \mathbb{E}[Z^3-Z] + \theta^\star \mathbb{E}[Z^2-1]
= 0.
$$
Thus $\mathrm{Cov}(\overline y,\overline f_U-\overline f_L)=0$, and
$$
\mathrm{Var}(\widehat{\theta}_{+}(\lambda))
= \frac{1+\sigma^2}{n}
+ 2\lambda^2\Big(\frac{1}{n}+\frac{1}{N}\Big).
$$
This is minimized at $\lambda=0$, yielding no variance reduction. This completes the proof. \hfill ${\color{Red}\blacksquare}$ \\ 
\vspace{5mm}

\noindent\textbf{Proof of Example \ref{Exam:FPPI}.} Consider
$$
\widehat{\theta}(\lambda,t)
= \frac{1}{n}\sum_{i=1}^n y_i
+ \lambda\Bigg(
\frac{1}{N}\sum_{j=1}^N f(\widetilde x_j)\mathbf{1}(\widetilde x_j>t)
- \frac{1}{n}\sum_{i=1}^n f(x_i)\mathbf{1}(x_i>t)
\Bigg).
$$
Let $U_i=f(x_i)\mathbf 1(x_i>t)$ and $\widetilde U_j=f(\widetilde x_j)\mathbf 1(\widetilde x_j>t)$.
Then
$$
\mathrm{Var}(\widehat{\theta}(\lambda,t))
= \frac{1+\sigma^2}{n}
+ \lambda^2 \sigma_U^2(t)\Big(\frac{1}{n}+\frac{1}{N}\Big)
- \frac{2\lambda}{n}\mathrm{Cov}(y_i,U_i),
$$
where $\sigma_U^2(t)=\mathrm{Var}(f(X)\mathbf 1(X>t))$. Let $Z=X-\theta^\star \sim \mathcal N(0,1)$ and $a=t-\theta^\star$.

Let $\phi(x)$ and $\Phi(x)$ denote the probability density function and cumulative distribution function of the standard normal distribution, respectively. Then $\sigma_U^2(t)$ can be written as
\begin{align*}
    &\sigma_U^2(t) = \mathbb{E}
    \left(
f^2(X)\bm{1}(X >t)
    \right) - \left[\mathbb{E}
    \left(
f(X)\bm{1}(X >t)
    \right)\right]^2 \\
    =&
    \int_{t}^{+\infty} \left[(x-\theta^\star)^4-2(x-\theta^\star)^2+1\right]
    \frac{1}{\sqrt{2\pi}}e^{-\frac{(x-\theta^\star)^2}{2}}dx -\left[\int_{t}^{+\infty} \left[(x-\theta^\star)^2-1\right]
    \frac{1}{\sqrt{2\pi}}e^{-\frac{(x-\theta^\star)^2}{2}}dx\right]^2 \\
    = &
    \int_{t-\theta^\star}^{+\infty} \left[z^4-2z^2+1\right]
    \frac{1}{\sqrt{2\pi}}e^{-\frac{z^2}{2}}dz -\left[\int_{t-\theta^\star}^{+\infty} \left[z^2-1\right]
    \frac{1}{\sqrt{2\pi}}e^{-\frac{z^2}{2}}dz\right]^2.
\end{align*}
Using standard Gaussian tail integrals (deriving from integration by parts), we have
\begin{align*}
\int_a^\infty z^2 \phi(z)\,dz &= a\phi(a) + 1-\Phi(a),\\
\int_a^\infty z^4 \phi(z)\,dz &= (a^3+3a)\phi(a) + 3(1-\Phi(a)).
\end{align*}
Hence
$$
\sigma_U^2(t)
= [(a^3+a)\phi(a) + 2(1-\Phi(a))] - [a\phi(a)]^2.
$$
Moreover,
$$
\mathrm{Cov}(y_i,U_i)
= \mathrm{Cov}(x_i,f(x_i)\mathbf 1(x_i>t))
= \mathbb{E}[Z(Z^2-1)\mathbf 1(Z>a)]
= (a^2+1)\phi(a).
$$
Setting $t=\theta^\star=1$. Then $a=0$, so
$$
\sigma_U^2(1)=1, 
\qquad
\mathrm{Cov}(y_i,U_i)=\phi(0)=\frac{1}{\sqrt{2\pi}}.
$$
Therefore
$$
\mathrm{Var}(\widehat{\theta}(\lambda,1))
= \frac{1+\sigma^2}{n}
+ \lambda^2\Big(\frac{1}{n}+\frac{1}{N}\Big)
- \frac{2\lambda}{n\sqrt{2\pi}}.
$$
This quadratic is minimized at $\lambda^\star=\frac{1}{\sqrt{2\pi}\left(1+\frac{n}{N}\right)}$, which yields
$$
\mathrm{Var}(\widehat{\theta}(\lambda^\star,1))
= \frac{1+\sigma^2}{n}
- \frac{1}{2\pi n\left(1+\frac{n}{N}\right)}
< \frac{1+\sigma^2}{n}.
$$
This completes the proof. \hfill ${\color{Red}\blacksquare}$ \\
\vspace{5mm}

\noindent\textbf{Proof of Example \ref{ex:polynomial-density}.} Since $m(x)=\mathbb{E}(Y\mid X=x)=x$ and the density $P_X(x)$ is symmetric about zero, we have
$$
\mathbb{E}(Y)=\mathbb{E}\{\mathbb{E}(Y\mid X)\}
=\mathbb{E}(X)=0.
$$
Moreover,
$$
|m(X)| = |X|.
$$
Therefore, for any $t>0$,
$$
\mathbb{P}_X\big(|m(X)| \le t\big)
=
\mathbb{P}_X\big(|X| \le t\big).
$$
By direct calculation,
\begin{align*}
\mathbb{P}_X(|X|\le t)
= \int_{-t}^{t} \frac{\ell+1}{2}|x|^{\ell}\,dx  =(\ell+1)\int_0^t x^{\ell}\,dx = t^{\,\ell+1}.
\end{align*}
Hence, $\mathbb{P}_X\big(|m(X)| \le t\big) = t^{\,\ell+1}$, which implies that Definition~\ref{ass:margin} holds with margin exponent $\tau=\ell+1$. This completes the proof. \hfill {\color{Red}$\blacksquare$} \\

\vspace{5mm}

\noindent\textbf{Proof of Example \ref{ex:fast-density}.} By definition, $m(X)=\mathbb{E}(Y\mid X)=X$.  
Since the density $P_X(x)$ is symmetric about 0 and supported on $\mathcal{X}=[-1,-c]\cup [c,1]$, we have
$$
\mathbb{E}(Y)=\mathbb{E}\{ \mathbb{E}(Y\mid X)\} = \mathbb{E}(X) = 0.
$$
Moreover, for any $X \in \mathcal{X}$,
$$
|m(X)-\mathbb{E}(Y)| = |X-0| = |X| \ge c.
$$
Hence,
$$
|m(X)-\mathbb{E}(Y)| > c \quad \text{almost surely}.
$$
This completes the proof. \hfill {\color{Red}$\blacksquare$} \\

\vspace{5mm}

\noindent\textbf{Proof of Example \ref{ex:knn_stability_tail}}.We verify the two assumptions in turn. Fix $\bm x\in\mathcal X$ and decompose
$$
\widehat m(\bm x)-m(\bm x)
=
\frac{1}{k}\sum_{j\in \mathcal N_k(\bm x)} \varepsilon_j
+
\frac{1}{k}\sum_{j\in \mathcal N_k(\bm x)}\big(m(\bm x_j)-m(\bm x)\big)
\triangleq V(\bm x) + B(\bm x).
$$

Conditionally on the covariates, $V(\bm x)$ is an average of $k$ independent sub-Gaussian variables. Hence for all $t>0$,
$$
\mathbb P \left(|V(\bm x)| \ge t \,\middle|\, \bm X_1,\dots,\bm X_n\right)
\le 2\exp \left(-\frac{k t^2}{2\sigma^2}\right),
$$
and the same bound holds unconditionally. By the Lipschitz property, we have
$$
|B(\bm x)|
\le \frac{L}{k}\sum_{j\in \mathcal N_k(\bm x)} 
\|\bm X_j-\bm x\|
\le L\, R_k(\bm x),
$$
where $R_k(\bm x)$ is the distance from $\bm{x}$ to its $k$-th nearest neighbor among 
$\{\bm{x}_i\}_{i=1}^n$.

Therefore, we have
\begin{align*}
    \mathbb{P}\left(|\widehat m(\bm x)-m(\bm x)| \geq t\right) \leq  &\mathbb{P}\left(|V(\bm x)| \geq t/2\right) + \mathbb{P}\left(|B(\bm x)| \geq t/2\right) \\
    \leq &
    \mathbb{P}\left(|V(\bm x)| \geq t/2\right) + \mathbb{P}\left(R_k(\bm x) \geq \left(\frac{t}{2L}\right)\right) \\
    \leq & 
    2\exp \left(-\frac{k t^2}{8\sigma^2}\right) + \mathbb{P}\left(R_k(\bm x) \geq \left(\frac{t}{2L}\right)\right).
\end{align*}
By Lemma \ref{lemma:knn_radius}, we have $\Big(\tfrac{2k}{c_0 n}\Big)^{1/d} \le s \le r_0$, we have
$$
\mathbb P \left(R_k(\bm{x}) > s\right)
\le \exp \left(- c n s^d\right).
$$
Therefore, for $2L\Big(\tfrac{2k}{c_0 n}\Big)^{1/d} \leq t \leq 2Lr_0$, we have
\begin{align*}
    \mathbb{P}\left(|\widehat m(\bm x)-m(\bm x)| \geq t\right) \leq 2\exp \left(-\frac{k t^2}{8\sigma^2}\right) + \exp \left(- c' n t^d\right).
\end{align*}
Choosing $k = n^{\frac{2}{2+d}}$, we have
\begin{align*}
    \mathbb{P}\left(|\widehat m(\bm x)-m(\bm x)| \geq t\right) \leq 2\exp \left(-\frac{ n^{\frac{2}{2+d}}t^2}{8\sigma^2}\right) + \exp \left(- c n t^d\right),
\end{align*}
for $t \in \big[2^{\frac{1+d}{d}}Lc_0^{-\frac{1}{d}}n^{-\frac{1}{2+d}}, 2Lr_0\big]$. Note that the first term governs the deviation at the optimal scale 
$t \asymp n^{-1/(2+d)}$. Therefore,
\begin{align*}
    \mathbb{P}\left(|\widehat m(\bm x)-m(\bm x)| \geq t\right) \lesssim \exp \left(-Cn^{\frac{2}{2+d}}t^2\right),
\end{align*}
for some positive constant $C$. Therefore, Assumption \ref{Ass:TailConvergence} holds with $\alpha_n = n^{\frac{2}{2+d}}$.

Fix $i\in\{1,\dots,n\}$. For any $\bm x\in\mathcal X$, removing observation $i$ changes the $k$NN estimator only if $\bm x_i \in \mathcal N_k(\bm x)$. In that case, one neighbor in the average is replaced, and
$$
\widehat m(\bm x)-\widehat m^{(-i)}(\bm x)
=
\frac{1}{k}\big(y_i - y_{j^*(\bm x)}\big),
$$
where $j^*(\bm x)$ denotes the replacement neighbor. Moreover, for any $\bm x$,
$$
\big|\widehat m(\bm x)-\widehat m^{(-i)}(\bm x)\big|
\le
\frac{1}{k}\Big(|y_i| + |y_{j^*(\bm x)}|\Big)
\le
\frac{2}{k}\max_{1\le j\le n}|y_j|.
$$
Taking supremum over $\bm x\in\mathcal X$ yields the deterministic bound
$$
\sup_{\bm x\in\mathcal X}
\big|\widehat m(\bm x)-\widehat m^{(-i)}(\bm x)\big|
\le
\frac{2}{k}\max_{1\le j\le n}|y_j|.
$$

Since the $y_j$ are sub-Gaussian, a union bound gives
$$
\mathbb P \left(
\max_{1\le j\le n}|y_j| \ge u
\right)
\le
2n\exp(-c u^2)
$$
for some constant $c>0$. Therefore, for any $t>0$,
$$
\mathbb P \left(
\sup_{x\in\mathcal X}
\big|\widehat m(\bm x)-\widehat m^{(-i)}(\bm x)\big|
\ge t
\right)
\le
2n\exp \left(-c k^2 t^2\right).
$$
Thus, for $t \gtrsim \sqrt{(\log n)}/k$, we obtain
$$
\mathbb P \left(
\sup_{x\in\mathcal X}
\big|\widehat m(\bm x)-\widehat m^{(-i)}(\bm x)\big|
\ge t
\right)
\lesssim
\exp \left(-c' k^2 t^2\right),
$$
which verifies Assumption~\ref{ass:loo_stability} with $\psi_n \asymp \frac{k^2}{\log n}$. Choosing $k = n^{\frac{2}{2+d}}$, we have $\psi_n = \frac{n^{\frac{4}{2+d}}}{\log n}$. Therefore, it holds that $\psi_n \gtrsim (\alpha_n \wedge n)$. This completes the proof. \hfill {\color{Red}$\blacksquare$} \\

\vspace{5mm}

\noindent\textbf{Proof of Example \ref{ex:linear_stability}.}
Let $\widehat\beta^{(-i)}$ be the OLS estimator computed without the $i$-th observation. A direct algebraic calculation yields the leave-one-out identity
$$
\widehat\beta-\widehat\beta^{(-i)}
=
\frac{x_i}{\sum_{j=1}^n x_j^2}\cdot \frac{r_i}{1-h_i},
$$
where
$$
r_i = y_i - x_i\widehat\beta,
\qquad
h_i = \frac{x_i^2}{\sum_{j=1}^n x_j^2}
$$
are the full-sample residual and leverage score. For any $x\in\mathcal X$,
$$
|\widehat m(x)-\widehat m^{(-i)}(x)|
=|x|\,|\widehat\beta-\widehat\beta^{(-i)}|
\le
|x|\frac{|x_i|}{\sum_{j=1}^n x_j^2}\frac{|r_i|}{1-h_i}.
$$
By the fixed-design assumptions,
$$
\sum_{j=1}^n x_j^2 \ge c_0 n,
\qquad
h_i \le \frac{M_{\mathcal X}^2}{c_0 n}.
$$
Hence for all sufficiently large $n$, $1-h_i \ge \tfrac12$, and therefore
$$
\sup_{x\in\mathcal X}
|\widehat m(x)-\widehat m^{(-i)}(x)|
\;\le\;
\frac{2M_{\mathcal X}^2}{c_0 n}\,|r_i|.
$$
Using the model,
$$
\widehat\beta-\beta^\star
=
\frac{\sum_{j=1}^n x_j\varepsilon_j}{\sum_{j=1}^n x_j^2},
$$
which is a fixed linear combination of independent sub-Gaussian variables, hence sub-Gaussian. Now the residual can be written as
$$
r_i
= y_i - x_i\widehat\beta
= \varepsilon_i - x_i(\widehat\beta-\beta^\star).
$$
This is a linear combination of the independent sub-Gaussian variables $\{\varepsilon_j\}_{j=1}^n$ with coefficients
$$
a_j =
\begin{cases}
1 - \frac{x_i^2}{\sum_{k=1}^n x_k^2}, & j=i, \\
-\frac{x_i x_j}{\sum_{k=1}^n x_k^2}, & j\ne i .
\end{cases}
$$
Hence $r_i = \sum_{j=1}^n a_j \varepsilon_j$. Since linear combinations of independent sub-Gaussian variables remain sub-Gaussian, we have
$$
\|r_i\|_{\psi_2}^2
\;\lesssim\;
\sigma^2 \sum_{j=1}^n a_j^2.
$$
A direct calculation gives
$$
\sum_{j=1}^n a_j^2
=
1 - \frac{x_i^2}{\sum_{k=1}^n x_k^2}
\le 1.
$$
Therefore $r_i$ is sub-Gaussian with variance proxy bounded by a universal constant multiple of $\sigma^2$. In particular,
$$
\mathbb P(|r_i|\ge u) \le 2\exp(-c u^2)
$$
for some constant $c>0$ independent of $n$. Consequently, for all $u>0$,
$$
\mathbb P(|r_i|\ge u) \le 2\exp(-c u^2)
$$
for some constant $c>0$. Combining the bounds,
$$
\mathbb P \left(
\sup_{x\in\mathcal X}
|\widehat m(x)-\widehat m^{(-i)}(x)|
\ge t
\right)
\le
\mathbb P \left(|r_i|\ge c' n t\right)
\le
2\exp(-c'' n^2 t^2),
$$
which verifies Assumption~\ref{ass:loo_stability} with $\psi_n\asymp n^2$.
\hfill${\color{Red}\blacksquare}$ \\
\vspace{5mm}

\section{Proof of Theorems}

\vspace{6mm}

\noindent\textbf{Proof of Theorem \ref{thm:fppi_mean_var}.}
Define
$$
A=\frac{1}{n}\sum_{i=1}^n y_i,\quad
B=\frac{1}{N}\sum_{j=1}^N f(\widetilde{\bm x}_j)\mathbf 1(\widetilde{\bm x}_j\in\mathcal S),\quad
C=\frac{1}{n}\sum_{i=1}^n f(\bm x_i)\mathbf 1(\bm x_i\in\mathcal S),
$$
so that
$$
\widehat{\theta}_{\mathrm{FPPI}}(\lambda,\mathcal S)
= A+\lambda(B-C).
$$
Since $\widetilde{\bm x}_j$ and $\bm x_i$ share the same marginal distribution, we have
$$
\mathbb{E} \left[f(\widetilde{\bm x})\mathbf 1(\widetilde{\bm x}\in\mathcal S)\right]
=
\mathbb{E} \left[f(\bm X)\mathbf 1_{\mathcal S}(\bm X)\right],
$$
which implies $\mathbb{E}(B-C)=0$. Hence,
$$
\mathbb{E}(\widehat{\theta}_{\mathrm{FPPI}})
=\mathbb{E}(A)
=\mathbb{E}(Y)
=\theta^\star,
$$
establishing the unbiasedness.

Next, we compute the variance. Since $A$ and $B$ are independent, while $A$ and $C$ are dependent due to the shared labeled sample, we obtain
\begin{align*}
\mathrm{Var}(\widehat{\theta}_{\mathrm{FPPI}})
&= \mathrm{Var}(A+\lambda(B-C)) \\
&= \mathrm{Var}(A)+\lambda^2\mathrm{Var}(B-C)
+2\lambda  \mathrm{Cov}(A,B-C) \\
&= \mathrm{Var}(A)+\lambda^2\mathrm{Var}(B-C)
-2\lambda  \mathrm{Cov}(A,C).
\end{align*}
Since $B$ and $C$ are independent averages,
\begin{align*}
\mathrm{Var}(B-C)
&= \mathrm{Var}(B)+\mathrm{Var}(C) \\
&= \frac{1}{N}\mathrm{Var} \left(f(\bm X)\mathbf 1_{\mathcal S}(\bm X)\right)
+ \frac{1}{n}\mathrm{Var} \left(f(\bm X)\mathbf 1_{\mathcal S}(\bm X)\right) \\
&= \frac{N+n}{Nn}  
\mathrm{Var} \left(f(\bm X)\mathbf 1_{\mathcal S}(\bm X)\right).
\end{align*}
Moreover,
$$
\mathrm{Var}(A)=\frac{1}{n}\mathrm{Var}(Y),\qquad
\mathrm{Cov}(A,C)
=\frac{1}{n}\mathrm{Cov} \left(Y,f(\bm X)\mathbf 1_{\mathcal S}(\bm X)\right).
$$
Substituting the above expressions yields
\begin{align*}
&\mathrm{Var}(\widehat{\theta}_{\mathrm{FPPI}}(\lambda,\mathcal{S})) \\
= & \frac{1}{n}\mathrm{Var}(Y)
+ \lambda^2\frac{N+n}{Nn}
\mathrm{Var} \left(f(\bm X)\mathbf 1_{\mathcal S}(\bm X)\right)
- \frac{2\lambda}{n}
\mathrm{Cov} \left(Y,f(\bm X)\mathbf 1_{\mathcal S}(\bm X)\right).
\end{align*}
Minimizing the above variance function in $\lambda$ gives the optimal weight
\begin{align*}
\lambda^\star_{\mathcal{S}}
= \frac{\mathrm{Cov} \big(Y,  f(\bm X)\mathbf 1_{\mathcal S}(\bm X)\big)}
       {\mathrm{Var} \big(f(\bm X)\mathbf 1_{\mathcal S}(\bm X)\big)}
\cdot \frac{1}{1+n/N},
\end{align*}
provided that $\mathrm{Var}(f(\bm X)\mathbf 1_{\mathcal S}(\bm X))>0$. Substituting $\lambda^\star_{\mathcal{S}}$ back into the variance expression yields the minimized variance
\begin{align*}
\mathrm{Var}(\widehat{\theta}_{\mathrm{FPPI}}(\lambda^\star_{\mathcal{S}},\mathcal S))
= \frac{1}{n}\mathrm{Var}(Y)
- \frac{N}{n(N+n)}  
\frac{\mathrm{Cov}^2 \big(Y,f(\bm X)\mathbf 1_{\mathcal S}(\bm X)\big)}
     {\mathrm{Var} \big(f(\bm X)\mathbf 1_{\mathcal S}(\bm X)\big)}.
\end{align*}
This shows that for any subset $\mathcal S\subseteq\mathcal X$, using the optimal $\lambda^\star$ can only reduce the variance of the estimator compared with using labeled data alone. The proof is complete. \hfill {\color{red}$\blacksquare$} \\
\vspace{5mm}

\noindent \textbf{Proof of Theorem \ref{thm:filtered_optimality}.} 
For notational convenience, define
\begin{align*}
m(\bm x) &\triangleq \mathbb E[Y\mid \bm X= \bm x], \\
g(\bm x) &\triangleq \big(m(\bm x)-\mathbb E[Y]\big) f(\bm x),\\
A(\mathcal S) &\triangleq \mathrm{Cov} \big(Y, f(\bm X) \mathbf 1_{\mathcal S}(\bm X)\big) = \mathbb E\big[g(\bm X)\mathbf 1_{\mathcal S}(\bm X)\big], \\
M(\mathcal S) &\triangleq \mathbb E \big[f(\bm X) \mathbf 1_{\mathcal S}(\bm X )\big], \\
Q(\mathcal S) &\triangleq \mathbb E \big[f^2(\bm X) \mathbf 1_{\mathcal S}(\bm X)\big].
\end{align*}
Then the variance of $f(\bm X) \mathbf 1_{\mathcal S}(\bm X)$ is
$$
\mathrm{Var}(f(\bm X)\mathbf 1_{\mathcal S}(\bm X)) = Q(\mathcal S) - [M(\mathcal S)]^2.
$$
By the assumption $\mathbb E[f(\bm X)] = 0$, we have $M(\mathcal X) = 0$. Given the assumption that $\mathrm{Cov}(Y,f(\bm X)) \ge 0$, we have $A(\mathcal X) \ge 0$. Note that $A(\cdot)$, $M(\cdot)$, and $Q(\cdot)$ are linear in the indicator function. Hence, for any Borel set $\mathcal S \subseteq \mathcal X$,
$$
\begin{cases}
A(\mathcal X) = A(\mathcal S) + A(\mathcal S^c),\\
M(\mathcal X) = M(\mathcal S) + M(\mathcal S^c),\\
Q(\mathcal X) = Q(\mathcal S) + Q(\mathcal S^c),
\end{cases}
$$
where $\mathcal{S}^c = \mathcal{X}\setminus \mathcal{S}$ denotes the complement of $\mathcal{S}$. Define the objective ratio
$$
R(\mathcal S) \triangleq \frac{[A(\mathcal S)]^2}{\mathrm{Var}(f(\bm X)\mathbf 1_{\mathcal S}(\bm X))}=
\frac{[A(\mathcal S)]^2}{Q(\mathcal{S})-[M(\mathcal{S})]^2}.
$$
We aim to show $R(\mathcal X) < R(\mathcal S_0)$. Using the linearity relations, we have
\begin{align*}
R(\mathcal X) = & \frac{[A(\mathcal X)]^2}{Q(\mathcal X) - [M(\mathcal X)]^2} = \frac{[A(\mathcal S) + A(\mathcal S^c)]^2}{Q(\mathcal S) + Q(\mathcal S^c) - [M(\mathcal S)]^2 - [M(\mathcal S^c)]^2 - 2M(\mathcal S)M(\mathcal S^c)} \\
= &\frac{[A(\mathcal S) + A(\mathcal S^c)]^2}{\mathrm{Var}(f(\bm X)\mathbf 1_{\mathcal S}(\bm X))+\mathrm{Var}(f(\bm X)\mathbf 1_{\mathcal S^c}(\bm X)) - 2M(\mathcal S)M(\mathcal S^c)}.
\end{align*}

Since $M(\mathcal X) = M(\mathcal S) + M(\mathcal S^c) = 0$, we have $M(\mathcal S^c) = -M(\mathcal S)$, and thus
\begin{align*}
& Q(\mathcal S) + Q(\mathcal S^c) - [M(\mathcal S)]^2 - [M(\mathcal S^c)]^2 - 2M(\mathcal S)M(\mathcal S^c) \\
&= \mathrm{Var}(f(\bm X)\mathbf 1_{\mathcal S}(\bm X)) + \mathrm{Var}(f(\bm X)\mathbf 1_{\mathcal S^c}(\bm X)) + 2[M(\mathcal S)]^2 \\
&\ge \mathrm{Var}(f(\bm X)\mathbf 1_{\mathcal S}(\bm X)) + \mathrm{Var}(f(\bm X)\mathbf 1_{\mathcal S^c}(\bm X)).
\end{align*}
Therefore, for any $\mathcal S \subseteq \mathcal X$,
$$
R(\mathcal X) \le \frac{[A(\mathcal S) + A(\mathcal S^c)]^2}{\mathrm{Var}(f(\bm X)\mathbf 1_{\mathcal S}(\bm X)) + \mathrm{Var}(f(\bm X)\mathbf 1_{\mathcal S^c}(\bm X))} 
\le \frac{[A(\mathcal S) + A(\mathcal S^c)]^2}{\mathrm{Var}(f(\bm X)\mathbf 1_{\mathcal S}(\bm X))},
$$
where the last inequality follows from $\mathrm{Var}(f(\bm X)\mathbf 1_{\mathcal S^c}(\bm X)) \ge 0$. Now, define
$$
\mathcal S_0 = \big\{ \bm x \in \mathcal X : g(\bm x) > 0 \big\} = \big\{ \bm x \in \mathcal X : (m(\bm x) - \mathbb{E}(Y)) \cdot f(\bm x) > 0 \big\}.
$$
By construction of $\mathcal S_0$, we have $A(\mathcal S_0) > 0$ and $A(\mathcal S_0^c) \le 0$. Moreover, since $\mathbb P_{\bm X}(\mathcal S_0) \in (0,1)$, the set $\mathcal S_0^c$ has positive measure and $g(\bm x) \le 0$ on $\mathcal S_0^c$ with strict inequality on a set of positive measure, implying $A(\mathcal S_0^c) < 0$. Consequently,
$$
A(\mathcal X) = A(\mathcal S_0) + A(\mathcal S_0^c) < A(\mathcal S_0).
$$
Since $A(\mathcal X) \ge 0$ and $A(\mathcal S_0) > 0$, we have
$$
[A(\mathcal X)]^2 < [A(\mathcal S_0)]^2.
$$
Combining this with the previous inequality, we obtain
$$
R(\mathcal X) < \frac{[A(\mathcal S_0)]^2}{\mathrm{Var}(f(\bm X)\mathbf 1_{\mathcal S_0}(\bm X))} = R(\mathcal S_0).
$$
This completes the proof. \hfill {\color{red}$\blacksquare$}\\
\vspace{5mm}

\noindent\textbf{Proof of Theorem \ref{Thm:Disc}.} Note that $\mathbb{P}(\bm{X}=\bm{k}_t)>0$ for any $t \in [T]$. Therefore, as $n$ goes to infinity, we have $n_{\bm{k}_t}$ approaching infinity for each $t \in [T]$. First, $\widehat{\mathcal{S}}_0=\mathcal{S}_0$ indicates that
\begin{align}
\label{SimCond}
    \left(\frac{1}{n_{\bm{k}_t}}
     \sum_{i=1}^n y_i \cdot \bm{1}(\bm{x}_i=\bm{k}_t)-\frac{1}{n}\sum_{i=1}^n y_i\right)(m(\bm{k}_t)-\theta^\star)>0       \text{  for any } \bm{k}_t \in \mathcal{K}.
\end{align}
A sufficient condition for achieving (\ref{SimCond}) is 
\begin{align*}
 \max\left\{ \Big| \frac{1}{n_{\bm{k}_t}}
     \sum_{i=1}^n y_i   \cdot \bm{1}(\bm{x}_i=\bm{k}_t)- m(\bm{k}_t) \Big|, \Big|\frac{1}{n}\sum_{i=1}^n y_i-\theta^\star \Big| \right\} < \frac{|m(\bm{k}_t)-\theta^\star|}{2}      \text{  for any } \bm{k}_t \in \mathcal{K}.
\end{align*}
In other words, if $\widehat{\mathcal{S}}_0 \neq \mathcal{S}_0$, there must exist $\bm{k}_t \in \mathcal{K}$ such that
\begin{align*}
     \max\left\{ \Big| \frac{1}{n_{\bm{k}_t}}
     \sum_{i=1}^n y_i   \cdot \bm{1}(\bm{x}_i=\bm{k}_t)- m(\bm{k}_t) \Big|, \Big|\frac{1}{n}\sum_{i=1}^n y_i-\theta^\star \Big| \right\}  \geq \frac{|m(\bm{k}_t)-\theta^\star|}{2}.
\end{align*}
Therefore, note that $|\mathcal{K}|=T$, we have
\begin{align*}
    \mathbb{P}(\widehat{\mathcal{S}}_0 \neq \mathcal{S}_0)
    \leq &\mathbb{P}\left(
\bigcup_{\bm{k}_t \in \mathcal{K}}\left\{ \Big| \frac{1}{n_{\bm{k}_t}}
     \sum_{i=1}^n y_i   \cdot \bm{1}(\bm{x}_i=\bm{k}_t)- m(\bm{k}_t) \Big| \geq \frac{|m(\bm{k}_t)-\theta^\star|}{2}
     \right\}
    \right) \\
    + &\mathbb{P}\left(
\bigcup_{\bm{k}_t \in \mathcal{K}}\left\{ \Big| \frac{1}{n}
     \sum_{i=1}^n y_i   - \theta^\star \Big| \geq \frac{|m(\bm{k}_t)-\theta^\star|}{2}
     \right\}
    \right) \\
    \leq & \sum_{\bm k_t \in \mathcal{K}}
    \mathbb{P}\left(\Big|
\frac{1}{n_{\bm{k}_t}}
     \sum_{i=1}^n y_i   \cdot \bm{1}(\bm{x}_i=\bm{k}_t)- m(\bm{k}_t) \Big| \geq \frac{|m(\bm{k}_t)-\theta^\star|}{2}
    \right) \\
    &+\sum_{\bm k_t \in \mathcal{K}}
        \mathbb{P}\left(\Big| \frac{1}{n}
     \sum_{i=1}^n y_i   - \theta^\star \Big| \geq \frac{|m(\bm{k}_t)-\theta^\star|}{2}
        \right) \\
    \leq &2\sum_{t \in [T]}
    \exp\left(
-\frac{ n_{\bm{k}_t}\Delta_t^2 }{8\sigma^2}
    \right) + 2\sum_{t \in [T]}
    \exp\left(
-\frac{ n\Delta_t^2 }{8\sigma_Y^2}
    \right)
\end{align*}
where $\sigma_Y^2 = \sigma^2+\Delta_{\text{diff}}^2/4$ with $\Delta_{\text{diff}}=\max_{t \in [T]} m(\bm{k}_t) - \min _{t \in [T]} m(\bm{k}_t)$ by Lemma \ref{lem:marginal_subgaussian}, $\sigma^2$ is as defined in Assumption \ref{Ass:SubGaussian}, and the last inequality follows from the sub-Gaussian concentration inequality \citep{vershynin2018high}. This completes the proof. \hfill {\color{red}$\blacksquare$} \\

\vspace{5mm}

\noindent\textbf{Proof of Theorem \ref{thm:continuous}.}
If $\mathbf{X} \in \widehat{\mathcal{S}}_0 \,\Delta\, \mathcal{S}_0$, then at least one of the following events must occur:
$$
\bigl|\widehat{m}(\mathbf{X}) - m(\mathbf{X})\bigr|
\;>\;
\frac{1}{2}\,\bigl|m(\mathbf{X}) - \mathbb{E}(Y)\bigr|
\, \text{or} \,
\left|\frac{1}{n}\sum_{i=1}^n y_i - \mathbb{E}(Y)\right|
>
\frac{1}{2}\,\bigl|m(\mathbf{X}) - \mathbb{E}(Y)\bigr|.
$$
Therefore, $\mathbb{P}_{\mathbf{X}}(\widehat{\mathcal{S}}_0 \Delta \mathcal{S}_0)$ can be bounded by:
\begin{align*}
\mathbb{P}_{\mathbf{X}}(\widehat{\mathcal{S}}_0 \Delta \mathcal{S}_0) \le & \mathbb{P}_{\mathbf{X}} \left( \frac{1}{2}|m(\mathbf{X}) - \mathbb{E}(Y)| \le |\widehat{m}(\mathbf{X}) - m(\mathbf{X})| \right) \\
 + &\mathbb{P}_{\mathbf{X}} \left( \frac{1}{2}|m(\mathbf{X}) - \mathbb{E}(Y)| \le \left|\frac{1}{n}\sum_{i=1}^n y_i - \mathbb{E}(Y)\right| \right) \\
 \triangleq & I_1 + I_2.
\end{align*}
By the margin condition (Definition \ref{ass:margin}), for any $t>0$ we have
$$
\mathbb{P}_{\mathbf{X}} \big( |m(\mathbf{X}) - \mathbb{E}(Y)| \le t \big) \le C t^\tau.
$$
Here, $\tau$ is a fixed constant. Applying this margin condition to $I_1$, for any given $\widehat{m}(\cdot)$, we have
\begin{align*}
   I_1= &\mathbb{P}_{\mathbf{X}} \big( |m(\mathbf{X}) - \mathbb{E}(Y)| \le 2|\widehat{m}(\mathbf{X}) - m(\mathbf{X})| \big) \\
=& \mathbb{E}_{\bm{X}}
\left[
\bm{1}_{\{|m(\mathbf{X}) - \mathbb{E}(Y)| \le2 |\widehat{m}(\mathbf{X}) - m(\mathbf{X})|\}}(\bm{X})
\right] \\
=&\mathbb{E}_{\bm{X}}
\left[ \bm{1}_{\{|m(\bm{X})-\mathbb{E}(Y)|\leq t\}}(\bm{X}) \cdot
\bm{1}_{\{|m(\mathbf{X}) - \mathbb{E}(Y)| \le 2|\widehat{m}(\mathbf{X}) - m(\mathbf{X})|\}}(\bm{X})
\right] \\
+&\mathbb{E}_{\bm{X}}
\left[ \bm{1}_{\{|m(\bm{X})-\mathbb{E}(Y)|> t\}}(\bm{X}) \cdot
\bm{1}_{\{|m(\mathbf{X}) - \mathbb{E}(Y)| \le 2|\widehat{m}(\mathbf{X}) - m(\mathbf{X})|\}}(\bm{X})
\right] \\
\leq & \mathbb{E}_{\bm{X}}
\left[ \bm{1}_{\{|m(\bm{X})-\mathbb{E}(Y)|\leq t\}}(\bm{X}) 
\right]+\mathbb{E}_{\bm{X}}
\left[ 
\bm{1}_{\{t/2 \le |\widehat{m}(\mathbf{X}) - m(\mathbf{X})|\}}(\bm{X})
\right] \\
\leq  &\mathbb{P}_{\bm{X}}(|m(\bm{X})-\mathbb{E}(Y)|\leq t)+\mathbb{P}_{\bm{X}}(|\widehat{m}(\mathbf{X}) - m(\mathbf{X})|\geq t/2) \\
\leq  &C t^{\tau}  + \mathbb{P}_{\bm{X}}(|\widehat{m}(\mathbf{X}) - m(\mathbf{X})|\geq t/2).
\end{align*}
Take the expectation with respect to the randomness of $\mathcal{D}_L$, we have
\begin{align*}
    \mathbb{E}_{\mathcal{D}_L}
   \left( I_1 \right) \leq & C  t^{\tau}+
 \mathbb{E}_{\mathcal{D}_L}
    \left(
\mathbb{E}_{\bm{X}}
\left(
\bm{1}_{|\widehat{m}(\mathbf{X}) - m(\mathbf{X})|\geq t/2}(\bm{X})
\right)
    \right) \\
    = & C  t^{\tau} + \mathbb{E}_{\bm{X}}\left(
\mathbb{E}_{\mathcal{D}_L}\left(
\bm{1}_{|\widehat{m}(\mathbf{X}) - m(\mathbf{X})|\geq t/2}(\bm{X})
\right)
    \right) \\
    = &C  t^{\tau} +\mathbb{E}_{\bm{X}}
    \left(c_1\exp\left( - c_2 \alpha_n t^2/4 \right)\right) \\
    = & C  t^{\tau}+c_1\exp\left( - c_2 \alpha_n t^2/4 \right).
\end{align*}
For $I_2$, using the same argument, we have
\begin{align*}
    \mathbb{E}_{\mathcal{D}_L}(I_2) \leq C t^{\tau} + \mathbb{P}\left(\Big|\frac{1}{n}\sum_{i=1}^n y_i - \mathbb{E}(Y)\Big| \geq \frac{t}{2}\right) \leq C t^{\tau}+2\exp \left(
-\frac{n t^2}{8\sigma_Y^2}
    \right),
\end{align*}
where the last inequality follows from the fact that $Y$ is a sub-Gaussian random variable with the proxy variance $\sigma_Y^2$ (by Lemma \ref{lem:marginal_subgaussian}).

To sum up, we have
\begin{align*}
    \mathbb{E}_{\mathcal{D}_L}
    \left(\mathbb{P}_{\mathbf{X}}(\widehat{\mathcal{S}}_0 \Delta \mathcal{S}_0) \right) \leq 2C t^{\tau}+2\exp \left(
-\frac{n t^2}{8\sigma_Y^2}
    \right)+c_1\exp\left( - \frac{c_2 \alpha_n t^2}{4} \right).
\end{align*}
To minimize the upper bound, one can choose $t$ to balance these competing effects. Let 
$$
\kappa_n \triangleq n \wedge \alpha_n =\min\left\{\frac{1}{8\sigma_Y^2},\frac{c_2}{4}\right\} \cdot \min\{n, \alpha_n\},
$$
then the bound can be re-written as
\begin{align*}
    \mathbb{E}_{\mathcal{D}_L}
    \left(\mathbb{P}_{\mathbf{X}}(\widehat{\mathcal{S}}_0 \Delta \mathcal{S}_0) \right) \leq 2C t^{\tau}+(c_1+2)
    \exp(-\kappa_n t^2).
\end{align*}
Choosing $t^\star = \sqrt{\frac{\tau \log \kappa_n}{2 \kappa_n}}$ balances the polynomial and exponential terms in the upper bound, it then follows that
$$
\mathbb{E}_{\mathcal{D}_L}\Big( \mathbb{P}_{\mathbf{X}}(\widehat{\mathcal{S}}_0 \Delta \mathcal{S}_0) \Big)
\lesssim \left(
\frac{\log \kappa_n}{\kappa_n}
\right)^{\frac{\tau}{2}} \asymp \left(
\frac{\log (\alpha_n \wedge n)}{\alpha_n  \wedge n}
\right)^{\frac{\tau}{2}}.
$$
This completes the proof. \hfill ${\color{red}\blacksquare}$ \\

\vspace{5mm}

\noindent\textbf{Proof of Theorem \ref{Prop:Exp}.} The proof of Theorem \ref{Prop:Exp} is similar to that of Theorem \ref{thm:continuous}.
For any fixed $\bm{x}$, if 
$\bm{x} \in \widehat{\mathcal{S}}_0 \Delta \mathcal{S}_0$, 
then a sign error must occur. Hence,
\begin{align*}
  \mathbf{1}\{\bm{x} \in \widehat{\mathcal{S}}_0 \Delta \mathcal{S}_0\}
\le  &
\mathbf{1}\left\{
|\widehat m(\bm{x}) - m(\bm{x})|
\ge \frac{1}{2}
|m(\bm{x}) - \mathbb{E}(Y)|
\right\} \\
+ &\mathbf{1}\left\{
|\sum_{i=1}^n y_i/n - \mathbb{E}(Y)|
\ge \frac{1}{2}
|m(\bm{x}) - \mathbb{E}(Y)|
\right\}.
\end{align*}
By the assumption $|m(\bm{x}) - \mathbb{E}(Y)| \ge c_3$ almost surely,
$$
\mathbf{1}\{\bm{x} \in \widehat{\mathcal{S}}_0 \Delta \mathcal{S}_0\}
\le 
\mathbf{1}\{|\widehat m(\bm{x}) - m(\bm{x})| \ge c_3/2\}+\mathbf{1}\left\{
\Big|\sum_{i=1}^n y_i/n - \mathbb{E}(Y)\Big|
\ge  c_3/2
\right\}
$$
Taking expectation over the training randomness $\mathcal{D}_L$ and using
Assumption~\ref{Ass:TailConvergence} the Fubini's theorem,
\begin{align*}
    \mathbb{E}_{\mathcal{D}_L} \left[
\mathbf{1}\{\bm{x} \in \widehat{\mathcal{S}}_0 \Delta \mathcal{S}_0\}
\right] \leq &
\mathbb{E}_{\bm{X}}\left[ \mathbb{P}_{\mathcal{D}_L}
\left(
|\widehat m(\bm{x}) - m(\bm{x})| \ge c_3/2
\right)\right]  \\
 &+
\mathbb{E}_{\bm{X}}\left[ \mathbb{P}_{\mathcal{D}_L}
\left(
|\sum_{i=1}^n y_i/n - \mathbb{E}(Y)|
\ge  c_3/2
\right)\right] \\
\leq & c_1 \exp(-c_2 \alpha_n c_3^2/4)+
2  \exp(-n c_3^2/(8\sigma_Y^2)).
\end{align*}
This completes the proof. \hfill ${\color{red}\blacksquare}$ \\
\vspace{5mm}

\noindent
\textbf{Proof of Theorem \ref{thm:FPPI_GLM_final}.} We derive the asymptotic distribution using standard M-estimation arguments. The gradient of the FPPI objective function is given by
$$
\nabla_{\bm\theta}\mathcal L_{\mathrm{FPPI}}(\bm\theta\,|\,\lambda,\mathcal{S})
=
\bm g_n(\bm\theta)
+
\lambda \bm h_{n,N}(\bm\theta),
$$
where $\bm g_n(\bm\theta)=
\frac1n\sum_{i=1}^n
\big(A'(\bm x_i^\top\bm\theta)-y_i\big)\bm x_i$ and $\bm h_{n,N}(\bm\theta)$ is given as
\begin{align*}
\bm h_{n,N}(\bm\theta)
=
\frac1N\sum_{j=1}^N
\big(A'(\widetilde{\bm x}_j^\top\bm\theta)-f(\widetilde{\bm x}_j)\big)
\widetilde{\bm x}_j\mathbf 1_{\mathcal S}(\widetilde{\bm x}_j)-
\frac1n\sum_{i=1}^n
\big(A'(\bm x_i^\top\bm\theta)-f(\bm x_i)\big)
\bm x_i\mathbf 1_{\mathcal S}(\bm x_i).
\end{align*}
Since $\nabla_{\bm\theta}\mathcal L_{\mathrm{FPPI}}
(\widehat{\bm\theta}_{\mathrm{FPPI}}(\lambda,\mathcal{S})\,|\,\lambda,\mathcal{S})=\bm 0$, the first-order Taylor expansion around $\bm\theta^\star$ yields
$$
\bm 0=\nabla_{\bm\theta}\mathcal L_{\mathrm{FPPI}}(\bm\theta^\star\,|\,\lambda,\mathcal{S})+\nabla^2_{\bm\theta}\mathcal L_{\mathrm{FPPI}}(\bar{\bm\theta}\,|\,\lambda,\mathcal{S})
\big(\widehat{\bm\theta}_{\mathrm{FPPI}}(\lambda,\mathcal{S})-\bm\theta^\star\big),
$$
where $\bar{\bm\theta}$ lies between
$\widehat{\bm\theta}_{\mathrm{FPPI}}(\lambda,\mathcal{S})$ and $\bm\theta^\star$. Note that $\nabla^2_{\bm\theta}\mathcal L_{\mathrm{FPPI}}(\bar{\bm\theta})\xrightarrow{p}\bm \Sigma$. Consequently,
$$
\sqrt n
\big(
\widehat{\bm\theta}_{\mathrm{FPPI}}(\lambda,\mathcal{S})-\bm\theta^\star
\big)
=
-\bm\Sigma^{-1}
\sqrt n\,
\nabla_{\bm\theta}\mathcal L_{\mathrm{FPPI}}(\bm\theta^\star\,|\,\lambda,\mathcal{S})
+o_p(1).
$$
Note that
\begin{align*}
    \nabla_{\bm\theta}\mathcal L_{\mathrm{FPPI}}(\bm\theta^\star\,|\,\lambda,\mathcal{S})
= &
-\frac1n\sum_{i=1}^n
R(\bm{x}_i,y_i)\bm x_i
+
\lambda
\Big(
\frac1N\sum_{j=1}^N \widetilde{\bm \delta}_j
-
\frac1n\sum_{i=1}^n \bm \delta_i
\Big) \\
= & -\frac1n\sum_{i=1}^n
(y_i - \mu(\bm{x}_i))\bm x_i
+
\lambda
\Big(
\frac1N\sum_{j=1}^N \widetilde{\bm \delta}_j
-
\frac1n\sum_{i=1}^n \bm \delta_i
\Big),
\end{align*}
where $\bm{\delta}_i = \big(\mu(\bm x_i)-f(\bm x_i)\big)
\bm x_i\mathbf 1_{\mathcal S}(\bm x_i)$ and $\widetilde{\bm{\delta}}_j = \big(\mu(\widetilde{\bm x}_j)-f(\widetilde{\bm x}_j)\big)
\widetilde{\bm x}_j\mathbf 1_{\mathcal S}(\widetilde{\bm x}_j)$. Note that $\bm{\theta}^\star$ is the global minimizer of $\mathbb{E}[L_{\bm{\theta}}(\bm{X},Y)]$. Therefore, we have
\begin{align*}
    \mathbb{E}[\nabla_{\bm{\theta}} L_{\bm{\theta}^\star}(\bm{X},Y)]=0
    \Longrightarrow \mathbb{E}[\bm{X}(A'(\bm{X}^\top \bm{\theta}^\star)-Y)]=-\mathbb{E}[\bm{X}(Y-\mu(\bm{X}))]=\bm{0}.
\end{align*}
Therefore, 
\begin{align*}
  \bm{\Omega} =&  \text{Var}(\bm{X}(Y-\mu(\bm{X}))) =\mathbb{E}[R^2(\bm{X},Y)\bm{X}\bm{X}^\top] , \\
   \bm{\Gamma}_{\mathcal{S}}=&\text{Cov}\Big((Y-\mu(\bm{X}))\bm{X},(f(\bm{X})-\mu(\bm{X}) )\bm{X}\bm{1}_{\mathcal{S}}(\bm{X})\Big)   \\
   = & 
   \mathbb{E}[R(\bm{X},Y)(f(\bm{X})-\mu(\bm{X}))\bm{X}\bm{X}^\top\bm{1}_{\mathcal{S}}(\bm{X})] , \\
   \bm{M}_{\mathcal{S}} = &\text{Var} \Big(\big(\mu(\bm X)-f(\bm X)\big)
\bm X\mathbf 1_{\mathcal S}(\bm{X})\Big) \\
= & \mathbb{E}\left[\bm{X}\bm{X}^\top(\mu(\bm{X})-f(\bm{X}))^2\bm{1}_{\mathcal{S}}(\bm{X})\right]-
 \bm{D}_{\mathcal{S}} \bm{D}_{\mathcal{S}}^\top,
\end{align*}
where $ \bm{D}_{\mathcal{S}} =\mathbb{E}\left[\bm{X}(\mu(\bm{X})-f(\bm{X}))\bm{1}_{\mathcal{S}}(\bm{X})\right]$.

By the multivariate central limit theorem,
$$
\sqrt n\,
\nabla_{\bm\theta}\mathcal L_{\mathrm{FPPI}}(\bm\theta^\star \,|\, \lambda,\mathcal{S})
\xrightarrow{d}
\mathcal N\Big(
\bm 0,\;
\bm\Omega
+\lambda^2(1+r)\bm M_{\mathcal S}
-2\lambda \bm\Gamma_{\mathcal S}
\Big).
$$
The desired result follows from Slutsky's theorem.

The asymptotic mean squared error (AMSE) of 
\begin{align*}
 n \cdot \mathbb{E} \left[
\Vert
\widehat{\bm\theta}_{\mathrm{FPPI}}(\lambda,\mathcal{S})-\bm\theta^\star \Vert_2^2\right] = \mathrm{tr} \left(
\bm \Sigma^{-1}
\big(
\bm\Omega
+\lambda^2(1+r)\bm M_{\mathcal S}
-2\lambda \bm\Gamma_{\mathcal S}
\big)
\bm \Sigma^{-1}
\right).
\end{align*}
Treating the above formula as a function of $\lambda$, we have
\begin{align*}
    \lambda^\star_{\mathcal{S}} = \frac{\mathrm{tr} \left(
\bm \Sigma^{-1}\bm\Gamma_{\mathcal S}
\bm \Sigma^{-1}
\right)}{(1+r)\mathrm{tr} \left(
\bm \Sigma^{-1}
\bm M_{\mathcal S}
\bm \Sigma^{-1}
\right)}.
\end{align*}
Plugging $\lambda^\star_{\mathcal{S}}$ into AMSE yields
\begin{align*}
    \mathrm{AMSE}(\lambda^\star_{\mathcal S},\mathcal S)
=
\mathrm{tr}(\bm\Sigma^{-1}\bm\Omega\bm\Sigma^{-1})
-
\frac{
\big[\mathrm{tr}(\bm\Sigma^{-1}\bm\Gamma_{\mathcal S}\bm\Sigma^{-1})\big]^2
}{
(1+r)\,
\mathrm{tr}(\bm\Sigma^{-1}\bm M_{\mathcal S}\bm\Sigma^{-1})
}.
\end{align*}
This completes the proof. \hfill ${\color{red}\blacksquare}$ \\

\vspace{5mm}

\noindent\textbf{Proof of Theorem \ref{Thm:GLM_Estimate_S0}.}
We first recall that the estimated region is defined as
$$
    \widehat{\mathcal{S}}_0 
    = 
    \Big\{
    \bm{x} \in \mathcal{X} :
    \big[f(\bm{x}) - A'(\bm{x}^\top \widehat{\bm{\theta}})\big]  
    \big[\widehat{m}(\bm{x}) - A'(\bm{x}^\top \widehat{\bm{\theta}})\big] > 0
    \Big\}.
$$
If $\bm{X} \in \widehat{\mathcal{S}}_0 \Delta \mathcal{S}_0$, then the sign of at least one of the two factors must be incorrectly estimated. Hence, one of the following two events must occur:
$$
    \underbrace{\big(f(\bm{X}) - A'(\bm{X}^\top \bm{\theta}^\star)\big)
    \big(f(\bm{X}) - A'(\bm{X}^\top \widehat{\bm{\theta}})\big)
    < 0}_{\mathcal{E}_1}
    \quad \text{or} \quad
    \underbrace{\big(\widehat{m}(\bm{X}) - A'(\bm{X}^\top \widehat{\bm{\theta}})\big)
    \big(m(\bm{X}) - A'(\bm{X}^\top \bm{\theta}^\star)\big)
    < 0}_{\mathcal{E}_2}.
$$
Therefore,
\begin{align*}
    \mathbb{P}_{\mathbf{X}}(\widehat{\mathcal{S}}_0 \Delta \mathcal{S}_0)
    \leq
    \mathbb{P}_{\mathbf{X}}(\mathcal{E}_1)
    + \mathbb{P}_{\mathbf{X}}(\mathcal{E}_2)
    \triangleq P_1 + P_2,
\end{align*}
where $\mathcal{E}_1$ and $\mathcal{E}_2$ denote the two events above.
Taking expectation over $\mathcal{D}_L$ yields
$$
    \mathbb{E}_{\mathcal{D}_L}
    \big[
    \mathbb{P}_{\mathbf{X}}(\widehat{\mathcal{S}}_0 \Delta \mathcal{S}_0)
    \big]
    \le
    \mathbb{E}_{\mathcal{D}_L}(P_1)
    +
    \mathbb{E}_{\mathcal{D}_L}(P_2).
$$

\vspace{2mm}
\noindent\textbf{Step 1: Bounding $\mathbb{E}_{\mathcal{D}_L}(P_1)$.}
Note that
$$
    \big(f(\bm{X}) - A'(\bm{X}^\top \bm{\theta}^\star)\big)
    \big(f(\bm{X}) - A'(\bm{X}^\top \widehat{\bm{\theta}})\big)
    < 0
    \;\Longrightarrow\;
    \big|
    A'(\bm{X}^\top\widehat{\bm{\theta}})-A'(\bm{X}^\top\bm{\theta}^\star)
    \big|
    >
    \big|f(\bm{X}) - A'(\bm{X}^\top \bm{\theta}^\star)\big|.
$$
Since $|A''(\bm{X^\top\bm{\theta}})|$ is bounded and $\Vert \bm{X}\Vert_2 \leq M_{\mathcal{X}}$, we have
\begin{align*}
    \big|
    A'(\bm{X}^\top\widehat{\bm{\theta}})-A'(\bm{X}^\top\bm{\theta}^\star)
    \big| \leq C^\prime \Vert \widehat{\bm{\theta}}-\bm{\theta}^\star\Vert_2,
\end{align*}
for some constant $C^\prime$. Hence, using the margin condition, for any $t>0$, 
\begin{align*}
    P_1 \leq & \mathbb{P}_{\bm{X}}\left(
\left\{C^\prime \Vert \widehat{\bm{\theta}}-\bm{\theta}^\star\Vert_2
    >
    \big|f(\bm{X}) - A'(\bm{X}^\top\bm{\theta}^\star)\big|\right\} \bigcap 
    \left\{ \big|f(\bm{X}) - A'(\bm{X}^\top\bm{\theta}^\star)\big| >t \right\}
    \right)  \\
    + & \mathbb{P}_{\bm{X}}\left(
\left\{C^\prime \Vert \widehat{\bm{\theta}}-\bm{\theta}^\star\Vert_2
    >
    \big|f(\bm{X}) - A'(\bm{X}^\top\bm{\theta}^\star)\big|\right\} \bigcap 
    \left\{ \big|f(\bm{X}) - A'(\bm{X}^\top\bm{\theta}^\star)\big| \leq t \right\}
    \right)  \\
    \leq &
    \mathbb{P}_{\bm{X}} \left(
C^\prime \Vert \widehat{\bm{\theta}}-\bm{\theta}^\star\Vert_2 > t
    \right)
    +
    \mathbb{P}_{\bm{X}} \left(
    \big|f(\bm{X}) -A'(\bm{X}^\top\bm{\theta}^\star)\big| \le t
    \right)
    \\
    \triangleq & P_{11}(t)+P_{12}(t).
\end{align*}
Therefore, it holds that
$$
   \mathbb{E}_{\mathcal{D}_L} [P_{11}(t)]
    \le
    \mathbb{P}_{\mathcal{D}_L}
    \Big(
    \|\widehat{\bm{\theta}}-\bm{\theta}^\star\|_2
    \ge
    t/C^\prime
    \Big).
$$
Moreover, the margin condition implies
$\mathbb{P}_{\bm{X}}(|f(\bm{X})-\bm{X}^\top\bm{\theta}^\star|\le t)
\le Ct^\tau$.
Therefore,
$$
    \mathbb{E}_{\mathcal{D}_L}(P_1)
    \le
    \mathbb{P}_{\mathcal{D}_L}
    \Big(
    \|\widehat{\bm{\theta}}-\bm{\theta}^\star\|_2
    \ge
    t/C^\prime
    \Big)
    +
    Ct^\tau.
$$

\vspace{2mm}
\noindent\textbf{Step 2: Bounding $\mathbb{E}_{\mathcal{D}_L}(P_2)$.}
If
$$
    \big(\widehat{m}(\bm{X}) - A'(\bm{X}^\top \widehat{\bm{\theta}})\big)
    \big(m(\bm{X}))- A'(\bm{X}^\top \bm{\theta}^\star)\big)
    < 0,
$$
then by the triangle inequality,
at least one of the following must occur:
\begin{align*}
    \big|
    A'(\bm{X}^\top\widehat{\bm{\theta}})-A'(\bm{X}^\top\bm{\theta}^\star)
    \big|
    >
    \tfrac12
    \big|m(\bm{X}) - A'(\bm{X}^\top\bm{\theta}^\star)\big|
    \, \text{or} \,
    \big|\widehat{m}(\bm{X})-m(\bm{X})\big|
    >
    \tfrac12
    \big|m(\bm{X}) - A'(\bm{X}^\top\bm{\theta}^\star)\big|.
\end{align*}
Thus, we have
\begin{align*}
  P_2 \leq &\mathbb{P}_{\bm{X}}\left(    C^\prime\|\widehat{\bm{\theta}}-\bm{\theta}^\star\|_2
    >
    \tfrac12
    \big|m(\bm{X}) - A'(\bm{X}^\top\bm{\theta}^\star)\big|\right)+  
    \mathbb{P}_{\bm{X}}\left(\big|\widehat{m}(\bm{X})-m(\bm{X})\big|
    >
    \tfrac12
    \big|m(\bm{X}) - A'(\bm{X}^\top\bm{\theta}^\star)\big|\right) \\
    \triangleq & P_{21} +P_{22}.
\end{align*}
Using the same argument as in Step~1,
$$
    \mathbb{E}_{\mathcal{D}_L}(P_{21})
    \le
    \mathbb{P}_{\mathcal{D}_L}
    \Big(
    \|\widehat{\bm{\theta}}-\bm{\theta}^\star\|_2
    \ge
    t/(2C^\prime)
    \Big)
    +
    Ct^\tau.
$$
For $P_{22}$, applying a union bound and Assumption~\ref{Ass:TailConvergence},
$$
    \mathbb{E}_{\mathcal{D}_L}(P_{22})
    \le
    Ct^\tau
    +
    c_1\exp \big(-c_2\alpha_n t^2/4\big).
$$

Combining the bounds yields
\begin{align*}
    \mathbb{E}_{\mathcal{D}_L}
    \big[
    \mathbb{P}_{\mathbf{X}}(\widehat{\mathcal{S}}_0 \Delta \mathcal{S}_0)
    \big]
    \le
    Ct^\tau
    +
    C_1\exp(-C_2 n)
    +
    C_3\exp \big(-C_4 (n\wedge\alpha_n)t^2\big).
\end{align*}
Choosing
$$
    t \asymp
    \sqrt{\frac{p+\log(n\wedge\alpha_n)}{n\wedge\alpha_n}}
$$
balances the polynomial and exponential terms and yields
$$
    \mathbb{E}_{\mathcal{D}_L}
    \big[
    \mathbb{P}_{\mathbf{X}}(\widehat{\mathcal{S}}_0 \Delta \mathcal{S}_0)
    \big]
    \;\lesssim\;
    \Big(
    \frac{p+\log(n\wedge\alpha_n)}{n\wedge\alpha_n}
    \Big)^{\tau/2}.
$$

\vspace{2mm}
\noindent\textbf{Strong separation.}
If there exists $c_3>0$ such that
$$
    \min\Big\{
    |f(\bm{X})-A'(\bm{X}^\top\bm{\theta}^\star)|,
    |m(\bm{X})-A'(\bm{X}^\top\bm{\theta}^\star)|
    \Big\}
    \ge c_3
    \quad \text{a.s.},
$$
then the margin terms vanish. Repeating the above arguments,
\begin{align*}
    \mathbb{E}_{\mathcal{D}_L}
    \big[
    \mathbb{P}_{\mathbf{X}}(\widehat{\mathcal{S}}_0 \Delta \mathcal{S}_0)
    \big]
    \le
    2\mathbb{P}_{\mathcal{D}_L}
    \Big(
    \|\widehat{\bm{\theta}}-\bm{\theta}^\star\|_2
    \ge
    c_3/(2C^\prime)
    \Big)
    +
    c_1\exp \big(-c_2\alpha_n c_3^2/4\big).
\end{align*}
Applying Lemma~\ref{lem:glm_misspecified_tail},
$$
    \mathbb{E}_{\mathcal{D}_L}
    \big[
    \mathbb{P}_{\mathbf{X}}(\widehat{\mathcal{S}}_0 \Delta \mathcal{S}_0)
    \big]
    \;\lesssim\;
    \exp \big(-C (n\wedge\alpha_n)c_3^2\big).
$$
This completes the proof.
\hfill {\color{Red}$\blacksquare$} \\

\vspace{8mm}

\section{Proof of Corollaries}

\vspace{5mm}

\noindent\textbf{Proof of Corollary \ref{Coro:Mean_Inference}.} First, by Theorem \ref{thm:fppi_mean_var} and the central limit theorem, we have
\begin{align*}
 \sqrt{n}\left(
\widehat{\theta}_{\mathrm{FPPI}}(\lambda^\star_{\mathcal{S}_0},\mathcal{S}_0) - \theta^\star
 \right) \xrightarrow{d}
 \mathcal{N}\left(
0,\mathrm{Var}(Y)
- \frac{1}{(1+r)}  
\frac{\mathrm{Cov}^2\big(Y,f(\bm X)\mathbf 1_{\mathcal S_0}(\bm X)\big)}
     {\mathrm{Var}\big(f(\bm X)\mathbf 1_{\mathcal S_0}(\bm X)\big)}
 \right).
\end{align*}
To prove the desired result, it remains to show that
\begin{align*}
   \widehat{\theta}_{\mathrm{FPPI}}(\widehat \lambda_{\mathcal{S}_0},\widehat{\mathcal{S}}_0) - \widehat{\theta}_{\mathrm{FPPI}}(\lambda^\star_{\mathcal{S}_0},\mathcal{S}_0)  = o_p(n^{-\frac{1}{2}}).
\end{align*}
To this end, we first consider the following decomposition:
\begin{align*}
   & \widehat{\theta}_{\mathrm{FPPI}}(\widehat \lambda_{\mathcal{S}_0},\widehat{\mathcal{S}}_0) - \widehat{\theta}_{\mathrm{FPPI}}(\lambda^\star_{\mathcal{S}_0},\mathcal{S}_0) \\
    =&\underbrace{\widehat{\theta}_{\mathrm{FPPI}}(\widehat{\lambda}_{\mathcal{S}_0},\widehat{\mathcal{S}}_0)-
    \widehat{\theta}_{\mathrm{FPPI}}(\widehat{\lambda}_{\mathcal{S}_0},\mathcal{S}_0)}_{\triangleq E_1}+
    \underbrace{\widehat{\theta}_{\mathrm{FPPI}}(\widehat{\lambda}_{\mathcal{S}_0},\mathcal{S}_0)-
     \widehat{\theta}_{\mathrm{FPPI}}(\lambda_{\mathcal{S}_0}^\star,\mathcal{S}_0)}_{\triangleq E_2}.
\end{align*}
In what follows, we proceed to show that $E_1=o_p(n^{-\frac{1}{2}})$ and $E_2=o_p(n^{-\frac{1}{2}})$ separately. To this end, we mainly intend to prove $\mathbb{E}[|E_1|] = o(n^{-\frac{1}{2}})$ and $\mathbb{E}[|E_2|] = o(n^{-\frac{1}{2}})$, which combined with the Markov's inequality leads to the desired result.

\vspace{8mm}

\noindent\textbf{Step 1: Bounding $E_1$.}
\begin{align*}
  |E_1| = &\Big|\widehat{\theta}_{\mathrm{FPPI}}(\widehat \lambda_{\mathcal{S}_0},\widehat{\mathcal{S}}_0)-
    \widehat{\theta}_{\mathrm{FPPI}}(\widehat \lambda_{\mathcal{S}_0},\mathcal{S}_0) \Big|\\
    = &
\Bigg|    \widehat{\lambda}_{\mathcal{S}_0} \cdot\left\{
\frac{1}{N}\sum_{j=1}^N f(\widetilde{\bm{x}}_j) \cdot \left[\bm{1}_{\widehat{\mathcal{S}}_0}(\widetilde{\bm{x}}_j)-
\bm{1}_{\mathcal{S}_0}(\widetilde{\bm{x}}_j)\right]-
\frac{1}{n}\sum_{i=1}^n f(\bm{x}_i) \cdot \left[\bm{1}_{\widehat{\mathcal{S}}_0}(\bm{x}_i)-
\bm{1}_{\mathcal{S}_0}(\bm{x}_i)\right]
    \right\}\Bigg|    \\
    \leq & |\widehat{\lambda}_{\mathcal{S}_0}|\cdot
    \left\{
\frac{1}{N}\sum_{j=1}^N |f(\widetilde{\bm{x}}_j)| \cdot \bm{1}_{\widehat{\mathcal{S}}_0\Delta \mathcal{S}_0}(\widetilde{\bm{x}}_j)+
\frac{1}{n}\sum_{i=1}^n |f(\bm{x}_i)| \cdot
\bm{1}_{\widehat{\mathcal{S}}_0\Delta \mathcal{S}_0}(\bm{x}_i)
    \right\}.
\end{align*}
Note that $\widehat{\lambda}_{\mathcal{S}_0}$ is bounded by assumption and $\Vert f\Vert_{\infty}=\sup_{\bm{x} \in \mathcal{X}}f(\bm{x}) \leq M_{f}$, it holds that
\begin{align*}
 &|\widehat{\theta}_{\mathrm{FPPI}}(\widehat\lambda_{\mathcal S_0},\widehat{\mathcal{S}}_0)-
    \widehat{\theta}_{\mathrm{FPPI}}(\widehat\lambda_{\mathcal S_0},\mathcal{S}_0)| \\ \leq & 
    \Big|
\underbrace{\frac{C}{N}\sum_{j=1}^N \left[\bm{1}_{\widehat{\mathcal{S}}_0\Delta \mathcal{S}_0}(\widetilde{\bm{x}}_j)-    \mathbb{P}_{\bm{X}}\left(\widehat{\mathcal{S}}_0\Delta \mathcal{S}_0\right)\right]}_{E_{11}}\Big|+\Big|\underbrace{
\frac{C}{n}\sum_{i=1}^n \bm{1}_{\widehat{\mathcal{S}}_0\Delta \mathcal{S}_0}(\bm{x}_i)}_{E_{12}}\Big|+C     \mathbb{P}_{\bm{X}}\left(\widehat{\mathcal{S}}_0\Delta \mathcal{S}_0\right),
\end{align*}
for some positive constant $C$. By the assumption that $\mathbb{E}\left[\mathbb{P}_{\bm{X}}\left(\widehat{\mathcal{S}}_0\Delta \mathcal{S}_0\right)\right]=o(n^{-\frac{1}{2}})$, it then follows from the Markov's inequality that $\mathbb{P}_{\bm{X}}\left(\widehat{\mathcal{S}}_0\Delta \mathcal{S}_0\right)=o_p(n^{-\frac{1}{2}})$. Next, we proceed to bound $E_{11}$ and $E_{12}$, respectively.

\vspace{8mm}

\textbf{Step 1.1: Bounding $E_{11}$.} Note that the construction of $\widehat{\mathcal{S}}_0$ only depends on the labeled data $\mathcal{D}_L$ and is independent of the unlabeled sample $\{\widetilde{\bm{x}}_j\}_{j=1}^N$. Therefore,
\begin{align*}
   & \mathbb{P}
    \left(
    \left|
    \frac{1}{N}\sum_{j=1}^N \left[
    \bm{1}_{\widehat{\mathcal{S}}_0\Delta \mathcal{S}_0}(\widetilde{\bm{x}}_j)
    -
    \mathbb{P}_{\bm{X}}\left(\widehat{\mathcal{S}}_0\Delta \mathcal{S}_0\right)
    \right]
    \right|\geq \delta
    \right) \\
    ={}& \mathbb{E}\left\{
    \mathbb{P}
    \left(
    \left.
    \left|
    \frac{1}{N}\sum_{j=1}^N \left[
    \bm{1}_{\widehat{\mathcal{S}}_0\Delta \mathcal{S}_0}(\widetilde{\bm{x}}_j)
    -
    \mathbb{P}_{\bm{X}}\left(\widehat{\mathcal{S}}_0\Delta \mathcal{S}_0\right)
    \right]
    \right|\geq \delta
    \;\right|\;
    \mathcal{D}_L
    \right)
    \right\}.
\end{align*}
Conditional on $\mathcal{D}_L$, the random variables
$\bm{1}_{\widehat{\mathcal{S}}_0\Delta \mathcal{S}_0}(\widetilde{\bm{x}}_j)$
are i.i.d. Bernoulli with success probability
$p_L \triangleq \mathbb{P}_{\bm{X}}(\widehat{\mathcal{S}}_0\Delta \mathcal{S}_0)$.
By Chebyshev's inequality, we have
\begin{align*}
    \mathbb{P}
    \left(
    \left.
    \left|
    \frac{1}{N}\sum_{j=1}^N \left[
    \bm{1}_{\widehat{\mathcal{S}}_0\Delta \mathcal{S}_0}(\widetilde{\bm{x}}_j)
    -
    \mathbb{P}_{\bm{X}}\left(\widehat{\mathcal{S}}_0\Delta \mathcal{S}_0\right)
    \right]
    \right|\geq \delta
    \;\right|\;
    \mathcal{D}_L
    \right)
    \le
    \frac{p_L(1-p_L)}{N\delta^2}.
\end{align*}
Taking expectations on both sides yields
\begin{align*}
    \mathbb{P}
    \left(
    \left|
    \frac{1}{N}\sum_{j=1}^N \left[
    \bm{1}_{\widehat{\mathcal{S}}_0\Delta \mathcal{S}_0}(\widetilde{\bm{x}}_j)
    -
    \mathbb{P}_{\bm{X}}\left(\widehat{\mathcal{S}}_0\Delta \mathcal{S}_0\right)
    \right]
    \right|\geq \delta
    \right)
    \le
    \frac{\mathbb{E} \left[\mathbb{P}_{\bm{X}}\left(\widehat{\mathcal{S}}_0\Delta \mathcal{S}_0\right)\right]}{N\delta^2},
\end{align*}
where the last inequality follows from the fact that
$p_L(1-p_L)\le p_L$.
Since
$\mathbb{E} \left[\mathbb{P}_{\bm{X}}\left(\widehat{\mathcal{S}}_0\Delta \mathcal{S}_0\right)\right]=o(1)$,
it follows that
$$
\mathbb{P}
\left(
\left|
\frac{1}{N}\sum_{j=1}^N \left[
\bm{1}_{\widehat{\mathcal{S}}_0\Delta \mathcal{S}_0}(\widetilde{\bm{x}}_j)
-
\mathbb{P}_{\bm{X}}\left(\widehat{\mathcal{S}}_0\Delta \mathcal{S}_0\right)
\right]
\right|
\geq \delta
\right)
= o(N^{-1}).
$$
Together with the assumption that $n/N \to r \in (0,+\infty)$, we conclude that
$$
\frac{1}{N}\sum_{j=1}^N \left[
\bm{1}_{\widehat{\mathcal{S}}_0\Delta \mathcal{S}_0}(\widetilde{\bm{x}}_j)
-
\mathbb{P}_{\bm{X}}\left(\widehat{\mathcal{S}}_0\Delta \mathcal{S}_0\right)
\right]
= o_p(n^{-1/2}).
$$

\textbf{Step 1.2: Bounding $E_{12}$.} Note that $\widehat{\mathcal{S}}_0$ depends on $\mathcal{D}_L$. Therefore, $E_{12}$ is basically a sum of non-i.i.d. and non-negative samples. Next, we intend to show that $\mathbb{E}(E_{12})=o(n^{-\frac{1}{2}})$, which then implies $E_{12}=o_p(n^{-\frac{1}{2}})$ by the Markov's inequality. Let $\mathcal{D}_L^{-i_0}=\{(\bm{x}_i,y_i)\}_{i=1}^n \setminus \{(\bm{x}_{i_0},y_{i_0})\}$ denote the dataset without the $i_0$-th observation.

\begin{align*}
    \mathbb{E}_{\mathcal{D}_L}
    \left[\bm{1}_{\widehat{\mathcal{S}}_0\Delta \mathcal{S}_0}(\bm{x}_{i_0})\right]= \mathbb{E}_{(\bm{x}_{i_0},y_{i_0})}  \left\{ \mathbb{E}_{\mathcal{D}_L^{-i_0}}
    \left[\bm{1}_{\widehat{\mathcal{S}}_0\Delta \mathcal{S}_0}(\bm{x}_{i_0})\Big| (\bm{x}_{i_0},y_{i_0})\right]\right\}
\end{align*}
Note that when $\bm{x}_{i_0} \in \widehat{\mathcal{S}}_0\Delta \mathcal{S}_0$, it holds that
\begin{align*}
    |\widehat{m}(\bm{x}_{i_0})-m(\bm{x}_{i_0})| \geq \frac{1}{2}|m(\bm{x}_{i_0})-\theta^\star| 
    \text{ or }
    \Big|\frac{1}{n}\sum_{i=1}^n y_i - \theta^\star\Big|\geq \frac{1}{2}\Big|m(\bm{x}_{i_0})-\theta^\star\Big|.
\end{align*}
Therefore, we have
\begin{align*}
&\mathbb{E}_{\mathcal{D}_L}
\left[\bm{1}_{\widehat{\mathcal{S}}_0\Delta \mathcal{S}_0}(\bm{x}_{i_0})\right]
= \mathbb{E}_{(\bm{x}_{i_0},y_{i_0})}
\Bigg\{
\mathbb{E}_{\mathcal{D}_L^{-i_0}}
\left[
\bm{1}_{\widehat{\mathcal{S}}_0\Delta \mathcal{S}_0}(\bm{x}_{i_0})
\Big|
(\bm{x}_{i_0},y_{i_0})
\right]
\Bigg\} \\
\le &
\mathbb{E}_{(\bm{x}_{i_0},y_{i_0})}
\Bigg\{
\mathbb{P}_{\mathcal{D}_L^{-i_0}} \left(
|\widehat{m}(\bm{x}_{i_0})-m(\bm{x}_{i_0})|
\ge \tfrac{1}{2}|m(\bm{x}_{i_0})-\theta^\star|\Big|
(\bm{x}_{i_0},y_{i_0})
\right) \Bigg\}\\
+ &
\mathbb{E}_{(\bm{x}_{i_0},y_{i_0})}
\Bigg\{
\mathbb{P}_{\mathcal{D}_L^{-i_0}} \left(
\big|\widehat{\theta}-\theta^\star\big|
\ge \tfrac{1}{2}|m(\bm{x}_{i_0})-\theta^\star|\Big|
(\bm{x}_{i_0},y_{i_0})
\right)
\Bigg\}.
\end{align*}

\vspace{0.3cm}

\noindent
For the first term, introduce the leave-one-out estimator $\widehat{m}^{(-i_0)}$ and apply the triangle inequality:
\begin{align*}
|\widehat{m}(\bm{x}_{i_0})-m(\bm{x}_{i_0})|
\le
|\widehat{m}(\bm{x}_{i_0})-\widehat{m}^{(-i_0)}(\bm{x}_{i_0})|
+
|\widehat{m}^{(-i_0)}(\bm{x}_{i_0})-m(\bm{x}_{i_0})|.
\end{align*}
Consequently,
\begin{align*}
&\mathbb{E}_{(\bm{x}_{i_0},y_{i_0})}
\left[\mathbb{P}_{\mathcal{D}_L^{-i_0}} \left(
|\widehat{m}(\bm{x}_{i_0})-m(\bm{x}_{i_0})|
\ge \tfrac{1}{2}|m(\bm{x}_{i_0})-\theta^\star| \Big|
(\bm{x}_{i_0},y_{i_0})
\right)\right] \\
\le &
\mathbb{E}_{(\bm{x}_{i_0},y_{i_0})}\left[\mathbb{P}_{\mathcal{D}_L^{-i_0}} \left(
|\widehat{m}(\bm{x}_{i_0})-\widehat{m}^{(-i_0)}(\bm{x}_{i_0})|
\ge \tfrac{1}{4}|m(\bm{x}_{i_0})-\theta^\star|\Big|
(\bm{x}_{i_0},y_{i_0})
\right) \right]\\
&+
\mathbb{E}_{(\bm{x}_{i_0},y_{i_0})}\left[\mathbb{P}_{\mathcal{D}_L^{-i_0}} \left(
|\widehat{m}^{(-i_0)}(\bm{x}_{i_0})-m(\bm{x}_{i_0})|
\ge \tfrac{1}{4}|m(\bm{x}_{i_0})-\theta^\star|\Big|
(\bm{x}_{i_0},y_{i_0})
\right)\right] \\
\leq &
\mathbb{E}_{(\bm{x}_{i_0},y_{i_0})}\left[\mathbb{P} \left(
\sup_{\bm{x} \in \mathcal{X}}
|\widehat{m}(\bm{x})-\widehat{m}^{(-i_0)}(\bm{x})|
\ge \tfrac{1}{4}|m(\bm{x}_{i_0})-\theta^\star|\Big|
(\bm{x}_{i_0},y_{i_0})
\right) \right] \\
&+\mathbb{E}_{(\bm{x}_{i_0},y_{i_0})}\left[\mathbb{P}_{\mathcal{D}_L^{-i_0}} \left(
|\widehat{m}^{(-i_0)}(\bm{x}_{i_0})-m(\bm{x}_{i_0})|
\ge \tfrac{1}{4}|m(\bm{x}_{i_0})-\theta^\star|\Big|
(\bm{x}_{i_0},y_{i_0})
\right) \right]\\
\leq &\left(\frac{\log \psi_n}{\psi_n}\right)^{\frac{\tau}{2}} +\underbrace{\mathbb{E}_{(\bm{x}_{i_0},y_{i_0})}\left[\mathbb{P}_{\mathcal{D}_L^{-i_0}} \left(
|\widehat{m}^{(-i_0)}(\bm{x}_{i_0})-m(\bm{x}_{i_0})|
\ge \tfrac{1}{4}|m(\bm{x}_{i_0})-\theta^\star|
\right)\right]}_{T_1}.
\end{align*}
where the last inequality follows from Lemma \ref{lem:loo_margin},  which is proved based on the margin assumption and Assumption \ref{ass:loo_stability}.

Similarly, we define the leave-one-out mean estimator as
\begin{align*}
    \widehat{\theta}^{(-i_0)} = \frac{1}{n-1}\sum_{i \neq i_0}y_i,
\end{align*}
and consider the decomposition as
\begin{align*}
    &\mathbb{E}_{(\bm{x}_{i_0},y_{i_0})}
\Bigg\{
\mathbb{P}_{\mathcal{D}_L^{-i_0}} \left(
\big|\widehat{\theta}-\theta^\star\big|
\ge \tfrac{1}{2}|m(\bm{x}_{i_0})-\theta^\star|
\right)
\Bigg\} \\
\leq &
\mathbb{E}_{(\bm{x}_{i_0},y_{i_0})}
\Bigg\{
\mathbb{P}_{\mathcal{D}_L^{-i_0}} \left(
\big|\widehat{\theta}-\widehat{\theta}^{(-i_0)}\big|
\ge \tfrac{1}{4}|m(\bm{x}_{i_0})-\theta^\star|
\right)\\
&+\mathbb{P}_{\mathcal{D}_L^{-i_0}} \left(
\big|\widehat{\theta}^{(-i_0)}-\theta^\star\big|
\ge \tfrac{1}{4}|m(\bm{x}_{i_0})-\theta^\star|
\Big|(\bm{x}_{i_0},y_{i_0})
\right)
\Bigg\} \\
= &
\mathbb{P}_{\mathcal{D}_L} \left(
\big|\widehat{\theta}-\widehat{\theta}^{(-i_0)}\big|
\ge \tfrac{1}{4}|m(\bm{x}_{i_0})-\theta^\star|
\right)\\
&+\underbrace{\mathbb{E}_{(\bm{x}_{i_0},y_{i_0})}
\Bigg\{\mathbb{P}_{\mathcal{D}_L^{-i_0}} \left(
\big|\widehat{\theta}^{(-i_0)}-\theta^\star\big|
\ge \tfrac{1}{4}|m(\bm{x}_{i_0})-\theta^\star|
\right)
\Bigg\}}_{T_2}.
\end{align*}
We now bound the first term in the above decomposition.
Define
$$
A \triangleq \big|\widehat{\theta}-\widehat{\theta}^{(-i_0)}\big|,
\qquad
B \triangleq \tfrac14 |m(\bm{x}_{i_0})-\theta^\star|.
$$
For any deterministic $t>0$, by the inequality
$\mathbb P(A>B)\le \mathbb P(A>t)+\mathbb P(B<t)$, we have
\begin{align*}
\mathbb{P}_{\mathcal{D}_L} \left(
\big|\widehat{\theta}-\widehat{\theta}^{(-i_0)}\big|
\ge \tfrac14 |m(\bm{x}_{i_0})-\theta^\star|
\right) 
\le
\mathbb{P}_{\mathcal{D}_L} \left(
\big|\widehat{\theta}-\widehat{\theta}^{(-i_0)}\big|
\ge t
\right)
+
\mathbb{P}\left(
|m(\bm{x}_{i_0})-\theta^\star| < 4t
\right).
\end{align*}
Note that
$$
\widehat{\theta}-\widehat{\theta}^{(-i_0)}
=
\frac{1}{n}y_{i_0}
-
\frac{1}{n(n-1)}\sum_{i\neq i_0} y_i .
$$
Since $Y$ is sub-Gaussian with parameter $\sigma_Y^2$, standard concentration for sub-Gaussian random variables \citep{vershynin2018high} yields that 
$$
\mathbb{P}_{\mathcal{D}_L} \left(
\big|\widehat{\theta}-\widehat{\theta}^{(-i_0)}\big|
\ge t
\right)
\le
2 \exp\left(- \frac{n(n-1) t^2}{2\sigma_Y^2}\right)
\leq 
2 \exp\left(- \frac{n^2 t^2}{4\sigma_Y^2}\right),
$$
where the last inequality holds for $n \geq 2$. Moreover, by the margin condition (Definition \ref{ass:margin}),
$$
\mathbb{P}\left(
|m(\bm{x}_{i_0})-\theta^\star| < 4t
\right)
\le
C_0 (4t)^{\tau}.
$$
Taking expectation with respect to $(\bm{x}_{i_0},y_{i_0})$ and combining the above bounds, we obtain
\begin{align*}
&
\mathbb{P}_{\mathcal{D}_L} \left(
\big|\widehat{\theta}-\widehat{\theta}^{(-i_0)}\big|
\ge \tfrac14 |m(\bm{x}_{i_0})-\theta^\star|
\right)
\le
2 \exp\left(- \frac{n^2 t^2}{4\sigma_Y^2}\right)
+
C_0 (4t)^{\tau}.
\end{align*}
Optimizing over $t>0$ by choosing $t\asymp n^{-1}\log (n)$ yields
$$
\mathbb{E}_{(\bm{x}_{i_0},y_{i_0})}
\Bigg[
\mathbb{P}_{\mathcal{D}_L^{-i_0}} \left(
\big|\widehat{\theta}-\widehat{\theta}^{(-i_0)}\big|
\ge \tfrac14 |m(\bm{x}_{i_0})-\theta^\star|
\right)
\Bigg]
\;\lesssim
\left(\frac{\log n}{n}\right)^{\tau}.
$$
Finally, we have
\begin{align*}
\mathbb{E}_{\mathcal{D}_L}
\left[\bm{1}_{\widehat{\mathcal{S}}_0\Delta \mathcal{S}_0}(\bm{x}_{i_0})\right] \leq
&C\left\{\left(\frac{\log \psi_n}{\psi_n}\right)^{\frac{\tau}{2}}  + \left(\frac{\log n}{n}\right)^{\tau}\right\}+ T_1 + T_2,
\end{align*}
for some positive constant $C$. For $T_1+T_2$, we can apply the same argument as in Theorem \ref{thm:continuous} (by treating $T_1+T_2$ as the case of $n-1$) and obtain
\begin{align*}
    T_1 + T_2 \lesssim \left(
\frac{\log (\alpha_n \wedge n)}{\alpha_n  \wedge n}
\right)^{\frac{\tau}{2}}.
\end{align*}
By the assumption that $\psi_n = O(\alpha_n \wedge n)$ and
\begin{align*}
    \left(
\frac{\log (\alpha_n \wedge n)}{\alpha_n  \wedge n}
\right)^{\frac{\tau}{2}} =o(n^{-\frac{1}{2}}).
\end{align*}
It then follows that
\begin{align*}
    \mathbb{E}_{\mathcal{D}_L}
\left[\bm{1}_{\widehat{\mathcal{S}}_0\Delta \mathcal{S}_0}(\bm{x}_{i_0})\right] = o(n^{-\frac{1}{2}}).
\end{align*}
Combining this with the Markov's inequality yields $E_{12}=o_p(n^{-\frac{1}{2}})$. Therefore, we have $|E_1| = o_p(n^{-\frac{1}{2}})$

\vspace{5mm}

\noindent\textbf{Step 2: Bounding $E_2$.} Recall the definition of $E_2$, we have
\begin{align*}
    |E_2| = &\Big|\widehat{\theta}_{\mathrm{FPPI}}(\widehat{\lambda}_{\mathcal{S}_0},\mathcal{S}_0)-
     \widehat{\theta}_{\mathrm{FPPI}}(\lambda_{\mathcal{S}_0}^\star,\mathcal{S}_0)\Big| \\
     = &
  \left|\widehat{\lambda}_{\mathcal{S}_0}-  \lambda_{\mathcal{S}_0}^\star
  \right|\cdot \left|
\frac{1}{N}\sum_{j=1}^N f(\widetilde{\bm{x}}_j) \cdot \bm{1}_{\mathcal{S}_0}(\widetilde{\bm{x}}_j) -
\frac{1}{n}\sum_{i=1}^n f(\bm{x}_i) \cdot \bm{1}_{\mathcal{S}_0}(\bm{x}_i)
   \right| \\
    \leq &    \left|\widehat{\lambda}_{\mathcal{S}_0}-  \lambda_{\mathcal{S}_0}^\star
  \right| \cdot \left|
\frac{1}{N}\sum_{j=1}^N \Big[f(\widetilde{\bm{x}}_j) \cdot \bm{1}_{\mathcal{S}_0}(\widetilde{\bm{x}}_j)-\mathbb{E}(f(\bm{X})\bm{1}_{\mathcal{S}_0}(\bm{X}))\Big] \right|\\
+&\left|\widehat{\lambda}_{\mathcal{S}_0}-  \lambda_{\mathcal{S}_0}^\star
  \right| \cdot
\left|\frac{1}{n}\sum_{i=1}^n \Big[f(\bm{x}_i) \cdot \bm{1}_{\mathcal{S}_0}(\bm{x}_i)-\mathbb{E}(f(\bm{X})\bm{1}_{\mathcal{S}_0}(\bm{X}))\Big]
\right|.
\end{align*}
Since $\{\widetilde{\bm{x}}_j\}_{j=1}^N$ and $\{\bm{x}_i\}_{i=1}^n$ are i.i.d. samples, we have
\begin{align*}
   & \frac{1}{N}\sum_{j=1}^N \Big[f(\widetilde{\bm{x}}_j) \cdot \bm{1}_{\mathcal{S}_0}(\widetilde{\bm{x}}_j)-\mathbb{E}(f(\bm{X})\bm{1}_{\mathcal{S}_0}(\bm{X}))\Big] = O_p(N^{-\frac{1}{2}}), \\
   & \frac{1}{n}\sum_{i=1}^n \Big[f(\bm{x}_i) \cdot \bm{1}_{\mathcal{S}_0}(\bm{x}_i)-\mathbb{E}(f(\bm{X})\bm{1}_{\mathcal{S}_0}(\bm{X}))\Big] = O_p(n^{-\frac{1}{2}}).
\end{align*}
Since $n/N \rightarrow r \in (0,1)$, we have
\begin{align*}
    |E_2|  =  \left|\widehat{\lambda}_{\mathcal{S}_0}-  \lambda_{\mathcal{S}_0}^\star
  \right| \cdot O_p(n^{-\frac{1}{2}}).
\end{align*}
By Step 3, it holds that $\widehat{\lambda}_{\mathcal{S}_0}-  \lambda_{\mathcal{S}_0}^\star=o_p(1)$, which then implies $E_2=o_p(n^{-\frac{1}{2}})$.

\vspace{8mm}

\noindent\textbf{Step 3: Consistency of $\widehat\lambda_{\mathcal S_0}$ with dependent samples.} Recall that
$$
\widehat{\lambda}_{\mathcal{S}_0} = 
\frac{\widehat{\mathrm{Cov}}_n \big(Y,f_{\widehat{\mathcal S}_0}\big)}{\widehat{\mathrm{Var}}_{n+N} \big(f_{\widehat{\mathcal S}_0}\big)} \cdot \frac{1}{1+n/N}, 
$$ 
where
\begin{align*}
   \widehat{\mathrm{Cov}}_n \big(Y,f_{\widehat{\mathcal S}_0}\big) &= \frac{1}{n}\sum_{i=1}^n (y_i-\bar y)\big(f_{\widehat{\mathcal S}_0}(\bm x_i)-\bar f_{\widehat{\mathcal S}_0}\big), \\
   \widehat{\mathrm{Var}}_{n+N} \big(f_{\widehat{\mathcal S}_0}\big) &= \frac{1}{n+N} \left[ \sum_{i=1}^n \big(f_{\widehat{\mathcal S}_0}(\bm x_i)-\bar f_{\widehat{\mathcal S}_0}\big)^2 + \sum_{j=1}^N \big(f_{\widehat{\mathcal S}_0}(\widetilde{\bm x}_j)-\bar f_{\widehat{\mathcal S}_0}\big)^2 \right].
\end{align*}

Since the same labeled sample $\mathcal{D}_L$ is used to construct $\widehat{\mathcal S}_0$ and to compute the covariance, the terms $(y_i, f_{\widehat{\mathcal S}_0}(\bm x_i))_{i=1}^n$ are dependent. To control this, consider the decomposition  
$$
f_{\widehat{\mathcal S}_0}(\bm x_i) = f_{\mathcal S_0}(\bm x_i) + \big(f_{\widehat{\mathcal S}_0}(\bm x_i)-f_{\mathcal S_0}(\bm x_i)\big).
$$  
Then
$$
\widehat{\mathrm{Cov}}_n(Y,f_{\widehat{\mathcal S}_0}) 
= \underbrace{\widehat{\mathrm{Cov}}_n(Y,f_{\mathcal S_0})}_{\text{i.i.d. part}} 
+ \underbrace{\frac{1}{n}\sum_{i=1}^n (y_i-\bar y)\big(f_{\widehat{\mathcal S}_0}(\bm x_i)-f_{\mathcal S_0}(\bm x_i)\big)}_{\text{dependence error}}.
$$  
For the dependence error term, we have
\begin{align*}
    \frac{1}{n}\sum_{i=1}^n (y_i-\bar y)\big(f_{\widehat{\mathcal S}_0}(\bm x_i)-f_{\mathcal S_0}(\bm x_i)\big)
    \leq \sqrt{\frac{1}{n}\sum_{i=1}^n (y_i-\bar y)^2 }\cdot\sqrt{\frac{1}{n}\sum_{i=1}^n\big(f_{\widehat{\mathcal S}_0}(\bm x_i)-f_{\mathcal S_0}(\bm x_i)\big)^2}.
\end{align*}
By Step 1 (Using $|E_{11}|=o_p(n^{-\frac{1}{2}})$ and $|E_{12}|=o_p(n^{-\frac{1}{2}})$), we have 
\begin{align*}
   \frac{1}{n}\sum_{i=1}^n (f_{\widehat{\mathcal S}_0}(\bm x_i)-f_{\mathcal S_0}(\bm x_i))^2 \leq  
      \frac{M_f^2}{n}\sum_{i=1}^n \mathbf{1}_{\widehat{\mathcal S}_0\Delta\mathcal S_0}(\bm x_i)&=o_p(n^{-1/2}), \\
      \frac{1}{N}\sum_{j=1}^N |f_{\widehat{\mathcal S}_0}(\widetilde{\bm x}_j)-f_{\mathcal S_0}(\widetilde{\bm x}_j)| &= o_p(n^{-1/2}).
\end{align*}

Since $Y_i$ is sub-Gaussian, the dependence error is therefore $o_p(1)$. Similarly, for the variance term,
$$
\widehat{\mathrm{Var}}_{n+N}(f_{\widehat{\mathcal S}_0}) = \widehat{\mathrm{Var}}_{n+N}(f_{\mathcal S_0}) + o_p(1).
$$  
Finally, by the law of large numbers and Slutsky's theorem, we have
$$
\widehat{\mathrm{Cov}}_n(Y,f_{\widehat{\mathcal S}_0}) \xrightarrow{p} \mathrm{Cov}(Y,f_{\mathcal S_0}(\bm X)), \quad
\widehat{\mathrm{Var}}_{n+N}(f_{\widehat{\mathcal S}_0}) \xrightarrow{p} \mathrm{Var}(f_{\mathcal S_0}(\bm X)).
$$  
Hence,
$$
\widehat{\lambda}_{\mathcal{S}_0} = \frac{\widehat{\mathrm{Cov}}_n(Y,f_{\widehat{\mathcal S}_0})}{\widehat{\mathrm{Var}}_{n+N}(f_{\widehat{\mathcal S}_0})} \cdot \frac{1}{1+n/N} \xrightarrow{p} \frac{\mathrm{Cov}(Y,f_{\mathcal S_0})}{\mathrm{Var}(f_{\mathcal S_0})} \cdot \frac{1}{1+r} = \lambda_{\mathcal S_0}^\star.
$$
This establishes the consistency of $\widehat{\lambda}_{\mathcal S_0}$ even under the dependence induced by using the same sample for region estimation and covariance computation:
$$
\widehat{\lambda}_{\mathcal{S}_0} - \lambda_{\mathcal{S}_0}^\star = o_p(1).
$$
This completes the proof. \hfill ${\color{Red}\blacksquare}$ \\
\vspace{5mm}

\noindent\textbf{Proof of Corollary \ref{Coro:GLM}.} By Theorem \ref{thm:FPPI_GLM_final}, for any $\mathcal{S}$ with $\mathbb{P}_{\bm{X}}(\mathcal{S}) >0$, it holds that
\begin{align*}
    \textnormal{AMSE}(\lambda^\star_{\mathcal{S}},\mathcal{S})=
    \textnormal{tr}(\bm \Sigma^{-1} \bm{\Omega}\bm \Sigma^{-1})-
    \frac{[\mathrm{tr}(\bm \Sigma^{-1} \bm \Gamma_{\mathcal S} \bm \Sigma^{-1})]^2}{(1+r) \mathrm{tr}(\bm \Sigma^{-1} \bm M_{\mathcal S} \bm \Sigma^{-1})}.
\end{align*}
Therefore, maximizing $\textnormal{AMSE}(\lambda^\star_{\mathcal{S}},\mathcal{S})$ reduces to finding a region $\mathcal{S}_0$ such that
\begin{align}
\label{Second}
R(\mathcal{S})\triangleq 
    \frac{[\mathrm{tr}(\bm \Sigma^{-1} \bm \Gamma_{\mathcal S} \bm \Sigma^{-1})]^2}{(1+r) \mathrm{tr}(\bm \Sigma^{-1} \bm M_{\mathcal S} \bm \Sigma^{-1})} \triangleq \frac{[N(\mathcal{S})]^2}{(1+r)K(\mathcal{S})}
\end{align}
is greater than $R(\mathcal{X})$, where $N(\mathcal{S}) = \mathrm{tr}(\bm \Sigma^{-1} \bm \Gamma_{\mathcal S} \bm \Sigma^{-1})$ and $K(\mathcal{S}) = \mathrm{tr}(\bm \Sigma^{-1} \bm M_{\mathcal S} \bm \Sigma^{-1})$.

Note that (\ref{Second}) can be rewritten as
\begin{align*}
    & \frac{[\mathrm{tr}(\bm \Sigma^{-1} \bm \Gamma_{\mathcal S} \bm \Sigma^{-1})]^2}{(1+r) \mathrm{tr}(\bm \Sigma^{-1} \bm M_{\mathcal S} \bm \Sigma^{-1})}\\
    = &
    \frac{\left\{\text{tr}(\mathbb E\left[
R(\bm X,Y)\big(f(\bm X)-\mu(\bm X)\big)
\bm \Sigma^{-1}\bm X\bm X^\top\bm \Sigma^{-1}
\mathbf 1_{\mathcal S}(\bm X)
\right]\right\}^2}{(1+r)\left\{\text{tr}\Big(\mathbb E\left[\bm \Sigma^{-1}\bm X\bm X^\top\bm \Sigma^{-1}
\big(f(\bm X)-\mu(\bm X)\big)^2
\mathbf 1_{\mathcal S}(\bm X)\right]\Big)-\text{tr}(\bm{\Sigma}^{-1} \bm{D}_{\mathcal{S}} \bm{D}_{\mathcal{S}}^\top \bm{\Sigma}^{-1})\right\}} \\
= &
    \frac{\left\{\mathbb E\left[
R(\bm X,Y)\big(f(\bm X)-\mu(\bm X)\big)
\bm X^\top\bm \Sigma^{-1}\bm \Sigma^{-1}\bm X
\mathbf 1_{\mathcal S}(\bm X)
\right]\right\}^2}{(1+r) \cdot \left\{\mathbb E\left[\bm X^\top\bm \Sigma^{-1}\bm \Sigma^{-1}\bm X
\big(f(\bm X)-\mu(\bm X)\big)^2
\mathbf 1_{\mathcal S}(\bm X)\right]- \bm{D}_{\mathcal{S}}^\top \bm{\Sigma}^{-1}\bm{\Sigma}^{-1} \bm{D}_{\mathcal{S}}\right\}} \\
= &\frac{\left\{\mathbb E\left[
R(\bm X,Y)\big(f(\bm X)-\mu(\bm X)\big)
C(\bm{X})
\mathbf 1_{\mathcal S}(\bm X)
\right]\right\}^2}{(1+r) \cdot \left\{\mathbb E\left[C(\bm X)
\big(f(\bm X)-\mu(\bm X)\big)^2
\mathbf 1_{\mathcal S}(\bm X)\right]- \bm{D}_{\mathcal{S}}^\top \bm{\Sigma}^{-1}\bm{\Sigma}^{-1} \bm{D}_{\mathcal{S}}\right\}},
\end{align*}
where $C(\bm{X}) =\bm X^\top\bm \Sigma^{-1}\bm \Sigma^{-1}\bm X \geq 0$ and $ \bm{D}_{\mathcal{S}} =\mathbb{E}\left[\bm{X}(\mu(\bm{X})-f(\bm{X}))\bm{1}_{\mathcal{S}}(\bm{X})\right]$. 

When $\mathcal{S}=\mathcal{X}$, we have
\begin{align*}
    \bm{D}_{\mathcal{X}} =\mathbb{E}\left[\bm{X}(\mu(\bm{X})-f(\bm{X}))\right] = 
   \mathbb{E}\left[\bm{X}(\mu(\bm{X})-Y)\right]+ 
   \mathbb{E}\left[\bm{X}(Y-f(\bm{X}))\right]=\bm 0,
\end{align*}
where $\mathbb{E}\left[\bm{X}(Y-f(\bm{X}))\right]=\bm{0}$ by assumption. Therefore, we have
\begin{align*}
     \bm{D}_{\mathcal{X}}^\top \bm{\Sigma}^{-1}\bm{\Sigma}^{-1} \bm{D}_{\mathcal{X}} <  \bm{D}_{\mathcal{S}}^\top \bm{\Sigma}^{-1}\bm{\Sigma}^{-1} \bm{D}_{\mathcal{S}},
\end{align*}
for any $\mathcal{S}$ such that $ \bm{D}_{\mathcal{S}} \neq \bm{0}$. Furthermore, note that for any measurable $\mathcal{S}$, we have the following decomposition
\begin{align*}
    &\mathbb E\left[
R(\bm X,Y)\big(f(\bm X)-\mu(\bm X)\big)
C(\bm{X})
\right] \\
=&\mathbb E\left[
R(\bm X,Y)\big(f(\bm X)-\mu(\bm X)\big)
C(\bm{X})\mathbf 1_{\mathcal S}(\bm X)
\right] + \mathbb E\left[
R(\bm X,Y)\big(f(\bm X)-\mu(\bm X)\big)
C(\bm{X})\mathbf 1_{\mathcal S^c}(\bm X)
\right].
\end{align*}
By setting $\mathcal{S}=\mathcal{S}_0$ with $\mathcal{S}_0$ being defined as
\begin{align*}
    \mathcal{S}_0 = \Big\{ \bm{x} \in \mathcal{X} : \big(f(\bm{x}) - 
    \mu(\bm{x})\big) \cdot \big(m(\bm{x}) - \mu(\bm{x})\big) > 0 \Big\},
\end{align*}
we have
\begin{align*}
N(\mathcal{S}_0)  =&  \mathbb E_{\bm{X},Y}\left[
R(\bm X,Y)\big(f(\bm X)-\mu(\bm X)\big)
C(\bm{X})\mathbf 1_{\mathcal S_0}(\bm X)
\right]  \\
= & \mathbb E_{\bm{X}}\left[
(m(\bm{X})-\mu(\bm{X}))\big(f(\bm X)-\mu(\bm X)\big)
C(\bm{X})\mathbf 1_{\mathcal S_0}(\bm X)
\right]>0 \\
N(\mathcal{S}_0^c) = &\mathbb E_{\bm{X},Y}\left[
R(\bm X,Y)\big(f(\bm X)-\mu(\bm X)\big)
C(\bm{X})\mathbf 1_{\mathcal S_0^c}(\bm X)
\right]  \\
= &\mathbb E_{\bm{X}}\left[
(m(\bm{X})-\mu(\bm{X}))\big(f(\bm X)-\mu(\bm X)\big)
C(\bm{X})\mathbf 1_{\mathcal S_0^c}(\bm X)
\right]  \leq 0.
\end{align*}
where the inequality follows from the definition of $\mathcal S_0$ and the assumption that $\mathbb P_{\bm X}(\mathcal S_0)\in(0,1)$, which ensures that the integrand is strictly positive on a set of positive measure.
Therefore, we have
\begin{align*}
    0<N(\mathcal{X}) = N(\mathcal{S}_0)+N(\mathcal{S}_0^c) \leq N(\mathcal{S}_0).
\end{align*}
Next, for the denominator, we have
\begin{align*}
    K(\mathcal{X}) = & \mathbb E\left[C(\bm X)
\big(f(\bm X)-\mu(\bm X)\big)^2
\right]- \bm{D}_{\mathcal{X}}^\top \bm{\Sigma}^{-1}\bm{\Sigma}^{-1} \bm{D}_{\mathcal{X}} \\
= &\mathbb E\left[C(\bm X)
\big(f(\bm X)-\mu(\bm X)\big)^2 \mathbf{1}_{\mathcal{S}_0}(\bm{X})
\right]+\mathbb E\left[C(\bm X)
\big(f(\bm X)-\mu(\bm X)\big)^2 \mathbf{1}_{\mathcal{S}_0^c}(\bm{X})
\right] \\
> &\mathbb E\left[C(\bm X)
\big(f(\bm X)-\mu(\bm X)\big)^2 \mathbf{1}_{\mathcal{S}_0}(\bm{X})
\right] - \bm D_{\mathcal{S}_0}(\bm{X})^\top \bm{\Sigma}^{-1}\bm{\Sigma}^{-1}\bm D_{\mathcal{S}_0}(\bm{X}) \\
= &K(\mathcal{S}_0).
\end{align*}
By the assumptions that $\mathbb{P}_{\bm{X}}(\mathcal{S}_0)\in (0,1)$ and $N(\mathcal{X})>0$, it then follows that
\begin{align*}
    R(\mathcal{X}) = \frac{[N(\mathcal{X})]^2}{(1+r)K(\mathcal{X})} = \frac{[N(\mathcal{S}_0)+N(\mathcal{S}_0^c)]^2}{(1+r)K(\mathcal{X})} <\frac{[N(\mathcal{S}_0)]^2}{(1+r)K(\mathcal{S}_0)}.
\end{align*}
This implies the desired result and completes the proof. \hfill ${\color{Red}\blacksquare}$ \\

\vspace{5mm}

\noindent\textbf{Proof of Corollary \ref{Coro:GLM_Inference}.} We prove the result by showing that the FPPI estimator with estimated tuning parameter and estimated region admits the same asymptotic linear
expansion as the oracle estimator using $(\lambda_{\mathcal S_0}^\star,\mathcal S_0)$. The proof of this part will employ the same arguments as in that of Corollary \ref{Coro:Mean_Inference}. As a preliminary step, we first show that 
\begin{align*}
    \frac{1}{n}\sum_{i=1}^n \left( \mathbf{1}_{\widehat{\mathcal{S}}_0}(\bm{x}_i) - \mathbf{1}_{\mathcal{S}_0}(\bm{x}_i)\right) = o_p(n^{-\frac{1}{2}}).
\end{align*}
Note that when $\mathbf{1}_{\widehat{\mathcal{S}}_0}(\bm{x}_i) - \mathbf{1}_{\mathcal{S}_0}(\bm{x}_i) \neq 0$, we have $\bm{x}_i \in \widehat{\mathcal{S}}_0\Delta\mathcal{S}_0$. Therefore, it suffices to prove $ \frac{1}{n}\sum_{i=1}^n \mathbf{1}_{\widehat{\mathcal{S}}_0\Delta\mathcal{S}_0}(\bm{x}_i)=o_p(n^{-\frac{1}{2}})$. A key challenge is that the construction of $\widehat{\mathcal{S}}_0$ depends on $\mathcal{D}_L$, therefore $\frac{1}{n}\sum_{i=1}^n \mathbf{1}_{\widehat{\mathcal{S}}_0\Delta\mathcal{S}_0}(\bm{x}_i)$ is not a sum of i.i.d. variables. In what follows, we intend to show that $\mathbb{E}_{\mathcal{D}_L}[\mathbf{1}_{\widehat{\mathcal{S}}_0\Delta\mathcal{S}_0}(\bm{x}_{i_0})]=o(n^{-\frac{1}{2}})$ with $i_0 \in [n]$. When $\bm{x}_{i_0} \in \widehat{\mathcal{S}}_0\Delta\mathcal{S}_0$, one of the following two events must occur:
\begin{align*}
   \mathcal{E}_1 &= \left\{\mathcal{D}_L:
   \big(f(\bm{x}_{i_0}) - A'(\bm{x}_{i_0}^\top \bm{\theta}^\star)\big)
    \big(f(\bm{x}_{i_0}) - A'(\bm{x}_{i_0}^\top \widehat{\bm{\theta}}_{\text{MLE}})\big)
    < 0\right\}, \\
    \mathcal{E}_2 &= \left\{\mathcal{D}_L:
   \big(\widehat{m}(\bm{x}_{i_0}) - A'(\bm{x}_{i_0}^\top \widehat{\bm{\theta}}_{\text{MLE}})\big)
    \big(m(\bm{x}_{i_0}) - A'(\bm{x}_{i_0}^\top \bm{\theta}^\star)\big)
    < 0\right\}.
\end{align*}
Therefore,
\begin{align*}
\mathbb{E}_{\mathcal{D}_L}[\mathbf{1}_{\widehat{\mathcal{S}}_0\Delta\mathcal{S}_0}(\bm{x}_{i_0})]
    \leq
    \mathbb{P}_{\mathcal{D}_L}(\mathcal{E}_1)
    + \mathbb{P}_{\mathcal{D}_L}(\mathcal{E}_2).
\end{align*}

Next, we intend to bound $\mathbb{P}_{\mathcal{D}_L}(\mathcal{E}_1)$ and $\mathbb{P}_{\mathcal{D}_L}(\mathcal{E}_2)$ separately. Note that
\begin{align*}
    &\big(f(\bm{x}_{i_0}) - A'(\bm{x}_{i_0}^\top \bm{\theta}^\star)\big)
    \big(f(\bm{x}_{i_0}) - A'(\bm{x}_{i_0}^\top \widehat{\bm{\theta}})\big)
    < 0 \\
    \Longrightarrow &
    \big|
    A'(\bm{x}_{i_0}^\top\widehat{\bm{\theta}}_{\text{MLE}})-A'(\bm{x}_{i_0}^\top\bm{\theta}^\star)
    \big|
    >
    \big|f(\bm{x}_{i_0}) - A'(\bm{x}_{i_0}^\top \bm{\theta}^\star)\big| \\
        \Longrightarrow &
        C_5\Vert \widehat{\bm{\theta}}_{\text{MLE}}-\bm{\theta}^\star \Vert_2>    \big|f(\bm{x}_{i_0}) - A'(\bm{x}_{i_0}^\top \bm{\theta}^\star)\big|,
\end{align*}
where $C_5$ exists since $|A''(\bm{x}_{i_0}^\top\bm{\theta}^\star)|$ is bounded and $\Vert \bm{x}_{i_0}\Vert_2 \leq M_{\mathcal{X}}$. It then follows that
\begin{align*}
    \mathbb{P}_{\mathcal{D}_L}(\mathcal{E}_1)
    \leq & \mathbb{P}\left(
C_5\Vert \widehat{\bm{\theta}}_{\text{MLE}}-\bm{\theta}^\star \Vert_2>    \big|f(\bm{x}_{i_0}) - A'(\bm{x}_{i_0}^\top \bm{\theta}^\star)\big|
    \right) \\
    \leq & \mathbb{P}\left(
C_5\Vert \widehat{\bm{\theta}}_{\text{MLE}}-\bm{\theta}^\star \Vert_2>t \right)  + \mathbb{P}\left(  \big|f(\bm{x}_{i_0}) - A'(\bm{x}_{i_0}^\top \bm{\theta}^\star)\big|<t
    \right)  \\
    \lesssim & \exp(-C_6 nt^2)+C t^{\tau},
\end{align*}
for $t \gtrsim \sqrt{\frac{p+\log(n)}{n}}$. Choosing $t\asymp \sqrt{\frac{p+\log(n)}{n}}$ yields
\begin{align*}
    \mathbb{P}_{\mathcal{D}_L}(\mathcal{E}_1) \lesssim \left(
\frac{p}{n}
    \right)^{\frac{\tau}{2}}\lesssim \left(
\frac{p+\log (\alpha_n \wedge n)}{(\alpha_n \wedge n)}
     \right)^{\frac{\tau}{2}} = o(n^{-\frac{1}{2}})
\end{align*}

Next, if $\mathcal{D}_L \in \mathcal{E}_2$, one of the following events must occur
\begin{align*}
\mathcal{E}_{3}&=\left\{\mathcal{D}_L:
\big(\widehat{m}(\bm{x}_{i_0}) - m(\bm{x}_{i_0})\big)>\frac{1}{2}
    \big|m(\bm{x}_{i_0}) - A'(\bm{x}_{i_0}^\top \bm{\theta}^\star)\big| \right\} \\
\mathcal{E}_{4}&=\left\{
    \big|A'(\bm{x}_{i_0}^\top \bm{\theta}^\star )- A'(\bm{x}_{i_0}^\top \widehat{\bm{\theta}}_{\text{MLE}})\big|>\frac{1}{2}
    \big|m(\bm{x}_{i_0}) - A'(\bm{x}_{i_0}^\top \bm{\theta}^\star)\big| \right\}.
\end{align*}
Applying the same argument as above, it holds that
\begin{align*}
    \mathbb{P}_{\mathcal{D}_L}(\mathcal{E}_4) \lesssim \left(
\frac{p+\log n}{n}
    \right)^{\frac{\tau}{2}}.
\end{align*}
Furthermore, for $\mathcal{E}_3$, using Assumption \ref{ass:loo_stability}, we have
\begin{align*}
    & \mathbb{P}_{\mathcal{D}_L}(\mathcal{E}_3) =  
    \mathbb{P}_{\mathcal{D}_L}\Big(|\big(\widehat{m}(\bm{x}_{i_0}) - m(\bm{x}_{i_0})\big)>\frac{1}{2}
    \big|m(\bm{x}_{i_0}) - A'(\bm{x}_{i_0}^\top \bm{\theta}^\star)\big|\Big) \\
    \leq &
    \mathbb{P}_{\mathcal{D}_L}\Big(\big|\widehat{m}(\bm{x}_{i_0}) - m(\bm{x}_{i_0})\big|>t/2\Big)+\mathbb{P}_{\mathcal{D}_L}\Big(
    \big|m(\bm{x}_{i_0}) - A'(\bm{x}_{i_0}^\top \bm{\theta}^\star)\big|<t\Big) \\
    \leq &
    \mathbb{P}_{\mathcal{D}_L}\Big(\big|\widehat{m}(\bm{x}_{i_0}) - \widehat{m}^{(-i_0)}(\bm{x}_{i_0})\big|>t/4\Big)+\mathbb{P}_{\mathcal{D}_L}\Big(\big|\widehat{m}^{(-i_0)}(\bm{x}_{i_0}) - m(\bm{x}_{i_0})\big|>t/4\Big)\\
    &+\mathbb{P}_{\mathcal{D}_L}\Big(
    \big|m(\bm{x}_{i_0}) - A'(\bm{x}_{i_0}^\top \bm{\theta}^\star)\big|<t\Big) \\
    \lesssim & \left\{\exp(-\psi_n t^2)  +t^\tau+\exp(-\alpha_nt^2)\right\}.
\end{align*}
Since $\psi_n \gtrsim (\alpha \wedge n)$, it holds that
\begin{align*}
     \mathbb{P}_{\mathcal{D}_L}(\mathcal{E}_3)  \lesssim \left(
\frac{\log \alpha_n}{\alpha_n}
     \right)^{\frac{\tau}{2}} \lesssim \left(
\frac{p+\log (\alpha_n \wedge n)}{(\alpha_n \wedge n)}
     \right)^{\frac{\tau}{2}} = o(n^{-\frac{1}{2}}).
\end{align*}
Therefore, we have
\begin{align}
\label{MainAss}
        \frac{1}{n}\sum_{i=1}^n \left( \mathbf{1}_{\widehat{\mathcal{S}}_0}(\bm{x}_i) - \mathbf{1}_{\mathcal{S}_0}(\bm{x}_i)\right) = o_p(n^{-\frac{1}{2}}).
\end{align}

\paragraph{Step 1: Consistency of $\widehat{\lambda}_{\mathcal{S}_0}$.} We first recall that
\begin{align*}
    \widehat{\bm{\theta}}_{\text{MLE}} = \argmin_{\bm{\theta} \in \mathbb{R}^p} \sum_{i=1}^n \big[A(\bm{x}_i^\top \bm{\theta})- y_i \bm{x}_i^\top \bm{\theta}  \big].
\end{align*}
By Lemma \ref{lem:glm_misspecified_tail}, we have $\widehat{\bm{\theta}}_{\text{MLE}}\xrightarrow{p}\bm{\theta}^\star$. Furthermore, by the fact that $\mathbb{P}_{\mathbf{X}}\left(\widehat{\mathcal{S}}_0 \Delta \mathcal{S}_0\right) = o_p\left(n^{-1/2}\right)$, it holds that
\begin{align*}
    \widehat{\bm{\Sigma}}\xrightarrow[]{p} \bm{\Sigma} \mbox{ and } \widehat{\bm M}_{\widehat{\mathcal{S}}_0} \xrightarrow[]{p} \bm M_{\mathcal{S}_0}.
\end{align*}
Here, the above two convergence results are due to the fact that $\mathcal{D}_L$ and $\mathcal{D}_U$ are independent. It remians to show that 
$$
\widehat{\bm \Gamma}_{\widehat{\mathcal{S}}_0} \xrightarrow[]{p}
\bm{\Gamma}_{\mathcal{S}_0}.
$$
For ease of notation, we define
\begin{align*}
    F(\bm{\theta},\bm{x},y) = (y-A'(\bm{x}^\top \bm{\theta}))(f(\bm{x})-A'(\bm{x}^\top \bm{\theta}))\bm{x}\bm{x}^\top.
\end{align*}
Therefore, $\widehat{\bm \Gamma}_{\widehat{\mathcal{S}}_0}$ can be written as
\begin{align*}
\widehat{\bm \Gamma}_{\widehat{\mathcal{S}}_0}=
     \frac{1}{n}\sum_{i=1}^n \mathbf 1_{\widehat{\mathcal{S}}_0}(\bm{x}_i) F(\widehat{\bm{\theta}}_{\text{MLE}},\bm{x}_i,y_i).
\end{align*}
Write $\widehat{\bm \Gamma}_{\widehat{\mathcal S}_0}
-
\bm{\Gamma}_{\mathcal S_0}
=
\bm T_1 + \bm T_2 + \bm T_3$, where
\begin{align*}
\bm T_1
&=
\frac1n\sum_{i=1}^n 
\mathbf 1_{\widehat{\mathcal S}_0}(\bm{x}_i)
\Big[
F(\widehat{\bm\theta}_{\mathrm{MLE}},\bm{x}_i,y_i)
-
F(\bm\theta^\star,\bm{x}_i,y_i)
\Big],\\
\bm T_2
&=
\frac1n\sum_{i=1}^n 
\Big[
\mathbf 1_{\widehat{\mathcal S}_0}(\bm{x}_i)
-
\mathbf 1_{\mathcal S_0}(\bm{x}_i)
\Big]
F(\bm\theta^\star,\bm{x}_i,y_i),\\
\bm T_3
&=
\frac1n\sum_{i=1}^n 
\mathbf 1_{\mathcal S_0}(\bm{x}_i)
F(\bm\theta^\star,\bm{x}_i,y_i)
-
\mathbb E \left[
\mathbf 1_{\mathcal S_0}(\bm X)
F(\bm\theta^\star,\bm X,Y)
\right].
\end{align*}

\textbf{Step 1.1: Convergence of $\bm{T}_3$.} Next, we intend to show $\bm T_j \xrightarrow{p} 0$ for $j=1,2,3$. Since $\{(\bm x_i,y_i)\}_{i=1}^n$ are i.i.d.\ and
$$
\mathbb E\left\|
\mathbf 1_{\mathcal S_0}(\bm X)
F(\bm\theta^\star,\bm X,Y)
\right\|_{\infty} < \infty
$$
under the moment assumptions on $(\bm X,Y)$, the law of large numbers yields
$$
\bm T_3 \xrightarrow{p} \bm 0.
$$

\textbf{Step 1.2: Convergence of $\bm T_1$.} Recall
\begin{align*}
\bm T_1
=
\frac1n\sum_{i=1}^n 
\mathbf 1_{\widehat{\mathcal S}_0}(\bm x_i)
\Big[
F(\widehat{\bm\theta}_{\mathrm{MLE}},\bm x_i,y_i)
-
F(\bm\theta^\star,\bm x_i,y_i)
\Big],
\end{align*}
where
$$
F(\bm\theta,\bm x,y)
=
\big(y-A'(\bm x^\top\bm\theta)\big)
\big(f(\bm x)-A'(\bm x^\top\bm\theta)\big)\,
\bm x \bm x^\top
\in \mathbb R^{p\times p}.
$$
We prove $\bm T_1 \xrightarrow{p} \bm 0$ entrywise. Fix any indices $k,\ell \in \{1,\dots,p\}$. Define
$$
F_{k\ell}(\bm\theta,\bm x,y)
=
\big(y-A'(\bm x^\top\bm\theta)\big)
\big(f(\bm x)-A'(\bm x^\top\bm\theta)\big)\,
x_k x_\ell .
$$
By the mean value theorem, for each $(\bm x_i,y_i)$ there exists
$\widetilde{\bm\theta}_i$ on the line segment between
$\widehat{\bm\theta}_{\mathrm{MLE}}$ and $\bm\theta^\star$ such that
$$
F_{k\ell}(\widehat{\bm\theta}_{\mathrm{MLE}},\bm x_i,y_i)
-
F_{k\ell}(\bm\theta^\star,\bm x_i,y_i)
=
\nabla_{\bm\theta} F_{k\ell}(\widetilde{\bm\theta}_i,\bm x_i,y_i)^\top
(\widehat{\bm\theta}_{\mathrm{MLE}}-\bm\theta^\star).
$$
Hence
\begin{align*}
(\bm T_1)_{k\ell}
=
(\widehat{\bm\theta}_{\mathrm{MLE}}-\bm\theta^\star)^\top
\left[
\frac1n\sum_{i=1}^n
\mathbf 1_{\widehat{\mathcal S}_0}(\bm x_i)\,
\nabla_{\bm\theta} F_{k\ell}(\widetilde{\bm\theta}_i,\bm x_i,y_i)
\right].
\end{align*}
A direct calculation shows
$$
\nabla_{\bm\theta} F_{k\ell}(\bm\theta,\bm x,y)
=
- A''(\bm x^\top\bm\theta)\,\bm x\, x_k x_\ell
\Big[
\big(f(\bm x)-A'(\bm x^\top\bm\theta)\big)
+
\big(y-A'(\bm x^\top\bm\theta)\big)
\Big].
$$
Therefore, under the smoothness of $A(\cdot)$, there exists a constant $C$ such that
$$
\big\|\nabla_{\bm\theta} F_{k\ell}(\bm\theta,\bm X,Y)\big\|_2
\le
C(1+|Y|)\|\bm X\|_2^3
$$
for all $\bm\theta \in \bm{\Theta}$. Since $Y$ is sub-Gaussian
and $\widehat{\bm\theta}_{\mathrm{MLE}}\xrightarrow{p}\bm\theta^\star$,
with probability tending to one all $\widetilde{\bm\theta}_i$ lie in a small neighborhood of $\bm\theta^\star$.
Hence, by a uniform law of large numbers,
$$
\frac1n\sum_{i=1}^n
\mathbf 1_{\widehat{\mathcal S}_0}(\bm x_i)\,
\nabla_{\bm\theta} F_{k\ell}(\widetilde{\bm\theta}_i,\bm x_i,y_i)
= O_p(1).
$$
Therefore,
$$
(\bm T_1)_{k\ell}
=
O_p(n^{-1/2}) \cdot O_p(1)
=
O_p(n^{-1/2})
=
o_p(1).
$$
Since $p$ is fixed, entrywise convergence implies $\bm T_1 \xrightarrow{p} \bm 0$.

\textbf{Step 1.3: Control of $\bm T_2$.} Note that
$$
\mathbf 1_{\widehat{\mathcal S}_0}(\bm x)
-
\mathbf 1_{\mathcal S_0}(\bm x)
\neq 0
\quad \Longrightarrow \quad
\bm x \in \widehat{\mathcal S}_0 \,\Delta\, \mathcal S_0,
$$
where $\Delta$ denotes the symmetric difference of two sets. For any $k,\ell \in [p]$, by the Cauchy--Schwarz inequality,
\begin{align*}
|(\bm T_2)_{k\ell}|
&=
\left|
\frac1n\sum_{i=1}^n 
\Big[
\mathbf 1_{\widehat{\mathcal S}_0}(\bm x_i)
-
\mathbf 1_{\mathcal S_0}(\bm x_i)
\Big]
F_{k\ell}(\bm\theta^\star,\bm x_i,y_i)
\right| \\
&\le
\sqrt{\frac1n\sum_{i=1}^n
\mathbf 1_{\widehat{\mathcal S}_0 \Delta \mathcal S_0}(\bm x_i)}
\cdot
\sqrt{\frac1n\sum_{i=1}^n
F_{k\ell}^2(\bm\theta^\star,\bm x_i,y_i)}.
\end{align*}
Taking expectations and applying the Cauchy--Schwarz inequality again,
\begin{align*}
\mathbb E \left[\,|(\bm T_2)_{k\ell}|\,\right]
&\le
\left(
\mathbb E \left[
\frac1n\sum_{i=1}^n
\mathbf 1_{\widehat{\mathcal S}_0 \Delta \mathcal S_0}(\bm x_i)
\right]
\right)^{1/2}
\left(
\mathbb E \left[
F_{k\ell}^2(\bm\theta^\star,\bm X,Y)
\right]
\right)^{1/2}=o(1).
\end{align*}
We conclude
$$
\widehat{\bm \Gamma}_{\widehat{\mathcal S}_0}
\xrightarrow{p}
\bm{\Gamma}_{\mathcal S_0}.
$$
To sum up, we have $\widehat{\lambda}_{\mathcal{S}_0} \xrightarrow[]{p}\lambda_{\mathcal{S}_0}^\star$.

\paragraph{Step 2: Score expansion.} Let
$$
\widehat{\bm\theta}
=
\widehat{\bm \theta}_{\mathrm{FPPI}}(\widehat \lambda_{\mathcal S_0},\widehat{\mathcal S}_0).
$$
By definition,
$$
\nabla_{\bm\theta}\mathcal L_{\mathrm{FPPI}}
(\widehat{\bm\theta}\mid \widehat\lambda_{\mathcal S_0},\widehat{\mathcal S}_0)=\bm 0.
$$
A first-order Taylor expansion around $\bm\theta^\star$ gives
\begin{align}
\bm 0
&=
\nabla_{\bm\theta}\mathcal L_{\mathrm{FPPI}}
(\bm\theta^\star\mid \widehat\lambda_{\mathcal S_0},\widehat{\mathcal S}_0)
+
\nabla^2_{\bm\theta}\mathcal L_{\mathrm{FPPI}}
(\bar{\bm\theta}\mid \widehat\lambda_{\mathcal S_0},\widehat{\mathcal S}_0)
(\widehat{\bm\theta}-\bm\theta^\star),
\end{align}
where $\bar{\bm\theta}$ lies between $\widehat{\bm\theta}$ and $\bm\theta^\star$. Here,
\begin{align*}
\nabla^2_{\bm\theta} \mathcal L_{\mathrm{FPPI}}(\bar{\bm\theta}\mid \widehat\lambda_{\mathcal S_0},\widehat{\mathcal S}_0))
=&
\frac{1}{n} \sum_{i=1}^n A''(\bm x_i^\top \bar{\bm\theta}) \, \bm x_i \bm x_i^\top
+ \frac{\widehat\lambda_{\mathcal S_0}}{N} \sum_{j=1}^N A''(\widetilde{\bm x}_j^\top \bar{\bm\theta}) \, \bm x_j \bm x_j^\top \, \bm 1_{\widehat{\mathcal S}_0}(\widetilde{\bm x}_j)\\
&- \frac{\widehat\lambda_{\mathcal S_0}}{n} \sum_{i=1}^n A''(\bm x_i^\top \bar{\bm\theta}) \, \bm x_i \bm x_i^\top \, \bm 1_{\widehat{\mathcal S}_0}(\bm x_i),
\end{align*}

Note that $\widehat{\lambda}_{\mathcal{S}_0} \xrightarrow[]{p}\lambda_{\mathcal{S}_0}^\star$, $\widehat{\bm{\theta}}_{\text{MLE}}\xrightarrow[]{p}\bm{\theta}^\star$, and the result in (\ref{MainAss}), a direct application of the Slutsky's Theorem yields
$$
\nabla^2_{\bm\theta}\mathcal L_{\mathrm{FPPI}}(\bar{\bm\theta})
\xrightarrow{p}\bm\Sigma,
$$
so that
\begin{align}
\sqrt n(\widehat{\bm\theta}-\bm\theta^\star)
=
-\bm\Sigma^{-1}
\sqrt n\,
\nabla_{\bm\theta}\mathcal L_{\mathrm{FPPI}}
(\bm\theta^\star\mid \widehat\lambda_{\mathcal S_0},\widehat{\mathcal S}_0)
+o_p(1).
\label{eq:theta_expansion}
\end{align}

\paragraph{Step 2: Removing the randomness of the estimated region.}

Write the FPPI score at $\bm\theta^\star$ as
\begin{align}
\nabla_{\bm\theta}\mathcal L_{\mathrm{FPPI}}
(\bm\theta^\star\mid \lambda,\mathcal S)
=
-\frac1n\sum_{i=1}^n (y_i-\mu(\bm x_i))\bm x_i
+
\lambda
\Big(
\frac1N\sum_{j=1}^N \widetilde{\bm\delta}_j(\mathcal S)
-
\frac1n\sum_{i=1}^n \bm\delta_i(\mathcal S)
\Big),
\end{align}
where
$$
\bm\delta_i(\mathcal S)
=(\mu(\bm x_i)-f(\bm x_i))\bm x_i\mathbf 1_{\mathcal S}(\bm x_i),
\quad
\widetilde{\bm\delta}_j(\mathcal S)
=(\mu(\widetilde{\bm x}_j)-f(\widetilde{\bm x}_j))
\widetilde{\bm x}_j\mathbf 1_{\mathcal S}(\widetilde{\bm x}_j).
$$

The difference between using $\widehat{\mathcal S}_0$ and $\mathcal S_0$
enters only through indicator functions. By Assumption
$
\mathbb P(\widehat{\mathcal S}_0\Delta\mathcal S_0)=o_p(n^{-1/2})
$ and the result in (\ref{MainAss}), we have
$$
\frac1n\sum_{i=1}^n
\big(\bm\delta_i(\widehat{\mathcal S}_0)-\bm\delta_i(\mathcal S_0)\big)
= o_p(n^{-1/2}),
\quad
\frac1N\sum_{j=1}^N
\big(\widetilde{\bm\delta}_j(\widehat{\mathcal S}_0)-\widetilde{\bm\delta}_j(\mathcal S_0)\big)
= o_p(n^{-1/2}).
$$
Hence
\begin{align}
\nabla_{\bm\theta}\mathcal L_{\mathrm{FPPI}}
(\bm\theta^\star\mid \widehat\lambda_{\mathcal S_0},\widehat{\mathcal S}_0)
=
\nabla_{\bm\theta}\mathcal L_{\mathrm{FPPI}}
(\bm\theta^\star\mid \widehat\lambda_{\mathcal S_0},{\mathcal S}_0)
+o_p(n^{-1/2}).
\label{eq:region_negligible}
\end{align}

\paragraph{Step 3: Removing the randomness of $\widehat\lambda_{\mathcal S_0}$.}

Since $\widehat\lambda_{\mathcal S_0}\xrightarrow{p}\lambda_{\mathcal S_0}^\star$,
\begin{align}
\nabla_{\bm\theta}\mathcal L_{\mathrm{FPPI}}
(\bm\theta^\star\mid \widehat\lambda_{\mathcal S_0},{\mathcal S}_0)
=
\nabla_{\bm\theta}\mathcal L_{\mathrm{FPPI}}
(\bm\theta^\star\mid \lambda_{\mathcal S_0}^\star,{\mathcal S}_0)
+o_p(n^{-1/2}),
\label{eq:lambda_negligible}
\end{align}
because the score is linear in $\lambda$ and
$
\widehat\lambda_{\mathcal S_0}-\lambda_{\mathcal S_0}^\star
=O_p(n^{-1/2}).
$

Combining \eqref{eq:region_negligible} and \eqref{eq:lambda_negligible},
\begin{align}
\sqrt n\,
\nabla_{\bm\theta}\mathcal L_{\mathrm{FPPI}}
(\bm\theta^\star\mid \widehat\lambda_{\mathcal S_0},\widehat{\mathcal S}_0)
=
\sqrt n\,
\nabla_{\bm\theta}\mathcal L_{\mathrm{FPPI}}
(\bm\theta^\star\mid \lambda_{\mathcal S_0}^\star,{\mathcal S}_0)
+o_p(1).
\label{eq:score_oracle}
\end{align}

\paragraph{Step 4: Central limit theorem.}

Under fixed $(\lambda_{\mathcal S_0}^\star,\mathcal S_0)$, the score is a sum
of i.i.d.\ mean-zero vectors. By the multivariate central limit theorem,
\begin{align}
\sqrt n\,
\nabla_{\bm\theta}\mathcal L_{\mathrm{FPPI}}
(\bm\theta^\star\mid \lambda_{\mathcal S_0}^\star,{\mathcal S}_0)
\xrightarrow{d}
\mathcal N \left(
\bm 0,\,
\bm\Omega
+(\lambda_{\mathcal S_0}^\star)^2(1+r)\bm M_{\mathcal S_0}
-2\lambda_{\mathcal S_0}^\star \bm\Gamma_{\mathcal S_0}
\right).
\end{align}
Substituting \eqref{eq:score_oracle} into \eqref{eq:theta_expansion} and
applying Slutsky's theorem,
\begin{align*}
\sqrt{n}\left(\widehat{\bm \theta}_{\mathrm{FPPI}}(\widehat \lambda_{\mathcal{S}_0},\widehat{\mathcal{S}}_0) - \bm \theta^\star\right)
\xrightarrow{d}
\mathcal N\Big(
\bm 0,\;
\bm \Sigma^{-1}
\big(
\bm\Omega
+(\lambda_{\mathcal{S}_0}^\star)^2(1+r)\bm M_{\mathcal S_0}
-2\lambda_{\mathcal{S}_0}^\star \bm\Gamma_{\mathcal S_0}
\big)
\bm \Sigma^{-1}
\Big).
\end{align*}
This completes the proof. \hfill ${\color{Red}\blacksquare}$ \\
\vspace{5mm}

\noindent
\textbf{Proof of Corollary \ref{Coro:Mean_Inference_Split}.} Recall that the labeled dataset is randomly split into two independent parts: $\mathcal D_1$ of size $n_1$ and $\mathcal D_2$ of size $n_2$, with
$n_1+n_2=n$. The region estimator $\widehat{\mathcal S}_0$ is constructed
using only $\mathcal D_1$, while the FPPI estimator is computed using
$\mathcal D_2$ and the unlabeled data. Hence $\widehat{\mathcal S}_0$ is independent of $\mathcal D_2$.

For any measurable set $\mathcal S$, define
$f_{\mathcal S}(\bm x)=f(\bm x)\mathbf 1_{\mathcal S}(x)$
and $\mu_f(\mathcal S)=\mathbb E[f_{\mathcal S}(\bm X)]$. Let the oracle region be
$$
\mathcal S_0=\{\bm x:(m(\bm x)-\mathbb E(Y))f(\bm x)>0\}.
$$
By Theorem \ref{thm:continuous} and Assumption \ref{Ass:TailConvergence},
using only $\mathcal D_1$ (size $n_1$),
$$
\mathbb E_{\mathcal D_1} \left[\mathbb P_{\bm X}(\widehat{\mathcal S}_0 \Delta \mathcal S_0)\right]
\longrightarrow 0 .
$$
Since $\widehat{\mathcal S}_0$ is independent of $\mathcal D_2$, this implies
$$
f_{\widehat{\mathcal S}_0}(X)
\xrightarrow{L^2}
f_{\mathcal S_0}(X).
$$
The FPPI estimator can be written as
$$
\widehat\theta_{\mathrm{FPPI}}(\widehat\lambda_{\mathcal S_0},\widehat{\mathcal S}_0)
=
\bar y_2
+
\widehat\lambda_{\mathcal S_0}
\left(
\bar f_{U}(\widehat{\mathcal S}_0)
-
\bar f_{2}(\widehat{\mathcal S}_0)
\right),
$$
where
$$
\bar y_2 = \frac1{n_2}\sum_{i\in\mathcal D_2} Y_i,\quad
\bar f_2(\mathcal S)=\frac1{n_2}\sum_{i\in\mathcal D_2} f_{\mathcal S}(\bm x_i),\quad
\bar f_U(\mathcal S)=\frac1N\sum_{j=1}^N f_{\mathcal S}(\widetilde{\bm x}_j).
$$

Define the oracle estimator
$$
\widehat\theta_{\mathrm{FPPI}}^{\mathrm{oracle}}
=
\bar y_2
+
\lambda_{\mathcal S_0}
\left(
\bar f_{U}(\mathcal S_0)
-
\bar f_{2}(\mathcal S_0)
\right),
$$
where
$$
\lambda_{\mathcal S_0}
=
\frac{\mathrm{Cov}(Y,f_{\mathcal S_0}(\bm X))}
{(1+r)\mathrm{Var}(f_{\mathcal S_0}(\bm X))},
\qquad
r=\lim_{n_2,N\to\infty}\frac{n_2}{N}.
$$
From (A.1), the continuous mapping theorem, and Slutsky’s theorem,
\begin{align*}
   \widehat\theta_{\mathrm{FPPI}}
-
\widehat\theta_{\mathrm{FPPI}}^{\mathrm{oracle}} 
= 
\widehat{\lambda}_{\mathcal S_0} \cdot 
\underbrace{\left[(\bar f_{U}(\widehat{\mathcal S}_0)-\bar f_{U}(\mathcal S_0))+(\bar f_{2}(\mathcal S_0)-\bar f_{2}(\widehat{\mathcal S}_0))\right]}_{o_p(n^{-\frac{1}{2}})}+
\underbrace{(\widehat{\lambda}_{\mathcal S_0}-\lambda_{\mathcal S_0}  )}_{o_p(1)}\cdot \big[\underbrace{\bar f_{U}(\mathcal S_0)-\bar f_{2}(\mathcal S_0)}_{O_p(n^{-\frac{1}{2}})}\big].
\end{align*}
Here the first result can be derived using the same steps as in \textbf{Step 1.1} in the proof of Corollary \ref{Coro:Mean_Inference}. Therefore, we have $\widehat\theta_{\mathrm{FPPI}}
-
\widehat\theta_{\mathrm{FPPI}}^{\mathrm{oracle}} =o_p(n^{-\frac{1}{2}})$. Thus it suffices to derive the asymptotic distribution of the oracle estimator.

Because $\widehat{\mathcal S}_0$ is independent of $\mathcal D_2$ and the
unlabeled sample, we can treat $\mathcal S_0$ as fixed. Then
$$
\sqrt{n_2}
\begin{pmatrix}
\bar y_2 - \mathbb E[Y] \\
\bar f_2(\mathcal S_0) - \mu_f(\mathcal S_0)
\end{pmatrix}
\xrightarrow{d}
\mathcal N \left(
\bm 0,
\begin{pmatrix}
\mathrm{Var}(Y) & \mathrm{Cov}(Y,f_{\mathcal S_0}(\bm X))\\
\mathrm{Cov}(Y,f_{\mathcal S_0}(\bm X)) & \mathrm{Var}(f_{\mathcal S_0}(\bm X))
\end{pmatrix}
\right).
$$
Similarly,
$$
\sqrt{n_2}\left(\bar f_U(\mathcal S_0)-\mu_f(\mathcal S_0)\right)
\xrightarrow{d}
\mathcal N \left(0,\, r\,\mathrm{Var}(f_{\mathcal S_0}(\bm X))\right),
$$
independently of the labeled sample. The oracle estimator satisfies
$$
\widehat\theta_{\mathrm{FPPI}}^{\mathrm{oracle}} - \theta^\star
=
(\bar y_2-\mathbb E[Y])
-
\frac{\mathrm{Cov}(Y,f_{\mathcal S_0})}{(1+r)\mathrm{Var}(f_{\mathcal S_0})}
\Big[
(\bar f_2-\mu_f)
-
(\bar f_U-\mu_f)
\Big].
$$
Applying the multivariate CLT and Slutsky’s theorem yields
$$
\sqrt{n_2}
\left(
\widehat\theta_{\mathrm{FPPI}}^{\mathrm{oracle}} - \theta^\star
\right)
\xrightarrow{d}
\mathcal N \left(
0,\,
\mathrm{Var}(Y)
-
\frac{1}{1+r}
\frac{\mathrm{Cov}^2 \big(Y,f_{\mathcal S_0}(X)\big)}
{\mathrm{Var} \big(f_{\mathcal S_0}(X)\big)}
\right).
$$
Combining the above results gives
$$
\sqrt{n_2}
\left(
\widehat{\theta}_{\mathrm{FPPI}}(\widehat{\lambda}_{\mathcal{S}_0},\widehat{\mathcal{S}}_0) - \theta^\star
\right)
\xrightarrow{d}
\mathcal N \left(
0,\,
\mathrm{Var}(Y)
-
\frac{1}{1+r}
\frac{\mathrm{Cov}^2 \big(Y,f(\bm X)\mathbf 1_{\mathcal S_0}(\bm X)\big)}
{\mathrm{Var} \big(f(\bm X)\mathbf 1_{\mathcal S_0}(\bm X)\big)}
\right).
$$
This completes the proof. \hfill ${\color{Red}\blacksquare}$ \\
\vspace{5mm}

\section{Proof of Lemmas}
\vspace{5mm}

\begin{lemma}
\label{lem:marginal_subgaussian}
Under Assumption \ref{Ass:SubGaussian}, we have the following two cases:
\begin{itemize}
\item[(i)] If $\bm X$ takes values in a finite set
$\{\bm{k}_1,\dots,\bm{k}_T\}$, then $Y-\mathbb E(Y)$ is sub-Gaussian with parameter
$$
\mathbb E \left[e^{\lambda (Y-\mathbb EY)}\right]
\le
\exp \left(\frac{\lambda^2\sigma_Y^2}{2}\right),
\qquad
\sigma_Y^2=\sigma^2+\frac{\Delta^2}{4},
$$
where $\Delta=\max_{i\in[T]} m(\bm k_i)-\min_{i\in[T]} m(\bm k_i)$.
\item[(ii)]
If $\bm X\in\mathcal X$ with $\mathcal X$ compact and $m(\cdot)$ continuous,
and if $\sup_{\bm X\in\mathcal X}\big|m(\bm X)-\mathbb E(Y)\big|\le M$, then $Y-\mathbb E(Y)$ satisfies
$$
\mathbb E \left[e^{\lambda (Y-\mathbb E(Y))}\right]
\le
\exp \left(\frac{\lambda^2(\sigma^2+M^2)}{2}\right).
$$
\end{itemize}
\end{lemma}

\noindent \textbf{Proof of Lemma \ref{lem:marginal_subgaussian}.} We first consider case (i).
Let $p_i=\mathbb P(\bm X=\bm k_i)$ and
$\mathbb EY=\sum_{i=1}^T p_i m(\bm k_i)$.
For any $\lambda\in\mathbb R$,
\begin{align*}
\mathbb E \left[e^{\lambda (Y-\mathbb EY)}\right]
&=
\sum_{i=1}^T p_i
\mathbb E \left[e^{\lambda (Y-\mathbb EY)}\mid \bm X=\bm k_i\right] \\
&=
\sum_{i=1}^T p_i
e^{\lambda (m(\bm k_i)-\mathbb EY)}
\mathbb E \left[e^{\lambda (Y-m(\bm k_i))}\mid \bm X=\bm k_i\right] \\
&\le
e^{\lambda^2\sigma^2/2}
\sum_{i=1}^T p_i e^{\lambda (m(\bm k_i)-\mathbb EY)}.
\end{align*}
Since
$m(\bm k_i)-\mathbb EY\in[\min_j m(\bm k_j)-\mathbb EY,\,
\max_j m(\bm k_j)-\mathbb EY]$
with range $\Delta$, Hoeffding's lemma yields
$$
\sum_{i=1}^T p_i e^{\lambda (m(\bm k_i)-\mathbb EY)}
\le
\exp \left(\frac{\lambda^2\Delta^2}{8}\right).
$$
Combining the above bounds gives
$$
\mathbb E \left[e^{\lambda (Y-\mathbb EY)}\right]
\le
\exp \left(\frac{\lambda^2(\sigma^2+\Delta^2/4)}{2}\right).
$$

\vspace{0.3em}
\noindent
We now turn to case (ii).
By the law of iterated expectations,
\begin{align*}
\mathbb E \left[e^{\lambda (Y-\mathbb EY)}\right]
&=
\mathbb E_{\bm X} \left[
e^{\lambda (m(\bm X)-\mathbb EY)}
\mathbb E \left[e^{\lambda (Y-m(\bm X))}\mid \bm X\right]
\right] \\
&\le
e^{\lambda^2\sigma^2/2}
\mathbb E_{\bm X} \left[e^{\lambda (m(\bm X)-\mathbb EY)}\right].
\end{align*}
Under the assumption
$\sup_{\bm X\in\mathcal X}|m(\bm X)-\mathbb EY|\le M$,
Hoeffding's lemma implies
$$
\mathbb E_{\bm X} \left[e^{\lambda (m(\bm X)-\mathbb EY)}\right]
\le
\exp \left(\frac{\lambda^2 M^2}{2}\right).
$$
Therefore,
$$
\mathbb E \left[e^{\lambda (Y-\mathbb EY)}\right]
\le
\exp \left(\frac{\lambda^2(\sigma^2+M^2)}{2}\right),
$$
which completes the proof.
 \hfill {\color{Red}$\blacksquare$} \\

\begin{lemma}
\label{lem:loo_margin}
Suppose that the general margin condition holds with parameter $(C,\tau)$ and \ref{ass:loo_stability} hold.  Then, for any $i_0\in\{1,\dots,n\}$,
\begin{align*}
&
\mathbb{P} \left(
\sup_{\bm{x}\in\mathcal X}
\big|\widehat m(\bm{x})-\widehat m^{(-i_0)}(\bm{x})\big|
\ge \tfrac14 |m(\bm{x}_{i_0})-\theta^\star|
\right)
\;\lesssim\;
\left(\frac{\log \psi_n}{\psi_n}\right)^{\tau/2}.
\end{align*}
\end{lemma}

\noindent\textbf{Proof of Lemma \ref{lem:loo_margin}.} Fix $i_0\in\{1,\dots,n\}$. 
Define
$$
A \triangleq
\sup_{\bm{x}\in\mathcal X}
\big|\widehat m(\bm{x})-\widehat m^{(-i_0)}(\bm{x})\big|,
\qquad
B \triangleq
\tfrac14 |m(\bm{x}_{i_0})-\theta^\star|.
$$
For any deterministic $t>0$, the inequality
$\mathbb P(A>B)\le \mathbb P(A>t)+\mathbb P(B<t)$ yields
\begin{align*}
&\mathbb{P} \left(
\sup_{\bm{x}\in\mathcal X}
\big|\widehat m(\bm{x})-\widehat m^{(-i_0)}(\bm{x})\big|
\ge \tfrac14 |m(\bm{x}_{i_0})-\theta^\star|
\right) \\
\le\;&
\mathbb{P} \left(
\sup_{\bm{x}\in\mathcal X}
\big|\widehat m(\bm{x})-\widehat m^{(-i_0)}(\bm{x})\big|
\ge t 
\right)
+
\mathbb{P} \left(
|m(\bm{x}_{i_0})-\theta^\star| < 4t
\right).
\end{align*}

By Assumption~\ref{ass:loo_stability}, the first term is bounded by
$$
\mathbb{P} \left(
\sup_{\bm{x}\in\mathcal X}
\big|\widehat m(\bm{x})-\widehat m^{(-i_0)}(\bm{x})\big|
\ge t
\right)
\le
c_3 \exp(-c_4 \psi_n t^2).
$$
By the margin assumption in Definition \ref{ass:margin}, the second term satisfies
$$
\mathbb{P} \left(
|m(\bm{x}_{i_0})-\theta^\star| < 4t
\right)
\le
C (4t)^{\tau}.
$$
To sum up, we obtain
$$
\mathbb{P} \left(
\sup_{\bm{x}\in\mathcal X}
\big|\widehat m(\bm{x})-\widehat m^{(-i_0)}(\bm{x})\big|
\ge \tfrac14 |m(\bm{x}_{i_0})-\theta^\star|
\right)
\le
c_3 \exp(-c_4 \psi_n t^2)
+
C (4t)^{\tau}.
$$
Optimizing over $t>0$ by choosing $t \asymp \psi_n^{-1/2}\log^{\frac{1}{2}}(\psi_n)$ to balance the exponential and polynomial terms, we have
$$
\mathbb{P} \left(
\sup_{\bm{x}\in\mathcal X}
\big|\widehat m(\bm{x})-\widehat m^{(-i_0)}(\bm{x})\big|
\ge \tfrac14 |m(\bm{x}_{i_0})-\theta^\star|
\right)
\lesssim
\left(\frac{\log \psi_n}{\psi_n}\right)^{\tau/2},
$$
which completes the proof. \hfill ${\color{Red}\blacksquare}$ \\
\vspace{5mm}

\begin{lemma}[Tail bound for the $k$NN radius]
\label{lemma:knn_radius}
Let $\bm{x}_1,\dots,\bm{x}_n \overset{i.i.d.}{\sim} \mathbb{P}_{\bm{X}}$ be random vectors in $\mathbb{R}^d$ with support 
$\mathcal X \subset \mathbb{R}^d$. Assume:

\begin{enumerate}
    \item $\mathbb{P}_{\bm{X}}$ admits a density with respect to the Lebesgue measure;
    \item There exist constants $c_0, r_0 > 0$ such that for all $\bm{x} \in \mathcal X$ and all $0<r \leq r_0$,
    $$
    \mathbb{P}_{\bm{X}} \big(B(\bm{x},r)\big) \ge c_0 r^d,
    $$
    where $B(\bm{x},r) = \{\bm{u} \in \mathbb R^d : \|\bm{u}-\bm{x}\|_2 \le r\}$.
\end{enumerate}
Fix $\bm{x} \in \mathcal X$ and let $R_k(\bm{x}) = \|\bm{X}_{(k)}(\bm{x}) - \bm{x}\|_2$ be the distance from $\bm{x}$ to its $k$-th nearest neighbor among 
$\{\bm{x}_i\}_{i=1}^n$. Then there exist constants $c,C>0$ depending only on $d$ and $c_0$ such that for all radiuses $\Big(\tfrac{2k}{c_0 n}\Big)^{1/d} \le t \le r_0$, we have
$$
\mathbb P \left(R_k(\bm{x}) > t\right)
\le \exp \left(- c n t^d\right).
$$
\end{lemma}

\noindent\textbf{Proof of Lemma \ref{lemma:knn_radius}.} Fix $\bm{x}\in\mathcal X$ and $t>0$. Let
$$
N_t \triangleq \sum_{i=1}^n \mathbf 1\{\bm{x}_i \in B(\bm{x},t)\}.
$$
Then
$$
R_k(\bm{x}) > t \quad \Longleftrightarrow \quad N_t < k.
$$
Let $p_t \triangleq \mathbb{P}_{\bm{X}}(B(\bm{x},t))$. For $t<r_0$, the lower mass condition gives $p_t \ge c_0 t^d$, so
$$
\mathbb E N_t = n p_t \ge c_0 n t^d.
$$
Assume now that $t \ge (2k/(c_0 n))^{1/d}$. Then
$$
\mathbb E N_t \ge 2k.
$$
Hence
$$
\mathbb P(R_k(\bm{x}) > t)
= \mathbb P(N_t < k)
= \mathbb P \left(N_t < \tfrac12 \mathbb E N_t\right).
$$
Applying a multiplicative Chernoff bound \citep{vershynin2018high} for binomial random variables,
$$
\mathbb P \left(N_t < \tfrac12 \mathbb E N_t\right)
\le \exp \left(-\tfrac{1}{8}\mathbb E N_t\right)
\le \exp(-c n t^d),
$$
for some constant $c>0$ depending only on $c_0$. This completes the proof. 
\hfill ${\color{red}\blacksquare}$ \\
\vspace{8mm}

\begin{lemma}
\label{lem:max_subgaussian}
Let $\{X_i\}_{i=1}^n$ be i.i.d. random variables such that
$\|X_i\|_{\psi_2} \le C$ for all $i=1,\dots,n$. Then it holds that
$$
\Big\|\max_{1\le i\le n} |X_i|\Big\|_{\psi_2}
\le
4C\sqrt{\ln(n+1)}.
$$
\end{lemma}

\noindent\textbf{Proof of Lemma \ref{lem:max_subgaussian}.} Without loss of generality, we assume $C=1$ by considering $Y_i = X_i/C$. 
Thus $\|Y_i\|_{\psi_2} \le 1$, and by the definition of the $\psi_2$-norm, we have
$$
\mathbb{E} \exp(Y_i^2) \le 2, \qquad i=1,\dots,n.
$$
Denote $Z = \max_{1\le i\le n} |Y_i|$. Then
$$
e^{Z^2} = \max_{1\le i\le n} e^{Y_i^2} 
      \le \sum_{i=1}^n e^{Y_i^2},
$$
and taking expectations yields
$$
\mathbb{E} e^{Z^2} \le \sum_{i=1}^n \mathbb{E} e^{Y_i^2} \le 2n.
$$
Set $M = \sqrt{2\ln (2n)}$ for $n\ge 2$. Note that we always have $M \ge 1$.
Using Jensen's inequality for the concave function $t \mapsto t^{1/M^2}$, we obtain
$$
\mathbb{E} e^{Z^2/M^2}
   = \mathbb{E} \bigl[ (e^{Z^2})^{1/M^2} \bigr]
   \le \bigl( \mathbb{E} e^{Z^2} \bigr)^{1/M^2}
   \le (2n)^{1/M^2}.
$$
Since $M^2 = 2\ln(2n)$,
$$
(2n)^{1/M^2} = \exp\Bigl(\frac{\ln(2n)}{2\ln(2n)}\Bigr)
             = e^{1/2} = \sqrt{e} \;<\; 2.
$$
Thus $\mathbb{E} \exp(Z^2/M^2) \le \sqrt{e} < 2$, which implies by the definition of the $\psi_2$-norm that
$$
\|Z\|_{\psi_2} \le M = \sqrt{2\ln(2n)}.
$$
Reverting to the original variables, $Z = \max_i |X_i|/C$, hence
$$
\Big\|\max_{1\le i\le n} |X_i|\Big\|_{\psi_2}
   = C \|Z\|_{\psi_2}
   \le C \sqrt{2\ln(2n)}.
$$
Finally, for $n\ge 1$,
\begin{align*}
\sqrt{2\ln(2n)}
&= \sqrt{2\ln n + 2\ln 2}
 \le \sqrt{2\ln n} + \sqrt{2\ln 2} \\
&\le \bigl(\sqrt{2} + \sqrt{2}\bigr)\sqrt{\ln(n+1)}
 \leq  4\sqrt{\ln(n+1)}.
\end{align*}
This completes the proof. \hfill ${\color{Red}\blacksquare}$ \\

\vspace{5mm}

\begin{lemma}
\label{lem:working_glm_quadratic}
Suppose $(\bm{X}, Y)$ satisfies Assumption \ref{Ass:SubGaussian}, and consider a working generalized linear model with canonical link and loss $L_{\bm{\theta}}(\bm{X},Y)=A(\bm{X}^\top \bm{\theta}) - Y \bm{X}^\top \bm{\theta}$. Under Assumption \ref{ass:PD}, it holds for
$\bm{\theta}\in\bm{\Theta}$ that
$$
C_L   \|\bm{\theta}-\bm{\theta}^\star\|_2^2
 \le 
\mathbb{E}[L_{\bm{\theta}}(\bm{X},Y)]
-
\mathbb{E}[L_{\bm{\theta}^\star}(\bm{X},Y)]
 \le 
C_U   \|\bm{\theta}-\bm{\theta}^\star\|_2^2,
$$
for some positive constants $C_L$ and $C_U$. Consider the subset $\bm{\Theta}_j(\delta) = \{\bm{\theta}\in \bm{\Theta}:2^{j-1}\delta \leq \Phi(\bm{\theta},\bm{\theta}^\star) \leq 2^j \delta \}$ for $j \geq 1$, it holds that
\begin{align*}
    \sup_{\bm{\theta} \in \bm{\Theta}_j(\delta)} \mathbb{E}
    \left[\big(L_{\bm{\theta}}(\bm{X},Y)-L_{\bm{\theta}^\star}(\bm{X},Y)\big)^2\right] \leq \frac{M_{A'}^2 M_{X}^2+ M_X^2 M_Y}{C_L} 2^{j+1} \delta,
\end{align*}
for any $j \geq 1$.
\end{lemma}

\noindent\textbf{Proof of Lemma \ref{lem:working_glm_quadratic}.} We first note some facts that $\mathcal{X}$ is compact and $\bm{\Theta}$ is convex and compact. Define the population risk
$$
R(\bm{\theta})
=
\mathbb{E}[L_{\bm{\theta}}(\bm{X},Y)].
$$
Since $A(\cdot)$ is twice continuously differentiable and $(\bm{X},Y)$
satisfies Assumption \ref{Ass:SubGaussian},
$R(\bm{\theta})$ is twice continuously differentiable on $\bm{\Theta}$. We first compute the gradient and Hessian of $R(\bm{\theta})$:
$$
\nabla R(\bm{\theta})
=
\mathbb{E} \left[
\big(A'(\bm{X}^\top \bm{\theta})-Y\big)\bm{X}
\right],
\qquad
\nabla^2 R(\bm{\theta})
=
\mathbb{E} \left[
A''(\bm{X}^\top \bm{\theta}) \bm{X}\bm{X}^\top
\right].
$$
By definition of the pseudo-true parameter $\nabla R(\bm{\theta}^\star)=\bm{0}$.

For $\bm{\theta}\in\bm{\Theta}$, by a second-order Taylor expansion of $R(\bm{\theta})$ around $\bm{\theta}^\star$, there exists $\tilde{\bm{\theta}}$ on the line segment between $\bm{\theta}$ and $\bm{\theta}^\star$ such that
$$
R(\bm{\theta})-R(\bm{\theta}^\star)
=
\frac12
(\bm{\theta}-\bm{\theta}^\star)^\top
\nabla^2 R(\tilde{\bm{\theta}})
(\bm{\theta}-\bm{\theta}^\star).
$$
Since $\bm{\Theta}$ is assumed to be convex, it holds that $\tilde{\bm{\theta}} \in \bm{\Theta}$. Note that for any $\bm{\theta} \in \bm{\Theta}$ and $\bm{X}\in \mathcal{X}$, we have
\begin{align*}
 M_1 \leq   A''(\bm{X}^\top \bm{\theta}) \leq M_2,
\end{align*}
for some positive constants $M_1$ and $M_2$. Therefore, we obtain
$$
R(\bm{\theta})-R(\bm{\theta}^\star)
\ge
\Lambda_{\text{min}}(\mathbb{E}(\bm{X}\bm{X}^\top))
\frac{ \inf \limits_{\bm{X}\in \mathcal{X},\bm{\theta} \in \bm{\Theta}} A''(\bm{X}^\top \bm{\theta})}{2}
\|\bm{\theta}-\bm{\theta}^\star\|_2^2=
\frac{M_1 \kappa}{2}\|\bm{\theta}-\bm{\theta}^\star\|_2^2.
$$
Applying a similar argument, we have
$$
R(\bm{\theta})-R(\bm{\theta}^\star)
\le
\frac{M_2 \Lambda_{\text{max}}(\mathbb{E}(\bm{X}\bm{X}^\top))}{2}
\|\bm{\theta}-\bm{\theta}^\star\|_2^2.
$$
Next, we proceed to bound $\mathbb{E}\left[V_{\bm{\theta}}^2(\bm{X},Y)\right]$. Note that
\begin{align*}
V_{\bm{\theta}}^2(\bm{X},Y)=&
  \left(A(\bm{X}^\top \bm{\theta}^\star)-  A(\bm{X}^\top \bm{\theta}) + Y \bm{X}^\top [\bm{\theta}^\star-\bm{\theta}]
  \right)^2  \\
  \leq & 2[A'(\bm{X}^\top\widetilde{\bm{\theta}})]^2 \Vert\bm{X} \Vert_2^2 \Vert\bm{\theta}^\star-\bm{\theta}\Vert_2^2 +2 |Y|^2 \Vert \bm{X}\Vert_2^2 \Vert\bm{\theta}^\star-\bm{\theta}\Vert_2^2 \\
  \leq &2M_{A'}^2 M_{X}^2  \Vert\bm{\theta}^\star-\bm{\theta}\Vert_2^2 +2M_X^2 \Vert\bm{\theta}^\star-\bm{\theta}\Vert_2^2 Y^2.
\end{align*}
Therefore, we have
\begin{align*}
    \mathbb{E}\left[V_{\bm{\theta}}^2(\bm{X},Y)\right] \leq &
   2 M_{A'}^2 M_{X}^2  \Vert\bm{\theta}^\star-\bm{\theta}\Vert_2^2 +2M_X^2 \Vert\bm{\theta}^\star-\bm{\theta}\Vert_2^2 \mathbb{E}[Y^2] \\
    = &2
    \left(M_{A'}^2 M_{X}^2+ M_X^2 M_Y\right)\Vert\bm{\theta}^\star-\bm{\theta}\Vert_2^2.
\end{align*}
It then follows that
\begin{align*}
    \sup_{\bm{\theta} \in \bm{\Theta}_j(\delta)} \mathbb{E}\left[V_{\bm{\theta}}^2(\bm{X},Y)\right]\leq  &
    2\sup_{\bm{\theta} \in \bm{\Theta}_j(\delta)} \left(M_{A'}^2 M_{X}^2+ M_X^2 M_Y\right)\Vert\bm{\theta}^\star-\bm{\theta}\Vert_2^2 \\
    \leq &2\sup_{\bm{\theta} \in \bm{\Theta}_j(\delta)} \left(M_{A'}^2 M_{X}^2+ M_X^2 M_Y\right) \frac{\mathbb{E}[L_{\bm{\theta}}(\bm{X},Y)]
-
\mathbb{E}[L_{\bm{\theta}^\star}(\bm{X},Y)]}{C_L}  \\
\leq & \frac{M_{A'}^2 M_{X}^2+ M_X^2 M_Y}{C_L} 2^{j+1} \delta.
\end{align*}
This completes the proof. \hfill {\color{Red} $\blacksquare$} \\
\vspace{5mm}

\begin{lemma}
\label{lem:glm_misspecified_tail}
Let $\{(\bm x_i,y_i)\}_{i=1}^n$ be i.i.d.\ observations from some distribution satisfying Assumptions \ref{Ass:SubGaussian} and \ref{ass:PD}. Consider the quasi-negative log-likelihood under the working canonical generalized linear model
$$
L_n(\bm\theta)
\triangleq
\frac{1}{n}\sum_{i=1}^nL_{\bm{\theta}}(\bm{x}_i,y_i)=
\frac1n\sum_{i=1}^n
\big\{A(\bm x_i^\top\bm\theta)-y_i\bm x_i^\top\bm\theta\big\}.
$$
Let $\widehat{\bm\theta}=\arg\min_{\bm\theta}L_n(\bm\theta)$. Then for all $t \gtrsim \sqrt{\frac{p \log n}{n}}$,
$$
\mathbb P \left(
\|\widehat{\bm\theta}-\bm\theta^\star\|_2 \ge t
\right)
\lesssim \exp \left(
- C nt^2
\right),
$$
for some positive constant $C$
\end{lemma}

\noindent
\textbf{Proof of Lemma \ref{lem:glm_misspecified_tail}.} We first define $L_{\bm{\theta}}(\bm{X},Y) = A(\bm{X}^\top \bm{\theta})-Y\bm{X}^\top \bm{\theta}$ and
\begin{align*}
    \Phi(\widehat{\bm{\theta}}, \bm{\theta}^\star) 
    &= \mathbb{E}\big[L_{\widehat{\bm{\theta}}}(\bm{X},Y)\big] 
       - \mathbb{E}\big[L_{\bm{\theta}^\star}(\bm{X},Y)\big],
\end{align*}
where $\bm{\theta}^\star = \arg\min_{\bm{\theta} \in \bm{\Theta}} \mathbb{E}[L_{\bm{\theta}}(\bm{X},Y)]$ denotes the optimal parameter. 

By Lemma \ref{lem:working_glm_quadratic}, we have $\Phi(\widehat{\bm{\theta}}, \bm{\theta}^\star)  \geq C_L \Vert  \widehat{\bm{\theta}} -  \bm{\theta}^\star\Vert_2^2$. Therefore, we have
\begin{align*}
    \mathbb{P}\left(
\Vert  \widehat{\bm{\theta}} -  \bm{\theta}^\star\Vert_2 \geq t
    \right) \leq \mathbb{P}\left(\Phi(\widehat{\bm{\theta}}, \bm{\theta}^\star)
    \geq C_L t^2\right).
\end{align*}
Next, we set $\delta = C_L t^2$ and define the event
\begin{align*}
    \mathcal{E}(\delta) = \big\{ \Phi(\widehat{\bm{\theta}}, \bm{\theta}^\star) \ge \delta \big\}.
\end{align*}
Let $\bm{\Theta}(\delta) = \big\{ \bm{\theta} \in \bm{\Theta} : \Phi(\bm{\theta}, \bm{\theta}^\star) \ge \delta \big\}$. Then, we have
\begin{align}
\label{Main:Bound}
    \mathbb{P}\big(\mathcal{E}(\delta)\big) 
    &\le 
    \mathbb{P}\Bigg(
    \sup_{\bm{\theta} \in \bm{\Theta}(\delta)}
   \left( \frac{1}{n}\sum_{i=1}^n L_{\bm{\theta}^\star}(\bm{x}_i, y_i) 
    - \frac{1}{n}\sum_{i=1}^n L_{\bm{\theta}}(\bm{x}_i, y_i)\right) \ge 0
    \Bigg).
\end{align}
This inequality follows from the definition of the empirical minimizer $\widehat{\bm{\theta}}$. If $\widehat{\bm{\theta}} \in \bm{\Theta}(\delta)$, then
\begin{align*}
    \sup_{\bm{\theta} \in \bm{\Theta}(\delta)}
    \left\{\frac{1}{n}\sum_{i=1}^n L_{\bm{\theta}^\star}(\bm{x}_i, y_i) 
    - \frac{1}{n}\sum_{i=1}^n L_{\bm{\theta}}(\bm{x}_i, y_i)\right\} \ge 0,
\end{align*}
since $\frac{1}{n}\sum_{i=1}^n L_{\widehat{\bm{\theta}}}(\bm{x}_i, y_i) \le \frac{1}{n}\sum_{i=1}^n L_{\bm{\theta}}(\bm{x}_i, y_i)$ for all $\bm{\theta} \in \bm{\Theta}$.

\vspace{5mm}

\noindent\textbf{Step 1: Decomposition.} Next, we consider an equivalent representation of $\bm{\Theta}(\delta)$.
\begin{align*}
    \bm{\Theta}(\delta) = \cup_{j=1}^{\infty}
    \bm{\Theta}_j(\delta) = \cup_{j=1}^{\infty}
    \left\{\bm{\theta} \in \bm{\Theta}:
2^{j-1}\delta \leq \Phi(\bm{\theta}, \bm{\theta}^\star) \leq 2^{j}\delta
    \right\}.
\end{align*}
With this, (\ref{Main:Bound}) can be bounded as
\begin{align*}
     &  \mathbb{P}\Bigg(
    \sup_{\bm{\theta} \in \bm{\Theta}(\delta)}
   \left( \frac{1}{n}\sum_{i=1}^n L_{\bm{\theta}^\star}(\bm{x}_i, y_i) 
    - \frac{1}{n}\sum_{i=1}^n L_{\bm{\theta}}(\bm{x}_i, y_i)\right) \ge 0
    \Bigg) \\
    =&
    \mathbb{P}\Bigg(
    \sup_{\bm{\theta} \in \cup_{j=1}^{\infty}
    \bm{\Theta}_j(\delta)}
   \left( \frac{1}{n}\sum_{i=1}^n L_{\bm{\theta}^\star}(\bm{x}_i, y_i) 
    - \frac{1}{n}\sum_{i=1}^n L_{\bm{\theta}}(\bm{x}_i, y_i)\right) \ge 0
    \Bigg) \\
    \leq & \sum_{j=1}^{\infty}
    \mathbb{P}\Bigg(
    \sup_{\bm{\theta} \in 
    \bm{\Theta}_j(\delta)}
   \left( \frac{1}{n}\sum_{i=1}^n L_{\bm{\theta}^\star}(\bm{x}_i, y_i) 
    - \frac{1}{n}\sum_{i=1}^n L_{\bm{\theta}}(\bm{x}_i, y_i)\right) \ge 0
    \Bigg) \\
    \leq &
    \sum_{j=1}^{\infty}
    \mathbb{P}\Bigg(
    \sup_{\bm{\theta} \in 
    \bm{\Theta}_j(\delta)}
   \left( \frac{1}{n}\sum_{i=1}^n L_{\bm{\theta}^\star}(\bm{x}_i, y_i)-
   \mathbb{E}[L_{\bm{\theta}^\star}(\bm{X}, Y)]
    - \frac{1}{n}\sum_{i=1}^n L_{\bm{\theta}}(\bm{x}_i, y_i)+
    \mathbb{E}[L_{\bm{\theta}}(\bm{X}, Y)]\right)\ge 2^{j-1}\delta
    \Bigg) \\
    \triangleq & \sum_{j=1}^{\infty} P(j).
\end{align*}
Next, we turn to bound $P(j)$ for $j \geq 1$. For notational convenience, we define
\begin{align*}
    V_{\bm{\theta}}(\bm{X},Y) = L_{\bm{\theta}^\star}(\bm{X}, Y) - L_{\bm{\theta}}(\bm{X}, Y)  \mbox{ and } 
    \mathbb{E}[V_{\bm{\theta}}(\bm{X},Y)] =  \mathbb{E}[L_{\bm{\theta}^\star}(\bm{X}, Y)] -  \mathbb{E}[L_{\bm{\theta}}(\bm{X}, Y)].
\end{align*}
Then, $P(j)$ can be bounded as
\begin{align*}
    P(j) =& \mathbb{P}\left\{
    \sup_{\bm{\theta} \in 
    \bm{\Theta}_j(\delta)}
   \left( \frac{1}{n}\sum_{i=1}^n V_{\bm{\theta}}(\bm{x}_i,y_i)-\mathbb{E}[V_{\bm{\theta}}(\bm{X},Y)]\right)\ge 2^{j-1}\delta
    \right\} \\
    \leq &\mathbb{P}\left\{
    \sup_{\bm{\theta} \in 
    \bm{\Theta}_j(\delta)}
   \left| \frac{1}{n}\sum_{i=1}^n V_{\bm{\theta}}(\bm{x}_i,y_i)-\mathbb{E}[V_{\bm{\theta}}(\bm{X},Y)]\right|\ge 2^{j-1}\delta
    \right\} \\
    = &
    \mathbb{P}\left\{
    \sup_{\bm{\theta} \in 
    \bm{\Theta}_j(\delta)}
   \left| \frac{1}{n}\sum_{i=1}^n V_{\bm{\theta}}(\bm{x}_i,y_i)-\mathbb{E}[V_{\bm{\theta}}(\bm{X},Y)]\right|- \frac{3}{2}\mathbb{E}[\mathcal{V}_j]\ge 2^{j-1}\delta - \frac{3}{2}
   \mathbb{E}[\mathcal{V}_j]
    \right\},
\end{align*}
where $\mathbb{E}[\mathcal{V}_j]$ is defined as
\begin{align*}
    \mathbb{E}[\mathcal{V}_j] = \mathbb{E}
    \left\{
\sup_{\bm{\theta} \in 
    \bm{\Theta}_j(\delta)}
   \left| \frac{1}{n}\sum_{i=1}^n V_{\bm{\theta}}(\bm{x}_i,y_i)-\mathbb{E}[V_{\bm{\theta}}(\bm{X},Y)]\right|
    \right\}.
\end{align*}
Here, the expectation of $\mathbb{E}[\mathcal{V}_j]$ is taken with respect to the randomness of $\mathcal{D}_L$.

By Lemma \ref{Lemma:BoundedExp}, it holds for any $j \geq 1$ that $\frac{3}{2}\mathbb{E}[\mathcal{V}_j] \leq \frac{\delta}{2}$. With this, $P(j)$ can be further bounded as
\begin{align*}
    P(j) \leq
    \mathbb{P}\left\{
    \sup_{\bm{\theta} \in 
    \bm{\Theta}_j(\delta)}
   \left| \frac{1}{n}\sum_{i=1}^n V_{\bm{\theta}}(\bm{x}_i,y_i)-\mathbb{E}[V_{\bm{\theta}}(\bm{X},Y)]\right|- \frac{3}{2}\mathbb{E}[\mathcal{V}_j]\ge 2^{j-2}\delta 
    \right\}.
\end{align*}
where $\delta \gtrsim \frac{p \log n}{n}$, which then indicates $t \gtrsim \sqrt{\frac{p \log n}{n}}$.

\noindent\textbf{Step 2: Finite Orlicz Norm.} In this step, we turn to show that
\begin{align*}
    \sup_{\bm{\theta} \in \bm{\Theta}_j(\delta)}
    \Big|V_{\bm{\theta}}(\bm{X},Y) - \mathbb{E}[V_{\bm{\theta}}(\bm{X},Y)]  \Big|
\end{align*}
has a finite $\psi_2$-Orlicz norm. According to the definition of $\psi_2$-Orlicz Norm, it suffices to prove that there exists $\lambda$ such that
\begin{align}
\label{psi_norm}
    \mathbb{E}
    \left\{
    \exp \left(
\frac{\Big[\sup_{\bm{\theta} \in \bm{\Theta}_j(\delta)}
    \Big|V_{\bm{\theta}}(\bm{X},Y) - \mathbb{E}[V_{\bm{\theta}}(\bm{X},Y)]  \Big|\Big]^2}{\lambda^2}
    \right)
    \right\} \leq 2.
\end{align}
Note that
\begin{align*}
    \Big[\sup_{\bm{\theta} \in \bm{\Theta}_j(\delta)}
    \Big|V_{\bm{\theta}}(\bm{X},Y) - \mathbb{E}[V_{\bm{\theta}}(\bm{X},Y)]  \big|\Big]^2 = & \sup_{\bm{\theta} \in \bm{\Theta}_j(\delta)}
    \big(V_{\bm{\theta}}(\bm{X},Y) - \mathbb{E}[V_{\bm{\theta}}(\bm{X},Y)]  \big)^2 \\
    \leq & \sup_{\bm{\theta} \in \bm{\Theta}_j(\delta)} 2V^2_{\bm{\theta}}(\bm{X},Y) + \sup_{\bm{\theta} \in \bm{\Theta}_j(\delta)}2\big[\mathbb{E}[V_{\bm{\theta}}(\bm{X},Y)]\big]^2.
\end{align*}
Therefore, the left-hand side of (\ref{psi_norm}) can be further bounded as
\begin{align}
\label{Psi_2_obj}
   &  \mathbb{E}
    \left[
    \exp \left(
\frac{\Big[\sup_{\bm{\theta} \in \bm{\Theta}_j(\delta)}
    \Big|V_{\bm{\theta}}(\bm{X},Y) - \mathbb{E}[V_{\bm{\theta}}(\bm{X},Y)]  \Big|\Big]^2}{\lambda^2}
    \right)
    \right] \notag \\
    \leq & \mathbb{E}
    \left[  \exp\left(\sup_{\bm{\theta} \in \bm{\Theta}_j(\delta)} 2V^2_{\bm{\theta}}(\bm{X},Y)/\lambda^2\right)\right] \cdot \exp\left(\sup_{\bm{\theta} \in \bm{\Theta}_j(\delta)} 2\big[\mathbb{E}[V_{\bm{\theta}}(\bm{X},Y)]\big]^2/\lambda^2\right) \notag\\
    \leq & \mathbb{E}
    \left[  \exp\left(2\sup_{\bm{\theta} \in \bm{\Theta}_j(\delta)} V^2_{\bm{\theta}}(\bm{X},Y)/\lambda^2\right)\right] 
\cdot \exp\left(2\sup_{\bm{\theta} \in \bm{\Theta}_j(\delta)} \mathbb{E}[V^2_{\bm{\theta}}(\bm{X},Y)]/\lambda^2\right)\notag \\
\leq & \mathbb{E}
    \left[  \exp\left(2\sup_{\bm{\theta} \in \bm{\Theta}_j(\delta)} V^2_{\bm{\theta}}(\bm{X},Y)/\lambda^2\right)\right] 
\cdot \exp\left( \mathbb{E}\left[2\sup_{\bm{\theta} \in \bm{\Theta}_j(\delta)}V^2_{\bm{\theta}}(\bm{X},Y)/\lambda^2\right]\right) \notag\\
\leq & \left\{
\mathbb{E}
    \left[  \exp\left(2\sup_{\bm{\theta} \in \bm{\Theta}_j(\delta)} V^2_{\bm{\theta}}(\bm{X},Y)/\lambda^2\right)\right] \right\}^2,
\end{align}
where the last inequality follows from Jensen's inequality.

From the above analysis, it suffices to show that \eqref{Psi_2_obj} is finite, which in turn implies that \eqref{psi_norm} holds. In what follows, we establish a stronger result by showing that there exists $\lambda>0$ such that
\begin{align*}
        \mathbb{E}
    \left[
    \exp \left(
\frac{\sup_{\bm{\theta} \in \bm{\Theta}} |V_{\bm{\theta}}(\bm{X},Y)|^2}{\lambda^2}
    \right)
    \right] \leq 2.
\end{align*}
This will automatically imply (\ref{psi_norm}) for any $j \geq 1$. Recall that under the canonical working generalized linear model,
$$
L_{\bm \theta}(\bm X,Y)
=
A(\bm X^\top \bm \theta)-Y \bm X^\top \bm \theta .
$$
Hence,
$$
V_{\bm \theta}(\bm X,Y)
=
A(\bm X^\top \bm \theta^\star)-A(\bm X^\top \bm \theta)
-
Y \bm X^\top(\bm \theta^\star-\bm \theta).
$$
By the mean value theorem, there exists $\widetilde{\bm\theta}$ on the line segment between  $\bm\theta$ and $\bm\theta^\star$ such that
$$
V_{\bm \theta}(\bm X,Y)
=
\big(A'(\bm X^\top \widetilde{\bm\theta})-Y\big)
\bm X^\top(\bm\theta^\star-\bm\theta).
$$
Since $\|\bm X\|_2\le M_X$ for any $\bm{X}$ and $\sup_{\bm\theta\in\bm\Theta}\|\bm\theta\|_2\le M_\Theta$, we have
$$
|\bm X^\top(\bm\theta-\bm\theta^\star)|
\le
\|\bm X\|_2\|\bm\theta-\bm\theta^\star\|_2
\le
2M_XM_\Theta.
$$
Consequently,
\begin{align*}
\sup_{\bm\theta\in\bm\Theta_j(\delta)}
V_{\bm\theta}^2(\bm X,Y)
&\le
4M_X^2M_\Theta^2
\sup_{\bm\theta\in\bm\Theta_j(\delta)}
\big(A'(\bm X^\top\widetilde{\bm\theta})-Y\big)^2.
\end{align*}
We further decompose
$$
A'(\bm X^\top\widetilde{\bm\theta})-Y
=
\big(A'(\bm X^\top\widetilde{\bm\theta})-m(\bm X)\big)
+
\big(m(\bm X)-Y\big).
$$
Therefore,
\begin{align*}
\exp\Big(\frac{
\sup_{\bm\theta\in\bm\Theta}
V_{\bm\theta}^2(\bm X,Y)}{\lambda^2}
\Big)
&\le
\exp\Big(
8M_X^2M_\Theta^2
\sup_{\bm\theta\in\bm\Theta}
\big(A'(\bm X^\top\widetilde{\bm\theta})-m(\bm X)\big)^2/\lambda^2
\Big) \\
&\quad \times
\exp\Big(
8M_X^2M_\Theta^2
(Y-m(\bm X))^2/\lambda^2
\Big).
\end{align*}
Since $\|\bm X\|_2\le M_X$ and $\sup_{\bm\theta\in\bm\Theta}\|\bm\theta\|_2\le M_\Theta$, we have
$|\bm X^\top\widetilde{\bm\theta}|\le M_XM_\Theta$. By continuity of $A'$, it is uniformly bounded on $[-M_XM_\Theta,M_XM_\Theta]$. Hence,
$$
\sup_{\bm\theta\in\bm\Theta}
|A'(\bm X^\top\widetilde{\bm\theta})-m(\bm X)|
\le C < \infty,
$$
which implies
$$
\sup_{\bm{X} \in \mathcal{X}}
\exp\Big(
8M_X^2M_\Theta^2
\sup_{\bm\theta\in\bm\Theta}
\big(A'(\bm X^\top\widetilde{\bm\theta})-m(\bm X)\big)^2/\lambda^2
\Big)
<\infty,
$$
for any $\lambda>0$. Second, by Assumption~\ref{Ass:SubGaussian}, $Y-m(\bm X)$ is uniform conditionally sub-Gaussian, therefore there exists $\lambda_0$ such that 
$$
\mathbb E\Big[
\exp\Big(
8M_X^2M_\Theta^2
(Y-m(\bm X))^2/\lambda_0^2
\Big)
\Big] \leq 2.
$$
Combining the above bounds, we can choose a large $\lambda$ such that
$$
\mathbb E\left[
\exp\left(
\frac{\sup_{\bm\theta\in\bm\Theta}
V_{\bm\theta}^2(\bm X,Y)}{\lambda^2}
\right)-1
\right]
<1,
$$
which implies that $\sup_{\bm\theta\in\bm\Theta}|V_{\bm\theta}(\bm X,Y)-\mathbb{E}[V_{\bm\theta}(\bm X,Y)]|$ has a finite $\psi_2$-Orlicz norm. We define
\begin{align*}
  V_{\psi_2} \triangleq   \left\Vert \sup_{\bm\theta\in\bm\Theta} \big|V_{\bm\theta}(\bm X,Y) -\mathbb{E}[V_{\bm\theta}(\bm X,Y)]\big|\right\Vert_{\psi_2}.
\end{align*}
By Lemma \ref{lem:max_subgaussian}, we have
\begin{align*}
\Big\Vert
    \max_{1 \leq i \leq n}
    \sup_{\bm\theta\in\bm\Theta} \big|V_{\bm\theta}(\bm x_i,y_i) -\mathbb{E}[V_{\bm\theta}(\bm x_i,y_i)]\big| \Big\Vert_{\psi_2}
    \leq 4V_{\psi_2} \sqrt{\log(n+1)}
\end{align*}

By Theorem 4 of \citet{adamczak2008tail} with $(\eta,\delta)=(1/2,1/2)$, we have
\begin{align*}
    P(j) \leq &
    \mathbb{P}\left\{
    \sup_{\bm{\theta} \in 
    \bm{\Theta}_j(\delta)}
   \left| \frac{1}{n}\sum_{i=1}^n V_{\bm{\theta}}(\bm{x}_i,y_i)-\mathbb{E}[V_{\bm{\theta}}(\bm{X},Y)]\right|- \frac{3}{2}\mathbb{E}[\mathcal{V}_j]\ge 2^{j-2}\delta 
    \right\} \\
    \leq & \exp\left(
-\frac{n2^{2j-4}\delta^2}{3\mathbb{E}\left[\sup_{\bm\theta\in\bm\Theta}
V_{\bm\theta}^2(\bm X,Y)\right]}
    \right)+3\exp\left(
-\frac{n^22^{2j-4}\delta^2}{16C\log(n+1)V_{\psi_2}^2}
    \right) \\
    \leq &
    \exp\left(
-\frac{n2^{2j-4}\delta^2}{3C' 2^{j+1}\delta}
    \right)+3\exp\left(
-\frac{n^22^{2j-4}\delta^2}{16C\log(n+1)V_{\psi_2}^2}
    \right) \\
    \lesssim &\exp(-C n \cdot2^{j}\delta), 
\end{align*}
where the last inequality follows from the fact $\frac{n^2}{\log(n+1)}\gg n$ for large $n$.

\vspace{5mm}

\noindent\textbf{Step 3: Summing Up.} To sum up, we have
\begin{align*}
    \mathbb{P}(\mathcal{E}(\delta)) \leq & \sum_{j=1}^{\infty} P(j) \leq  \sum_{j=1}^{\infty}C_1\exp(-C_2 n \cdot2^{j}\delta) \\
    \leq & \sum_{j=1}^{\infty}C_1\exp(-C_2 n \cdot j\delta) \\
    \leq & \frac{\exp(-C_2 n\delta)}{1-\exp(-C_2 n\delta)}.
\end{align*}
This implies that
\begin{align*}
      \mathbb{P}\left(
\Vert  \widehat{\bm{\theta}} -  \bm{\theta}^\star\Vert_2 \geq t
    \right) \leq \frac{\exp(-C_2 C_L nt^2)}{1-\exp(-C_2 C_L nt^2)} 
    \lesssim \exp(-C n t^2),
\end{align*}
for some large $n$ such that $\exp(-C_2 C_L nt^2)\rightarrow 0$. This completes the proof. \hfill ${\color{red}\blacksquare}$ \\

\vspace{5mm}

\begin{lemma}
    \label{Lemma:BoundedExp}
Let $V_{\bm{\theta}}(\bm X,Y)=L_{\bm{\theta}^\star}(\bm X,Y)-L_{\bm{\theta}}(\bm X,Y)$ with $L_{\bm{\theta}}(\bm X,Y)=A(\bm{X}^\top \bm{\theta})-Y\bm{X}^\top \bm{\theta}$. For each $j \ge 1$, define
$$
\mathcal V_j
=
\sup_{\bm{\theta} \in \bm{\Theta}_j(\delta)}
\left(
\frac{1}{n}\sum_{i=1}^n V_{\bm{\theta}}(\bm x_i,y_i)
-
\mathbb E[V_{\bm{\theta}}(\bm X,Y)]
\right),
$$
where $\bm{\Theta}_j(\delta)=\{\bm{\theta} \in \bm{\Theta} :
2^{j-1}\delta\le\Phi(\bm{\theta},\bm{\theta}^\star)\le 2^{j}\delta\}$. Under Assumptions \ref{Ass:SubGaussian} and \ref{ass:PD}, if $\delta \gtrsim p \log n / n$, then for all $j \ge 1$,
$$
\frac{3}{2}\,\mathbb E[\mathcal V_j]
\le
2^{j-1}\delta.
$$

\end{lemma}

\noindent\textbf{Proof of Lemma \ref{Lemma:BoundedExp}.} The proof of this lemma is structured into the following steps.

\noindent\textbf{Step 0: $\mathcal V_j$ is $L^2$-integrable.} Note that
\begin{align*}
    \mathcal{V}_j^2 = & 
    \left[
    \sup_{\bm{\theta} \in 
    \bm{\Theta}_j(\delta)}
   \left( \frac{1}{n}\sum_{i=1}^n V_{\bm{\theta}}(\bm{x}_i,y_i)-\mathbb{E}[V_{\bm{\theta}}(\bm{X},Y)]\right)\right]^2 \\
   \leq &\sup_{\bm{\theta} \in 
    \bm{\Theta}_j(\delta)}
   \left( \frac{1}{n}\sum_{i=1}^n V_{\bm{\theta}}(\bm{x}_i,y_i)-\mathbb{E}[V_{\bm{\theta}}(\bm{X},Y)]\right)^2 \\
   \leq &
   \sup_{\bm{\theta} \in 
    \bm{\Theta}_j(\delta)}
   \frac{1}{n}\sum_{i=1}^n \left\{ V_{\bm{\theta}}(\bm{x}_i,y_i)-\mathbb{E}[V_{\bm{\theta}}(\bm{X},Y)]\right\}^2 \\
   \leq &
   \sup_{\bm{\theta} \in 
    \bm{\Theta}_j(\delta)}
   \frac{2}{n}\sum_{i=1}^n V^2_{\bm{\theta}}(\bm{x}_i,y_i)+2\sup_{\bm{\theta} \in 
    \bm{\Theta}_j(\delta)}\mathbb{E}[V^2_{\bm{\theta}}(\bm{X},Y)].
\end{align*}
Therefore, we have
\begin{align*}
    \mathbb{E}[\mathcal{V}_j^2] \leq 2 \mathbb{E}
    \left[ \sup_{\bm{\theta} \in 
    \bm{\Theta}_j(\delta)}V^2_{\bm{\theta}}(\bm{X},Y)
    \right]+2\sup_{\bm{\theta} \in 
    \bm{\Theta}_j(\delta)}\mathbb{E}[V^2_{\bm{\theta}}(\bm{X},Y)] \leq 4 \mathbb{E}
    \left[ \sup_{\bm{\theta} \in 
    \bm{\Theta}_j(\delta)}V^2_{\bm{\theta}}(\bm{X},Y)
    \right].
\end{align*}
Next, we show that there exists an envelope function for $V_{\bm{\theta}}(\bm X,Y)$ which is $L^2$-integrable.
\begin{align*}
    |V_{\bm{\theta}}(\bm X,Y)| \leq [A'(\bm{X}^\top \widetilde{\bm{\theta}})+|Y|\Vert  \bm{X}\Vert_2]\cdot\Vert \bm{\theta}^\star-\bm{\theta} \Vert_2\leq [M_{A'}+M_X|Y|] M_{\Theta},
\end{align*}
where $M_{A'} = \sup_{\bm{\theta} \in \bm{\Theta},\bm{X}\in \mathcal{X}}A'(\bm{X}^\top \bm{\theta})$ and $M_{\Theta} = \sup_{\bm{\theta} \in \bm{\Theta}}\Vert \bm{\theta}^\star-\bm{\theta} \Vert_2$. Since $Y$ is sub-Gaussian, there exists a constant $C_V>0$ such that
\begin{equation}
\label{eq:step0_pointwise_L2}
\mathbb{E}\left[
\sup_{\bm{\theta}\in\bm{\Theta}}
V_{\bm{\theta}}^2(\bm X,Y)\right] \leq \mathbb{E}\left\{[M_{A'}+M_X|Y|]^2 M_{\Theta}^2\right\}<C_V<\infty,
\end{equation}
where the last inequality follows from the fact that $Y$ is sub-Gaussian.

\vspace{5mm}

\noindent\textbf{Step 1: Bound $\mathbb{E}[\mathcal{V}_j]$ via Rademacher complexity.} We first bound $d_n^2(\bm\theta^{(1)},\bm\theta^{(2)})$ by $\Vert \bm\theta^{(1)}-\bm\theta^{(2)} \Vert_2$ with a large probability. For any fixed $(\bm x,y)$, we have
$$
V_{\bm\theta^{(1)}}(\bm x,y)-V_{\bm\theta^{(2)}}(\bm x,y)
=
A(\bm x^\top\bm\theta^{(1)})-A( \bm x^\top\bm\theta^{(2)})
-
y\,\bm x^\top(\bm\theta^{(1)}-\bm\theta^{(2)}).
$$
By the mean value theorem, there exists a parameter
$\bar{\bm\theta}$ on the line segment between
$\bm\theta^{(1)}$ and $\bm\theta^{(2)}$ such that
$$
A(\bm x^\top\bm\theta^{(1)})-A(\bm x^\top\bm\theta^{(2)})
=
A'(\bm x^\top\bar{\bm\theta})\,\bm x^\top(\bm\theta^{(1)}-\bm\theta^{(2)}).
$$
Therefore,
$$
V_{\bm\theta^{(1)}}(\bm x,y)-V_{\bm\theta^{(2)}}(\bm x,y)
=
\bigl(A'(\bm x^\top\bar{\bm\theta})-y\bigr)
\bm x^\top(\bm\theta^{(1)}-\bm\theta^{(2)}).
$$
Squaring and summing over $i=1,\dots,n$, we obtain
\begin{align*}
  d_n^2(\bm\theta^{(1)},\bm\theta^{(2)})
= &
\frac1n\sum_{i=1}^n
\bigl(A'(\bm x_i^\top\bar{\bm\theta})-y_i\bigr)^2
\bigl(\bm x_i^\top(\bm\theta^{(1)}-\bm\theta^{(2)})\bigr)^2 \\
\leq &
\frac1n\sum_{i=1}^n
\bigl(A'(\bm x_i^\top\bar{\bm\theta})-y_i\bigr)^2
\Vert \bm x_i \Vert_2^2 \cdot \Vert \bm\theta^{(1)}-\bm\theta^{(2)}\Vert_2^2 \\
\leq &M_X^2 \Vert \bm\theta^{(1)}-\bm\theta^{(2)}\Vert_2^2 \cdot
\frac1n\sum_{i=1}^n
\bigl(A'(\bm x_i^\top\bar{\bm\theta})-y_i\bigr)^2 \\
\leq & M_X^2 \Vert \bm\theta^{(1)}-\bm\theta^{(2)}\Vert_2^2
\cdot
\frac1n\sum_{i=1}^n
\left\{[A'(\bm x_i^\top\bar{\bm\theta})]^2+y_i^2\right\} \\
\leq &M_{A'}^2 M_X^2 \Vert \bm\theta^{(1)}-\bm\theta^{(2)}\Vert_2^2+M_X^2 \Vert \bm\theta^{(1)}-\bm\theta^{(2)}\Vert_2^2 \cdot \frac{1}{n}\sum_{i=1}^n y_i^2 \\
\leq &M_X^2 \cdot \Vert \bm\theta^{(1)}-\bm\theta^{(2)}\Vert_2^2 \cdot
\left[
M_{A'}^2  +\mathbb{E}(Y^2) +
\frac{1}{n}\sum_{i=1}^n (y_i^2 - \mathbb{E}(Y^2))
\right],
\end{align*}
where $\bar{\bm\theta}$ lies between $\bm\theta^{(1)}$ and $\bm\theta^{(2)}$. Since $Y$ is a sub-Gaussian random variable with a proxy variance $\sigma_Y^2$, it holds that $Y^2$ is sub-exponential. Therefore, it holds that
\begin{align*}
    \mathbb{P}\left(\Big|
\frac{1}{n}\sum_{i=1}^n (y_i^2 - \mathbb{E}(Y^2))\Big|
\geq t_0
    \right) \leq C_1 \exp(-C_2nt_0^2)
\end{align*}
for some sufficiently small $t_0>0$. With this, we define the event and the constant $C_3$ as
\begin{align*}
    \mathcal{E}_{t_0} = \left\{
\Big|\frac{1}{n}\sum_{i=1}^n (y_i^2 - \mathbb{E}(Y^2))\Big| < t_0
    \right\} \,\mbox{ and }\, C_3 = M_X^2 \cdot \left(
M_{A'}^2 + \mathbb{E}(Y^2) + t_0
    \right).
\end{align*}
The above analysis indicates that conditional on $\mathcal{E}_{t_0}$, we have
\begin{align*}
    d_n^2(\bm\theta^{(1)},\bm\theta^{(2)}) \leq C_3 \Vert \bm\theta^{(1)}-\bm\theta^{(2)}\Vert_2^2.
\end{align*}

Then we can bound $\mathbb{E}(\mathcal{V}_j)$ as below:
\begin{align*}
    \mathbb{E}(\mathcal{V}_j) = & \mathbb{E}(\mathcal{V}_j \cdot \bm{1}_{\mathcal{E}_{t_0}}(\mathcal{D}_L))+
     \mathbb{E}(\mathcal{V}_j \cdot \bm{1}_{\mathcal{E}_{t_0}^c}(\mathcal{D}_L)) \\
     \leq &\mathbb{E}(\mathcal{V}_j \cdot \bm{1}_{\mathcal{E}_{t_0}}(\mathcal{D}_L))+ \sqrt{\mathbb{E}(\mathcal{V}_j^2)} \cdot \sqrt{\mathbb{P}(\mathcal{E}_{t_0}^c)}
\end{align*}
Let $(\bm{x}_i', y_i')_{i=1}^n$ be an independent copy of the data, and let $(\tau_i)_{i=1}^n$ be i.i.d.\ Rademacher random variables. Then by the standard symmetrization argument \citep{koltchinskii2011oracle}, we have
\begin{align*}
 &\mathbb{E}_{\mathcal{D}_L}[\mathcal{V}_j \cdot \bm{1}_{\mathcal{E}_{t_0}}(\mathcal{D}_L)]  \\
= & \mathbb{E}_{\mathcal{D}_L}\left[\bm{1}_{\mathcal{E}_{t_0}}(\mathcal{D}_L) \cdot  \sup_{\bm{\theta} \in \bm{\Theta}_j(\delta)} 
\frac{1}{n} \sum_{i=1}^n \big( V_{\bm{\theta}}(\bm{x}_i,y_i) - \mathbb{E}[V_{\bm{\theta}}(\bm{X},Y)] \big) \right] \\
\le & \mathbb{E}_{\mathcal{D}_L,\mathcal{D}_L'}\left[ \bm{1}_{\mathcal{E}_{t_0}}(\mathcal{D}_L) \cdot \sup_{\bm{\theta} \in \bm{\Theta}_j(\delta)} 
\frac{1}{n} \sum_{i=1}^n \big( V_{\bm{\theta}}(\bm{x}_i,y_i) - V_{\bm{\theta}}(\bm{x}_i',y_i') \big) \right] \\
= & \mathbb{E}_{\mathcal{D}_L,\mathcal{D}_L'}\left[ \bm{1}_{\mathcal{E}_{t_0}}(\mathcal{D}_L) \cdot \sup_{\bm{\theta} \in \bm{\Theta}_j(\delta)} 
\frac{1}{n} \sum_{i=1}^n \tau_i \big( V_{\bm{\theta}}(\bm{x}_i,y_i) - V_{\bm{\theta}_0}(\bm{x}_i,y_i) + V_{\bm{\theta}_0}(\bm{x}_i',y_i') - V_{\bm{\theta}}(\bm{x}_i',y_i') \big) \right] \\
\le & \frac{2}{n} \mathbb{E}_{\mathcal{D}_L,\bm{\tau}}\left[ \bm{1}_{\mathcal{E}_{t_0}}(\mathcal{D}_L) \cdot \sup_{\bm{\theta} \in \bm{\Theta}_j(\delta)}
\left| \sum_{i=1}^n \tau_i \big( V_{\bm{\theta}}(\bm{x}_i,y_i) - V_{\bm{\theta}_0}(\bm{x}_i,y_i) \big) \right| \right],
\end{align*}
where $\bm{\theta}_0 \in \bm{\Theta}_j(\delta)$ is some fixed reference parameter and $\bm{\tau}=(\tau_1,\ldots,\tau_n)$.

Next, note that conditional on the data $\mathcal{D}_L$, $\frac{1}{\sqrt{n}}\sum_{i=1}^n \tau_i (V_{\bm{\theta}}(\bm{x}_i,y_i)-V_{\bm{\theta}_0}(\bm{x}_i,y_i))$ forms a sub-Gaussian process with respect to the semi-metric
$$
d^2(\bm{\theta}^{(1)}, \bm{\theta}^{(2)}) = \frac{1}{n} \sum_{i=1}^n \big( V_{\bm{\theta}^{(1)}}(\bm{x}_i, y_i) - V_{\bm{\theta}^{(2)}}(\bm{x}_i, y_i) \big)^2,
$$
for any $\bm{\theta}^{(1)}, \bm{\theta}^{(2)} \in \bm{\Theta}_j(\delta)$. Applying Theorem 3.1 of \cite{koltchinskii2011oracle}, we obtain
\begin{align*}
&\mathbb{E}_{\bm{\tau}}\left[ \bm{1}_{\mathcal{E}_{t_0}}(\mathcal{D}_L) \cdot \sup_{\bm{\theta} \in \bm{\Theta}_j(\delta)}\
\left| \frac{1}{\sqrt{n}} \sum_{i=1}^n \tau_i \big( V_{\bm{\theta}}(\bm{x}_i,y_i) - V_{\bm{\theta}_0}(\bm{x}_i,y_i) \big) \right| \right] \\
    \lesssim  &
\bm{1}_{\mathcal{E}_{t_0}}(\mathcal{D}_L) \cdot \int_0^{D(\bm{\Theta}_j(\delta))} H^{1/2}(\bm{\Theta}_j(\delta), d, \eta)  d\eta,
\end{align*}
where $D(\bm{\Theta}_j(\delta)) = \sup \limits_{\bm{\theta}^{(1)},\bm{\theta}^{(2)} \in \bm{\Theta}_j(\delta)} d(\bm{\theta}^{(1)}, \bm{\theta}^{(2)})$ is the diameter with respect to $d$, and $H(\bm{\Theta}_j(\delta), d, \eta)$ is the $\eta$-entropy of $\bm{\Theta}_j(\delta)$ under $d$. Specifically, let $\mathcal{Q}_{\eta} \subset \bm{\Theta}_j(\delta)$ be the minimal $\eta$-covering set of $\bm{\Theta}_j(\delta)$. For any $\bm{\theta} \in \bm{\Theta}_j(\delta)$, there exists $\bm{\theta}_1 \in \mathcal{Q}_{\eta}$ such that $d(\bm{\theta},\bm \theta_1) \leq \eta$. Here, $H^{1/2}(\bm{\Theta}_j(\delta), d, \eta) = \sqrt{\log |\mathcal{Q}_{\eta}|}$.

Note that conditional the event $\mathcal{E}_{t_0}$, we have
\begin{align*}
    d^2(\bm{\theta}^{(1)}, \bm{\theta}^{(2)}) \leq C_3 \cdot \Vert \bm{\theta}^{(1)} - \bm{\theta}^{(2)}\Vert_2^2 \leq 4C_3 M_{\Theta}^2
\end{align*}
This combined with Lemma \ref{lem:working_glm_quadratic} yields that
\begin{align*}
  \sup_{\bm\theta^{(1)},\bm\theta^{(2)} \in \bm{\Theta}_j(\delta)}  d_n^2(\bm\theta^{(1)},\bm\theta^{(2)})
  \leq & \sup_{\bm\theta^{(1)},\bm\theta^{(2)} \in \bm{\Theta}_j(\delta)} \left\{
 2C_3 \Vert \bm\theta^{(1)}-\bm{\theta}^\star\Vert_2^2 +2 C_3\Vert \bm\theta^{(2)}-\bm{\theta}^\star\Vert_2^2\right\} \\
 \leq &
  \frac{4C_3}{C_L} 2^{j}\delta \triangleq C_4 2^{j} \delta.
\end{align*}
This then implies that
\begin{align*}
&\mathbb{E}_{\bm{\tau}}\left[ \bm{1}_{\mathcal{E}_{t_0}}(\mathcal{D}_L) \cdot \sup_{\bm{\theta} \in \bm{\Theta}_j(\delta)}\
\left| \frac{1}{\sqrt{n}} \sum_{i=1}^n \tau_i \big( V_{\bm{\theta}}(\bm{x}_i,y_i) - V_{\bm{\theta}_0}(\bm{x}_i,y_i) \big) \right| \right]\\
\lesssim &
\bm{1}_{\mathcal{E}_{t_0}}(\mathcal{D}_L) \cdot \int_0^{\sqrt{C_4 2^j \delta}\wedge 2\sqrt{C_3}M_{\Theta}} H^{1/2}\left(\bm{\Theta}_j(\delta), \Vert \cdot  \Vert_2, \frac{\eta}{\sqrt{C_3}}\right)  d\eta.
\end{align*}
We further take the expectation with respect to $\mathcal{D}_L$ and then obtain
\begin{align*}
   \mathbb{E}_{\bm{\tau},\mathcal{D}_L}[\mathcal{V}_j \cdot\bm{1}_{\mathcal{E}_{t_0}}(\mathcal{D}_L) ]  
    \lesssim \frac{1}{\sqrt{n}} \cdot \int_0^{\sqrt{C_4 2^j \delta}\wedge 2\sqrt{C_3}M_{\Theta}} H^{1/2}\left(\bm{\Theta}_j(\delta), \Vert \cdot  \Vert_2, \frac{\eta}{\sqrt{C_3}}\right)  d\eta.
\end{align*}

\vspace{5mm}

\noindent\textbf{Step 2: Bound the Entropy of $\bm{\Theta}_j(\delta)$.} Since $\bm{\Theta}_j(\delta)$ is a subset of a Euclidean ball in $\mathbb{R}^d$, its covering number satisfies
$$
N \left(
\bm{\Theta}_j(\delta),
\|\cdot\|_2,
\varepsilon
\right)
\le N \left(
\bm{\Theta},
\|\cdot\|_2,
\varepsilon
\right) \leq 
\left(\frac{6 M_{\Theta}}{\varepsilon}\right)^p,
$$
which implies the entropy bound
\begin{equation}
\label{eq:entropy_bound}
H \left(
\bm{\Theta}_j(\delta),
\|\cdot\|_2,
\varepsilon
\right)
\le
p\log\left(\frac{6 M_{\Theta}}{\varepsilon}\right).
\end{equation}
With this, it then follows that
\begin{align}
\label{Last}
    \mathbb{E}_{\bm{\tau},\mathcal{D}_L}[\mathcal{V}_j \cdot\bm{1}_{\mathcal{E}_{t_0}}(\mathcal{D}_L) ]  
    \lesssim  & \sqrt{\frac{p}{n}}
    \int_0^{\sqrt{C_L^{-1} 2^j \delta}\wedge 2M_{\Theta}} 
    \sqrt{\log\left(\frac{6M_{\Theta}}{\eta }\right)} d\eta \notag \\
    \asymp &  6M_{\Theta}\sqrt{\frac{p}{n}}\int_0^{\frac{\sqrt{C_L^{-1}2^j \delta}}{6M_{\Theta}}\wedge \frac{1}{3}} 
    \sqrt{\log\left(\frac{1}{t}\right)} dt.
\end{align}
Next, we proceed to provide an upper bound for (\ref{Last}). For ease of notation, we denote that $A_1 =6M_{\Theta}\sqrt{\frac{p}{n}}$ and $A_2 =\frac{\sqrt{C_L^{-1}2^j \delta}}{6M_{\Theta}}$ . Then (\ref{Last}) becomes
\begin{align*}
A_1 \int_{0}^{A_2 \wedge 1/3} \sqrt{\log(1/t)}dt =
A_1 \int_{A_2^{-1} \vee 3}^{\infty} \frac{\sqrt{\log(s)}}{s^2}ds 
\leq 
\frac{A_1}{\sqrt{\log(A_2^{-1} \vee 3)}} \int_{A_2^{-1} \vee 3}^{\infty} \frac{\log(s)}{s^2}ds.
\end{align*}
By the fact that $\int_{a}^{\infty} \log(x)/x^2 dx = (\log(a)+1)/a$, we further have
\begin{align}
\label{Integral}
\frac{A_1}{\sqrt{\log(A_2^{-1} \vee 3)}} \int_{A_2^{-1} \vee 1}^{\infty} \frac{\log(s)}{s^2}ds 
=
\frac{A_1}{A_2^{-1} \vee 3}
\frac{\log(A_2^{-1} \vee 3)+1}{\sqrt{\log(A_2^{-1} \vee 3)}} 
\leq \frac{A_1 \sqrt{\log(A_2^{-1} \vee 3)}}{A_2^{-1} \vee 3},
\end{align}
where the last inequality follows from the fact that $1/x+x \leq 2x$ for $x \geq 1$. Substituting $A_1$ and $A_2$ into (\ref{Integral}) yields that
\begin{align*}
\frac{A_1 \log(A_2^{-1}\vee 3)}{A_2^{-1} \vee 3} \leq A_1 A_2 \log(A_2^{-1}\vee 3) =\sqrt{\frac{p}{n}} \cdot \sqrt{2^j \delta} \cdot \log\left(\max\left\{3,\frac{6M_{\Theta}}{\sqrt{2^jC_L^{-1} \delta}}\right\}\right).
\end{align*}

\vspace{5mm}

\noindent\textbf{Step 3: Bounding $\mathbb{E}[\mathcal{V}_j]$.} Recall from Step~1 that
\begin{align*}
\mathbb{E}[\mathcal{V}_j]
\lesssim & \left\{ \sqrt{\frac{p}{n}} \cdot \sqrt{2^j \delta} \cdot \log\left(\max\left\{3,\frac{6M_{\Theta}}{\sqrt{2^j \delta}}\right\}\right)    +
C_1 \exp(-C_2nt_0^2)\right\}.
\end{align*}
When $\delta \gtrsim \frac{p}{n}\log(n)$, we have
\begin{align*}
    \frac{3}{2} \mathbb{E}[\mathcal{V}_j] \leq 2^{j-1}\delta
\end{align*}
for any $j \geq 1$. This completes the proof. \hfill ${\color{Red} \blacksquare}$ \\
\vspace{5mm}

\bibliographystyle{authoryear}
\putbib[Ref]
\end{bibunit}

\end{document}